\documentclass[11pt, a4paper]{amsbook}
\usepackage{amsmath, amsthm,,amscd, amsfonts, amssymb, comment ,mathrsfs, textcomp, epsfig, a4wide, graphicx, IEEEtrantools, rotating, color, comment, tikz}

\numberwithin{equation}{chapter}

\usetikzlibrary{matrix}

 \RequirePackage{rotating}                   
    \def\HMt{%
       \setbox0=\hbox{$\widehat{\mathit{HM}}$}
       \setbox1=\hbox{$\mathit{HM}$}
       \dimen0=1.1\ht0
       \advance\dimen0 by 1.17\ht1
       \smash{\mskip2mu\raise\dimen0\rlap{%
          \begin{turn}{180}
              {$\widehat{\phantom{\mathit{HM}}}$}
           \end{turn}} \mskip-2mu    
                \mathit{HM}
    }{\vphantom{\widehat{\mathit{HM}}}}{}}
    
     \RequirePackage{rotating}                   
    \def\HSt{%
       \setbox0=\hbox{$\widehat{\mathit{HS}}$}
       \setbox1=\hbox{$\mathit{HS}$}
       \dimen0=1.1\ht0
       \advance\dimen0 by 1.17\ht1
       \smash{\mskip2mu\raise\dimen0\rlap{%
          \begin{turn}{180}
              {$\widehat{\phantom{\mathit{HS}}}$}
           \end{turn}} \mskip-2mu    
                \mathit{HS}
                    }{\vphantom{\widehat{\mathit{HS}}}}{}}

\newcommand{\C}{\mathbb{C}}
\newcommand{\R}{\mathbb{R}}
\newcommand{\Z}{\mathbb{Z}}
\newcommand{\F}{\mathfrak{F}}
\newcommand{\V}{\mathcal{V}}
\newcommand{\T}{\mathcal{T}}
\newcommand{\J}{\mathcal{J}}
\newcommand{\K}{\mathcal{K}}
\newcommand{\A}{\mathcal{A}}
\newcommand{\ztwo}{\mathbb{F}}
\newcommand{\spin}{\mathfrak{s}}
\newcommand{\Co}{\mathcal{C}}

\newcommand{\Bo}{\mathcal{B}}
\newcommand{\G}{\mathcal{G}}

\newcommand{\Pin}{\mathrm{Pin}(2)}
\newcommand{\CSD}{\mathcal{L}}
\newcommand{\CSd}{\rlap{$-$}\mathcal{L}}
\newcommand{\Cs}{\mathcal{C}^{\sigma}}
\newcommand{\Bs}{\mathcal{B}^{\sigma}}

\newcommand{\Ct}{\mathcal{C}^{\tau}}
\newcommand{\Bt}{\mathcal{B}^{\tau}}

\newcommand{\N}{\mathcal{N}}
\newcommand{\Sl}{\mathcal{S}}
\newcommand{\Rin}{\mathcal{R}}
\newcommand{\q}{\mathfrak{q}}

\newcommand{\acr}{\mathfrak{a}}
\newcommand{\bcr}{\mathfrak{b}}
\newcommand{\Hess}{\text{Hess}}
\newcommand{\HSb}{\overline{\mathit{HS}}}
\newcommand{\HSf}{\widehat{\mathit{HS}}}
\newcommand{\HMb}{\overline{\mathit{HM}}}
\newcommand{\HMf}{\widehat{\mathit{HM}}}
\newcommand{\gr}{\mathrm{gr}}
\newcommand{\ev}{\mathrm{ev}}
\newcommand{\Cr}{\mathfrak{C}}
\newcommand{\Lr}{\mathfrak{C}}

\theoremstyle{plain}
\newtheorem{teor}{Theorem}[section]
\newtheorem{prop}[teor]{Proposition}
\newtheorem{lemma}[teor]{Lemma}
\newtheorem{cor}[teor]{Corollary}
\newtheorem*{theorem*}{Theorem}

\theoremstyle{definition}
\newtheorem{defn}[teor]{Definition}

\theoremstyle{remark}
\newtheorem{example}[teor]{Example}

\newtheorem{remark}[teor]{Remark}

\begin{document}
\title{A Morse-Bott approach to monopole Floer homology and the Triangulation conjecture}
\author{Francesco Lin}
\address{Department of Mathematics, Massachusetts Institute of Technology} 
\email{linf@math.mit.edu}

\begin{abstract}
In the present work we generalize the construction of monopole Floer homology due to Kronheimer and Mrowka to the case of a gradient flow with Morse-Bott singularities. Focusing then on the special case of a three-manifold equipped equipped with a spin$^c$ structure which is isomorphic to its conjugate, we define the counterpart in this context of Manolescu's recent $\Pin$-equivariant Seiberg-Witten-Floer homology. In particular, we provide an alternative approach to his disproof of the celebrated Triangulation conjecture.
\end{abstract}

\makeatletter
\thispagestyle{empty}
\begin{center}
\vspace{24pt}
{\huge\bfseries\openup\medskipamount\@title\par}
\vspace{24pt}
{\LARGE\mdseries\def\and{\par\medskip}\authors\par}

\end{center}
\bigskip
\@setabstract
\vspace*{\fill}
\@setaddresses
\makeatother

\tableofcontents

\chapter*{Introduction}

The present work has two purposes. The first one is to generalize Kronheimer and Mrowka's construction of monopole Floer homology (\cite{KM}) to the case in which the gradient flow has Morse-Bott singularities. The second one is to define in this framework the counterpart of Manolescu's recent $\Pin$-equivariant Seiberg-Witten-Floer homology (\cite{Man2}), and use it to give an alternative approach to his recent breakthrough, the disproof of the longstanding Triangulation conjecture.
\\
\par
Monopole Floer homology is a set of invariants of three manifolds equipped with a spin$^c$ structure. They are obtained by studying the gradient flow equation of the Chern-Simons-Dirac functional, the Seiberg-Witten monopole equations. These equations have a natural $S^1$-symmetry. The construction of \cite{KM} deals with case in which the singularities of this flow are non degenerate, meaning that the Hessian is invertible at those points. The analogue in the finite dimensional case are the usual Morse singularities.
However, in many situations it is convenient to be able to deal with a more general kind of singularity, which is analogous to the Morse-Bott case in a finite dimensional setting. These critical points form smooth submanifolds, and the Hessian is invertible in the normal directions.
For example, these type of objects naturally arise when dealing with the Chern-Simons-Dirac functional on Seifert-fibered spaces, see \cite{MOY}.
\par
Another context in which this kind of singularities arises is the case in which the spin$^c$ structure is actually induced by a genuine spin structure. In this case, it has been known for a long time (see for example \cite{Mor}) that the problem has more symmetry due to the the quaternionic structure of the spinor bundle. By exploiting this additional data, Manolescu was able to construct his new invariants of three manifolds based on the Lie group
\begin{equation*}
\Pin=S^1\cup j\cdot S^1\subset \mathbb{H},
\end{equation*}
and disproof the longstanding Triangulation conjecture.

\vspace{0.8cm}

\textbf{The Triangulation conjecture. }The present historical discussion is taken from the nice exposition \cite{Man3}, to which we refer for more details. The Triangulation conjecture, first formulated by Kneser in \cite{Kne}, asserts that every topological manifold is homeomorphic to a simplicial complex. This was known to be true for manifolds of dimension at most three (\cite{Rad}, \cite{Moi}) and false in dimension four (\cite{AM}), while the answer was unknown in dimension at least five.
\par
One could also ask the stronger question of whether every topological manifold $M$ admits a $\textsc{pl}$-structure, i.e it is homeomorphic to a simplicial complex such that the link of each vertex is $\textsc{pl}$-homeomorphic to a sphere. This was answered by Kirby and Siebenmann in \cite{KS}. In particular, they construct a class
\begin{equation*}
\Delta(M)\in H^4(M;\mathbb{Z}/2\mathbb{Z}) 
\end{equation*}
which vanishes if an only if $M$ admits a $\textsc{pl}$-structure, and show that in each dimension at least five this obstruction is effective.
\par
The case of general simplicial complexes is more subtle, and has very deep connections with low dimensional topology. Define the homology cobordism group $\Theta^H_3$ to be the group whose elements are oriented $\textsc{pl}$-homology three spheres up to homology cobordism and for which the sum is given by connected sum. The definition makes sense in every dimension, but by a result of Kervaire \cite{Ker} it is trivial in all these other cases. On the other hand, one can show that $\Theta^H_3$ is not trivial because of the existence of a surjective map
\begin{equation*}
\mu: \Theta^H_3\rightarrow \mathbb{Z}/2\mathbb{Z},
\end{equation*}
the Rokhlin homomorphism. This is obtained by sending a homology three sphere to $\mathrm{sign}(W)/8$ where $W$ is any spin four manifold bounding it.
\par
Given a triangulation $K$ on $M$ (which we suppose to be closed and oriented), one can form the so called Sullivan-Cohen-Sato class
\begin{equation*}
c(K)=\sum_{\sigma\in K^{(n-4)}}[\mathrm{link}_K(\sigma)]\cdot \sigma \in H_{n-4}(M;\Theta_3)\cong H^4(M;\Theta_3).
\end{equation*}
The short exact sequence of groups
\begin{equation}\label{split}
0\rightarrow \mathrm{ker}\mu\rightarrow \Theta^H_3\stackrel{\mu}{\longrightarrow}\mathbb{Z}/2\mathbb{Z}\rightarrow 0
\end{equation}
induces the Bockstein exact sequence
\begin{equation*}
H^4(M;\Theta^H_3)\stackrel{\mu_*}{\longrightarrow}H^4(M;\mathbb{Z}/2\mathbb{Z})\stackrel{\delta}{\longrightarrow}H^5(M;\mathrm{ker}\mu)
\end{equation*}
and it can be shown that the image of $c(K)$ is the class $\Delta(M)$ discussed above. In particular this implies that if $M$ admits a triangulation then $\delta(\Delta(M))$ is zero. Work of Galewski-Stern \cite{Gal} and Matumoto \cite{Mat} shows that the vanishing of this element is also a sufficient condition for the existence of a triangulation. Of course, if the exact sequence (\ref{split}) splits then $\delta(\Delta(M))$ is always zero. In \cite{Gal} and \cite{Mat} it is also shown that this condition is sufficient, hence in order to disproof the Triangulation conjecture it is sufficient to show that the sequence does not split. In fact, Manolescu proved the following result.
\begin{teor}[\cite{Man2}]\label{triangulation}
There are no elements of order two and Rokhlin invariant one in the homology cobordism group. Hence the Triangulation conjecture is false in all dimensions at least five.
\end{teor}
In general, the structure of the homology cobordism group remains a mystery. It is known that it is not finitely generated (\cite{Fur1}, \cite{FS1}), but for example it is not known whether it has torsion or not. In order to prove this theorem, Manolescu defines for each homology sphere $Y$ a $\mathbb{Z}$-valued invariant $\beta(Y)$ satisfying the following properties:
\begin{enumerate}
\item it is invariant under homology cobordism;
\item it reduces modulo two to the Rokhlin invariant $\mu$;
\item $\beta(-Y)=-\beta(Y)$.
\end{enumerate}
The theorem follows because if $2[Y]=0$ in $\Theta_3^H$, then $[Y]=[-Y]$ so
\begin{equation*}
\beta([Y])=\beta(-[Y])=-\beta([Y])
\end{equation*}
hence $\beta([Y])$ is zero, and so is the Rokhlin invariant. The invariant $\beta$ arises as an analogue in the $\Pin$-equivariant context of Fr\o yshov's invariant in Seiberg-Witten Floer homology (\cite{Fro}, \cite{KM}) and Ozsvath-Szabo's correction term in Heegaard Floer homology (\cite{OSd}). We now describe its construction in our theory. 

\vspace{0.8cm}
\textbf{An overview of $\Pin$-monopole Floer homology. }Manolescu's construction of his new invariants follows the framework of his previous work \cite{Man1}, and relies on the theory of Conley index and finite dimensional approximations of the monopole equations. In the present work we define the counterpart of these objects in Kronheimer and Mrowka's theory. We expect the two definitions to agree for rational homology spheres. Even though the construction is quite involved, the final result has some desirable features that are missing in the case of Manolescu's invariants. To each closed oriented three manifold $Y$ equipped with a self-conjugate spin$^c$ structure $\spin$ (i.e. $\spin$ is isomorphic to its conjugate $\bar{\spin}$) we associate three graded topological abelian groups called \textit{$\Pin$-monopole Floer homology groups}. These are denoted by
\begin{equation*}
\HSt_{\bullet}(Y,\spin)\quad\HSf_{\bullet}(Y,\spin)\quad\HSb_{\bullet}(Y,\spin)
\end{equation*}
where the $S$ stands for \textit{spin}, and they are pronounced ``\textit{H-S-to}", ``\textit{H-S-from}" and ``\textit{H-S-bar}" respectively.
These groups are also graded topological modules over the graded topological ring
\begin{IEEEeqnarray*}{c}
\Rin=\ztwo[[V]][Q]/(Q^3)
\end{IEEEeqnarray*}
where $V$ and $Q$ have degree respectively $-4$ and $-1$, and $\ztwo$ is the field with two elements. The three groups are related by a long exact sequence of $\Rin$-modules
\begin{equation}\label{exactsequencein}
\dots\stackrel{i_*}{\longrightarrow} \HMt_k(Y,\spin)\stackrel{j_*}{\longrightarrow} \HMf_k(Y,\spin)\stackrel{p_*}{\longrightarrow} \HMb_{k-1}(Y,\spin)\stackrel{i_*}{\longrightarrow} \HMt_{k-1}(Y,\spin)\stackrel{j_*}{\longrightarrow}\dots.
\end{equation} 
When compared to Manolescu's invariants, the groups resulting from this alternative approach have the following two important additional features:
\begin{itemize}
\item they are defined for \textit{every} three manifold, not only for rational homology spheres;
\item they are \textit{functorial} in the sense that a spin$^c$ cobordism $(X,\spin_X)$ from $(Y_0,\spin_0)$ to $(Y_1,\spin_1)$ determines a group homomorphism
\begin{equation*}
\HSt_{\bullet}(X,\spin_X):\HSt_{\bullet}(Y_0,\spin_0)\rightarrow \HSt_{\bullet}(Y_1,\spin_1),
\end{equation*}
and similarly for the other versions. In the case the spin$^c$ structure $\spin_X$ is actually a genuine spin structure, the map is in fact a homomorphism of graded ${\Rin}$-modules.
\end{itemize}
Furthermore, in forthcoming work (see for example \cite{Lin}) we will develop in this setting some computational tools that are available in monopole Floer homology but not in Manolescu's theory. For many of these developments it is useful to also take into account also not self-conjugate spin$^c$ structures. Denote by $\jmath$ the action by conjugation on the set $\mathrm{Spin}^c(Y)$ of spin$^c$ structures of $Y$, and by $[\spin]$ the equivalence class of $\spin$. We then define for $\spin\neq\bar{\spin}$ the group $\HSt_{\bullet}(Y,[\spin])$ as
\begin{equation*}
\HMt_{\bullet}(Y,\spin)=\HMt_{\bullet}(Y,\bar{\spin})
\end{equation*}
where these are canonically identified via conjugation, and the total group
\begin{equation*}
\HSt_{\bullet}(Y)=\bigoplus_{[\spin]\in \mathrm{Spin}^c(Y)/\jmath} \HSt_{\bullet}(Y,[\spin]).
\end{equation*}
This defines a functor from the category $\textsc{cob}$ of compact connected oriented three manifolds and isomorphism classes of cobordism between them to the category of topological $\ztwo[[V]]$-modules.
\\
\par
In the case of $S^3$, the $\Pin$-monopole Floer homology groups can be identified as the graded $\Rin$-modules
\begin{align*}
\HMb_{\bullet}(S^3)&= \ztwo[V^{-1},V]][Q]/(Q^3)\{-2\}\\
\HMt_{\bullet}(S^3)&= \left(\ztwo[V^{-1},V]][Q]/(Q^3)\{-2\}\right)/\Rin\cdot1\\
\HMf_{\bullet}(S^3)&= \Rin\{-1\}
\end{align*}
where the braces indicate grading shifts. In particular, the minimum degree of an element is the \textit{to} group is zero. For a general homology three sphere the \textit{bar} group is isomorphic (up to grading shift) to $\HMb_{\bullet}(S^3)$, the \textit{to} group is bounded below and the \textit{from} group is bounded above. The exact sequence (\ref{exactsequencein}) implies that the $\Rin$ module given by the image
\begin{equation*}
i_*\left(\HMb_{\bullet}(S^3)\right)\subset \HMt_{\bullet}(S^3)
\end{equation*}
decomposes as the direct sum of three copies of the $\ztwo[[V]]$-module $\ztwo[[V^{-1},V]]/\ztwo[[V]]$, which are related to each other by the action $Q$. The invariant $\beta(Y)$ used by Manolescu to prove Theorem \ref{triangulation} is then defined so that $2\beta(Y)+1$ is the minimum grading of an element in the middle tower, i.e. the one on which $Q$ acts non trivially and $Q^2$ acts trivially. For example, it is zero in the case of $S^3$.
\\
\par
The key point of Kronheimer and Mrowka's construction of monopole Floer homology is the introduction of the \textit{blow-up} of the configuration space on which the functional is defined. This is done in order to deal properly with the reducible configurations. In fact, the natural $S^1$-action on the configuration space is not free and this caused serious invariance and functoriality issues in the earlier approaches to the Floer homology of the monopole equations. These have been tackled in various ways in the last twenty years (see for example \cite{MW}, \cite{Man1} and \cite{Fro}). The approach of \cite{KM} naturally leads to consider Morse flows on manifolds where the boundary is invariant for the gradient flow (and is therefore \textit{not} Morse-Smale). This requires a definition of Morse homology for a manifold with boundary which differs from the more classical one as it has to deal with some new phenomena, in particular the existence of \textit{boundary obstructed trajectories} (see Chapter $2$ of \cite{KM} for an introduction to the finite dimensional case).
\\
\par
There are many approaches to Morse-Bott homology in literature, and for many reasons the one that fits our problem of defining the invariants the best is the one introduced by Fukaya (\cite{Fuk}) in the context of instanton Floer homology. The chain complex associated to a Morse-Bott function $f$ on a smooth manifold $X$ that he defines has underlying vector space the direct sum of (some modified version of) the singular chain complexes of the critical submanifolds
\begin{equation*}
C_*(X,f)=\bigoplus_{\mathcal{C}\in \mathrm{Crit}(f)}\tilde{C}_*(\mathcal{C})
\end{equation*} 
and the differential of a chain $\sigma$ is given by
\begin{equation*}
\partial \sigma= \tilde{\partial}\sigma+\sum_{\mathcal{C}'\in \mathrm{Crit}(f)}\sigma\times \breve{M}^+(\mathcal{C},\mathcal{C}').
\end{equation*}
Here $\tilde{\partial}$ denotes the differential in the modified chain complex $\tilde{C}_*(\mathcal{C})$ and each element in the sum is the fibered product of $\sigma$ and the compactified moduli space of trajectories $\breve{M}^+(\mathcal{C},\mathcal{C}')$ connecting $\mathcal{C}$ and $\mathcal{C}'$ (using the evaluation map on the negative end), which is a chain in $\mathcal{C}'$ via the evaluation map on the positive end. Of course $\sigma$ belongs to some appropriate class of geometric objects which is closed under fibered products with the moduli spaces (in \cite{Fuk} the author considers polyhedra with singularities in codimension at least two). Furthermore, we require some transversality conditions in order for the fibered product to be in this class. The nice feature of this construction is that it works for generic (in a suitable sense) Morse-Bott perturbations, and it can be adapted to work in our case where also boundary obstructedness phenomena come into play. A manifestation of the latter is that the compactified moduli spaces of trajectories are not in general manifolds with corners, as they have a more complicated structure both combinatorially and topologically (see Section 19.5 in \cite{KM}). Nevertheless we will construct a version of the singular chain complex of a smooth manifold (inspired by the work \cite{Lip}), and prove that it computes the usual singular homology. With this in hand, we will define the three Floer chain complexes by adapting the construction sketched above to the case where the manifold $X$ (which in our case is the moduli space of configurations) has boundary and the gradient of $f$ is tangent to it.
\\
\par
When the spin$^c$ structure is self-conjugate, the configuration space admits a natural $\Pin$-action, and the idea is to exploit this additional symmetry to construct a chain complexes computing monopole Floer homology groups with an additional chain involution induced by the action of the element $j$ in $\Pin$. The $\Pin$-monopole Floer homology groups will be then defined as the homologies of the invariant subcomplexes. In order to do this we will perturb the problem while respecting the additional symmetry. Unfortunately this class of perturbations will not allow us to achieve genericity in the sense of \cite{KM}, but nevertheless they will be generic enough so that we can arrange the singularities to be Morse-Bott and apply our approach.
\\
\par
Many of the proofs and results in the present work (and especially in the first part) are generalizations of the ones contained in \cite{KM}. In order to keep the length of the present work somehow contained, we will often rely on the work already done there, and we will only expand the details which are significantly different or which are notably interesting. We will assume that the reader has a reasonable understanding of the content of \cite{KM}, and in order to help her/him we will always give precise references of the omitted passages. We will refer to \cite{KM} as \textit{the book}.

\vspace{0.8cm}

\textbf{Plan of the work.} In Chapter $1$, we quickly review the main protagonists of the present work, namely the Chern-Simons-Dirac functional and its gradient flow equation, the Seiberg-Witten monopole equations. In particular, we briefly discuss the content of the Chapter $4-11$ in the book regarding the monopole equations, the blow-up of the configurations spaces, their completions and finally the theory of tame perturbations, which will be essential in the construction of $\Pin$-monopole Floer homology.
\par
Chapter $2$ is dedicated to the local and global analysis of Morse-Bott singularities, following Chapters $12-19$ in the book. After defining them, we will discuss the properties of the space of solutions to the monopole equations connecting two such critical submanifolds, and prove the fundamental transversality, compactness and gluing results.
\par
In Chapter $3$, following Chapters $22-25$ in the book, we show how to define monopole Floer homology when the singularities of the flow are Morse-Bott. We construct a modified version of the singular chain complex of a smooth manifold and adapt Fukaya's approach to our context. We show that the functoriality and invariance properties of the Floer groups hold in this setting, proving among the other things that the result of the new construction canonically agrees with the one of Kronheimer and Mrowka.
\par
Finally, in Chapter $4$ we focus on the case of a three manifold equipped with self-conjugate spin$^c$ structure. We show how to perturb the equations in a way which is compatible with the additional symmetry of the equations. This naturally leads to Morse-Bott singularities. The symmetry of the equations will give rise a symmetry of the chain complex defining our invariants, and exploiting this we will be able to define the $\Pin$-monopole Floer homology and prove all the properties we briefly discussed above. After providing some simple calculations (in particular for $S^1\times S^2$ and the three-torus), we will give an alternative disproof of the Triangulation conjecture.

\vspace{1.5cm}
\textbf{Acknowledgements. }The author would sincerely like to thank his advisor Tom Mrowka for introducing him to the subject, for suggesting the present problem, and for his patient help and support throughout the development of the project. Without his guidance and expertise this work would not have been possible at all. The author would also like to express his gratitude to Jonathan Bloom, especially for the interesting discussions related to the content of Chapter $3$. Finally, he would like to thank Michael Andrews, Lucas Culler, Michael Hutchings, Ciprian Manolescu, Roberto Svaldi and Umut Varolgunes for the useful conversations. This work was partially funded by NSF grants DMS-0805841 and DMS-1005288.

\chapter{Basic setup}

This chapter contains a quick discussion of the background required for the construction of Floer homology in the Seiberg-Witten setting. We start by describing the differential geometry needed to write down the monopole equations and then introduce the fundamental construction of Kronheimer and Mrowka's approach, the blow up of the moduli spaces. We then set up the functional spaces on which we will study the analytical problems, and recall the class of perturbation we allow. The material is treated in the same way as in the book and the aim of our discussion is to review the content and the notation of  Chapters $4$ to $11$. In particular, the result will be cited without proofs (for which we refer the reader to the book).

\vspace{1.5cm}
\section{The monopole equations}
In order to introduce the Seiberg-Witten equations we first have to discuss spin$^c$ structures. This can be done in many ways and we expose a very concrete one which will be the most useful in the rest of the work. Let $Y$ be an oriented closed connected riemannian $3$-manifold. A \textit{spin$^c$ structure} $\spin$ on $Y$ is given by a rank-$2$ hermitian vector bundle $S$ on $Y$ together with a map
\begin{equation*}
\rho: TX\rightarrow \mathrm{End}(S),
\end{equation*}
called \textit{Clifford multiplication}, satisfying the following properties. It is a bundle map that identifies $TX$ isometrically with the subbundle $\mathfrak{su}(S)$ of traceless skew-adjoint endomorphisms of $S$ (equipped with the metric $\frac{1}{2}\mathrm{tr}(a^*b)$) and it respects orientations, i.e. if $e_1, e_2, e_3$ is an oriented orthonormal frame then
\begin{equation*}
\rho(e_1)\rho(e_2)\rho(e_3)=1_S.
\end{equation*}
This means that at any point we can always find a basis of the fiber such that the Clifford multiplication by $\rho(e_i)$ are given by the Pauli matrix $\sigma_i$:
\begin{equation*}
\sigma_1=\left(\begin{matrix}
i & 0 \\
0 & -i
\end{matrix}\right)
\qquad \sigma_2=\left(\begin{matrix}
0 & -1 \\
1 & 0
\end{matrix}\right)
\qquad \sigma_3=\left(\begin{matrix}
0 & i \\
i & 0
\end{matrix}\right).
\end{equation*}
The action of $\rho$ is extended to cotangent vectors using the metric, and then to (complex valued) forms using the rule
\begin{equation*}
\rho(\alpha\wedge\beta)=\frac{1}{2}\left(\rho(\alpha)\rho(\beta)+(-1)^{\deg(\alpha)\deg(\beta)}\rho(\beta)\rho(\alpha)\right).
\end{equation*}
On a $3$-manifold $Y$ spin$^c$ structures always exist. In fact one can take a trivialization of the tangent bundle $TY$ and define on $\C^2\times Y$ a Clifford multiplication given globally using the Pauli matrices. Furthermore, given a spin$^c$ structure $(S_0,\rho_0)$ and an hermitian line bundle $L\rightarrow Y$ we can define a new spin$^c$ structure given by
\begin{IEEEeqnarray*}{c}
S=S_0\otimes L \\
\rho(e)=\rho_0(e)\otimes 1_L.
\end{IEEEeqnarray*}
This construction gives the space of spin$^c$ structures the structure of an affine space over the group of isomorphism classes of complex line bundles over $Y$ or, equivalently, $H^2(Y;\Z)$.
\begin{remark}
Notice that our definition depends on the prior choice of a Riemannian metric. However, we identify two spin$^c$ structures $\spin_0$ and $\spin_1$ associated to two metrics $g_0$ and $g_1$ such that there is a path of metrics $g_t$ joining  $g_0$ and $g_1$ and a corresponding continous family $\spin_t$ over $[0,1]$. Hence we think of the set of isomorphism classes of spin$^c$ structures as being associated to a smooth closed connected oriented manifold $Y$.
\end{remark}
On a $4$-manifold the story is analogous. Let $X$ be a closed oriented riemannian $4$-manifold. A spin$^c$ structure $\spin_X$ is given by a rank $4$ hermitian vector bundle $S_X$ on $X$ together with a Clifford multiplication $\rho_X$, i.e. a bundle map
\begin{equation*}
\rho_X:TX\rightarrow \mathrm{End}(S_X)
\end{equation*}
such that at each $x\in X$ we can find an oriented orthonormal frame $e_0,e_1,e_2,e_3$ with
\begin{equation*}
\rho_X(e_0)=\left(\begin{matrix}
0 & -I_2 \\
I_2 & 0
\end{matrix}\right)\qquad
 \rho_X(e_i)=\left(\begin{matrix}
0 & -\sigma_i^* \\
\sigma_i & 0
\end{matrix}\right) \quad i=1,2,3,
\end{equation*}
in some orthonormal basis of the fiber. Here $I_2$ is the $2\times 2$ identity matrix and the $\sigma_i$ are the Pauli matrices as above. Extending the Clifford multiplication to (complex) forms as in the $3$-dimensional case, we have that in the same basis
\begin{equation*}
\rho(\mathrm{vol}_x)=\left(\begin{matrix}
-I_2 & 0 \\
0 & I_2
\end{matrix}\right)
\end{equation*}
where $\mathrm{vol}=e_0\wedge e_1\wedge e_2\wedge e_3$ is the oriented volume form. Hence we get a orthogonal decomposition
\begin{equation*}
S_X=S^+\oplus S^-
\end{equation*}
respectively as the $-1$ and $+1$ eigenspaces of $\rho(\mathrm{vol})$. Then the Clifford multiplication by a tangent vector interchanges the two factors, and we have the bundle isometry
\begin{equation*}
\rho: \Lambda^+X\rightarrow \mathfrak{su}(S^+)
\end{equation*}
where $\Lambda^+X$ is the space of self-dual $2$-forms on $X$.
The existence and classification results for spin$^c$ structures on a $4$-manifold are identical to the $3$-dimensional case, even though the existence of at least one such structure is more subtle.
\\
\par
Finally we discuss the relation between spin$^c$ structures on $3$ and $4$-manifolds. Suppose $Y=\partial X$, and consider a spin$^c$ structure on $X$. Then using the outward normal vector field $\nu$ one obtains an identification
\begin{equation*}
\rho(\nu): S^+|_Y\rightarrow S^-|_Y.
\end{equation*}
Hence we recover a spin$^c$ structure on $Y$ with
\begin{align*}
S&= S^+|_Y\cong S^-|_Y\\
 \rho(v)&=\rho_X(\nu)^{-1}\rho_X(v).
\end{align*}

\vspace{0.8cm}
Fix now a spin$^c$ structure $\spin$ on a riemannian $3$-manifold $Y$. This data allows us to consider the following two objects. First, one has \textit{spinors}, i.e. sections $\Psi$ of $S$. Then one has \textit{spin$^c$ connections}, which are unitary connections $B$ on $S$ such that $\rho$ is parallel, namely for every vector field $\xi$ and every spinor $\Psi$ one has
\begin{equation*}
\nabla_B(\rho(\xi)\Psi)=\rho(\nabla \xi)\Psi+ \rho(\xi)\nabla_B(\Psi),
\end{equation*}
where $\nabla$ is the Levi-Civita connection.
Any such connection determines to a connection $B^t$ on the line bundle $\mathrm{det}(S)$, and given a base spin$^c$ connection $B_0$, all the other ones can be written as
\begin{equation*}
B=B_0+b
\end{equation*}
for some $b\in i\Omega^1(Y;\R)$.
Notice that this also implies $B^t=B_0^t+2b$. We denote the space of spin$^c$ connections by $\mathcal{A}(Y,\spin)$. We define the configuration space
\begin{equation*}
\Co(Y,\spin)=\left\{(B,\Psi)\right\}=\mathcal{A}(Y,\spin)\times \Gamma(S).
\end{equation*}
The group of automorphisms of the spin$^c$ structure
\begin{equation*}
\G=\mathrm{Map}(Y, S^1),
\end{equation*}
also called the \textit{gauge group}, acts on $\Co(Y,\spin)$ via the map
\begin{equation*}
u\cdot(B,\Psi)=(B-u^{-1}du,u\Psi).
\end{equation*}
We denote the quotient space of this action by $\Bo(Y,\spin)$, the moduli space of configurations.
\\
\par
We are interested in paths in $\Co(Y,\spin)$ satisfying the following flow equations:
\begin{equation}\label{monopoleflow}
\left\{
\begin{aligned}
\frac{d}{dt}B&=-\left(\frac{1}{2}\ast F_{B^t}+\rho^{-1}(\Psi\Psi^*)_0\right)\otimes1_S\\
\frac{d}{dt}\Psi&=-D_B\Psi.
\end{aligned}
\right.
\end{equation}
Here $(\Psi\Psi^*)_0$ denotes the traceless part of the endomorphism $\Psi\Psi^*$., and $D_B$ is the \textit{Dirac operator} associated to $B$, i.e. the composition
\begin{equation*}
\Gamma(S)\xrightarrow{\nabla_B} \Gamma(T^*X\otimes S)\xrightarrow{\rho} \Gamma(S).
\end{equation*}
The Dirac operator $D_B$ is a first order elliptic self-adjoint operator, hence it has index zero. These equations can be thought as describing the trajectories for the flow of the formal $L^2$ gradient of the Chern-Simons-Dirac functional
\begin{equation}
\CSD(B,\Psi)=-\frac{1}{8}\int_Y (B^t-B^t_0)\wedge (F_{B^t}+F_{B^t_0})+\frac{1}{2}\int_Y\langle D_B\Psi, \Psi\rangle d\mathrm{vol},
\end{equation}
where $B_0$ is some fixed spin$^c$ connection. The equations for the critical points
\begin{equation}\label{criticalpoints}
\left\{
\begin{aligned}
\frac{1}{2}\rho(F_{B^t})-(\Psi\Psi^*)_0&=0\\
D_B\Psi&=0
\end{aligned}\right.
\end{equation}
are called the \textit{Seiberg-Witten equations} or \textit{monopole equations}. The Chern-Simons-Dirac functional is not invariant under $\G$, but one has the relation
\begin{equation*}
\CSD(u\cdot(B,\Psi))-\CSD(B,\Psi)=2\pi^2([u]\cup c_1(S))[Y]
\end{equation*}
where $[u]\in H^1(Y;\Z)$ is the cohomology class corresponding to the map $u:Y\rightarrow S^1$. In particular, $\CSD$ is invariant under the identity component $\G^e$ of the gauge group, and descends to a well defined function with values in $\R/(2\pi^2\Z)$.
\\
\par
The monopole equations have a special class of solutions. We say that a configuration $(B,\Psi)$ (not necessarily a solution) is \textit{reducible} if $\Psi=0$, and \textit{irreducible} otherwise. The subset of irreducible configurations is denoted by $\Co^*(Y,\spin)$. If there is a reducible solution $(B,0)$ to the monopole equations (\ref{criticalpoints}), then the connection $B^t$ on $^\mathrm{det}(S)$ is flat, hence $c_1(S)$ is torsion. On the other hand, if the first Chern class of $S$ is torsion, there are always reducible solutions, and the space of reducible solutions up to gauge transformations can be identified with a torus
\begin{equation*}
H^1(Y;i\R)/2\pi iH^1(Y; \Z)\subset\Bo(Y,\spin).
\end{equation*}
The distinction between irreducible and reducible plays an important role because an irreducible configuration has trivial stabilizer under the action of the gauge group, while a reducible configuration has stabilizer $S^1$ consisting of constant gauge transformations.
\\
\par
The flow equations (\ref{monopoleflow}) can also be interpreted purely in $4$-dimensional terms. We can define the configuration space on a $4$-manifold $X$ as
\begin{equation*}
\Co(X,\spin)=\mathcal{A}(X,\spin)\times \Gamma(S^+)
\end{equation*}
where the first factor is the space of spin$^c$ connections on $S_X$. The gauge group $\G=\mathrm{Map}(X,S^1)$ acts on it with quotient $\Bo(X,\spin)$. A spin$^c$ structure on $Y$ naturally induces a spin$^c$ structure on the infinite cylinder $Z=\R\times Y$ as follows. The bundle $S_Z$ will be (the pull-back of) $S\oplus S$, and the Clifford multiplication $\rho_Z$ is defined as
\begin{equation*}
\rho_Z(\partial/\partial t)=\left(\begin{matrix}
0 & -I_S \\
I_S & 0
\end{matrix}\right)\qquad
\rho_Z(v)=\left(\begin{matrix}
0 & -\rho(v)^* \\
\rho(v) & 0
\end{matrix}\right)\text{ for }v\in TY.
\end{equation*}
A time-dependent spinor naturally defines a section of $S^+$, while a time dependent connection $B$ on $Y$ gives rise to a connection $A$ on $Z$ defined as
\begin{equation}\label{temporalgauge}
\nabla_A=\frac{d}{dt}+\nabla_B.
\end{equation}
In this setting, one can write the flow equations (\ref{monopoleflow}) as a gauge invariant equations for a pair $(A,\Phi)$ of the form
\begin{equation*}
\left\{
\begin{aligned}
\frac{1}{2}\rho_Z(F_{A^t}^+)&=(\Phi\Phi^*)_0\\
D_A^+\Phi&=0
\end{aligned}\right.
\end{equation*}
where $D_A^+$ is the Dirac operator arising as the composition
\begin{equation*}
\Gamma(S^+)\xrightarrow{\nabla_A} \Gamma(T^*X\otimes S^+)\xrightarrow{\rho} \Gamma(S^-).
\end{equation*}
We will also think the equations as the differential operator
\begin{IEEEeqnarray*}{c}
\F: \Co(Z,\spin)\rightarrow C^{\infty}(X; i\mathfrak{su}(S^+)\oplus S^-)\\
(A,\Phi)\mapsto\left(\frac{1}{2}\rho(F_{A^t}^+)-(\Phi\Phi^*)_0,D_A^+\Phi\right).
\end{IEEEeqnarray*}
It is important to notice that the last set of equations makes sense for a general configuration $(A,\Phi)$ on a general $4$-manifold $X$, and these are indeed the original Seiberg-Witten equations (\cite{Mor}). 
\begin{remark} In the case $X$ is compact without boundary it follows from the Atiyah-Singer index theorem that $D_A^+$ has complex index 
\begin{equation*}
\mathrm{Ind} D_A^+=\frac{1}{8}\left(c_1(S^+)^2[X]-\sigma(X)\right),
\end{equation*}
where $\sigma(X)$ is the signature of the intersection form of $X$.
\end{remark}
On the infinite cylinder $Z$ the $4$-dimensional equations make sense for every spin$^c$ connection $A$, not necessary coming from a time dependent connection $B$ on $Y$ as in equation (\ref{temporalgauge}) (which we call \textit{in temporal gauge}). In general we can always write such a connection as
\begin{equation*}
A=B+(cdt)\otimes 1_S,
\end{equation*}
where $c$ is a time dependent imaginary valued function on $Y$, and writing down the Seiberg-Witten equations in a $3$-dimensional fashion we obtain the gauge invariant equations
\begin{equation*}
\left\{
\begin{aligned}
\frac{d}{dt}B-dc&=-\left(\frac{1}{2}\ast F_{B^t}+\rho^{-1}(\Psi\Psi^*)_0\right)\otimes1_S\\
\frac{d}{dt}\Psi+c\Psi&=-D_B\Psi.
\end{aligned}\right.
\end{equation*}
Given a configuration $(A,\Phi)$ on $Z$ we will denote by $(\check{A}(t),\check{\Phi}(t))$ the path of configurations on $Y$ it determines.

\vspace{1.5cm}

\section{Blowing up the configuration spaces}
The key construction in monopole Floer homology is the blow-up of the moduli space along the reducible locus. This is done in order to be able to deal properly with reducible solutions, which are the fixed points for the action of the constant gauge transformations. In particular, our goal will be to define in an infinite dimensional setting an analogue of $S^1$-equivariant homology.
\\
\par
Consider the configuration space $\Co(X,\spin)$ on a $4$-manifold (possibly with boundary) with a fixed spin$^c$ structure. Define the blown-up moduli space as
\begin{equation*}
\Cs(X,\spin)=\left\{(A,s,\phi)\mid s\geq0\text{ and } \|\phi\|_{L^2(X)}=1\right\}\subset \mathcal{A}(X,\spin)\times\R\times \Gamma(S^+).
\end{equation*}
This comes with the natural blow-down map
\begin{align*}
\pi:\Cs(X,\spin)&\rightarrow\Co(X,\spin)\\
(A,s,\phi)&\mapsto (A,s\phi).
\end{align*}
Such a map is a bijection over the irreducible locus $\Co^*(X,\spin)$, while the fiber over a reducible configuration $(A,0)$ is a sphere $\mathbb{S}(C^{\infty}(X;S^+))$.
\par
The Seiberg-Witten equations are defined on the blow up as follows. One can see the operator $\F$ as a section on the trivial vector bundle
\begin{equation*}
\V(X,\spin)\rightarrow \Co(X,\spin)
\end{equation*}
with fiber $C^{\infty}(X; i\mathfrak{su}(S^+)\oplus S^-)$. This vector bundle can be pulled back to $\Cs(X,\spin)$ along the projection $\pi$ to get a vector bundle $\V^{\sigma}(X,\spin)$. Then one defines the section
\begin{IEEEeqnarray*}{c}
\F^{\sigma}:\Cs(X,\spin)\rightarrow \V^{\sigma}(X,\spin)\\
(A,s,\phi)\rightarrow\big(\frac{1}{2}\rho(F_{A^t}^+)-s^2(\phi\phi^*)_0, D_A^+\phi\big).
\end{IEEEeqnarray*}
We call these the \textit{Seiberg-Witten equations on the blown up configuration space}. Notice that the gauge group $\G(X)$ naturally acts on both $\Cs(X,\spin)$ and $\V^{\sigma}(X,\spin)$ and the section $\F^{\sigma}(X,\spin)$ is equivariant for this action. 
\\
\par
The blow-up of the configuration space on a $3$-manifold is defined in an analogous way as
\begin{equation*}
\Cs(Y,\spin)=\left\{(B,r,\psi)\mid r\geq0\text{ and }\|\psi\|_{L^2(Y)}=1\right\}\subset \mathcal{A}(Y,\spin)\times\R\times \Gamma(S).
\end{equation*}
On a finite cylinder $Z=I\times Y$, a configuration
\begin{equation*}
\gamma^{\sigma}=(A,s,\phi)\in \Cs(Z,\spin)
\end{equation*}
gives rise to a path of configurations in $\Cs(Y,\spin)$
\begin{equation*}
\check{\gamma}^{\sigma}(t)=(\check{A}(t), s\|\check{\phi}(t)\|_{L^2(Y)}, \check{\phi}(t)/\|\check{\phi}(t)\|_{L^2(Y)})
\end{equation*}
provided
\begin{equation*}
\|\check{\phi}(t)\|_{L^2(Y)}\neq 0\text{ for every } t\in I.
\end{equation*}
Because of a unique continuation result for solutions of the equation
\begin{equation*}
\F^{\sigma}(\gamma^{\sigma})=0
\end{equation*}
(see Chapter $7$ in the book), this is true if we restrict to this class of paths. If $A$ is also in temporal gauge, the $4$-dimensional Seiberg-Witten equations on the blow-up can be interpreted as the equations for the path $(B(t), r(t),\psi(t))$ given by
\begin{equation*}\label{blowflow}
\left\{
\begin{aligned}
\frac{d}{dt}B&=-\left(\frac{1}{2}\ast F_{B^t}+r^2\rho^{-1}(\psi\psi^*)_0\right)\otimes 1_S\\
\frac{d}{dt}r&=-\Lambda(B,r,\psi)r\IEEEyesnumber\\
\frac{d}{dt}\psi&=-D_B\psi+\Lambda(B,r,\psi)\psi
\end{aligned}\right.
\end{equation*}
where we have defined the function
\begin{equation*}
\Lambda(B,r,\psi)=\langle\psi,D_B\psi\rangle_{L^2(Y)}.
\end{equation*}
The right hand side of the equations defines a vector field of $\Cs(Y,\spin)$, the \textit{blown-up gradient of the Chern-Simons-Dirac functional}, which we denote by $(\mathrm{grad} \CSD)^{\sigma}$. This vector field is the pull-back of $\mathrm{grad}\CSD$ on the irreducible locus $\Co^*(Y,\spin)$, and is tangent to the locus $r=0$, i.e. the boundary of $\Co^{\sigma}(Y,\spin)$. Notice that $\CSD$ also pulls back to a well defined function on $\Co^{\sigma}(Y,\spin)$, but it is important to notice that $(\mathrm{grad} \CSD)^{\sigma}$ is not the gradient of this function with respect to any natural metric. It is easy to identify the critical points of $(\mathrm{grad} \CSD)^{\sigma}$. In particular they are:
\begin{itemize}
\item configurations with $r\neq0$ such that $(B,r\psi)\in\Co$ is a critical point of $\mathrm{grad}\CSD$.
\item configurations with $r=0$, such that $(B,0)\in\Co$ is a critical point of $\mathrm{grad}\CSD$ and $\psi$ is an eigenvector of $D_B$ (with eigenvalue $\Lambda(B,0,\psi))$.
\end{itemize}
\vspace{0.5cm}
\par
When dealing with configurations on a cylinder $Z=I\times Y$, there is another version of the blow-up which will be very useful when studying flow lines. This is the $\tau$-model
\begin{equation*}
\Ct(Z,\spin)\subset \mathcal \A(Z,\spin)\times C^{\infty}(I,\R)\times \Gamma(S^+)
\end{equation*}
which consists of the set of triples $(A,s,\phi)$ satisfying
\begin{itemize}
\item $s(t)\geq 0$;
\item $\|\phi(t)\|_{L^2(Y)}=1$.
\end{itemize}
An element $\gamma=(A,s,\phi)$ in $\Ct(Y,\spin)$ determines a path $\check{\gamma}=(\check{A},s,\check{\phi})$ in $\Cs(Y,\spin)$, and the correspondence is bijective if we restrict to configurations in temporal gauge. In general, the flow equations (\ref{blowflow}) can be rewritten for our path in the following gauge-invariant form
\begin{equation*}
\left\{
\begin{aligned}
\frac{d}{dt}\check{A}&=-\left(\frac{1}{2}\ast F_{\check{A}^t}+dc+r^2\rho^{-1}(\check{\psi}\check{\psi}^*)_0\right)\otimes 1_S\\
\frac{d}{dt}s&=-\Lambda(\check{A},s,\check{\psi})s\\
\frac{d}{dt}\check{\psi}&=-D_{\check{A}}\check{\psi}-c\check{\phi}+\Lambda(\check{A},s,\check{\psi})\check{\psi}
\end{aligned}
\right.
\end{equation*}
where $A=\check{A}+cdt$.
It is useful to interpret these equations in the $\tau$-model in a $4$-dimensional fashion (which also makes gauge-invariance more manifest). The equations can be written as
\begin{equation*}
\left.
\begin{aligned}
\frac{1}{2}\rho(F_{A^t}^+)-s^2(\phi\phi^*)_0&=0\\
\frac{d}{dt}s+\mathrm{Re}\langle D_A^+\phi, \rho(dt)^{-1}\phi\rangle_{L^2(Y)}s&=0\\
D_A^+\phi-\mathrm{Re}\langle D_A^+\phi, \rho(dt)^{-1}\phi\rangle_{L^2(Y)}\phi&=0.
\end{aligned}
\right\}
\end{equation*}
The left hand side on this equations determine a section $\F^{\tau}$ of the vector bundle
\begin{equation*}
\V^{\tau}(Z,\spin)\rightarrow\Ct(Z,\spin)
\end{equation*}
with fiber over $(A,s,\phi)$ the vector space
\begin{equation*}
\left\{(\eta,r,\psi)\mid\langle\mathrm{Re}\check{\phi}(t),\check{\psi}(t)\rangle_{L^2(Y)}=0\text{ for all }t\right\}\subset C^{\infty}(Z,i\mathfrak{su}(S^+))\oplus C^{\infty}(I,\R)\oplus C^{\infty}(Z,S^-).
\end{equation*}
A configuration in $\Cs(Z,\spin)$ such that the restriction to each slice $\{t\}\times Y$ is non zero gives rise to a well defined element of $\Ct(Z,\spin)$ by rescaling. A unique continuation result (see Chapter $7$ in the book) tells us that if solution of the equations is such that the spinor vanishes on a slice, then the spinor vanishes everywhere. Hence we get a bijection between the set of solutions of $\F^{\sigma}$ and $\F^{\tau}$.

\vspace{1.5cm}
\section{Completion and slices}
We now define suitable completions of our configuration spaces in order to be able to work in a Hilbert (or Banach) manifold setting which will be essential for the kind of analysis we will do.
From now on, given a vector bundle $E$ over $M$ with an inner product and a connection $\nabla$, and given integer $k$, we will denote by $L^p_k(M;E)$ the Sobolev space obtained by completing the space of sections of such a bundle with respect to the Sobolev norm
\begin{equation*}
\|f\|_{L^p_k}^p=\int_M |f|^p+|\nabla f|^p+\dots+|\nabla^k f|^p d\mathrm{vol}.
\end{equation*} 
If $M$ is closed, we can also define the fractional Sobolev spaces $L^2_k(M;E)$ as the completion of the space of sections with respect to the norm
\begin{equation*}
\|f\|_{L^2_k}=\|(1+\Delta)^{k/2}f\|_{L^2}.
\end{equation*}
This definition is equivalent to the previous one in the case $k$ is an integer.
\\
\par
We then focus in the case when $M$ is a $3$-manifold or a $4$-manifold (possibly with boundary). We will write $W$ for the bundle $S$ in the former case and for $S^+$ in the latter. Fix a smooth connection $A_0$ on $M$. We define
\begin{equation*}
\Co_k(M,\spin)=(A_0,0)+L^2_k(M; iT^*M\oplus W)=\A_k(M,\spin)\times L^2_k(M;W)
\end{equation*}
where
\begin{equation*}
\A_k(M,\spin)=A_0+L^2_k( M; i T^*M)
\end{equation*}
is the space of $L^2_k$ connections on $M$. Similarly the completion of the blown-up configuration space is defined as
\begin{IEEEeqnarray*}{c}
\Cs_k(M,\spin)=\left\{(A,s,\phi)\mid s\geq0,\|\phi\|_{L^2(X)}=1\right\}\\
\subset
\A_k(M,\spin)\times\R^{\geq0}\times \mathbb{S}(L^2_k(M;W))\end{IEEEeqnarray*}
where the sphere is still taken with respect to the $L^2$ norm. This is a smooth Hilbert manifold with boundary given by the locus $s=0$. For $2(k+1)>\mathrm{dim}M$, the gauge group
\begin{equation*}
\G_{k+1}(M)\subset L^2_k(M;\C)
\end{equation*}
is defined as the subset of functions with pointwise norm $1$. With the topology inherited from $L^2_{k+1}(M;\C)$, this is a Hilbert Lie group which acts smoothly and freely on $\Cs_k(M,\spin)$.
\par
The tangent bundle of $\Cs_k(M;\spin)\subset \A_k\times\R\times L^2_k(M;W)$ can be regarded as the bundle with fiber at $\gamma=(A_0,s_0,\phi_0)$ given by
\begin{equation*}
\{(a,s,\phi)\mid\mathrm{Re}\langle\phi_0,\phi\rangle_{L^2}=0\}\subset L^2_k(M;iT^*M)\times\R\times L^2_k(M;W).
\end{equation*}
We also define the bundles $\T^{\sigma}_j\rightarrow \Cs_k(M,\spin)$ obtained from the tangent bundle by fiberwise completion in the $L^2_j$ norm for $j\leq k$. The fiber at a point $\gamma=(A,s,\phi)\in\Cs_k(M,\spin)$ is given by
\begin{equation*}
\{(a,s,\phi)\mid\mathrm{Re}\langle\phi_0,\phi\rangle_{L^2}=0\}\subset L^2_j(M;iT^*M)\times\R\times L^2_j(M;W).
\end{equation*}
In the case of non blown-up configuration spaces, the same construction yields the product bundles
\begin{equation*}
\T_j=L^2_j(M; iT^*M\oplus W)\times \Co_k(M,\spin).
\end{equation*}
In both cases, for $j=k$ we obtain the usual tangent bundle.
\\
\par
For $2(k+1)>\mathrm{dim} M$, we introduce the quotient spaces
\begin{IEEEeqnarray*}{c}
\Bo_k(M,\spin)=\Co_k(M,\spin)/\G_{k+1}(M) \\
\Bs_k(M,\spin)=\Cs_k(M,\spin)/\G_{k+1}(M)
\end{IEEEeqnarray*}
In the latter case, the action of the gauge group is free and the quotient space is Hausdorff (see Proposition $9.3.1$ in the book). To show that it is a Hilbert manifold, we will explicitly construct slices for the gauge group action. There is for every $j\leq k$ a smooth bundle decomposition
\begin{equation*}
\T^{\sigma}_j=\J^{\sigma}_j\oplus\K^{\sigma}_j
\end{equation*}
defined as follows. The fiber $\J_{\gamma,k}^{\sigma}$ at a point $\gamma=(A_0,s_0,\phi_0)$ is the image of the derivative of the gauge group action
\begin{align}\label{devgauges}
\mathbf{d}^{\sigma}_{\gamma}: T_e\G_{k+1}(M)&=L^2_{k+1}(M,i\R)\rightarrow T_{\gamma}\Co_k^{\sigma}(M) \\ \xi&\mapsto (-d\xi,0,\xi\phi_0)\nonumber,
\end{align}
and similarly we define $\J_{\gamma,j}^{\sigma}$ for $j\leq k$ by extending the map in Sobolev spaces of lower regularity. The fiber of $\K^{\sigma}_j$ over a point $\gamma$ is defined as the subspace of $\T^{\sigma}_{\gamma,j}$ consisting of the space of triples $(a,s,\phi)$ satisfying
\begin{equation*}
\left.
\begin{aligned}
\langle a, \nu\rangle&=0 \text{ at }\partial M\\
-d^*a+is_0^2\mathrm{Re}\langle i\phi_0,\phi\rangle&=0 \\
\mathrm{Re}\langle i\phi_0,\phi\rangle_{L^2(M)}&=0.
\end{aligned}
\right\}
\end{equation*}
This decomposition can be thought as the extension to the boundary of the analogous decomposition on $\Co^*_k(M,\spin)\subset\Co_k(M,\spin)$ given by
\begin{equation*}
\T_j=\J_j\oplus \K_j.
\end{equation*}
Here as before $\J_{\gamma,j}$ at the point $\gamma=(A_0,\Phi_0)$ is the completion of the image of the derivative of the gauge group action
\begin{IEEEeqnarray*}{c}\label{devgauge}
\mathbf{d}_{\gamma}: T_e\G_{k+1}(M)=L^2_{k+1}(M,i\R)\rightarrow T_{\gamma}\Co_k(M) \IEEEyesnumber\\
\xi\mapsto (-d\xi,\xi\Phi_0)
\end{IEEEeqnarray*}
and $\K_{\gamma,k}$ is its orthogonal complement with respect to the standard $L^2$ inner product, namely
\begin{equation}\label{gaugefix}
\left\{(a,\phi)\mid -d^*a+i\mathrm{Re}\langle i\Phi_0,\phi\rangle=0 \text{ and }\langle a,\nu\rangle=0 \text{ at }\partial M\right\}.
\end{equation}
It is important to notice that the decomposition $\T^{\sigma}_j=\J^{\sigma}_j\oplus\K^{\sigma}_j$ is \textit{not} orthogonal with respect to any natural metric on the blown-up configuration space.
\\
\par
Given a configuration $\gamma\in\Cs_k(M,\spin)$ we are now ready to define the \textit{Coulomb-Neumann slice}  $\Sl^{\sigma}_{k,\gamma}$ through it. This will be a closed Hilbert submanifold containing $\gamma$ whose tangent space at $\gamma$ is $\K_k^{\sigma}$. A small open neighborhood of $\mathcal{U}\subset \Sl^{\sigma}_{k,\gamma}$ of $\gamma$ will provide a local charts for the Hilbert manifold $\Bs_k(M,\spin)$ near $[\gamma]$, via the composition
\begin{equation}\label{slice}
\Sl^{\sigma}_{k,\gamma}\hookrightarrow \Cs_k(M,\spin)\rightarrow \Bs(M,\spin)
\end{equation}
where the first map is the inclusion and the second is the quotient map, see Corollary $9.3.8$ in the book. We define $\Sl^{\sigma}_{k,\gamma}$ with $\gamma=(A_0,s_0,\phi_0)$ to be the closed Hilbert submanifold of $\Cs_k(M,\spin)$ consisting of triples $(A_0+a,s,\phi)$ satisfying
\begin{equation*}
\left.
\begin{aligned}
\langle a, \nu\rangle&=0 \text{ at }\partial M\\
-d^*a+iss_0\mathrm{Re}\langle i\phi_0,\phi\rangle&=0 \\
\mathrm{Re}\langle i\phi_0,\phi\rangle_{L^2(M)}&=0.
\end{aligned}
\right\}
\end{equation*}
In the case $\gamma$ is irreducible, one can define the affine subspace $\Sl_{k,\gamma}\subset\Co_k(M,\spin)$ given by
\begin{equation*}
\Sl_{k,\gamma}=\left\{(A_0+a,\Phi)\mid| -d^*a+i\mathrm{Re}\langle i\Phi_0,\Phi\rangle=0 \text{ and } \langle a,\nu\rangle=0\text{ at }\partial M\right\},
\end{equation*}
and $\Sl^{\sigma}_{k,\gamma}$ is simply the proper transform of this under blow-up.
\par
In the case $\gamma=(A_0, 0)$ is a reducible configuration the equations defining $\Sl_{k,\gamma}$ are simply
\begin{equation*}
\left\{
\begin{aligned}
d^*a& = 0 \\
\langle a, \nu\rangle&=0\text{ at }\partial M.
\end{aligned}
\right.
\end{equation*}
These equations define a global slice for the gauge group action. Looking for a gauge transformation of the form $u=e^{\xi}$ to put a given configuration in the slice $\Sl_{k,\gamma}$ is equivalent to solve the problem
\begin{align*}
\langle d\xi,\nu\rangle&=\langle a, \nu\rangle \text{ at }\partial M\\
\Delta\xi&=d^*a
\end{align*}
which is known to have a unique solution $\xi\in L^2_{k+1}$ such that
\begin{equation*}
\int_M\xi=0.
\end{equation*}
If we introduce the subgroup 
\begin{equation*}
\G^{\perp}_{k+1}=\left\{e^{\xi}\mid\int_M\xi=0\right\},
\end{equation*}
the action of the gauge group gives us a diffeomorphism
\begin{align*}
\G^{\perp}_{k+1}\times \K_{k,\gamma_0}&\rightarrow \Co_k \\ (e^{\xi}, (a,\phi))&\mapsto (A_0+(a-d\xi)\otimes 1, e^{\xi}\phi).
\end{align*}
The gauge group can be decomposed as $\G_{k+1}=\G^h\times \G^{\perp}_{k+1}$, where $\G^h$ is the group of the harmonic maps $M\rightarrow S^1$ sitting in the short exact sequence
\begin{equation}\label{harmonicgauge}
1\rightarrow S^1 \rightarrow \G^h\rightarrow H^1(M;\Z)\rightarrow 0.
\end{equation}
Such a sequence splits (but not canonically). In particular, up to homotopy one can identify
\begin{equation*}
\Bs=\K_{k,\gamma_0}^*/\G^h.
\end{equation*}
Such a space is a fiber bundle (by projecting the spinor away)
\begin{equation*}
(L^2_k(M;W)\setminus\{0\})/S^1\hookrightarrow \K_{k,\gamma_0}^*/\G^h\rightarrow H^1(M;i\R)/H^1(M;i\Z)
\end{equation*}
which is trivial by Kuiper's theorem on the contractibility of the unitary group of a Hilbert space. Hence it has the homotopy type of a product
\begin{equation*}\mathbb{CP}^{\infty}\times H^1(M;i\R)/H^1(M;i\Z),
\end{equation*}
and its cohomology ring (over $\Z$) is isomorphic to
\begin{equation*}
\Z[U]\otimes \Lambda^*(H_1(M;\Z)/\text{torsion}),
\end{equation*}
where $\deg U=2$. Similarly, one shows that $\Bo_k(M,\spin)$ has the homotopy type of a torus $H^1(M;i\R)/H^1(M;i\Z)$.
\\
\par
We define the Seiberg-Witten equations on the completions of the configuration spaces. In the case of a compact $4$-manifold (perhaps with boundary), we define for each $j\leq k$ the trivial vector bundle $\V_j(X,\spin)$ over $\Co_k(X,\spin)$ with fiber $L^2_j(X;i\mathfrak{su}(S^+)\oplus S^-)$, and the bundle
\begin{equation*}
\V_j^{\sigma}(X,\spin)\rightarrow \Cs_k(X,\spin)
\end{equation*}
as its pull-back under the blow-down map. The section $\F^{\sigma}$ then extends to these Sobolev completions as a section
\begin{equation*}
\F^{\sigma}: \Cs_k(X,\spin)\rightarrow \V^{\sigma}_{k-1}\subset\V^{\sigma}_j.
\end{equation*}
The gauge group acts smoothly on such bundles for $j\leq k+1$, and the section $\F^{\sigma}$ is $\G_{k+1}(X)$-equivariant for $j=k-1$. The case of a $3$-dimensional manifold is analogous. In this case we obtain the sections
\begin{align*}
\mathrm{grad}\CSD&:\Co_k(Y,\spin)\rightarrow T_{k-1}\\
(\mathrm{grad}\CSD)^{\sigma}&:\Co_k(Y,\spin)\rightarrow \T^{\sigma}_{k-1},
\end{align*}
which are both smooth.
\\
\par
We can also easily adapt this story in the case of the $\tau$ model on a cylinder $I\times Y$. We define
\begin{equation*}
\Ct_k(Z,\spin)\subset \A_k(Z,\spin)\times L^2_k(I,\R)\times L^2_k (Z; S^+)
\end{equation*}
consisting of the triples $(A,s,\phi)$ with 
\begin{itemize}
\item $s(t)\geq 0$;
\item $\|\phi(t)\|_{L^2(Y)}=1$ for all $t\in I$.
\end{itemize}
Notice that this space is not a Hilbert manifold (nor even a manifold with boundary) because of the condition $s\geq 0$, but is a closed subspace of the Hilbert manifold
\begin{equation*}
\tilde{\mathcal{C}}_k^{\tau}(Z,\spin)\subset \A_k(Z,\spin)\times L^2_k(I,\R)\times L^2_k (Z; S^+)
\end{equation*}
consisting of the triples $(A,s,\phi)$ with $\|\phi(t)\|_{L^2(Y)}=1$ for all $t\in I$. There is a natural involution of this space
\begin{equation}\label{involution}
\begin{aligned}
\mathbf{i}: \tilde{\mathcal{C}}_k^{\tau}(Z,\spin)&\rightarrow \tilde{\mathcal{C}}_k^{\tau}(Z,\spin)\\
(A,s,\phi)&\mapsto (A,-s,\phi)
\end{aligned}
\end{equation}
and the gauge group $\G_{k+1}(Z)$ acts smoothly and freely on it. The completion and slices story in this framework is essentially unchanged, and we refer the reader to Sections 9.2 and 9.4 in the book for the details.

\vspace{1.5cm}
\section{Perturbations}
As it is usual in Floer theory, we need to introduce suitable perturbations of the Seiberg-Witten equations in order to construct monopole Floer homology. The space of perturbations has to be a large enough in order to achieve transversality for the moduli spaces of solutions. On the other hand, these perturbations have to be mild enough so that the perturbed equations still have the nice properties (smoothness, unique continuation and compactness among the others) of the original ones. In this section we review the construction of the perturbations we will use, the ones arising as formal gradients of cylinder functions. We do this in quite detail as we will perform a variant of such a construction in Chapter $4$, and we start by discussing the abstract theory of perturbations.
\\.
\par
The perturbations we will deal with will arise as the formal gradient of a given gauge-invariant function $f: \Co(Y)\rightarrow \R$, i.e. a section
\begin{equation*}
\q:\Co(Y)\rightarrow \T_0
\end{equation*}
such that for every path $\gamma:[0,1]\rightarrow \Co(Y)$ we have
\begin{equation*}
f\circ \gamma(1)-f\circ\gamma(0)=\int_0^1\langle \dot{\gamma}, \q\rangle_{L^2(Y)}dt.
\end{equation*}
We will write
\begin{equation*}
\CSd=\CSD+f
\end{equation*}
for the perturbed Chern-Simons-Dirac functional, and write
\begin{equation*}\mathrm{grad}\CSd=\mathrm{grad}\CSD+\q,
\end{equation*}
which is a $\G(Y)$-invariant section of $\T_0\rightarrow \Co(Y)$.
\par
We can then pull-back such a perturbation to the cylinders $I\times Y$ in order to obtain a section
\begin{equation*}
\hat{\q}:\Co(Z)\rightarrow \V_0(Z)
\end{equation*}
as follows. From a configuration $(A,\Phi)\in\Co(Z)$ we obtain by restricting to slices a continuous path $(\check{A}(t),\check{\Phi}(t))$ in $\Co(Y)$, hence a continuous path $\q(\check{A}(t),\check{\Phi}(t))$ in $L^2(Y;iT^*Y\oplus S)$. By identifying $iT^*Y\oplus S$ with $i\mathfrak{su}(S^+)\oplus S^-$ via the Clifford multiplication on the first factor we obtain an element of $\V_0(Z)$.
\\
\par
In the following definition we recollect the analytic requirements we make on such a perturbation.
\begin{defn}\label{tamepert}
Let $k\geq 2$ an integer. We say that a perturbation
\begin{equation*}
\q:\Co(Y)\rightarrow \T_0
\end{equation*}
is $k$-\textit{tame} if it is the formal gradient of a $\G(Y)$-equivariant function on $\Co(Y)$ and it satisfies the following properties:
\begin{enumerate}
\item the associated $4$-dimensional perturbation $\hat{\q}$ defines a smooth section
\begin{equation*}
\hat{\q}:\V_k(Z)\rightarrow \Co_k(Z);
\end{equation*}
\item  for every integer $j\in[1, k]$ the section $\hat{\q}$ extends to a continuous section
\begin{equation*}
\hat{\q}:\V_j(Z)\rightarrow \Co_j(Z);
\end{equation*}
\item for every integer $j=[-k, k]$ the first derivative
\begin{equation*}\mathcal{D}\hat{\q}\in C^{\infty}\left(\Co_k(Z), \mathrm{Hom}(T\Co_k(Z),\V_k(Z))\right)
\end{equation*}
extends to a map
\begin{equation*}\mathcal{D}\hat{\q}\in C^{\infty}\left(\Co_k(Z), \mathrm{Hom}(T\Co_j(Z),\V_j(Z))\right);
\end{equation*}
\item there is a constant $m_2$ such that
\begin{equation*}
\|\q(B,\Psi)\|_{L^2}\leq m_2\left(\|\Psi\|_{L^2}+1\right)
\end{equation*}
for every $(B,\Psi)\in \Co_k(Y)$;
\item for any fixed smooth connection $A_0$, there is a real function $\mu_1$ such that the inequality
\begin{equation*}
\|\hat{\q}(A,\Phi)\|_{L^2_{1,A}}\leq\mu_1(\|(A-A_0,\Phi)\|_{L^2_{1,A_0}})
\end{equation*}
holds for all configurations $(A,\Phi)\in \Co_k(Z)$;
\item the $3$-dimensional perturbation $\q$ defines a $C^1$ section
\begin{equation*}
\q:\Co_1(Y)\rightarrow\T_0.
\end{equation*}
\end{enumerate}
We say that $\q$ is \textit{tame} if it is $k$-tame for every $k\geq 2$.
\end{defn}
\vspace{0.5cm}
\par
The perturbed equations are defined as follows. Using the inclusion $\V_k\subset\V_{k-1}$, we can define the section
\begin{equation*}\F_{\q}=\F+\hat{\q}
\end{equation*}
and the perturbed Seiberg-Witten equations on $Z$ as
\begin{equation*}
\F_{\q}(A,\Phi)=0.
\end{equation*}
In a more expanded version, we write the perturbation as $\q=(\q^0,\q^1)$ with
\begin{align*}
\q^0&\in L^2(Y;iT^*Y)\\
\q^1&\in L^2(Y;S)
\end{align*}
and the induced $4$-dimensional perturbation as $\hat{\q}=(\hat{\q}^0,\hat{\q}^1)$, where
\begin{align*}\hat{\q}^0&\in L^2(Z; i\mathfrak{su}(S^+))\\
\hat{\q}^1&\in L^2(Z;S^-).
\end{align*}
The equation $\F_{\q}(A,\Phi)=0$ can then be written as 
\begin{equation*}
\left.
\begin{aligned}
\rho_Z(F_{A^t}^+)-2(\Phi\Phi^*)_0&= -2\hat{\q}^0(A,\Phi)\\
D_A^+\Phi& =-\hat{\q}^1(A,\Phi)
\end{aligned}
\right\}
\end{equation*}
or, when interpreted as a gradient-flow equation, in the form
\begin{equation*}
\left\{
\begin{aligned}
\frac{d}{dt}B^t& = -\ast F_{B^t}-2\rho^{-1}(\Psi\Psi^*)_0-2\q^0(B,\Psi)\\
\frac{d}{dt}\Psi &= -D_B\Psi-\q^1(B,\Psi).
\end{aligned}
\right.
\end{equation*}
To write the perturbed equations on the blow-up, one notices that the section $\hat{\q}$ of $\V_k$ gives rise to a section
\begin{equation*}
\hat{\q}^{\sigma}:\Cs_k(Z)\rightarrow \V_k^{\sigma}
\end{equation*}
as follows. Gauge invariance implies that
\begin{equation*}\hat{\q}^1(A,0)=0,
\end{equation*}
hence one can define the section
\begin{align*}
\hat{\q}^{1,\sigma}&:\Cs_k(Z)\rightarrow \V_k^{1,\sigma}\\
(A,s,\phi)&\mapsto\int_0^1 (\mathcal{D}_{(A,rs\phi)}\hat{\q}^1)(\phi)dr
\end{align*}
and finally
\begin{equation*}
\hat{\q}^{\sigma}=(\hat{\q}^0, \hat{\q}^{1,\sigma}).
\end{equation*}
It is straightforward that $\hat{\q}^{\sigma}$ is a smooth section as $\hat{\q}$ is. In any case,
\begin{equation*}
\F^{\sigma}_{\q}=\F^{\sigma}+\hat{\q}^{\sigma}: \Co_k^{\sigma} \rightarrow\V_{k-1}^{\sigma}
\end{equation*}
is a smooth section and $\F^{\sigma}_{\hat{\q}}=0$ are the \textit{perturbed Seiberg-Witten equations on the blow-up}.
\\
\par
There is also a natural $3$-dimensional counterpart of the last construction, giving rise to the perturbed $3$-dimensional gradient in the blow-up. The perturbation $\q$ gives rise to a  perturbation $\q^{\sigma}$ on the blow-up $\Cs_k(Y)$, and we can write 
\begin{equation*}(\mathrm{grad}\CSd)^{\sigma}=(\mathrm{grad}\CSD)^{\sigma}+\q^{\sigma},
\end{equation*}
which is a smooth section of the vector bundle $\T_{k-1}^{\sigma}\rightarrow\Co_k^{\sigma}$.
The equations of the gradient flow of a path $(B(t), r(t), \psi(t))$ are
\begin{equation*}
\left\{
\begin{aligned}
\frac{d}{dt}B^t& = -\ast F_{B^t}-2r^2\rho^{-1}(\psi\psi^*)_0-2\q^0(B,r\psi)\\
\frac{d}{dt}r &=-\Lambda_{\q}(B,r,\psi)r\\
\frac{d}{dt}\psi& = -D_B\Psi-\tilde{\q}^1(B,\Psi)+\Lambda _{\q}(B,r,\psi)r
\end{aligned}
\right.
\end{equation*}
where $\tilde{\q}^1$ is defined similarly to $\hat{\q}^{1,\sigma}$ to be
\begin{equation*}
\tilde{\q}^1(B,r,\psi)=\int_0^1\mathcal{D}_{(B,sr\psi)}\q^1(0,\psi)ds,
\end{equation*}
and we have defined the function
\begin{equation*}
\Lambda_{\q}(B,r,\psi)=\mathrm{Re}\langle \psi, D_B\psi+\tilde{\q}^1(B,r,\psi)\rangle_{L^2}.
\end{equation*}
Writing this in coordinates we have that
\begin{equation*}
\q^{\sigma}(B,r,\psi)=\left(\q^0(B,r\psi),\quad\langle\tilde{\q}^1(B,r,\psi),\psi\rangle_{L^2(Y)}r,\quad\tilde{\q}(B,r,\psi)^{\perp}\right)
\end{equation*}
where $\perp$ denotes the orthogonal projection to the real orthogonal complement of $\psi$. We denote by $D_{B,\q}$ the operator on sections of $S$ defined as
\begin{equation*}
D_{B,\q}\psi= D_B\psi+\mathcal{D}_{(B,0)}\q^1(0,\psi).
\end{equation*}
Then a critical point $(B,r,\psi)$ of the gradient is of the form:
\begin{itemize}
\item $r\neq 0$ and $(B,r\psi)$ is a critical point of $\mathrm{grad}\CSd$;
\item $r=0$, $(B,0)$ is a critical point of $\mathrm{grad}\CSd$ and $\psi$ is an eigenvector of $D_{B,\q}$.
\end{itemize}
The story for the $\tau$ model is essentially identical, see Section 10.4 in the book.

\vspace{0.8cm}
There is a very large class of tame perturbations arising as formal gradients of \textit{cylinder functions} $f:\Co(Y)\rightarrow\R$. We recall their construction (as it will be very useful later in Chapter $4$) and their most important properties. 
\par
Given a coclosed $1$-form $c\in\Omega^1(Y;i\R)$ we can define the function
\begin{IEEEeqnarray*}{c}
r_c: \Co(Y)\rightarrow\R\\
(B_0+b\otimes 1,\Psi)\mapsto \int_Y b\wedge\ast\bar{c}=\langle b,c\rangle_Y.
\end{IEEEeqnarray*}
This is generally invariant only under the identity component $\G^e$ of the gauge group, but it is fully gauge invariant if $c$ is coexact. One also has the $\G$-invariant map
\begin{IEEEeqnarray*}{c}
\Co(Y)\rightarrow  \mathbb{T}=H^1(Y;i\R)/(2\pi i H^1(Y;\Z))\\
(B_0+b\otimes 1,\Psi)\mapsto [b_{\mathrm{harm}}]
\end{IEEEeqnarray*}
where $b_{\mathrm{harm}}$ denotes the harmonic part of the $1$-form $b$ and the brackets denote the equivalence class. Picking an integral basis $\omega_1,\dots \omega_t$ for $H^1(Y;\R)$ we can identify the torus $\mathbb{T}$ with $\R^t/(2\pi\Z^t)$ and the map can be written as
\begin{equation*}
(B,\Psi)\mapsto \left(r_{\omega_1}(B,\Psi),\dots, r_{\omega_t}(B,\Psi)\right) \quad (\mathrm{mod}\text{ }2\pi \Z^t).
\end{equation*}
Consider now a splitting $v$ of the exact sequence in equation (\ref{harmonicgauge}),
\begin{equation*}
S^1\rightarrow \G^h\stackrel{v}{\leftrightarrows} H^1(Y;\Z).
\end{equation*}
We can choose it so that
\begin{equation*}
\G^{h,o}:=v(H^1(Y;\Z))=\{u\mid u(y_0)=1\}\subset \G^h
\end{equation*}
for some chosen basepoint $y_0\in Y$. We then define the subgroup
\begin{equation*}
\G^o(Y)=\G^{h,0}\times \G^{\perp}(Y)\subset\G(Y)
\end{equation*}
which acts freely on $\Co(Y)$, and consider the \textit{based configuration space} as the quotient of the respective completions
\begin{equation*}
\Bo^o_k(Y)=\Co_k(Y)/\G^o_{k+1}(Y).
\end{equation*}
This space is a Hilbert manifold and the space $\Bo_k(Y)$ is obtained as the quotient by the remaining circle action
\begin{equation*}
(B,\Psi)\mapsto (B,e^{i\vartheta}\Psi).
\end{equation*}

Consider now the smooth bundle $\mathbb{S}$ over $\mathbb{T}\times Y$ obtained as the quotient of $H^1(Y;i\R)\times S\rightarrow H^1(Y;i\R)\times Y$ by the group $\G^{h,o}$. Any smooth section $\Upsilon$ of $\mathbb{S}$ can be lifted to a section
\begin{equation*}
\tilde{\Upsilon}:H^1(Y;i\R)\times Y\rightarrow H^1(Y;i\R)\times S
\end{equation*}
respecting the following quasi-periodicity condition. For every $\kappa\in H^1(Y;\Z)$  we have
\begin{equation*}
\tilde{\Upsilon}_{\alpha+\kappa}(y)=u(y)\tilde{\Upsilon}_{\alpha}(y),
\end{equation*}
where $u=v(\kappa)\in \G^{h,o}$ and we write $\tilde{\Upsilon}_b(y)$ instead of $\tilde{\Upsilon}(b,y)$. Hence a section $\Psi$ of $\mathbb{S}$ gives rise to the $\G^o(Y)$-equivariant map
\begin{align*}
\Upsilon^{\dagger}: \Co(Y)&\rightarrow C^{\infty}(S)\\
(B_0+b\otimes 1,\Psi)&\mapsto e^{-Gd^*b}\tilde{\Upsilon}_{b_{\mathrm{harm}}},
\end{align*}
where
\begin{equation*}
G: L^2_{k-1}(Y)\rightarrow L^2_{k+1}(Y)
\end{equation*}
is the Green's operator of $\Delta=d^*d$. In turn we define the $\G^o(Y)$-invariant map
\begin{align*}
q_{\Upsilon}&:\Co(Y)\rightarrow \C \\
(B,\Psi)&\mapsto\int_Y \langle \Psi, \Upsilon^{\dagger}(B,\Psi)\rangle=\langle\Psi,\tilde{\Upsilon}^{\dagger}\rangle_Y,
\end{align*}
which is also equivariant for the remaining $S^1$ action if we make $S^1$ act on $\C$ by complex multiplication.
\\
\par
Choose a finite collection of coclosed $1$-forms $c_1,\dots ,c_{n+t}$ with the first $n$ being coexact and the remaining $t$ representing an integral basis of $H^1(Y;\R)$, and a collection of smooth sections $\Upsilon_1,\dots ,\Upsilon_m$ of $\mathbb{S}$. This gives rise to the map
\begin{IEEEeqnarray*}{c}
p:\Co(Y)\rightarrow \R^n\times\mathbb{T}\times\C^m \\
(B,\Psi)\mapsto \left(r_{c_1}(B,\Psi),\dots, r_{c_{n+t}}(B,\Psi),q_{\Upsilon_1}(B,\Psi),\dots , q_{\Upsilon_m}(B,\Psi)\right)\quad (\mathrm{mod }2\pi \Z^t)
\end{IEEEeqnarray*}
which is $\G^o(Y)$-invariant and equivariant with respect to the remaining $S^1$ action.

\begin{defn}
We say that a gauge invariant function $f:\Co(Y)\rightarrow \R$ is a \textit{cylinder function} if it is of the form $f=g\circ p$ where:
\begin{itemize}
\item the map $p:\Co(Y)\rightarrow \R^n\times\mathbb{T}\times\C^m$ is defined as above with any choice of $n$ coclosed $1$-forms and $m$ sections of $\mathbb{S}$, $n,m\geq0$;
\item the function $g$ is an $S^1$-invariant smooth function on $\R^n\times\mathbb{T}\times\C^m$ with compact support.
\end{itemize}
\end{defn}
We summarize the two main features of cylinder functions in the following proposition (see Chapter $11$ in the book).
\begin{prop}\label{densepert}
If $f$ is a cylinder function, then its formal gradient
\begin{equation*}
\mathrm{grad}f: \Co(Y)\rightarrow \T_0
\end{equation*}
is a tame perturbation. Furthermore, for every compact subset $K$ of a finite dimensional submanifold $M\subset \Bo_k^o(Y)$, which are both $S^1$-invariant, one can find a collection of coclosed forms $c_{\mu}$, sections $\Upsilon_{\nu}$ and a neighborhood $U$ of $K$ inside $M$ such that
\begin{equation*}
p:\Bo_k^o(Y)\rightarrow\R^n\times\mathbb{T}\times \mathbb{C}^m
\end{equation*}
gives an embedding of $U$. 
\end{prop}

\vspace{0.3cm}
Finally, in order to apply the Sard-Smale theorem we will need to be in a Banach manifold setup. To do this, we first specify a countable collection of cylinder functions as follows. For every pair $(n,m)$, choose a countable collection of $(n+m)$-tuples $(c_1,\dots, c_n,\Upsilon_1,\dots, \Upsilon_m)$ which are dense in the $C^{\infty}$ topology in the space of all such $(n+m)$-tuples. Choose a countable collection of compact subsets $K\subset \R^n\times\mathbb{T}\times \C^m$ which is dense in the Hausdorff topology and, for each $K$, a countable collection of $S^1$-invariant functions $g_{\alpha}$ with support in $K$ which are dense (in the $C^{\infty}$ topology) in the space of $S^1$-invariant functions with support in $K$. We also require that the subset of $g_{\alpha}$ that vanish on
\begin{equation*}
K\cap(\R^n\times\mathbb{T}\times\{0\})
\end{equation*}
is dense (in the $C^{\infty}$ topology) in the space of $S^1$-invariant functions with support on $K$ and vanishing on $K\cap (\R^n\times \mathbb{T}\times\{0\})$. Combining all these choices, we obtain a countable collection of cylinder functions, and denote the corresponding set of tame perturbations by $\{\q_i\}_{i\in\mathbb{N}}$. We can then construct a Banach space of perturbations containing all the ones in this family, as stated in the next result.

\begin{prop}\label{largepert}
There exists a separable Banach space $\mathcal{P}$ and a linear map
\begin{IEEEeqnarray*}{c}
\mathfrak{D}:\mathcal{P}\rightarrow\mathcal C^0(\Co(Y),\T_0)\\
\lambda\mapsto \q_{\lambda}
\end{IEEEeqnarray*}
such that every $\q_{\lambda}$ is a tame perturbation and the image contains the family $\{\q_i\}_{i\in\mathbb{N}}$. Furthermore we have that:
\begin{itemize}
\item for a cylinder $Z=I\times Y$ and all $k\geq2$, the map
\begin{equation*}
\mathcal{P}\times \Co_k(Z)\rightarrow \V_k(Z)
\end{equation*}
given by $(\lambda,\gamma)\mapsto \hat{\q}_{\lambda}(\gamma)$ is smooth;
\item the map $\mathcal{P}\times\Co_1(Y)\rightarrow\T_1(Y)$ given by $(\lambda,\beta)\mapsto\q_{\lambda}(\beta)$ is continuous and satifies bounds
\begin{align*}
\|\q_{\lambda}(B,\Psi)\|_{L^2}&\leq \|\lambda\| m_2(\|\Psi\|_{L^2}+1) \\
\|\q_{\lambda}(B,\Psi)\|_{L^2_{1,A_0}}&\leq\|\lambda\|\mu_1\left(\|B-B_0,\Psi\|_{L^2_{1,A_0}}\right)
\end{align*}
for some constants $m_2$ and a real function $\mu_1$.
\end{itemize}
\end{prop}
By a \textit{large Banach space of tame perturbations} we will mean a pair $(\mathcal{P},\mathfrak{D})$ arising from the proposition. In the rest of the work we will always assume that a large Banach space of perturbations is fixed, and we will identify it with its image (and denote it by $\mathcal{P}$), even if it is clear that the Banach space topology is not the topology as a subspace of the set of tame perturbation, for any any natural choice of the topology on the latter.

\chapter{The analysis of Morse-Bott singularities}
In this chapter we study the analytical and structural properties of the space of solutions to the Seiberg-Witten equations when the singularities of the (blown-up) gradient are \textit{Morse-Bott}. This kind of singularities is not generic, but arises quite naturally when dealing with many problems, e.g. when dealing with the equation on a Seifert-fibered space (see \cite{MOY}). The case in which we will be interested in the most is the case of a spin$^c$ structure induced by a spin structure, which will be discussed in the last chapter of the present work. The content of this chapter can be thought as a generalization of the material appearing in Chapters $12$ to $19$ in the book, to which we refer as the \textit{Morse} or \textit{classical} case. We will often make use of results proven there providing precise references.
\vspace{1.5cm}

\section{Hessians and Morse-Bott singularities}
We define the class of differential operators we will be dealing with, namely $k$-\textit{almost self-adjoint first-order elliptic operators} (abbreviated $k$-\textsc{asafoe}), see Section $12.2$ in the book. An operator $L$ is called $k$-\textsc{\textsc{asafoe}} if it can be written as
\begin{equation*}
L=L_0+h
\end{equation*}
where:
\begin{itemize}
\item $L_0$ is a first order self-adjoint elliptic differential operator with smooth coefficients, acting on sections of a vector bundle $E\rightarrow Y$;
\item $h$ is an operator on sections of $E$
\begin{equation*}
h:C^{\infty}(Y;E)\rightarrow L^2(Y;E)
\end{equation*}
which extends to a bounded map on the spaces $L^2_j(Y;E)$ for all $|j|\leq k$.
\end{itemize}
We also say that $L$ is \textsc{\textsc{asafoe}} if it is $k$-\textsc{\textsc{asafoe}} for every $k\geq0$. So $k$-\textsc{\textsc{asafoe}} operators are a class of (non necessarily symmetric) compact perturbations of first order elliptic self-adjoint operators. The main example to keep in mind is that of the perturbed Dirac operators $D_{\q,B}$, which indeed motivates the definition. This class of operators satifies very nice regularity and spectral properties, see Section $12.2$ in the book. We will call a $k$-\textsc{asafoe} operator \textit{hyperbolic} if its spectrum does not intersect the imaginary line.
\\
\par
We first define Morse-Bott singularities in the non blown-up configuration space. Recall that the irreducible part $\Co_k^*(Y)\subset\Co_k(Y)$ we have the decomposition
\begin{equation}\label{dectan}
\T_j\lvert_{\Co_k^*}=\J_j\oplus\K_j,
\end{equation}
see Section $3$ of Chapter $1$. As this decomposition is orthogonal, the vector field $\mathrm{grad}\CSd$ is a section of 
\begin{equation*}
\K_{k-1}\subset\T_{k-1},
\end{equation*}
so we can define the Hessian operator at a configuration $\alpha\in\Co^*_k(Y)$
\begin{equation*}
\Hess_{\q,\alpha}: \K_{k,\alpha}\rightarrow\K_{k-1,\alpha}
\end{equation*}
as the restriction to $\K_{k,\alpha}\subset\T_{k,\alpha}$ of the linear map
\begin{equation*}
\Pi_{\K_{k-1}}\circ\mathcal{D}_{\alpha}\left(\mathrm{grad}\CSd\right): T_{\alpha}\Co^*_k(Y)=\T_{k,\alpha}\rightarrow \K_{k-1,\alpha}
\end{equation*}
where $\Pi_{\K_{k-1}}:\T_{k-1}\rightarrow \K_{k-1}$ is the $L^2$ orthogonal projection with kernel $\J_{k-1}$. As $\alpha$ varies, one obtains a $\G_{k+1}$-equivariant smooth bundle map
\begin{equation*}
\Hess_{\q}: \K_k\rightarrow\K_{k-1}.
\end{equation*}
The operators $\Hess_{\q,\alpha}$ are symmetric, as they are the pull back of the covariant Hessian of the circled valued function $\CSd$ on $\Bo^*(Y)$. Here we identify $\K_k$ with the pull-back of the tangent bundle of $\Bo^*(Y)$, and $\K_{k-1}$ with the pull back of $[\T_{k-1}]=(\T_{k-1}/\J_{k-1})/\G_{k+1}$. It also satisfies nice spectral properties which are manifest from the following construction (see Proposition $12.3.1$ in the book). We can introduce the \textit{extended Hessian} at a configuration $\alpha$ as the map
\begin{equation*}
\widehat{\Hess}_{\q,\alpha}:\T_{k,\alpha}\oplus L^2_k(Y;i\R)\rightarrow \T_{k-1,\alpha}\oplus L^2_{k-1}(Y;i\R)
\end{equation*}
given by the matrix
\begin{equation*}
\widehat{\Hess}_{\q,\alpha}=
\begin{bmatrix}
\mathcal{D}_{\alpha} \mathrm{grad}\CSD & \mathbf{d}_{\alpha} \\
\mathbf{d}^*_{\alpha} & 0
\end{bmatrix}.
\end{equation*}
Here $\mathbf{d}_{\alpha}$ is the linearization of the gauge group action at $\alpha$, see equation (\ref{devgauge}), and $\mathbf{d}^*_{\alpha}$ is the linear operator (\ref{gaugefix})  defining $\K_{k,\alpha}$ (which coincides with the adjoint of $\mathbf{d}_{\alpha}$).
Such an operator is symmetric $k$-\textsc{\textsc{asafoe}},. This can be seen using the decomposition
\begin{equation*}
\T_{j,\alpha}= L^2_j(Y;S)\oplus L^2_j(Y;iT^*Y).
\end{equation*}
The standard Hessian appears using the decomposition
\begin{equation*}
\T_{j,\alpha}=\J_{j,\alpha}\oplus\K_{j,\alpha}
\end{equation*}
as a block in the matrix
\begin{equation*}
\widehat{\Hess}_{\q,\alpha}=
\begin{bmatrix}
0 & 0 & \mathbf{d}_{\alpha} \\
0 & \Hess_{\q,\alpha} & 0 \\
\mathbf{d}^*_{\alpha} & 0 & 0
\end{bmatrix}+ \tilde{h}
\end{equation*}
where $\tilde{h}$ is symmetric, bounded between the spaces of lower regularity and vanishes at a critical point. Hence we obtain that $\Hess_{\q,\alpha}$ has a complete orthornormal system $\{e_n\}$ of smooth eigenvectors with real eigenvalues which is dense in each $\K_{j,\alpha}$, such that there is a finite number of eigenvalues in every bounded interval. Finally $\Hess_{\q,\alpha}$ is Fredholm of index zero, hence it is surjective if and only if it is injective.
\\
\par
Suppose now that we have a smooth submanifold $[C]\subset \Bo^*_k(Y)$ (which we will always suppose to be closed, connected and finite dimensional) consisting of gauge equivalence classes critical points of $\mathrm{grad}\CSd$. A standard regularity argument implies that every $[\acr]\in[C]$ always admits a smooth representative (see Corollary $12.2.6$ in the book). In particular the set of gauge-equivalence classes of critical points is independent of the choice of $k$. Call $C\subset \Co_k^*(Y)$ the corresponding gauge-invariant critical set. Given $\alpha\in C$, define
\begin{equation*}
\T_{\alpha}C\subset \K_{k,\alpha}
\end{equation*}
to be the inverse image of $\T_{[\alpha]}[C]$ via the identification $\K_{k,\alpha}\rightarrow [T_k]$ arising from the local chart provided by the slice
\begin{equation*}
\bar{\i}:\mathcal{U}\subset \Sl_{k,\alpha}\rightarrow \Bo_k(Y).
\end{equation*}
This gives rise to a finite dimensional subbundle $\T C\subset \K_{j}\lvert_C$ over $C$ which is invariant under the action of the gauge group. We can then define the \textit{normal bundle} to $C$ as
\begin{equation*}
\N_j=\K_{j}/\T C\rightarrow C.
\end{equation*}
As $C$ consists of critical points for $\mathrm{grad}\CSd$, the space $\T_{\alpha}C$ is contained in the kernel of $\Hess_{\q,\alpha}$, and hence there is an induced map
\begin{equation*}
\Hess_{\q,\alpha}^{n}: \N_{k,\alpha}\rightarrow \N_{k-1,\alpha}
\end{equation*}
called the \textit{normal Hessian}. This is symmetric and inherits all the nice spectral properties of the Hessian discussed above. It is Fredholm of index $0$ and it defines a smooth gauge-equivariant map 
\begin{equation}\label{normalhess}
\Hess_{\q}^{\nu}: \N_{k}\rightarrow \N_{k-1}
\end{equation}
of bundles over $C$. Even though the decomposition (\ref{dectan}) does not extend smoothly to the whole configuration space, all the definitions we have provided continue to make sense without any change also for critical submanifolds $[C]$ \textit{consisting entirely of reducible configurations}. For such critical submanifolds one can define the normal Hessian as in the irreducible case \ref{normalhess}.
We are now ready to define the notion of Morse-Bott singularity.

\begin{defn}\label{Morsebott}
We say that a (closed, connected and finite dimensional) submanifold of gauge equivalence classes of critical points $[C]\subset \Bo_k(Y)$ with the property that if $[C]$ contains a reducible configuration then in consists entirely of reducible configurations is a \textit{Morse-Bott singularity} if for every $\alpha\in C$ the normal Hessian
\begin{equation*}
\Hess_{\q,\alpha}^{n}: \N_{k,\alpha}\rightarrow \N_{k-1,\alpha}
\end{equation*}
is an isomorphism. We say that $\mathrm{grad}\CSd$ and the perturbation $\q$ are \textit{Morse-Bott} if the set of gauge equivalence classes of critical points is an union of Morse-Bott singularities.
\end{defn}
For brevity, we will call a submanifold of gauge equivalence classes of critical points simply a \textit{critical submanifold}. It is an immediate consequence of the definition that for a Morse-Bott $\mathrm{grad}\CSd$ all the critical submanifolds are isolated, as in that case $\T_{\alpha}C$ is exactly the kernel of $\Hess_{\q,\alpha}$. Furthermore the requirement that the submanifold is finite dimensional is redundant as it follows from the spectral properties of the Hessian.
\par
The compactness properties of the space of solutions of the (perturbed) Seiberg-Witten equations (see Section $10.7$ in the book) imply that there will be only finitely many critical submanifolds. Finally, it is clear that Morse-Bott singularities are a generalization of non-degenerate singularities (see Section $12.1$ in the book), which are Morse-Bott singularities for which all critical submanifolds consist of points.

\vspace{0.8cm}
On the blown up configuration space, the situation is essentially analogous, with some small complications because the vector field $(\mathrm{grad}\CSd)^{\sigma}$ is not the gradient of a function in any natural way.
The Hessian
\begin{equation*}
\Hess_{\q}^{\sigma}:\K_{k}^{\sigma}\rightarrow \K_{k-1}^{\sigma}
\end{equation*}
is obtained by restricting to $\K_{k}^{\sigma}\subset\T_{k}^{\sigma}$ the smooth bundle map
\begin{IEEEeqnarray*}{c}
\T_{k}^{\sigma}\rightarrow\K_{k-1}^{\sigma} \\ x\mapsto \Pi_{\K_{k-1}^{\sigma}}\circ\mathcal{D}(\mathrm{grad}\CSd)^{\sigma}(x).
\end{IEEEeqnarray*}
Here $\mathcal{D}(\mathrm{grad}\CSd)^{\sigma}$ is the vector field obtained by differentiating $(\mathrm{grad}\CSd)^{\sigma}$ as a vector field along the submanifold
\begin{equation*}\Cs_k(Y)\subset\A_k\times\R\times L^2_k(Y;S).
\end{equation*}
As in the previous case, it is useful to work with the \textit{extended Hessian} 
\begin{equation*}
\widehat{\Hess}_{\q,\alpha}:\T_{k,\alpha}\oplus L^2_k(Y;i\R)\rightarrow \T_{k-1,\alpha}\oplus L^2_{k-1}(Y;i\R),
\end{equation*}
which will also be a protagonist when studying moduli spaces on infinite cylinders and is defined as follows. Consider the operator
\begin{equation*}
\mathbf{d}_{\acr}^{\sigma,\dagger}:\T_{k,\acr}^{\sigma}\rightarrow L_{k-1}^2(Y;i\R)
\end{equation*}
given at configuration $\acr=(B_0,s_0,\psi_0)$ by the map
\begin{equation*}
(b,s,\psi)\mapsto -d^*b+is_0^2\mathrm{Re}\langle i\psi_0,\psi\rangle+i|\psi_0|^2\mathrm{Re}(\mu_Y\langle i\psi_0,\psi\rangle).
\end{equation*}
This is the operator defining the subspace $\K^{\sigma}_{k,\acr}$ (see Section $3$ of Chapter $1$), and it is important to notice that this is \textit{not} the adjoint of the differential of the gauge group action
\begin{equation*}
\mathbf{d}_{\acr}^{\sigma}:L_{k-1}^2(Y;i\R)\rightarrow \T_{k,\alpha}^{\sigma},
\end{equation*}
defined in equation (\ref{devgauges}). Then we define the extended Hessian by the matrix
\begin{equation}\label{exthessian}
\widehat{\Hess}_{\q,\alpha}^{\sigma}=\begin{bmatrix}
\mathcal{D}_{\acr}(\mathrm{grad}\CSd)^{\sigma} & \mathbf{d}_{\alpha}^{\sigma} \\
\mathbf{d}_{\acr}^{\sigma,\dagger} & 0
\end{bmatrix}.
\end{equation}
Using the further decomposition
\begin{equation*}
\T^{\sigma}_{k,\acr}=\J^{\sigma}_{k,\acr}\oplus\K^{\sigma}_{k,\acr}
\end{equation*}
we see that the operator has the form
\begin{equation*}
\widehat{\Hess}_{\q,\alpha}^{\sigma}=\begin{bmatrix}
0 & x & \mathbf{d}_{\alpha}^{\sigma} \\
y & \Hess_{\q,\alpha}^{\sigma} & 0\\
\mathbf{d}_{\acr}^{\sigma,\dagger} & 0 & 0
\end{bmatrix},
\end{equation*}
and at a critical point $\acr$ the terms $x$ and $y$ vanish, so that the $\Hess_{\q,\acr}^{\sigma}$ is a direct summand of the extended Hessian. Notice though that the extended Hessian is not a $k$-\textsc{\textsc{asafoe}} operator, for the simple reason that it does not act on the space of sections of a fixed vector bundle over $Y$. This problem can be fixed by means of the following construction (which will be used again in the rest of the present work): we can convert an element
\begin{equation*}
(b,r,\psi)\in \T_{j,\acr}^{\sigma}
\end{equation*}
where $\acr=(B_0,r_0,\psi_0)$ to a section of the vector bundle $iT^*Y\oplus\R\oplus S$ simply by setting
\begin{equation}\label{asafoetrick}
(b,r,\psi)\mapsto (b,r, \boldsymbol\psi) \qquad \text{where} \qquad \boldsymbol\psi=\psi+r\psi_0.
\end{equation}
From this perspective we see that the extended Hessian is a $k$-\textsc{\textsc{asafoe}} operator. (See Section $12.4$ in the book for details).
\\
\par
As before, suppose we are given a smooth submanifold $[\mathfrak{C}]\subset\Bs_k(Y)$ consisting of critical points of $(\mathrm{grad}\CSd)^{\sigma}$. We will always suppose that such a submanifold is closed, connected and finite dimensional. Furthermore, we suppose that if it contains a reducible configuration, then all configurations in it are reducible. Call $\mathfrak{C}\subset\Cs_k(Y)$ the corresponding gauge-invariant critical set. As before, given $\acr\in\mathfrak{C}$, we can define
\begin{equation*}
\T_{\acr}\mathfrak{C}\subset \K_{k,\acr}
\end{equation*}
to be the inverse image of $\T_{\acr}[\mathfrak{C}]$ via the identification $\K^{\sigma}_{k,\alpha}\rightarrow [\T_k]$ provided by the local chart given by the slice
\begin{equation*}
\bar{\iota}:\mathcal{U}\subset\Sl^{\sigma}_{k,\acr}\rightarrow \Bo_k(Y).
\end{equation*}
This gives rise to a finite dimensional vector bundle
\begin{equation*}
\T\mathfrak{C}\rightarrow \mathfrak{C}
\end{equation*}
which is a subbundle of $\K^{\sigma}_j\lvert_{\Cr}$ for each $j=1,\dots k$, and we can define the normal bundle
\begin{equation*}
\N^{\sigma}_j=\K^{\sigma}_j/\T\mathfrak{C}.
\end{equation*}
Now, $\Cr$ consists of critical points of $(\mathrm{grad}\CSd)^{\sigma}$ hence $\T_{\acr}\Cr$ is contained in the kernel of $\Hess^{\sigma}_{\q,\acr}$. We can then define the \textit{normal Hessian in the blow up} as the operator
\begin{equation*}
\Hess^{\sigma,n}_{\q,\acr}:\N^{\sigma}_{\acr,k}\rightarrow \N^{\sigma}_{\acr,k-1}
\end{equation*}
which also induces a smooth gauge-equivariant map
\begin{equation*}
\Hess^{\sigma,n}_{\q}:\N^{\sigma}_{k}\rightarrow \N^{\sigma}_{k-1}
\end{equation*}
between bundles over $\Cr$.
\begin{defn}\label{Morsebotts}
We say that a (closed, connected and finite dimensional) submanifold of gauge equivalence classes of critical points $[\Cr]\subset\Bs_k(Y)$ is a \textit{Morse-Bott singularity} if the following conditions hold:
\begin{itemize}
\item $[\Cr]$ consists entirely of reducible \textit{or} irreducible configurations;
\item for every $\acr\in\Cr$ the normal Hessian
\begin{equation*}
\Hess^{\sigma,n}_{\q,\acr}:\N^{\sigma}_{\acr,k}\rightarrow \N^{\sigma}_{\acr,k-1}
\end{equation*}
is an isomorphism;
\item in the case $[\Cr]$ consists of reducible configurations, we also require the condition that its blow-down $[C]$ is also a reducible Morse-Bott critical submanifold (in the sense of Definition \ref{Morsebott}), and the restriction of the blowdown map is a fibration.
\end{itemize}
We say that $(\mathrm{grad}\CSd)^{\sigma}$ and the perturbation $\q$ are \textit{Morse-Bott} if all submanifolds of gauge equivalence classes of critical points are Morse-Bott singularities. \end{defn}

As before, for brevity we will call $[\Cr]$ a critical submanifold. Furthermore, Morse-Bott submanifolds are isolated, as $\T_{\acr}\Cr$ is exactly the kernel of $\Hess^{\sigma}_{\q,\acr}$. The blow down induces a bijective correspondence between irreducible critical points, and the condition of being a Morse-Bott singularity is preserved under blow down. The reducible case is slightly more complicated. We impose the condition that the blow down map is a fibration in order to ensure that the moduli spaces are regular in the sense of Section $3$ (see Definition \ref{regularpert} and the end of the proof in Theorem \ref{transversalitymain}). This condition is of course satisfied when the critical submanifold downstairs consists of a single point. It also implies that the corresponding eigenspaces of the Dirac operators all have the same dimension.

\begin{example}
Consider the unperturbed equations on a three manifold $Y$ with positive scalar curvature for some torsion spin$^c$ structure. Then the torus of flat connections is a Morse-Bott singularity (downstairs), because the corresponding Dirac operators do not have kernel (see Chapter $36$ in the book). For the unperturbed equation on the flat torus with the torsion spin$^c$ structure, the torus of flat connections is \textit{not} a Morse-Bott singularity. This is because the Dirac operator $D_{B_0}$ for the flat connection with trivial holonomy $B_0$ has two dimensional kernel consisting of constant section (see also Chapter $37$ in the book).
\end{example}

\begin{example}\label{seifert}
An example of Morse-Bott singularities (in the blown down configuration space) arises when studying Seifert fibered space homology spheres, as shown in \cite{MOY}. Notice that in the paper a non standard reducible connection is used instead of the Levi-Civita one, but all the definitions still apply.
\end{example}

\begin{example}
The key example we are interested in is the case of a $3$-manifold equipped with a self-conjugate spin$^c$ structure which will be extensively studied in Chapter $4$. In that case three kinds of singularities will arise:
\begin{itemize}
\item irreducible non-degenerate critical points;
\item reducible non-degenerate critical points;
\item reducible critical submanifolds diffeomorphic to $S^2$ blowing down to a single reducible configuration.
\end{itemize}
The present work is developed in order to be able to deal properly with the last kind of singularity.
\end{example}

\vspace{0.5cm}

We have the following easy lemmas.
\begin{lemma}\label{nonzerolambda}
If $[\Cr]$ is a Morse-Bott singularity consisting of reducible configurations, there is no point $[\bcr]\in[\Cr]$ such that $\Lambda_{\q}([\bcr])=0$.
\end{lemma}

\begin{proof}
This readily follows from the fact that one of the summands of the Hessian at a reducible critical point $[\bcr]$ is the multiplication map on $\R$ given by
\begin{IEEEeqnarray*}{c}
t\mapsto \Lambda_{\q}([\bcr])t,
\end{IEEEeqnarray*}
see Lemma $12.4.3$ in the book, and $\mathbb{R}$ is not in the tangent space of $[\Cr]$ because it only consists of reducibles.
\end{proof}

\begin{lemma}\label{hessian}
If $\acr\in\Cr$ is a Morse-Bott singularity, then the extended Hessian $\widehat{\Hess}^{\sigma}_{\q,\alpha}$ is Fredholm of index $0$, has real spectrum and kernel consisting exactly of $\T_{\acr}\Cr$.
\end{lemma}
\begin{proof}
The first two properties are already proven in the book, see Section $12.4$. For the last one, the operator at a critical point $\acr$ has the form
\begin{equation*}
\widehat{\Hess}_{\q,\alpha}^{\sigma}=\begin{bmatrix}
0 & 0 & \mathbf{d}_{\alpha}^{\sigma} \\
0 & \Hess_{\q,\alpha}^{\sigma} & 0\\
\mathbf{d}_{\acr}^{\sigma,\dagger} & 0 & 0
\end{bmatrix}
\end{equation*}
and the block
\begin{equation*}
\begin{bmatrix}
0 & \mathbf{d}_{\alpha}^{\sigma} \\
\mathbf{d}_{\acr}^{\sigma,\dagger} & 0
\end{bmatrix}
\end{equation*}
is invertible (see Section $12.4$ in the book also).
\end{proof}

\vspace{0.8cm}

It is clear that Morse-Bott singularies are not generic in any sense. If a perturbation $\q$ is such that $(\mathrm{grad}\CSd)^{\sigma}$ has Morse-Bott singularities, a nearby perturbation will generally have a different set of critical points. In what follows we will suppose that a Morse-Bott perturbation $\q_0$ is fixed, and we need to perturb the vector field away from the singularities in order to achieve transversality. This can be done as follows.
\par
In the based configuration space $\Bo^o_k(Y)=\Co_k(Y)/\G^o_{k+1}(Y)$ the image of the critical set is a finite collection of $S^1$-invariant submanifolds. This is because the set of critical points in the non blown-up setting is compact (see Corollary $10.7.4$ in the book), and the critical submanifolds are isolated because of the Morse-Bott condition. By Proposition \ref{densepert}, we can choose a map
\begin{equation*}
p_0:\Bo_k^o(Y)\rightarrow\R^n\times\mathbb{T}\times\C^m
\end{equation*}
defined by a collection of coclosed $1$-forms and sections of $\mathbb{S}$ such that $p_0$ is an embedding of image of the critical set. For each critical submanifold $[C]\subset \Bo_k(Y)$, let $\mathcal{O}_{[C]}\subset\Bo_k(Y)$ be an open $S^1$-invariant neighborhood of the corresponding $S^1$-orbit, chosen so that their images under $p_0$ have disjoint closures, and write
\begin{equation*}
\mathcal{O}=\bigcup_{[C]}\mathcal{O}_{[C]}\subset\Bo^o_k(Y).
\end{equation*}
We also require that each $\mathcal{O}_{[C]}$ is path connected and that the relative fundamental groups $\pi_1\left(\overline{p_0(\mathcal{O}_{[C]})}, p_0([C])\right)$ are trivial.
Consider the subset $\mathcal{P}_{\mathcal{O}}\subset\mathcal{P}$ of perturbations
\begin{equation}\label{pertequal}
\mathcal{P}_{\mathcal{O}}=\{\q\in\mathcal{P}\mid\q|_{\mathcal{O}}=\q_0|_{\mathcal{O}}\}
\end{equation}
which is a closed linear subspace of $\mathcal{P}$, hence a Banach space itself. Then there is an open neighborhood $\q_0$ in $\mathcal{P}_{\mathcal{O}}$ such that for all $\q$ in this neighborhood, the vector fields $\mathrm{grad}\CSD_{\q}$ and $(\mathrm{grad}\CSD_{\q})^{\sigma}$ have no zeroes outside $\mathcal{O}$, hence they have exactly the same zeroes as $\mathrm{grad}\CSD_{\q_0}$ and $(\mathrm{grad}\CSD_{\q_0})^{\sigma}$ (see Proposition $11.6.4$ in the book).

\begin{defn}\label{adapted}
Suppose we are given a tame Morse-Bott perturbation $\q_0$. We say that a tame perturbation $\q$ is \textit{adapted} to $\q_0$ if $\q$ agrees with $\q_0$ in a neighborhood $\mathcal{O}$ as above and the perturbed vector field $\mathrm{grad}\CSD+\q$ does not have any critical point outside of $\mathcal{O}$.
\end{defn}

The discussion above tells us that there is an open set of tame perturbations in $\mathcal{P}_{\mathcal{O}}$ adapted to a given $\q_0$ inside the Banach space of perturbation satisfying the property $(\ref{pertequal})$.
\par
Finally one should note that in general given a Morse-Bott perturbation downstairs it is not generally possible to find a close perturbation with the same critical submanifolds such that the critical submanifolds upstairs are Morse-Bott. This follows from example from the fact that in the space of operators we are considering (the space $\mathrm{Op}^{{sa}}$ in Section $12.6$ in the book) the operators with non simple spectrum form a subset of codimension three. On the other hand, the proof of the same result works for example when the reducible critical submanifolds consist of points (as in example \ref{seifert}).

\vspace{1.5cm}

\section{Moduli spaces of trajectories}
In this section, which is analogous of Chapter $13$ in the book, we provide two descriptions of the space of solutions connecting two given critical submanifolds. Each of them will be useful for different aspects of the theory in the following sections. In order to prove that these descriptions are equivalent, we need to prove some estimates for solutions of the perturbed Seiberg-Witten equations on a finite cylinder $Z=I\times Y$ which always stay in a small neighborhood of a given critical point. We study this problem in the first part of the section, and then focus in the second part on the basic properties of the solutions on the infinite cylinder. We will always assume that a perturbation $\q_0$ is chosen so that all the singularities are Morse-Bott. 
\\
\par
Let $\bcr$ be a critical point of the perturbed equations on $\Cs_k(Y)$ and let $\alpha$ be its image in $\Co_k(Y)$. Denote the corresponding translation-invariant solutions on $I\times Y$ as 
\begin{align*}
\gamma_{\acr}=(A_{\bcr},s_{\bcr},\phi_{\bcr})&\in \Ct_k(I\times Y)\\
\gamma_{\alpha}=(A_{\bcr},\Phi_{\bcr})&\in\Co_k(I\times Y).
\end{align*}
Given an element
\begin{equation*}
\gamma^{\tau}=(A,s,\phi)\in\Ct_k(I\times Y)
\end{equation*}
covering a configuration
\begin{equation*}
\gamma=(A,s\phi)=(A,\Phi)\in\C_k(I\times Y)
\end{equation*}
we can write 
\begin{align*}
A-A_{\bcr}&=b\otimes 1+(c\otimes 1)dt \\
\phi&=\phi_{\bcr}+\psi
\end{align*}
where $b$, $c$ are time dependent $1$ and $0$-forms respectively, and $\psi$ is a time dependent section of $S\rightarrow Y$ such that $\phi_{\acr}+\psi(t)$ is of unit norm for every $t\in I$.
In what follows Sobolev norm of a difference of two configurations in the $\tau$-model such as
\begin{equation*}
\|\gamma^{\tau}-\gamma_{\bcr}\|_{L^2_{k,A_{\bcr}}(I\times Y)}
\end{equation*}
is intended as the norm in the bigger affine space
\begin{equation*}L^2_k(I\times Y; iT^*Z)\times L^2_k(I;\R)\times L^2_{k,A_{\acr}}(I\times Y; S).
\end{equation*}

\vspace{0.5cm}
The result and proofs we are going to discuss are very close to the case of Morse singularities in the book, and they can be thought as a parametric version of them. Before stating them, we need to define a nice parametrization of a neighborhood of a critical point $\bcr\in\Cs_k(Y)$.

\begin{defn}
Fix $k\geq1$. Given a point $\bcr$ in a critical submanifold $\Cr$, an \textit{$L^2_k$-compatible product chart around} $\bcr$ is a pair $(\mathcal{U},\varphi)$ where
\begin{equation*}
\mathcal{U}\subset \T_{\bcr}\Cr\oplus \N^{\sigma}_{\bcr,k}\oplus \J^{\sigma}_{\bcr,k}
\end{equation*}
is an open neighborhood of $\{0\}$ and there exists a small neighborhood $U\subset\Cs_k(Y)$ of $\bcr$ and a map
\begin{equation*}
\varphi: \mathcal{U}\rightarrow U
\end{equation*}
such that the following hold:
\begin{itemize}
\item $\varphi$ is an $L^2_k$-compatible diffeomorphism, i.e. there exists $C>0$ such that
\begin{equation*}
\|v\|_{L^2_k}/C\leq \|d\varphi (v)\|_{L^2_k}\leq C\|v\|_{L^2_k}\qquad\text{for every }v\in T\mathcal{U},
\end{equation*}
and $\mathcal{D}_0\varphi=\mathrm{Id}$;
\item the restriction of $\varphi$ to $\mathcal{U}\cap (\T_{\bcr}\Cr\oplus \N^{\sigma}_{\bcr,k}\oplus\{0\})$ is a local chart for the Coulomb slice $\Sl_{k,\bcr}^{\sigma}$ through $\bcr$ such that the restriction to ${\mathcal{U}\cap (\T_{\bcr}\Cr\oplus \{0\}\oplus\{0\})}$ is a local chart for $\Cr\cap \Sl_{k,\bcr}^{\sigma}$ with $\varphi(0)=\bcr$;
\end{itemize}
\end{defn}

\begin{lemma}\label{compchart}
Every critical point $\bcr\in \Cs_k(Y)$ admits a $L^2_k$-compatible chart around it.
\end{lemma}
\begin{proof}
Consider the map
\begin{equation*}
\Psi: \T_{\bcr}\Cr\oplus \N^{\sigma}_{\bcr,k}\rightarrow \mathcal{S}^{\sigma}_{k,\acr}
 \end{equation*}
obtained by composing the local chart from Section $3$ in Chapter $1$
\begin{equation*}
\iota: \K_{k,\bcr}^{\sigma}\rightarrow \mathcal{S}^{\sigma}_{k,\bcr}
\end{equation*}
with a linear isomorphism of Hilbert spaces
\begin{equation*}
\psi:\T_{\bcr}\Cr\oplus \N^{\sigma}_{\bcr,k}\rightarrow \K_{k,\bcr}^{\sigma}
\end{equation*}
where $\psi$ is the inclusion on the first summand and a left inverse of the quotient map
\begin{equation*}
 \K_{k,\bcr}^{\sigma}\rightarrow \N^{\sigma}_{\bcr,k}
\end{equation*}
on the second summand. The map  $\Psi$ is a diffeomorphism in a neighborhood $\mathcal{U}'$ of $\{0\}\oplus\{0\}$, and by the implicit function theorem we can describe $\Psi^{-1}(\Cr)\cap\mathcal{U}'$ in a smaller neighborhood as the graph of a smooth function
\begin{equation*}
f: \T_{\bcr}\Cr\rightarrow \N^{\sigma}_{\bcr,k}.
\end{equation*}
The origin is a critical point of $f$ because $T_{\bcr}\Cr$ is by definition tangent to $\Cr$. To obtain a parametrization of a neighborhood in $\Cs_k(Y)$ we use the action of the gauge group as follows. The differential of the gauge group action gives the isomorphism
\begin{equation*}
d^{\sigma}_{\bcr}:T_{e}\G_{k+1}(Y)\rightarrow \J^{\sigma}_{\bcr,k}
\end{equation*}
and we define the map
\begin{align*}
\varphi&:\mathcal{U}'\oplus \J^{\sigma}_{\bcr,k}\rightarrow \Cs_k(Y)\\
(v^t, v^{n}, v^j)&\mapsto \mathrm{exp}\left((d^{\sigma}_{\bcr})^{-1}(v^j)\right)\cdot \Psi(v^t, v^{n}+f(v^t)).
\end{align*}
Here we are exponentiating the element in the Lie algebra to obtain a gauge transformation. The differential of this map at the origin is the identity by definition, and this is in fact $L^2_k$-compatible product chart by applying the inverse function theorem.
\end{proof}
\begin{remark}\label{jkcompatible}
If $1\leq j\leq k$ the construction above implies that if $\bcr$ is an $L^2_k$ configuration we can construct an $L^2_k$-compatible product chart which is the restriction of an $L^2_j$-compatible product chart. This follows from the fact that the differential at the origin is the identity and the uniqueness statement in the inverse function theorem. We will use this fact in Section $5$. Indeed, one can slightly modify the construction to obtain a $L^2_k$-compatible product chart which is defined in an $L^2_j$ neighborhood of the origin, see Lemma \ref{cylinderchart}. The chart we have constructed does not have this property, but it has the desirable property of being gauge invariant is some sense that will be made precise in Remark \ref{tangind}.
\end{remark}

\vspace{0.5cm}
Given such an $L^2_k$-compatible product chart $(\mathcal{U},\varphi)$, for a configuration $(B,r,\psi)\in U=\varphi(\mathcal{U})$ one can define its \textit{normal component}
\begin{equation*}
(B,r,\psi)^{\nu}=\varphi\big( \Pi^{\nu}\circ\varphi^{-1}(B,r,\psi)\big)\in U
\end{equation*}
where
\begin{equation*}
\Pi^{\nu}: \T_{\acr}\Cr\oplus \N^{\sigma}_{\acr,k}\oplus\J_{k,\acr}^{\sigma}
\rightarrow \{0\}\oplus \N^{\sigma}_{\acr,k}\oplus\J_{k,\acr}^{\sigma}
\end{equation*}
is the projection with kernel $\T_{\acr}\Cr\oplus\{0\}\oplus\{0\}$. Furthermore, given any $L^2_1$ path
\begin{equation*}
\check{\gamma}: I\rightarrow \Cs_k(Y)
\end{equation*}
such that $\check{\gamma}(t)$ is in $U$ for every $t\in I$, we can define its normal part $\gamma^{\nu}$ to be the path defined as
\begin{equation*}
\check{\gamma}^{\nu}(t)=\check{\gamma}(t)^{\nu}\qquad\text{for }t\in I.
\end{equation*}
This is still an $L^2_1$ path. The most interesting case is that of a solution $\gamma\in \Ct_k(I\times Y)$ of the Seiberg-Witten equations on the cylinder. After a suitable gauge transformation, this gives rise to a $L^2_1$ path
\begin{equation*}
\big(\check{\gamma}(t),c(t)\big)\in \Cs_k(Y)\times L^2_k(Y,i\R),
\end{equation*}
see Corollary $10.7.3$ in the book, and in the situation above we define its \textit{normal part}, denoted by $\gamma^{\nu}$, to be the $L^2_1$ path $(\check{\gamma}^{\nu}(t),c(t)\big)$.

\par
Similarly, there is a notion of \textit{tangent component}
\begin{equation*}
(B,r,\psi)^{t}=\varphi\big(\Pi^{t}\circ\varphi^{-1}(B,r,\psi)\big)\in\Sl^{\sigma}_{k,\bcr}
\end{equation*}
where
\begin{equation*}
\Pi^{t}:\T_{\acr}\Cr\oplus \N^{\sigma}_{\acr,k}\oplus\J_{k,\acr}^{\sigma}\rightarrow \T_{\acr}\Cr\oplus\{0\}\oplus\{0\}
\end{equation*}
is the projection with kernel $\{0\}\oplus \N^{\sigma}_{\acr,k}\oplus\J_{k,\acr}^{\sigma}$, and one can define the tangent part of a $L^2_1$ path in the identical fashion as above.
\begin{remark}\label{tangind}
It follows from the construction of Lemma \ref{compchart} that we can choose the $L^2_k$-compatible product chart so that the tangent part of any two gauge equivalent configurations in the image of the chart is the same.
\end{remark}

Furthermore, the images of the trivial bundles over $\mathcal{U}$ with fiber $T_{\acr}\Cr\oplus\{0\}\oplus\{0\}$ and $\{0\}\oplus\N_{\acr,k}^{\sigma}\oplus\{0\}$, which we denote by $\T_k^t$ and $\T_k^n$ determine via fiberwise completion a smooth bundle decomposition
\begin{equation*}
\T_l\lvert_{\tilde{U}}=\T_j^t\oplus \T_j^n\oplus\J^{\sigma}_j
\end{equation*}
for every $j\leq k$ (notice that the first summand is not affected by the completion) on some neighborhood $\tilde{U}\subset U$. We will restrict our chart to this neighborhood. In particular, we can decompose the gradient of the Chern-Simons-Dirac in its tangent and normal part $(\mathrm{grad}\CSd)^{\sigma,t}$ and $(\mathrm{grad}\CSd)^{\sigma,n}$.
There is an analogous notion of $L^2_k$-compatible product chart for a Morse-Bott critical point in the non blown-up configuration space, which will be useful in what follows.

\vspace{0.8cm}

\vspace{0.8cm}
We start by studying the situation downstairs. From now on we suppose that a $L^2_1$-compatible product chart is fixed. We have the following key estimates for near constant solutions (see Proposition $13.4.4$ in the book). Recall that $k$ is an integer not less than $2$.
\begin{prop}\label{nearconstestimate}
For any Morse-Bott critical point $\beta$ there exists a gauge invariant neighborhood $U$ of $\gamma_{\beta}$ in $C_k(I\times Y)$ and a constant $C_0$ such that any solution $\gamma$ in $U$ in the Coulomb-Neumann slice $S_{k,\gamma_{\beta}}$ satisfies the estimate
\begin{equation*}
\|\gamma^{\nu}\|^2_{L^2_1(I\times Y)}\leq C_0\big(\CSd(s_1)-\CSd(s_2)\big).
\end{equation*}
\end{prop}
Here we are choosing the gauge invariant neighborhood to be small enough so that the normal part is well defined for a configuration in Coulomb-Neumann slice.
The proof of this result proceeds as the one in the book, the only difference being the fact that the non degeneracy of the normal Hessian gives us control only on the normal part of the configuration and not on the whole configuration. Furthermore, we cannot estimate the $L^2_{k+1}$ norm (up to gauge and on a smaller interval $I'\subset I$ as in the statement in Proposition $13.4.1$) because the normal part of a solution is generally not a solution, hence the bootstrapping argument does not apply. On the other hand, the rest of the proof works few modifications. The key observation is the following.

\begin{lemma}
Given an irreducible Morse-Bott critical point $\beta$ there is a constant $C>0$ and a neighborhood
\begin{equation*}
(\beta,0)\in U^Y\subset \Co_1(Y)\times L^2_1(Y,i\R)
\end{equation*}
such that for every $(\beta+v,c)\in U^Y$ we have
\begin{equation*}
\|((\beta+v)^{\nu},c)\|^2_{L^2_1(Y)}\leq C\big(\|\mathbf{d}^*_{\beta}v\|^2+\|\mathbf{d}_{\beta+v}c\|^2+\|\mathrm{grad}\CSd(\beta+v)\|^2\big).
\end{equation*}
In the reducible case the same results holds with the additional hypothesis $\int_Y c=0$.
\end{lemma}
Notice that we need to choose first a neighborhood small enough of $\beta$ in order to have a notion of normal component.
\begin{proof}
Consider the map
\begin{IEEEeqnarray*}{c}
\Co_1(Y)\times L^2_1(Y;i\R)\rightarrow \J_0\oplus\T^n_0\oplus L^2(Y;i\R)\\
(\beta+v,c)\mapsto \big(\mathbf{d}_{\beta+v}c,(\mathrm{grad}\CSd)^{n}(\beta+v), \mathbf{d}^*_{\beta}v\big)
\end{IEEEeqnarray*}
where we use the compatible product chart to define the normal part of the gradient.
Using the identification
\begin{equation*}
\T_{1,\beta}\equiv \J_{1,\beta}\oplus T_{\beta}C\oplus\N_{1,\beta},
\end{equation*}
the linearization of this map at $(\beta,0)$ can be written as
\begin{equation*}
(v^J, v^{\tau}, v^{n}, c)\mapsto (\mathbf{d}_{\beta}c, \Hess^n_{\beta} v^{n}, \mathbf{d}^*_{\beta} v^J).
\end{equation*}
Here we use that the perturbation $\q$ is a $C^1$ section, as stated in condition $(6)$ in the Definition \ref{tamepert} in Chapter $1$.
As $\Hess^n_{\beta}:\N_{1,\beta}\rightarrow \N_{0,\beta}$ is an isomorphism and the map
\begin{equation*}
\mathbf{d}_{\beta}: L^2_1(Y;i\R)\rightarrow \J_{0,\beta}
\end{equation*}
is invertible in the irreducible case, we obtain that the linearization is an isomorphism when restricted to the subspace 
\begin{equation*}
(J_{1,\beta}\oplus\{0\}\oplus\N_{1,\beta})\oplus L^2_1(Y;i\R)\subset \T_{1,\beta}\oplus L^2_1(Y;i\R).
\end{equation*}
This, together with the $L^2_1$-compatibility of the chart, implies that the estimate
\begin{equation*}
\|((\beta+v)^{\nu},c)\|^2_{L^2_1(Y)}\leq C'\big(\|\mathbf{d}^*_{\beta}v\|^2+\|\mathbf{d}_{\beta+v}c\|^2+\|(\mathrm{grad}\CSd)^{n}(\beta+v)\|^2\big)
\end{equation*}
holds in some $L^2_1$ neighborhood of $(\beta,0)$. On the other hand, there is a neighborhood of $\beta$ and a costant $C''$ such that for $\beta+w$ in this neighborhood one has
\begin{equation}\label{normalbound}
\|(\mathrm{grad}\CSd)^{\nu}(\beta+v)\|^2\leq C''\|\mathrm{grad}\CSd(\beta+v)\|^2,
\end{equation}
hence we obtain the required estimate.
\end{proof}

\begin{proof}[Proof of Proposition \ref{nearconstestimate}]
Following the proof of Lemma $13.4.4$, we obtain an identity of the form
\begin{align*}
2\big(\CSd(s_1)-\CSd(s_2)\big)&=\\
&\begin{aligned}
=&\int_I\big( \| \frac{d}{ds}\check{\gamma}(s)\|^2+\|\frac{d}{ds}c(s)\|^2\big)ds \\
& + \int_I \big(\|\mathbf{d}^*_{\beta}(\check{\gamma}(s)-\beta)\|^2+\|\mathbf{d}_{\check{\gamma}(s)}c(s)\|^2+\|\mathrm{grad}\CSd(\check{\gamma}(s))\|^2\big)ds   \\
& +\int_I \big\langle (\check{\gamma}-\beta)\sharp c,\frac{d}{ds} \check{\gamma}\big\rangle ds
\end{aligned}
\end{align*}
where $\sharp$ is a bilinear operator involving only pointwise multiplication.
By the previous lemma (which can be applied to the restriction because $k\geq 2$) the second term on the right hand side bounds from above
\begin{equation*}
\int_I \|\big(\check{\gamma}^{\nu}(s), c(s)\big)\|^2_{L^2_1(Y)}ds.
\end{equation*}
We claim that, after possibly rescaling the neighborhood of the critical point $\beta$, we can obtain an estimate of the form
\begin{equation}\label{normtan}
\|\frac{d}{ds} \check{\gamma}^{t}(s)\|_{L^2}\leq K_0\|\check{\gamma}^{\nu}(s)\|_{L^2_1(Y)}
\end{equation}
for some $K_0>0$. From this, we see that the inequality can be rearranged (using that in dimension $4$ the $L^4$ norm is controlled by the $L^2_1$ norm) to be
\begin{equation*}
\|(\gamma-\gamma_{\beta})^{\nu}\|^2_{L^2_1(I\times Y)}\leq C' \big (\CSd(s_1)-\CSd(s_2)\big)+ K\|\gamma-\gamma_{\beta}\|_{L^2_1(I\times Y)}\|(\gamma-\gamma_{\beta})^{\nu}\|^2_{L^2_1(I\times Y)},
\end{equation*} 
hence we obtain the result by restricting to a suitably small neighborhood of $\gamma_{\beta}$.
\par
To prove the estimate (\ref{normtan}), we just notice that for a suitably small neighborhood of $\beta$ we have for every small $\varepsilon>0$ the estimate
\begin{multline*}
\|\check{\gamma}^{t}(s+\varepsilon)-\check{\gamma}^{t}(s)\|_{L^2(Y)}\leq\int_{s}^{s+\varepsilon} \|(\mathrm{grad}\CSd)^t(\check{\gamma}(s))\|_{L^2(Y)} ds\\
\leq C''\int_{s}^{s+\varepsilon} \|\mathrm{grad}\CSd(\check{\gamma}(s))\|_{L^2(Y)} ds\leq K_0\int_{s}^{s+\varepsilon} \|\check{\gamma}^{\nu}(s)\|_{L^2_1(Y)}ds.
\end{multline*}
where the last inequality comes from the fact that $\CSd$ is constant along the critical submanifold, and the fact that the gradient of $\CSd$ is $C^1$ in the $L^2_1$ topology. Then our desired inequality follows by dividing both sides by $\varepsilon$ and taking the limit for $\varepsilon$ going to zero.
\end{proof}

\vspace{1.5cm}
So far we have treated the case of configurations spaces on a finite cylinder $Z=I\times Y$. In the case of an infinite cylinder, the situation is a little different (see Section $13.1$ in the book). For any interval $I\subset\R$, one introduces the configuration space
\begin{align*}
\tilde{\Co}_{k,\mathrm{loc}}^{\tau}(I\times Y)&=\{(A,s,\phi)\mid \|\check{\phi}(t)\|_{L^2(Y)}=1\text{ for every }t\in I\}\\&
\subset \A_{k,\mathrm{loc}}(I\times Y)\times L^2_{k,\mathrm{loc}}(I;\R)\times L^2_{k,\mathrm{loc}}(I\times Y; S^+),
\end{align*}
where
\begin{equation*}
\A_{k,\mathrm{loc}}(I\times Y)=A_0+L^2_{k,\mathrm{loc}}(I\times Y; iT^*Z)
\end{equation*}
for any smooth spin$^c$ connection $A_0$. We consider on this space the topology of $L^2_k$ convergence on compact subsets (it is clear that for $I$ compact this definition coincides with the usual one). Furthermore, we define the closed subspace
\begin{equation*}
\Ct_{k,\mathrm{loc}}(I\times Y)\subset \tilde{\Co}_{k,\mathrm{loc}}^{\tau}(I\times Y)
\end{equation*}
consisting of configurations $(A,s,\phi)$ with $s\geq 0$. The appropriate gauge group to consider is $\G_{k+1,\mathrm{loc}}(I\times Y)$, the group of $L^2_{k+1,\mathrm{loc}}$ maps with values in the circle $S^1\subset \C$. The quotient spaces will be denoted by
\begin{align*}
\Bo_{k,\mathrm{loc}}^{\tau}(I\times Y)&=\Co_{k,\mathrm{loc}}^{\tau}(I\times Y)/\G_{k+1,\mathrm{loc}}(I\times Y)  \\
\tilde{\Bo}_{k,\mathrm{loc}}^{\tau}(I\times Y)&=\tilde{\Co}_{k,\mathrm{loc}}^{\tau}(I\times Y)/\G_{k+1,\mathrm{loc}}(I\times Y).
\end{align*}
\par
Suppose a perturbation $\q_0\in\mathcal{P}$ is fixed so that all critical points of $(\mathrm{grad}\CSd)^{\sigma}$ are Morse-Bott singularities. The choice of $\q_0$ determines the perturbed $4$-dimensional Seiberg-Witten equations, which is a section $\F_{\q}^{\tau}$ of the map
\begin{equation*}
\V_{k-1,\mathrm{loc}}^{\tau}(I\times Y)\rightarrow\tilde{\Co}_{k,\mathrm{loc}}^{\tau}(I\times Y).
\end{equation*}
Here the fiber at $\gamma=(A_0,s_0,\phi_0)$ is the subspace
\begin{equation*}
\V_{k-1,\mathrm{loc},\gamma}^{\tau}\subset L^2_{k-1,\mathrm{loc}}(I\times Y; i\mathfrak{su}(S^+))\oplus L^2_{k-1,\mathrm{loc}}\oplus L^2_{k-1,\mathrm{loc}}(I\times Y; S^-)
\end{equation*} 
consisting of triples $(a,s,\phi)$ with
\begin{equation*}
\mathrm{Re}\langle \check{\phi}_0(t),\check{\phi}(t)\rangle_{L^2(Y)}=0
\end{equation*}
for every $t\in I$. It is important to remark that here we do not use the language of vector bundles, as $\V_{k-1,\mathrm{loc}}^{\tau}$ is not locally trivial in any straightforward way.

\vspace{0.8cm}

If $\bcr\in\Cs_k(Y)$ is a critical point, then the corresponding translation invariant configuration $\gamma_{\bcr}$ is a solution of the equations, i.e. $\F_{\q}^{\tau}(\gamma_{\bcr})=0$, and we write $[\gamma_{\bcr}]$ for its gauge-equivalence class. We say that a configuration $[\gamma]\in\tilde{\Bo}^{\tau}_{k,\mathrm{loc}}$ is \textit{asymptotic} to $[\bcr]$ as $t\rightarrow\pm\infty$ if
\begin{equation*}
[\tau_t^*\gamma]\rightarrow [\gamma_{\bcr}]\quad\text{in }\Bt_{k,\mathrm{loc}}(Z)\text{ as }t\rightarrow\pm\infty
\end{equation*}
where
\begin{align*}
\tau_t&:Z\rightarrow Z\\
(s,y)&\mapsto(s+t,y)
\end{align*}
is the translation map. We will respectively write
\begin{equation*}
\lim_{\rightarrow}[\gamma]=[\bcr]\quad\text{and}\quad \lim_{\leftarrow}[\gamma]=[\bcr].
\end{equation*}

\begin{defn}Suppose $[\Cr_-]$ and $[\Cr_+]$ are critical submanifolds for $(\mathrm{grad}\CSd)^{\sigma}$. We write $M([\Lr_-],[\Lr_+])$ for the space of all configurations $[\gamma]$ in $\Bt_{k,\mathrm{loc}}(Z)$ which are asymptotic to a point in $[\Lr_-]$ for $t\rightarrow-\infty$, asymptotic to a point in $[\Lr_+]$ for $t\rightarrow+\infty$ and solve the perturbed Seiberg-Witten equations:
\begin{equation*}
M([\Lr_-],[\Lr_+])=\big\{[\gamma]\in\Bt_{k,\mathrm{loc}}(Z)\mid\F_{\q}^{\tau}(\gamma)=0,\lim_{\leftarrow}[\gamma]\in[\Lr_-], \lim_{\rightarrow}[\gamma]\in[\Lr_+]\big\}.
\end{equation*}
We refer to this as a \textit{moduli space of trajectories} on the cylinder $Z=\R\times Y$. We can similarly define the subset $\tilde{M}([\Lr_-],[\Lr_+])$ of the large space $\tilde{\Bo}_{k,\mathrm{loc}}^{\tau}(Z)$. Given open subsets $[\mathfrak{U}_{\pm}]\subset [\Cr_{\pm}]$ we write 
\begin{equation*}
M([\mathfrak{U}_-], [\mathfrak{U}_+])\subset M([\Lr_-],[\Lr_+])
\end{equation*}
for the subspace of configurations which are asymptotic at $\pm\infty$ to the critical points in the open set $[\mathfrak{U}_{\pm}]$.
\end{defn}

Notice that we suppressed the value of $k$ from our notation. Because of the bootstrapping properties of the Seiberg-Witten equations, the space is essentially independent of the choice of $k$. This is stated precisely in the next lemma (see Proposition $13.1.2$ in the book).
\begin{lemma}
Let $M([\Lr_-],[\Lr_+])_k$ temporarily denote the moduli space $M([\Lr_-],[\Lr_+])\subset \Bt_{k,\mathrm{loc}}(Z)$.
\begin{enumerate}
\item If $[\gamma]$ is in $M([\Lr_-],[\Lr_+])_k$, then there is a gauge representative $\gamma\in \Ct_{k,\mathrm{loc}}(Z)$ which is $C^{\infty}$ on $Z$;
\item The naturally induced bijections $M([\Lr_-],[\Lr_+])_{k_1}\rightarrow M([\Lr_-],[\Lr_+])_{k_2}$ for all $k_1,k_2\geq2$ are homeomorphisms.
\end{enumerate}
\end{lemma}

If $[\gamma]\in M([\Lr_-],[\Lr_+])$, then there is a corresponding (smooth) path $[\check{\gamma}]$ in $\Bs_k(Y)$ approaching $[\Lr_-]$ and $[\Lr_+]$ at the two ends. Hence we can decompose the space according to the relative homotopy class of the path
\begin{equation*}
z\in\pi_1(\Bs_k(Y),[\Lr_-],[\Lr_+])
\end{equation*}
and write
\begin{equation*}
 M([\Lr_-],[\Lr_+])=\bigcup_{z} M_z([\Lr_-],[\Lr_+]).
\end{equation*}
The set of homotopy classes is an affine space on $H^1(Y;\Z)$, the component group of the gauge group. Furthermore, we have the continuous \textit{evaluation maps}
\begin{IEEEeqnarray*}{c}
\ev_+:M([\Lr_-],[\Lr_+])\rightarrow [\Lr_+] \\
\ev_-:M([\Lr_-],[\Lr_+])\rightarrow [\Lr_-]
\end{IEEEeqnarray*}
obtained by sending a solution to its limit points at $\pm\infty$,
\begin{equation*}
\ev_{\pm}[\gamma]=\lim_{t\rightarrow \pm\infty}[\check{\gamma}(t)].
\end{equation*}
\\
\par
While the definition of the moduli space we have just given will be useful when studying compactness issues, it will not fit when dealing with transversality problems, because the spaces involved are not Banach manifolds in any natural way. Furthermore, unlike the case treated in the book, our operators are not well behaved on the $L^2_k$ spaces on the cylinder because we are dealing with Morse-Bott singularities. In particular, their linearization is not Fredholm on such spaces. We tackle the problem in the following way, see for example \cite{Don} and \cite{MMR}. Recall that for a vector bundle $E\rightarrow Z$ which is the pullback of a bundle on $Y$ the definition of the \textit{weighted Sobolev space} $L^2_{k,\delta}(Z;E)$ with weight $\delta\in\R^{>0}$. Pick a smooth function $f:\R\rightarrow\R^{>0}$ such that
\begin{equation*}
f(t)=e^{\delta|t|}\qquad\text{for }|t|>>0.
\end{equation*}
Then $L^2_{k,\delta}(Y;E)$ is the space $f^{-1}L^2_k(Z;E)$:
\begin{equation}\label{weightedsp}
s\in L^2_{k,\delta}(Z;E) \text{ if and only if } f\cdot s\in L^2_k(Z;E).
\end{equation}
In this case, the norm on $L^2_{k,\delta}(Y;E)$ is defined so that the multiplication by $f$ is an isometry $L^2_{k,\delta}\rightarrow L^2_k$. Furthermore, different functions define equivalent norms. It may be useful sometimes to use the equivalent norm 
\begin{equation}\label{weightedspequiv}
\|s\|_{k,\delta}^2=\int_Z f^2\left(|s|^2+|\nabla s|^2+\cdots +|\nabla^k s|^2\right)d\mathrm{vol}.
\end{equation}

With this in mind, the following embedding and multiplication results are straightforwardly adapted from the unweighted case (see Theorems $13.2.1$ and $13.2.2$ in the book).
\begin{prop}
There is a continuous inclusion
\begin{equation*}
L^p_{k,\delta}(Z)\hookrightarrow L^q_{l,\delta'}(Z)
\end{equation*}
for $k\geq l$, $\delta'\leq\delta$, $p\leq q$ and $(k-n/p)\geq (l-n/q)$, with the further assumption that if the last inequality is an equality $1<p\leq q<\infty$. This embedding is compact if and only if $\delta\>\delta'$.
\end{prop}

\begin{prop}
Suppose $\delta+\delta'\geq\delta''$, $k,l\geq m$ and $1/p+1/q\geq 1/r$, with $p,q,r\in(1,\infty)$. Then the multiplication
\begin{equation*}
L^p_{k,\delta}(Z)\times L^q_{l,\delta'}(Z)\rightarrow L^r_{m,\delta''}(Z)
\end{equation*}
is continuous in any of the following three cases:
\begin{enumerate}
\item $\mathrm{(a)}$ $(k-n/p)+(l-n/q)\geq m-n/r$, and \\
$\mathrm{(b)}$ $k-n/p<0$, and \\
$\mathrm{(c)}$ $l-n/q<0$;\\
or
\item $\mathrm{(a)}$ $\mathrm{min}\{(k-n/p),(l-n/q)\}\geq m-n/r$, and  \\
$\mathrm{(b)}$ either $k-n/p>0$ or $l-n/q>0$;\\
or
\item $\mathrm{(a)}$ $\mathrm{min}\{(k-n/p),(l-n/q)\}\geq r-n/m$, and  \\
$\mathrm{(b)}$ either $k-n/p=0$ or $l-n/q=0$.
\end{enumerate}
When the map is continuous, it is a compact operator as a function of $g$ for fixed $f$ provided $l>m$ and $l-n/q>m-n/r$.
\end{prop}

\vspace{0.5cm}

Suppose we are given critical submanifolds $[\Lr_-]$ and $[\Lr_+]$. Our aim is to define the space of trajectories connecting these two submanifolds that converge exponentially fast up to gauge transformation. In order to do so, first choose two contractible open sets $[\mathfrak{U}_-]\subset [\Lr_-]$ and $[\mathfrak{U}_+]\subset [\Lr_+]$ that admit smooth lifts $\mathfrak{U}_-$ and $\mathfrak{U}_+$. We can choose these open sets and their lifts so that the latter are contained in the slice passing through configurations $\bcr_{\pm}\in\Cs_k(Y)$ with $[\bcr_{\pm}]\in[\mathfrak{U}_{\pm}]$. Choose then a smooth family of smooth base configurations $\{\gamma_0\}$ parametrized by
\begin{equation*}(\acr_-,\acr_+)\in\mathfrak{U}_-\times\mathfrak{U}_+
\end{equation*}
such that $\gamma_0(\acr_-,\acr_+)$ agrees near $\pm\infty$ with the translation invariant configuration $\gamma_{\acr_-}$ and $\gamma_{\acr_+}$ respectively. Choosing the lifts in the appropriate component of the gauge group orbit, we can arrange that the family $[\gamma_0]$ defines any given relative homotopy class $z\in\pi_1(\Bs_k(Y),[\Lr_-],[\Lr_+])$. As $\Ct_{k,\mathrm{loc}}(Z)$ is a subset of an affine space
\begin{equation*}
\tilde{\Co}^{\tau}_{k,\mathrm{loc}}\subset\A_{k,\mathrm{loc}}(Z)\times L^2_{k,\mathrm{loc}}(\R,\R)\times L^2_{k,\mathrm{loc}}(Z;S^+),
\end{equation*}
we can interpret the difference of two configurations as elements of the vector space
\begin{equation*}
L^2_{k,\mathrm{loc}}(Z;iT^*Z)\times L^2_{k,\mathrm{loc}}(\R,\R)\times L^2_{k,\mathrm{loc}}(Z;S^+),
\end{equation*}
and it makes sense to ask for the $L^2_{k,\delta}$ norm to be finite. Hence we can introduce the configuration space $\Ct_{k,\delta}(\mathfrak{U_-},\mathfrak{U_+})$ defined as
\begin{multline*}
\{ \gamma\in\Ct_{k,\mathrm{loc}}(Z)\mid\gamma-\gamma_0(\acr_-,\acr_+)\in L^2_{k,\delta}(Z;iT^*Z)\times L^2_{k,\delta}(\R,\R)\times L^2_{k,\delta}(Z;S^+)\\ \text{ for some }(\acr_-,\acr_+)\in \mathfrak{U}_-\times\mathfrak{U}_+\}.
\end{multline*}
It is clear that such a definition is independent of the choice of the family of base configurations $\{\gamma_0\}$. Similarly we can introduce the larger space $\tilde{\Co}^{\tau}_{k,\delta}(\mathfrak{U_-},\mathfrak{U_+})$ as a subspace of $\tilde{\Ct}_{k,\mathrm{loc}}(Z)$. We also introduce the gauge group $\G_{k+1,\delta}(Z)$ which is the subgroup of $\G_{k+1,\mathrm{loc}}(Z)$ that preserves $\Ct_{k,\delta}(\mathfrak{U}_-,\mathfrak{U}_+)$. This group has the following simple characterization, whose proof is readily adapted from the one of Lemma $13.3.1$ in the book using the multiplication theorem for weighted Sobolev spaces we discussed above.
\begin{lemma}
The group $\G_{k+1,\delta}(Z)$ is independent of $\mathfrak{U}_-$ and $\mathfrak{U}_+$ and can be described as
\begin{equation*}
\G_{k+1,\delta}(Z)=\{u:Z\rightarrow S^1\mid 1-u\in L^2_{k+1,\delta}(Z;\C)\}
\end{equation*}
Furthermore, its component group is $\Z$, the identification given by the winding number of the map $t\mapsto u(t, y_0)$ for any fixed basepoint $y_0$.
\end{lemma}

We can then define the quotient spaces
\begin{IEEEeqnarray*}{c}
\Bt_{k,\delta,z}([\mathfrak{U}_-],[\mathfrak{U}_+])=\Ct_{,\delta}k(\mathfrak{U}_-,\mathfrak{U}_+)/\G_{k+1,\delta}(Z) \\
\tilde{\Bo}^{\tau}_{k,\delta,z}([\mathfrak{U}_-],[\mathfrak{U}_+])=\tilde{\Co}^{\tau}_{k,\delta}(\mathfrak{U}_-,\mathfrak{U}_+)/\G_{k+1,\delta}(Z) 
\end{IEEEeqnarray*}
where we have chosen the lifts $\mathfrak{U}_-$ and $\mathfrak{U}_+$ so that a path connecting them projects to a path in the class $z$. To within a canonical identification, $\Bt_{k,\delta,z}([\mathfrak{U}_-],[\mathfrak{U}_+])$ is independent of the choice of these smooth lifts. Also, if $[\mathfrak{U}_-']\subset[\mathfrak{U}_-]$ and $[\mathfrak{U}_+']\subset[\mathfrak{U}_+]$ we have a natural inclusion
\begin{equation*}
\Bt_{k,\delta,z}([\mathfrak{U}_-]',[\mathfrak{U}_+'])\hookrightarrow \Bt_{k,\delta,z}([\mathfrak{U}_-],[\mathfrak{U}_+]),
\end{equation*}
and an analogous one for the larger spaces with tildes. Using these identifications we can define the spaces of configurations 
\begin{IEEEeqnarray*}{c}
\Bt_{k,\delta,z}([\Lr_-],[\Lr_+])=\coprod \Bt_{k,\delta,z}([\mathfrak{U}_-],[\mathfrak{U}_+])/\sim \\
\tilde{\Bo}^{\tau}_{k,\delta,z}([\Lr_-],[\Lr_+])=\coprod \tilde{\Bo}^{\tau}_{k,\delta,z}([\mathfrak{U}_-],[\mathfrak{U}_+])/\sim
\end{IEEEeqnarray*}
and their union over all relative homotopy classes
\begin{IEEEeqnarray*}{c}
\Bt_{k,\delta}([\Lr_-],[\Lr_+]) =\bigcup_z\Bt_{k,\delta,z}([\Lr_-],[\Lr_+]) \\
\tilde{\Bo}^{\tau}_{k,\delta}([\Lr_-],[\Lr_+])=\bigcup_z\tilde{\Bo}^{\tau}_{k,\delta,z}([\Lr_-],[\Lr_+]).
\end{IEEEeqnarray*}

\vspace{0.8cm}
The key result of the section is the following.
\begin{teor}\label{modsame}
For a fixed tame perturbation $\q_0$ such that all singularities are Morse-Bott there exists a $\delta>0$ with this property. For any two critical submanifolds $[\Lr_-]$ and $[\Lr_+]$ and any $\gamma\in \Ct_{k,\mathrm{loc}}(Z)$ representing an element $[\gamma]\in M_z([\Lr_-],[\Lr_+])$, and let $\bcr_{\pm}$ be suitable lifts of $\ev_{\pm}([\gamma])$ as above. Then there is a gauge transformation $u\in \G_{k+1,\mathrm{loc}}$ such that $u\cdot \gamma\in \Ct_{k,\delta}(\bcr_-, \bcr_+)$, and if $u'$ is another such a gauge transformation, then $u^{-1}u'\in\G_{k+1}$. The resulting bijection is a homeomorphism
\begin{equation*}
M_z([\Lr_-],[\Lr_+])\rightarrow \left\{[\gamma]\in\Bt_{k,z,\delta}([\Lr_-],[\Lr_+])\mid\F_{\q}^{\tau}(\gamma)=0\right\}.
\end{equation*}
Finally, if the statement holds for $\delta$, then it holds for every $0<\delta'<\delta$.
\end{teor}
Before proving the theorem, we discuss a key auxiliary result (see Proposition $13.5.1$ in the book).
\begin{lemma}\label{expdecay}
For every solution $\gamma^{\tau}\in \Ct_{k,\mathrm{loc}}$ of the perturbed Seiberg-Witten equations on $[0,\infty)\times Y$ asymptotic to a Morse-Bott critical point $[\bcr]\in\Cs_k(Y)$ there exists a $t_0$ such that for all $t\geq t_0$
\begin{equation*}
\CSd(\gamma^{\tau}(t))-\CSd(\bcr)\leq Ce^{-\delta t}
\end{equation*}
where $C=\CSd(\gamma^{\tau}(t_0))-\CSd(\bcr)$. Furthermore, if a tame perturbation $\q_0$ such that all singularities are Morse-Bott is fixed, then such a $\delta>0$ can be chosen uniformly for all critical points.
\end{lemma}
\begin{proof}
This is analogous to the Morse case. Let $C$ be the critical submanifold $\beta$ it belongs to, and fix a $L^2_1$-compatible product chart around $\beta$. As $\CSd$ is $C^2$ on $\Co_1(Y)$ with vanishing derivative at $\beta=(B,\Phi)$, we have for some $C_1>0$ that
\begin{equation*}
|\CSd(\beta+w)-\CSd(\beta)|= |\CSd(\beta+w)-\CSd\left((\beta+w)^{\tau})\right)|\leq C_1\|(\beta+w)^{\nu}\|_{L^2_{1,B}}^2
\end{equation*}
for all $w\in T_{\beta} \Co_1(Y)$ with $[\beta+w]$ in some $L^2_1$ neighborhood $U_1$ of $[\beta]$ in $\Bo_1(Y)$ (here we use the $L^2_1$ compatibility of the local chart used to define the normal part). Then the non-degeneracy of the normal Hessian tells us that
\begin{equation*}
\|(\mathrm{grad}\CSd)^{n}(\beta+w)\|^2_{L^2}\geq C_2\|(\beta+w)^{\nu}\|^2_{L^2_1}
\end{equation*}
for $\beta+w$ in the Coulomb slice $\Sl_{1,\beta}$ (so that $(\beta+w)^{\nu}$ is also in the Coulomb slice) and $[\beta+w]$ in some $L^2_1$ neighborhood of $[\beta]$, and, after possibly restricting to a smaller neighborhood $U_2$,
\begin{equation*}
\|\mathrm{grad}\CSd(\beta+w)\|^2_{L^2}\geq C_2'\|(\beta+w)^{\nu}\|^2_{L^2_1}
\end{equation*}
As the path $[\check{\gamma}(t)]$ converges to $[\bcr]$, we can assume that $\check{\gamma}(t)$ lies in the Coulomb slice for $t\geq t_0$, so it lies in $U_1\cap U_2$ and we have that $\CSd$ satisfies the differential inequality
\begin{equation*}
\frac{d}{ds}\CSd(\check{\gamma}(s))=-\|\mathrm{grad}\CSd(\check{\gamma}(t))\|^2_{L^2}\leq -\delta(\CSd(\check{\gamma}(s))-\CSd(\beta))
\end{equation*}
where $\delta=C_2'/C_1$. This implies the exponential decay in the statement. The fact that $\delta>0$ can be chosen uniformly for every critical point follows from the compactness of the space of solutions in the blown down setting, see Corollary $10.7.4$ in the book.
\end{proof}

\vspace{0.5cm}
\begin{proof}[Proof of Theorem \ref{modsame}]
We just need to prove that for every solution $\gamma^{\tau}$ of the perturbed equations on $[0,\infty)\times Y$ converging to a Morse-Bott critical point $[\bcr]$ there exists a $\delta$ depending only in $\bcr$ and a gauge transformation $u$ such that
\begin{equation*}
u\cdot\gamma^{\tau}-\gamma_{\bcr}\in L^2_{k,A_{\bcr}, \delta}([0,\infty)\times Y).
\end{equation*}
We first focus on the case of the solutions in the blow down.
To prove this, consider the sequence of cylinders $[i-1,i+1]\times Y$. From Proposition \ref{nearconstestimate}, we see that there is an $i_0$ such that for every $i\geq i_0$ we can find a sequence of gauge transformations 
\begin{equation*}
u_i\in \G_{k+1}([i-1,i+1]\times Y)
\end{equation*}
such that if we call $\gamma_i=u_i\cdot \gamma^{\tau}$ we have
\begin{equation*}
\|\gamma_i^{\nu}\|_{L^2_{1,A_{\beta}}([i-1,i+1]\times Y)}\leq N_i
\end{equation*}
where
\begin{equation*}
N_i=C\big(\CSd(i-1)-\CSd(i+1))^{1/2}\big).
\end{equation*}
Indeed, $u_i$ is simply the gauge transformation that puts $\gamma^{\tau}$ in Coulomb-Neumann slice on the given interval. On the other hand, the norm of the tangent part $\gamma_i^{t}$ (which is gauge invariant by definition) can be estimated as follows. As in the proof of the inequality (\ref{normtan}), after passing to a smaller neighborhood $U'\subset U$, we have that for any $\varepsilon\in[0,1]$ and $i\geq j_0$ one has that the tangent part of the path satisfies the inequality
\begin{equation*}
\|\check{\gamma}_i^{t}(i-1/2+\varepsilon)-\check{\gamma}_i^{t}(i-1/2)\|_{L^2(Y)}\leq C \int_{i-1/2}^{i-1/2+\varepsilon} \|\check{\gamma}_i^{\nu}(s)\|_{L^2_{1,B}(Y)}ds.
\end{equation*}
As before, this implies the inequality for $s$ big enough
\begin{equation*}
\|\frac{d}{ds}\check{\gamma}_i^{t}(s)\|_{L^2(Y)}\leq C\|\check{\gamma}_i^{\nu}(s)\|_{L^2_{1,B}(Y)}.
\end{equation*}
Define the quantity
\begin{equation*}
M_i=\sum_{j=i}^{\infty} N_j
\end{equation*}
Suppose we have chosen an $L^2_1$-compatible product chart so that the tangent part of a configuration is gauge invariant (see Remark \ref{tangind}). Using the triangular inequality and the fact that $\tilde{\gamma}^t$ converges to zero by hypothesis, the same inequality implies a bound of the form
\begin{multline*}
\|\check{\gamma}_i^{t}(i-1/2+\varepsilon)\|_{L^2(Y)}\leq
\\
\leq \|\check{\gamma}_i^{t}(i-1/2+\varepsilon)-\check{\gamma}_i^{t}(i+1/2)\|_{L^2(Y)}+\sum_{j=i+1}^{\infty} \|\check{\gamma}_j^{t}(j-1/2)-\check{\gamma}_j^{t}(j+1/2)\|_{L^2(Y)}
\leq C M_i
\end{multline*}
for every $\varepsilon\in[0,1]$. Furthermore as the tangent part lies in a finite dimensional submanifold, this implies (after possibly restricting the neighborhood) a bound of the form
\begin{equation*}
\|\check{\gamma}_i^{t}(i-1/2+\varepsilon)\|_{L^2_{1,B}(Y)}\leq C' M_i.
\end{equation*}
By integrating we obtain
\begin{equation*}
\|\check{\gamma}_i^{t}\|_{L^2_{1,B}([i-1,i+1]\times Y)}\leq K\int \big(\|\frac{d}{ds}\check{\gamma}_i^{t}\|_{L^2(Y)}+\|\check{\gamma}_i^{t}\|_{L^2_{1,B}(Y)}\big)\leq K' M_i,
\end{equation*}
hence
\begin{equation*}
\|\check{\gamma}_i-\gamma_{\beta}\|_{L^2_{1,A_{\beta}}([i-1,i+1]\times Y)}\leq K''M_i
\end{equation*}
and finally by a bootstrapping argument
\begin{equation}\label{lknorm}
\|\check{\gamma}_i-\gamma_{\beta}\|_{L^2_{k,A_{\beta}}([i-3/4,i+3/4]\times Y)}\leq K'''M_i.
\end{equation}
The exponential decay of $N_i$ (which follows from Lemma \ref{expdecay}) implies the exponential decay for $M_i$. We can then choose a $\tilde{\delta}>0$ small enough so that
\begin{equation*}
\sum_{i=0}^{\infty} e^{\tilde{\delta} i}M_i< \infty.
\end{equation*}
Now we just need to glue carefully all the gauge transformations $u_i$ in order to obtain one that preserves the summability property. To do so, one just applies the construction in the proof of Proposition $13.6.1$ in the book, which works verbatim.
\par
With this exponential decay in hand, the case of the blow-up (which requires additional work only when the limit is a reducible configuration) follows in a similar fashion by taking also into account the function $\Lambda_{\q}$, which can be used to give bounds analogous to Lemma $13.4.6$ in the book. The only new point here is the exponential decay of the function, which is proved the following lemma.
\end{proof}

\begin{lemma}\label{expint}
Let $\gamma^{\tau}\in \Ct_{k,\mathrm{loc}}$  a solution of the perturbed Seiberg-Witten equations on $[0,\infty)\times Y$ with $\lim_{\rightarrow}[\gamma^{\tau}]=[\bcr]$, a Morse-Bott critical point. Then there is a $\delta>0$ such that the function
\begin{equation*}
f(t)=\Lambda_{\q}(t)-\Lambda_{\q}(\bcr)
\end{equation*}
satisfies the bound $|f(t)|\leq Ce^{-\delta t}$.
\end{lemma}
\begin{proof}
As mentioned above, we can assume that $\bcr$ is reducible. The result follows if we can prove a differential inequality of the form
\begin{equation*}
\frac{d}{dt}f\leq -\delta |f| + K e^{-\delta t}
\end{equation*}
for some $K>0$. By Lemma $13.4.5$ in the book we have the identity
\begin{equation*}
\frac{d}{dt} \Lambda_{\q}(\gamma^{\tau}(t))=-2\|\phi'\|^2+\langle \phi, L' \phi\rangle.
\end{equation*}
Notice that all three terms are gauge invariant. Working in temporal gauge, we have from the proof of Corollary $13.4.8$ in the book and the inequality (\ref{lknorm}) for the blow down of the trajectory that the norm of the second summand is $O(e^{-\delta t})$. We also have the lower bound on the first term
\begin{equation*}
\|\phi'\|^2=\|(\mathrm{grad}\CSd)^{\sigma,1}(\gamma^{\tau}(t))\|\geq C\|\phi^{\nu}\|^2_{L^2_{1,B_{\bcr}}}+O(e^{-\delta t}).
\end{equation*}
To see this, we work in the Coulomb slice $\Sl^{\sigma}_{k,\gamma_{\bcr}}$ on $[t-1,t+1]$. As in Lemma $13.4.6$ in the book we deduce from the fact that $\bcr$ is Morse-Bott and the differentiability of $(\mathrm{grad}\CSd)^{\sigma,1}$ that for all $(B,r,\phi_{\bcr}+\psi)$ in a neighborhood of $\bcr$ in $\Cs_1(Y)$ we have
\begin{equation*}
\|\psi^{\nu}\|^2_{L^2_{1,B_{\bcr}}}\leq C_1\|(\mathrm{grad}\CSd)^{\sigma,1}(B,r,\phi_{\bcr}+\psi)\|^2+K(\|B-B_{\bcr}\|^2_{L^2_1(Y)}+r^2).
\end{equation*}
The bound follows because of the exponential decay of the last term as proved in the proposition above, see equation (\ref{lknorm}).
Finally using again the exponential decay of the connection component we have for a configuration $(B,r,\phi)$ in neighborhood in $\Cs_1(Y)$ of $\bcr$ the estimate
\begin{multline*}
|\Lambda_{\q}(t)-\Lambda_{\q}(\bcr)|=|\mathrm{Re}\langle \phi,D_{\q,B}\phi\rangle-\mathrm{Re}\langle \phi_{\bcr},D_{\q,B_{\bcr}}\phi_{\bcr}\rangle|\leq \\ \leq|\mathrm{Re}\langle \phi,D_{\q,B_{\bcr}}\phi\rangle-\mathrm{Re}\langle \phi_{\bcr},D_{\q,B_{\bcr}}\phi_{\bcr}\rangle|+O(e^{-\delta t})
\leq  C' \|\phi^{\nu}\|^2_{L^2_{1,B_{\bcr}}}+O(e^{-\delta t})
\end{multline*}
where the last estimate follows from the fact that the map on the unit sphere in the $L^2$ norm
\begin{align*}
\mathbb{S}(L^2_1(Y;S))&\rightarrow \mathbb{R}\\
\psi &\mapsto \mathrm{Re}\langle \psi,D_{\q,B_{\bcr}} \psi\rangle
\end{align*}
is $C^2$ and has the unit sphere of the eigenspace to which $\phi_{\bcr}$ belongs to as a Morse-Bott critical submanifold. The claimed differential inequality follows from these two inequalities.
\end{proof}

\vspace{1.5cm}
\section{Transversality}
In this section we deal with transversality for the moduli spaces of solutions on an infinite cylinder. These are treated in the same way as in the Morse case (see Chapter 14 in the book), with the only difference that we need to work in the weighted Sobolev space setting as the extended Hessian now has kernel at the limit configurations. We start by defining the analytical setup for the problem. From now on we suppose that a tame perturbation $\q_0$ such that all critical points are Morse-Bott is fixed.
\\
\par
We recall a useful trick to study differential operators on weighted Sobolev spaces, see for example \cite{Don}). We focus on the first order case as it will be useful later. Suppose we have a vector bundle $E\rightarrow Z$ pulled back from $E\rightarrow Y$ together with a family of $k$-\textsc{asafoe} operators $\{L(t)\}$ on $E$. Then the action of the differential operator
\begin{equation*}
d/dt+L(t): L^2_{k,\delta}(Z;E)\rightarrow L^2_{k-1,\delta}(Z;E)
\end{equation*} 
is conjugated via the multiplication by $f$ (the function defining the weighted Sobolev norm, see equation (\ref{weightedsp})) to the action of the differential operator
\begin{equation}\label{weightedtrick}
d/dt+L(t)+\sigma(t): L^2_{k}(Z;E)\rightarrow L^2_{k-1}(Z;E),
\end{equation}
where again $\sigma$ is defined as $-f'/f$. In particular, we reduced to the more familiar study of the family of $k$-\textsc{asafoe} operators $\{L(t)+\sigma(t)\}$ acting on the unweighted spaces, see Chapter $14$ in the book. The key point of the introduction of the weighted spaces is that for suitable choice of $\delta$ this family of operators will be hyperbolic at the ends even though the one we started with was not.
\\
\par
Consider as in the previous section two contractible open sets $[\mathfrak{U}_-],[\mathfrak{U}_+]$ of two critical submanifolds $[\Lr_-]$ and $[\Lr_+]$. Choose smooth lifts $\mathfrak{U}_-,\mathfrak{U}_+\subset\Cs_k(Y)$ contained in the Coulomb slices $\Sl^{\sigma}_{k,\bcr_{\pm}}$ for some $\bcr_{\pm}$. Fix some $\delta>0$ so that Theorem \ref{modsame} holds. We can then consider Banach manifold of paths $\tilde{\Co}^{\tau}_{k,\delta}(\mathfrak{U}_-,\mathfrak{U}_+)$. We write $\T^{\tau}_{j,\delta}$ for the $L^2_{j,\delta}$ completion of its tangent bundle. In particular the fiber $\T^{\tau}_{j,\gamma}$ at $\gamma=(A_0,s_0,\phi_0)$ with
\begin{equation*}
\ev_{\pm}(\gamma)=\bcr_{\pm}\in \mathfrak{U}_{\pm}
\end{equation*}
is given by the subset
\begin{align*}
\left\{(a,s,\phi)\mid \mathrm{Re}\langle \phi_0(t), \phi(t)\rangle_{L^2(Y)}=0\right\}
\subset L^2_{j,\mathrm{loc}}(Z;iT^*Z)\oplus L^2_{j,\mathrm{loc}}(\R,\R)\oplus L^2_{j,loc,A_0}(Z;S^+)
\end{align*}
consisting of the configurations $v$ such that there exist $v_{\pm}\in T_{\bcr_{\pm}}\mathfrak{C}_{\pm}$ with
\begin{equation*}
v-\gamma_{v_-,v_+}\in  L^2_{j,\delta}(Z;iT^*Z)\oplus L^2_{j,\delta}(\R,\R)\oplus L^2_{j,,\delta, A_0}(Z;S^+)
\end{equation*}
where we define
\begin{equation*}
\gamma_{v_-, v_+}=\beta(-t)v_-+\beta(t)v_+
\end{equation*}
for any fixed smooth cut-off function $\beta$ such that
\begin{equation*}
\beta(s) =
\begin{cases}
0 & \text{if } x \leq 0,\\
1 & \text{if } x \geq 1.
\end{cases}
\end{equation*}
It is clear that such $v_{\pm}$ are unique and we call them the \textit{evaluations} $\ev_{\pm}(v)$. We also define the Hilbert norm of $v\in \T^{\tau}_{j,\gamma}$ as
\begin{equation*}
\|v\|^2_{L^2_{j,\delta}(Z)}:=\|v-\gamma_{\ev_{+}(v),\ev_{-}(v)}\|^2_{L^2_{j,\delta}(Z)}+\|\ev_-(v)\|^2_{L^2(Y)}+\|\ev_+(v)\|^2_{L^2(Y)}.
\end{equation*}
Different choices of the function $\beta$ clearly define equivalent norms. One can also define
\begin{equation*}
\T^{\tau}_{j,\delta,0}=\{v\mid\ev_-(v)=\ev_+(v)=0\}\subset\T^{\tau}_{j,\delta},
\end{equation*}
which is a closed subspace with finite codimension $\mathrm{dim}[\Cr_-]+\mathrm{dim}[\Cr_+]$.
\\
\par
The derivative of the gauge group action, regarded as a bundle map, is given by
\begin{align*}
\mathbf{d}^{\tau}:\mathrm{Lie}(\G_{j+1,\delta})\times \tilde{\Co}^{\tau}_{k,\delta}(\bcr_-,\bcr_+)&\rightarrow \T^{\tau}_{j,\delta}\\ (\xi,\gamma)&\mapsto (-d\xi, 0,\xi\phi_0),
\end{align*}
where $\gamma=(A_0,s_0, \phi_0)$, and the image is contained in the subspace $\T^{\tau}_{j,\delta,0}$.
For a fixed configuration $\gamma$ one can define the Coulomb gauge fixing condition
\begin{equation*}
\mathrm{Coul}^{\tau}_{\gamma}:\Co^{\tau}_{k,\delta}(\mathfrak{U}_-,\mathfrak{U}_+)\rightarrow L^2_{k-1,\delta}(Z;i\R)
\end{equation*}
given by the equation
\begin{equation*}
(A_0+a, s,\phi)\mapsto-d^*a-2\sigma(t)c+iss_0\mathrm{Re}\langle i\phi_0,\phi\rangle + i |\phi_0|^2\mathrm{Re}\left(\mu_Y(\langle i\phi_0,\phi\rangle)\right)\\
\end{equation*}
Here we are considering
\begin{equation*}
a=cdt+b.
\end{equation*}
for a time dependent imaginary valued function $c$ and a family $b$ of imaginary valued one-forms on $Y$, and $\sigma$ is the function depending only on time appearing in equation (\ref{weightedtrick}).
Notice that the image is contained in $L^2_{k-1,\delta}(Z;i\R)$ because we have chosen the lifts $\mathfrak{U}_-,\mathfrak{U}_+$ to be in a Coulomb slice and the family of functions $c$ is in $L^2_j$. For all $1\leq j\leq k$, the linearization of $\mathrm{Coul}^{\tau}_{\gamma}$ extends to the operator
\begin{equation*}
\mathbf{d}^{\tau,\dagger}_{\gamma}:\T^{\tau}_{j,\delta}\rightarrow L^2_{j-1,\delta}(Z;i\R)
\end{equation*}
given by the map
\begin{equation*}
(A_0+a, s,\phi)\mapsto-d^*a-2\sigma(t)c+is_0^2\mathrm{Re}\langle i\phi_0,\phi\rangle + i |\phi_0|^2\mathrm{Re}\left(\mu_Y(\langle i\phi_0,\phi\rangle)\right)
\end{equation*}
The definition of the Coulomb slice we have chosen differs from that in the book (in which the term with the function $\sigma$ is not present) because we are working in the weighted setting. In fact, the operator
\begin{equation*}
a\mapsto-d^*a-2\sigma(t)c
\end{equation*}
is the adjoint of the operator $-d$ in the weighted norm, as it can be easily seen by applying the trick in equation \ref{weightedtrick}.
\par
Define for a configuration $\gamma$ in $\tilde{\Co}^{\tau}_k(\mathfrak{U}_-,\mathfrak{U}_+)$ the subspaces
\begin{IEEEeqnarray*}{c}
\K^{\tau}_{j,\delta,\gamma}\subset \T^{\tau}_{j,\delta,\gamma}\\
\J^{\tau}_{j,\delta,\gamma}\subset \T^{\tau}_{j,\delta,\gamma}.
\end{IEEEeqnarray*}
to be respectively the kernel of $\mathbf{d}^{\tau,\dagger}_{\gamma}$ and the image of $\mathbf{d}^{\tau}_{\gamma}$. One then has the following result, which is the analogue of Proposition $14.3.2$ in the book.
\begin{prop}\label{closedaction}
For each $\delta>$ the subspaces $\J^{\tau}_{j,\delta,\gamma}$ and $\K^{\tau}_{j,\delta,\gamma}$ are complementary at each configuration $\gamma$, so they define a smooth bundle decomposition in closed subbundles
\begin{equation*}
\T^{\tau}_{j,\delta}=\J^{\tau}_{j,\delta}\oplus \K^{\tau}_{j,\delta}.
\end{equation*}
Furthermore, $\J^{\tau}_{j,\delta}=\J^{\tau}_{0,\delta}\cap \T^{\tau}_{j,\delta}$.
\end{prop}
\begin{proof}
We can restrictict our attention to the finite codimension subbundle $\T^{\tau}_{j,\delta,0,\gamma}$
We show that $\J^{\tau}_{j,\delta}$ is a closed subspace at each point. Recall that it is defined as the image of the operator
\begin{IEEEeqnarray*}{c}
\mathbf{d}^{\tau}_{\gamma}:L^2_{j+1,\delta}(Z;i\R)\rightarrow \T^{\tau}_{j,\delta,0}\\ (\xi,\gamma)\mapsto (-d\xi, 0,\xi\phi_0).
\end{IEEEeqnarray*}
We have the estimate
\begin{align*}
\|{\mathbf{d}}^{\tau}\xi\|^2_{L^2_{j,\delta}} &= \|d\xi\|^2_{L^2_{j,\delta}}+\|\xi\phi_0\|^2_{L^2_{j,\delta}}\\
& \geq\frac{1}{2}\|d\xi\|_{L^2_{j,\delta}}^2+\frac{1}{2}\int_{\R}f^2\cdot\left( \|d_Y\xi(t)\|^2_{L^2_j(Y)}+\|\xi(t)\phi_0(t)\|^2_{L^2_j(Y)}\right) dt\\
&\geq \frac{1}{2}\|d\xi\|^2_{L^2_{j,\delta}}+C\|\xi\|^2_{L^2_{\delta}}\\
& \geq C' \|\xi\|^2_{L^2_{j+1,\delta}},
\end{align*}
which shows that the image is closed. Here we used the alternative definition of the weighted norms (see equation (\ref{weightedspequiv})).
\par
The fact that the subspaces are in direct sum follows if we show that
\begin{equation*}
\mathbf{d}^{\tau,\dagger}_{\gamma}\mathbf{d}^{\tau}_{\gamma}: L^2_{j+1,\delta}(Z;i\R)\rightarrow L^2_{j-1,\delta}(Z; i\R)
\end{equation*}
is an isomorphism. This follows because we have chosen the operator $\mathbf{d}^{\tau,\dagger}_{\gamma}$ so that its part acting on one-forms is the adjoint of $-d$ on the weighted spaces. In particular this allows to proof injectivity by the usual integration by parts argument. Finally also the surjectivity and the closedness of the image follow as the proof in the book.
\end{proof}

\begin{cor}
The moduli space of configuration $\Bt_{k,\delta}([\mathfrak{U}_-],[\mathfrak{U}_+])$ is a Hilbert manifold, and the evaluation maps $\ev_{\pm}$ are smooth.
\end{cor}

\begin{proof}
The fact that the quotient is Hausdorff follows as Proposition $13.3.4$ in the book. The lemma above implies that the set
\begin{equation*}
\mathcal{S}^{\tau}_{k,\delta,\gamma}=(\mathrm{Coul}^{\tau}_{\gamma})^{-1}(0)
\end{equation*}
is a slice for the gauge group action. In particular, there is a neighborhood $\mathcal{U}_{\gamma}\subset\mathcal{S}^{\tau}_{k,\delta,\gamma}$ of $\gamma$ such that the restriction of the map induced by the quotient map
\begin{equation*}
\bar{\iota}:\mathcal{S}^{\tau}_{k,\delta,\gamma}\rightarrow \tilde{\Bo}^{\tau}_{k,\delta}([\mathfrak{U}_-],[\mathfrak{U}_+])
\end{equation*}
is a diffeomorphism onto its image, which is a neighborhood of $[\gamma]\in\tilde{\mathcal{B}}^{\tau}_{k,\delta}([\mathfrak{U}_-],[\mathfrak{U}_+])$. Finally, the smoothness of the evaluation maps follows from their smoothness when restricted to a local slice.
\end{proof}

\vspace{0.8cm}
We now describe the linearization of the Seiberg-Witten equations when restricted to a slice. One can define the vector bundle
\begin{equation*}
\V^{\tau}_{j,\delta}\rightarrow \Co^{\tau}_{k,\delta}(\mathfrak{U}_-,\mathfrak{U}_+)
\end{equation*}
whose fiber over $\gamma$ is the vector space
\begin{IEEEeqnarray*}{c}
\V^{\tau}_{j,\delta,\gamma}=\{(\eta, r,\psi)\mid\mathrm{Re}\langle \phi_0(t),\phi(t)\rangle_{L^2(Y)}=0\text{ for all }t\}\\
\subset L^2_{j,\delta}(Z;i\mathfrak{su}(S^+))\oplus L^2_{j,\delta}(\R;\R)\oplus L^2_{j,\delta, A_0}(Z;S^-). 
\end{IEEEeqnarray*}
As in Lemma $14.4.1$ of the book, the perturbed Seiberg-Witten equations define a smooth section $\F^{\tau}_{\q}$ of the vector bundle
\begin{equation*}
\V^{\tau}_{k-1,\delta}\rightarrow \Co^{\tau}_{k,\delta}(\mathfrak{U}_-,\mathfrak{U}_+).
\end{equation*}
As we have chosen $\delta>0$ small enough so that Theorem \ref{modsame} holds, the moduli space of trajectories $M_z([\mathfrak{U}_-],[\mathfrak{U}_+])$ arises as the quotient of the locus
\begin{equation*}
\left\{\F^{\tau}_{\q}=0\right\}
\end{equation*}
by the action of the gauge group. In order to understand its local structure we need to study the derivative of $\F^{\tau}_{\q}$. As $\V^{\tau}_{k,\delta}$ is not a trivial vector bundle, the definition of such a derivative involves a projection. In our case, we define the projection
\begin{equation*}
\Pi^{\tau}_{\gamma}:L^2_{j,\delta}(Z;i\mathfrak{su}(S^+))\oplus L^2_{j,\delta}(\R;\R)\oplus L^2_{j,\delta, A_0}\rightarrow \V^{\tau}_{j,\delta,\gamma}
\end{equation*}
obtained by applying the $L^2$ projection on each slice $\{t\}\times Y$, that is
\begin{equation*}
\Pi^{\tau}_{\gamma}(\eta,r,\psi)= (\eta,r,\Pi^{\perp}_{\phi_0(t)}\psi)
\end{equation*}
where
\begin{equation*}
\Pi^{\perp}_{\phi_0(t)}\psi=\psi-\mathrm{Re}\langle \check{\phi}_0(t),\psi(t)\rangle_{L^2(Y)}\phi_0.
\end{equation*}
Then the derivative $\mathcal{D}\F^{\tau}_{\q}$ is defined as the derivative in the ambient Hilbert space followed by the projection $\Pi^{\tau}_{\gamma}$. Because of condition $\mathrm{(iii)}$ in the definition of a tame perturbation (Definition \ref{tamepert} in Chapter $1$), the derivative extends to smooth bundle maps 
\begin{equation*}
\mathcal{D}\F^{\tau}_{\q}:\T^{\tau}_{j,\delta}\rightarrow \V^{\tau}_{j-1,\delta}
\end{equation*}
for any $j\in[-k,k]$.
\\
\par
The restriction of the operator $\mathcal{D}\F^{\tau}_{\q}$ to the finite codimension subspace $\T^{\tau}_{j,\delta,\gamma_0,0}$ can be interpreted in the following way (we will temporarily ignore the question of regularity). Let
\begin{equation*}
\check{\gamma}_0(t)=(B_0(t), r_0(t), \phi_0(t)):\R\rightarrow \Cs(Y)
\end{equation*}
be the path defined by $\gamma_0$. The codomain of the derivative can be then interpreted as sections along $\check{\gamma}_0$ of the tangent bundle $\T^{\sigma}$, by identifying an endomorphism $\eta$ of the spin bundle $S$ and with an imaginary valued one form $b$ on $Y$ via Clifford multiplication. The domain of the operator can be similarly interpreted as the space of sections $(V,c)$ along $\check{\gamma}_0$ of the bundle
\begin{equation*}
\T^{\sigma}\oplus L^2(Y;i\R)\rightarrow \Cs(Y).
\end{equation*}
Indeed a $1$-form $a$ on $\R\times Y$ can be written as $b+cdt$ with $b$ in temporal gauge and $c$ an imaginary valued function, and hence interpreted as a pair $(b,c)$ consisting of a path $b$ of $1$-forms on $Y$ and a path $c$ of functions on $Y$.
\par
Also in this case the derivative along the path involves a projection, as the vector bundle $\T^{\sigma}(Y)$ is not trivial along the path. We set
\begin{equation*}
\frac{D}{dt}V=(\frac{db}{dt}, \frac{dr}{dt}, \Pi^{\perp}_{\phi_0(t)}\frac{d\psi}{dt})
\end{equation*}
where as before $\Pi^{\perp}_{\phi_0(t)}$ is the orthogonal projection to the orthogonal complement of $\phi_0(t)$.
\par
With these identifications, one can write the operator $\mathcal{D}\F^{\tau}_{\q}$ as 
\begin{equation*}
(V,c)\mapsto \frac{D}{dt}V+\mathcal{D}(\mathrm{grad} \CSd)^{\sigma}(V)+\mathbf{d}^{\sigma}_{\gamma_0(t)}c
\end{equation*}
where $\mathcal{D}(\mathrm{grad} \CSd)^{\sigma}$ is the derivative of the vector field $(\mathrm{grad} \CSd)^{\sigma}$ and $\mathbf{d}^{\sigma}$ is the derivative of the gauge group action on $\Cs(Y)$, see Equation (\ref{devgauges}) in Chapter $1$.
\par
We then impose a Coulomb type gauge-fixing condition, namely the linearization of the Coulomb slice condition discussed above
\begin{equation*}
\mathbf{d}^{\tau,\dagger}_{\gamma_0}(V,c)=0
\end{equation*}
which can be rephrased in our new language by the condition
\begin{equation*}
\left(\frac{d}{dt}-2\sigma(t)c\right)+\mathbf{d}^{\sigma, \dagger}_{\gamma_0(t)}(V)=0
\end{equation*}
where $\mathbf{d}^{\sigma,\dagger}$ is the linearized gauge-fixing operator on $\Cs(Y)$ defined in Section $1$. The appropriate operator to study for our problem is then
\begin{equation*}
Q_{\gamma_0}=\mathcal{D}_{\gamma_0}\F^{\tau}_{\q}\oplus \mathbf{d}^{\tau,\dagger}_{\gamma_0}:
\T^{\tau}_{j,\delta,\gamma_0,0}\rightarrow \V^{\tau}_{j-1,\delta, \gamma_0}\oplus L^2_{j-1,\delta}(Z;i\R)
\end{equation*}
which can be written in path notation if $\gamma_0$ is in path notation
\begin{equation}\label{gaugefixedeq}
(V,c)\mapsto \frac{D}{dt}(V,c)+ L_{\gamma_0(t)}(V,c).
\end{equation}
Here if $\gamma_0(t)$ is $\bcr\in \Cs(Y)$ we have that
\begin{equation}\label{weightedextendedhessian}
L_{\gamma_0(t)}=\begin{bmatrix}
\mathcal{D}_{\bcr}(\mathrm{grad}\CSd)^{\sigma} & \mathbf{d}_{\bcr}^{\sigma} \\
\mathbf{d}_{\bcr}^{\sigma,\dagger} & -2\sigma(t)
\end{bmatrix}
\end{equation}
which is obtained from the extended Hessian we have encountered before in Section $1$ (equation (\ref{exthessian}) by adding the term
\begin{equation*}
-2\sigma(t):L^2_j(Y;i\mathbb{R})\rightarrow L^2_{j-1}(Y;i\mathbb{R})
\end{equation*}
the lower right entry. We call this operator the \textit{weighted extended Hessian} at time $t$, and denote it by $L_{\bcr,t}$. The proof of Lemma \ref{hessian} carries over without modifications to show the following.
\begin{lemma}\label{weighhessian}
If $\bcr\in\Cr$ is a Morse-Bott singularity, then for any $t$ the weighted extended Hessian $L_{\bcr,t}$ is Fredholm of index $0$, has real spectrum and kernel consisting exactly of $\T_{\bcr}\Cr$.
\end{lemma}

\vspace{0.8cm}
We are finally ready to state the basic Fredholm property for the linearized equations on the infinite cylinder. Indeed the choice of the weighted Sobolev spaces is made so that the next result holds.
\begin{prop}\label{fredholmQ}
Suppose we are given two critical submanifolds $\mathfrak{U}_-,\mathfrak{U}_+$ in Coulomb slice. Then for each $\gamma_0\in\Ct_{k,\delta}(\mathfrak{U}_-,\mathfrak{U}_+)$ the linear operator
\begin{equation*}
Q_{\gamma_0}:T^{\tau}_{j,\delta,\gamma_0}\rightarrow \V^{\tau}_{j-1,\delta, \gamma_0}\oplus L^2_{j-1,\delta}(Z;i\R)
\end{equation*}
is Fredholm for every $1\leq j\leq k$ and satisfies the G\aa rding inequality
\begin{equation*}
\|u\|_{L^2_{j,\delta}}\leq C( \|Q_{\gamma_0}u\|_{L^2_{j-1,\delta}}+C_2\|u\|_{L^2_{j-1,\delta}}).
\end{equation*}
The index of $Q_{\gamma_0}$ is independent of $j$ and $\delta>0$ sufficiently small.
\end{prop}

Here by sufficiently small we mean that $\delta>0$ is such that Theorem \ref{modsame} holds, and is also smaller that all values $|\Lambda_{\q}(\bcr)|$ as $\bcr$ varies among the critical submanifolds (see also Lemma \ref{nonzerolambda}).
The latter condition is included to assure invariance of the index, as it will be clear from the spectral flow interpretation. 

\begin{proof}
Here again we restrict to the finite codimension subspace $T^{\tau}_{j,\delta,\gamma_0,0}$, for which we can apply Atiyah-Patodi-Singer techniques. The family $\{L(t)\}$ of weighted extended Hessians in equation \ref{weightedextendedhessian} which is not necessarily hyperbolic at the limit points. As discussed in the introduction of the section idea is to study the operators on the weighted Sobolev spaces, and use the observation of equation $(\ref{weightedtrick})$ to obtain the operator
\begin{equation*}
\frac{d}{dt}+L(t)+\sigma(t)
\end{equation*}
acting on the unweighted spaces. But $\sigma(t)=-\delta$ for $t>>0$ and $\sigma(t)=\delta$ for $t<<0$, so for $\delta>0$ chosen small enough the family $\{L(t)+\sigma(t)\}$ is hyperbolic. Hence, we can apply the Atiyah-Patodi-Singer techniques to this family of operators and the proof of the proposition follows in the exact same way as in Theorem $14.4.2$ of the book.
\end{proof}
\begin{cor}
The restriction of the bundle map $\mathcal{D}\F^{\tau}_{\q}$
\begin{equation*}
\mathcal{D}\F^{\tau}_{\q}:\K^{\tau}_{j,\delta,\gamma}\rightarrow \V^{\tau}_{j-1,\delta, \gamma}
\end{equation*}
is Fredholm and has the same index as $Q_{\gamma}$.
\end{cor}
By Atiyah-Patodi-Singer techniques the index of $Q_{\gamma}$ on the subspace $T^{\tau}_{j,\delta,\gamma_0,0}$, which we denote by $\mathrm{Ind}_0 Q_{\gamma}$, is equal to the spectral flow of the family of operators
\begin{equation*}
\left\{L_{\check{\gamma}(t)}+\sigma(t)\right\},
\end{equation*}
and it should be clear from this description that it depends only on the critical submanifolds $\Cr_{\pm}$.
We then introduce the following definition.

\begin{defn}
Given critical points $\bcr_-,\bcr_+\in\Cs_k(Y)$ which have images respectively in $[\Lr_-]$ and $[\Lr_+]$, we define 
\begin{equation*}
\gr(\bcr_-,\bcr_+)=\mathrm{Ind}_0 Q_{\gamma}+\mathrm{dim}{[\Cr_+]}.
\end{equation*}
Given $[\bcr_-],[\bcr_+]\in \Bs_k(Y)$ and a relative homotopy class  $z\in\pi_1(\Bs_k(Y), [\Lr_-], [\Lr_+])$ connecting them we also define
\begin{equation*}
\gr_z([\bcr_-],[\bcr_-])=\gr(\bcr_-,\bcr_-)
\end{equation*}
for any pair of lifts $\bcr_-,\bcr_+\in \Cs_k(Y)$ such that a path connecting them defines in the quotient the given homotopy class. We call these quantities the \textit{relative gradings} between the critical points. In the similar way we define relative gradings between critical submanifolds.
\end{defn}

The interesting point of this definition is that even though the index of $Q_{\gamma}$ is not additive in the setting of weighted Sobolev spaces, the relative grading is.
\begin{lemma}\label{additivegrading}
If $\acr,\bcr$ and $\mathfrak{c}$ are three critical points, then
\begin{equation*}
\mathrm{gr}(\acr,\mathfrak{c})=\mathrm{gr}(\acr, \bcr)+\mathrm{gr}(\bcr, \mathfrak{c}).
\end{equation*}
\end{lemma}
\begin{proof}This is very well understood with the interpretation of the index as spectral flow. Call the submanifolds these critical points belong to $[\Cr_{\acr}],[\Cr_{\bcr}]$ and $[\Cr_{\mathfrak{c}}]$. Then the indices of $Q_{\gamma}$ and $Q_{\gamma'}$ for some paths $\gamma,\gamma'$ connecting $\acr$ to $\bcr$ and $\bcr$ to $\mathfrak{c}$ on the right hand side can be interpreted as the spectral flow of two families of $k$-\textsc{asafoe} operators
\begin{equation*}
L_t+\sigma(t) \qquad\text{and}\qquad L'_t+\sigma(t)
\end{equation*}
where $L(t)$ and $L'(t)$ are the weighted extended Hessians along a path connecting the critical points, and
\begin{equation*}
\lim_{t\rightarrow+\infty}L_t=\lim_{t\rightarrow-\infty} L'_t=L.
\end{equation*}
On the other hand, one can concatenate the two paths to obtain a third path $\gamma''$ with associated family of operators $L''(t)$. Then
\begin{multline*}
\mathrm{gr}(\acr,\mathfrak{c})=\mathrm{sf}(L''_t+\sigma(t))+\mathrm{dim}[\mathfrak{C}_{\mathfrak{a}}]\\=\left(\mathrm{sf}(L_t+\sigma(t))+\mathrm{sf}(L'_t+\sigma(t))+\mathrm{dim}(\mathrm{ker}L)\right)-\mathrm{dim}[\mathfrak{C}_{\mathfrak{a}}]=\mathrm{gr}(\acr, \bcr)+\mathrm{gr}(\bcr, \mathfrak{c}).
\end{multline*}
The result follows as $\mathrm{ker}L$ is identified with $T_{\bcr}[\mathfrak{C}_{\mathfrak{b}}]$.
\end{proof}

The ideas in the proof above can be exploited to prove the following adaptation of Lemma $14.4.6$ in the book.
\begin{lemma}
For the closed loop $z_u$ based at $[\bcr]\in\Bs_k(Y)$, we have
\begin{equation*}
\gr_{z_u}([\bcr],[\bcr])=\left([u]\cup c_1(S)\right)[Y]
\end{equation*}
where $[u]$ denotes the homotopy class of $u:Y\rightarrow S^1$, identified with an element of $H^1(Y;\mathbb{Z})$.
\end{lemma}

\vspace{0.8cm}

In general we cannot expect the operator $Q_{\gamma}$ to be surjective because of the boundary obstructedness phenomenon that arises already in the finite dimensional case (see Section $2.4$ in the book). We review this important concept in detail in what follows.

\begin{defn}
A reducible configuration $\acr\in\Cs_k(Y)$ (not necessarily a critical point) is \textit{boundary stable} if $\Lambda_{\q}(\acr)<0$ and \textit{boundary unstable} if $\Lambda_{\q}(\acr)>0$. We say that a pair of critical submanifolds $([\Lr_-],[\Lr_+])$ is \textit{boundary obstructed} if the points in $[\Lr_-]$ are boundary-stable and the points in $[\Lr_-]$ are boundary-unstable, and so we call the moduli space of trajectories $M([\Lr_-],[\Lr_+])$.
\end{defn}
To clarify the second part of the definition, notice that because of the Lemma \ref{nonzerolambda}  if a reducible critical point is boundary stable (unstable), then all the points in the critical manifold it belongs to are stable (unstable). Also, if a moduli space $M([\Lr_-],[\Lr_+])$ contains an irreducible trajectory, then $[\Lr_-]$ consists of irreducible or boundary unstable points and $[\Lr_+]$ consists of irreducible or boundary stable points (see Lemma $14.5.3$ in the book).
\\
\par
Suppose that $\gamma$ is a reducible trajectory. Then the operator $Q_{\gamma}$ decomposes as the sum of two operators $Q_{\gamma}^{\partial}$ and $Q_{\gamma}^{\nu}$, reflecting the decomposition of the involution
\begin{IEEEeqnarray*}{c}
\mathbf{i}:\tilde{\mathcal{C}}^{\tau}_{k,\delta}(\mathfrak{U}_-,\mathfrak{U}_+)\rightarrow \tilde{\mathcal{C}}^{\tau}_{k,\delta}(\mathfrak{U}_-,\mathfrak{U}_+)\\
(A,s,\phi)\mapsto (A,-s,\phi).
\end{IEEEeqnarray*}
The first operator is
\begin{equation*}
Q^{\partial}_{\gamma}=(\mathcal{D}_{\gamma}\F^{\tau}_{\q})^{\partial}\oplus \mathbf{d}^{\tau,\dagger}_{\gamma}
\end{equation*}
where
\begin{equation*}
(\mathcal{D}_{\gamma}\F^{\tau}_{\q})^{\partial}:(\T^{\tau}_{k,\delta,\gamma})^{\partial}\rightarrow (\V^{\tau}_{k,\delta,\gamma})^{\partial}
\end{equation*}
is the part invariant under the involution $\mathbf{i}$, while the second operator is
\begin{IEEEeqnarray*}{c}
Q^{\nu}_{\gamma} :L^2_{k,\delta}(\R;i\R)\rightarrow L^2_{k-1,\delta}(\R;i\R) \\ s\mapsto  \frac{ds}{dt}+\Lambda_{\q}(\check{\gamma})s.
\end{IEEEeqnarray*}
The following elementary result (which is Lemma $14.5.4$ in the book) characterizes kernel and cokernel of the operators $Q^{\nu}_{\gamma}$.
\begin{lemma}
The dimensions of the kernel and the cokernel of $Q^{\nu}_{\gamma}$ are:
\begin{itemize}
\item $1$ and $0$ if $\acr$ and $\bcr$ are boundary stable and unstable respectively;
\item $0$ and $1$ if $\acr$ and $\bcr$ are boundary unstable and stable respectively;
\item $0$ and $0$ in the other cases.
\end{itemize}
\end{lemma}

We introduce a first notion of transversality, which corresponds in the finite dimensional case to the condition that stable and unstable manifolds intersect transversely.

\begin{defn}\label{smalereg}
Let $\gamma$ be a solution in $M_z([\Cr_-],[\Cr_+])$. If the pair $([\Cr_-], [\Cr_+])$ is not boundary obstructed, we say that $\gamma$ is \textit{Smale-regular} if $Q_{\gamma}$ is surjective. In the boundary obstructed case, $\gamma$ must be reducible and we say that it is \textit{regular} if $Q_{\gamma}^{\partial}$ is surjective. We say that $M_z([\Cr_-],[\Cr_+])$ is \textit{Smale-regular} if it is \textit{Smale-regular} at any point.
\end{defn}

The point of the definition is that Smale-regular moduli spaces are transversely cut out smooth manifolds by the inverse function theorem, as stated in the following proposition (see Proposition $14.5.7$ in the book for a proof).
\begin{prop}\label{transversality}
Consider critical manifold $[\Cr_{\pm}]$ such that $M_z([\Cr_-],[\Cr_+])$ is Smale-regular, and set
\begin{equation*}
d=\gr_z([\Cr_-],[\Cr_+])+\mathrm{dim}[\Lr_-].
\end{equation*}
Then the moduli space $M_z([\Lr_-],[\Lr_+])$ is:
\begin{itemize}
\item a smooth $d$-manifold consisting entirely of irreducible solutions if either $[\Cr_-]$ or $[\Cr_+]$ is irreducible;
\item a smooth $d$-manifold with boundary $[\Cr_-]$ and $[\Cr_+]$ are boundary-unstable and stable respectively;
\item a smooth $d$-manifold consisting entirely of reducibles if $[\Cr_-]$ and $[\Cr_+]$ are either both boundary stable or unstable;
\item a smooth $d+1$-manifold consisting entirely of reducibles in the boundary obstructed case.
\end{itemize}
In the second case, the boundary of the moduli space consists of the reducible elements.
\end{prop}
\begin{remark}
It is important to notice that the zero locus of the section $\mathfrak{F}_{\q}^{\tau}$ modulo gauge action is \textit{not} the moduli space we are interested in, but a bigger space $\tilde{M}_z([\Cr_-],[\Cr_+])$. The actual space is the quotient of this space by the involution
\begin{IEEEeqnarray*}{c}
\mathbf{i}: \tilde{M}_z([\Cr_-],[\Cr_+])\rightarrow \tilde{M}_z([\Cr_-],[\Cr_+])\\
{[}A,s,\phi]\mapsto[A,-s,\phi],
\end{IEEEeqnarray*}
see also equation (\ref{involution}). Here we use the fact that for a solution the function $s$ is always positive, always negative or constantly zero.
\end{remark}

\vspace{0.8cm}

In what follows, we will need a slightly stronger notion of transversality. When $M_z([\Cr_-],[\Cr_+])$ is Smale-regular, the evaluation maps
\begin{equation*}
\ev_{\pm}:M_z([\Cr_-],[\Cr_+])\rightarrow [\Cr_{\pm}]
\end{equation*}
are smooth maps. Given a sequence of critical submanifolds
\begin{equation*}
\mathcal{C}=\left([\Cr_i]\right)_{i=0,\dots, n}
\end{equation*}
and relative homotopy classes
\begin{equation*}
\mathbf{z}=(z_i)_{i=0,\dots,n-1}
\end{equation*}
with $z_i\in \pi_1(\Bs_k(Y),[\Cr_i],[\Cr_{i+1}])$ we can define the space $M_{\mathbf{z}}(\mathcal{C})$ consisting of $n$-uples $[\gamma_i]\in M_{z_i}([\Cr_i],[\Cr_{i+1}])$ such that
\begin{equation*}
\ev_+[\gamma_i]=\ev_-[\gamma_{i+1}]
\end{equation*}
for every $0\leq i\leq n-2$. This space comes with natural continuous evaluation maps $\ev_{\pm}$ to the critical submanifolds $[\Cr_0]$ and $[\Cr_n]$ respectively.

\begin{defn}\label{regularpert}
Suppose we are given a Morse-Bott tame perturbation $\q$ such that all the moduli spaces are Smale-regular as in Definition \ref{smalereg}. We then say that $\q$ is \textit{regular} if the following holds. For every sequence of critical submanifolds $\mathcal{C}=\left([\Cr_i]\right)_{i=0,\dots, n}$ and corresponding relative homotopy classes $\mathbf{z}$, and any other critical submanifold $[\Cr_+]$ and homotopy class $z_+\in \pi_1(\Bs_k(Y),[\Cr_n],[\Cr_+])$, the maps
\begin{IEEEeqnarray*}{c}
\ev_+:M_{\mathbf{z}}(\mathcal{C})\rightarrow [\Cr_n]\\
\ev_-: M_{z_+}([\Cr_n],[\Cr_+])\rightarrow [\Cr_n]
\end{IEEEeqnarray*}
are transverse smooth maps.
\end{defn}

Notice that unless $M_{\mathbf{z}}(\mathcal{C})$ has a natural smooth structure, the condition of being transverse is not defined. The definition needs to be interpreted in the following inductive way. Suppose that $M_{\mathbf{z}}(\mathcal{C})$ has a canonical smooth structure for which the evaluation maps are smooth (this is true in the case the sequence has length one by the Smale-regularity assumption). Then the space
\begin{equation*}
M_{(\mathbf{z},z_+)}\left([\Cr_0],\dots,[\Cr_n],[\Cr_+]\right)
\end{equation*}
has a natural smooth structure as the fibered product of the two evaluation maps, as these two are smooth transverse map by the regularity assumption. Furthermore the evaluation maps to $[\Cr_0]$ and $[\Cr_-]$ are smooth. 

\begin{remark}
Some approaches (see for example \cite{AB}) impose stronger transversality condition on the moduli spaces, as the fact that the evaluation maps are submersions. On the other hand, it is not hard to construct examples in our problem such that this property does not hold (even after small prturbation). 
\end{remark}

\vspace{0.8cm}
Before proving the main transversality result we characterize the moduli spaces between two reducible critical submanifolds $(\Cr_1,\Cr_2)$ lying over the same reducible configuration $(B,0)\in C\subset\Co_k(Y)$, corresponding to eigenvalues $\lambda_1$ and $\lambda_2$ of $D_{\q,B}$. Let
\begin{equation*}
i=
\begin{cases}
|\{\mu\in\mathrm{Spec}(D_{\q,B})| \lambda_2\leq \mu<\lambda_1\}|, & \text{if } \lambda_1\geq \lambda_2,\\
-|\{\mu\in\mathrm{Spec}(D_{\q,B})| \lambda_2> \mu\geq\lambda_1\}|, & \text{if } \lambda_1\leq \lambda_2.
\end{cases}
\end{equation*}
Of course we are counting eigenvalues with multiplicity. Notice that this definition is slightly different from the one given in Section $14.6$ of the book, and fits the general case better. We have the relative homotopy class $z_0$ joining $[\Cr_1]$ and $[\Cr_2]$ in $\Bs_k(Y)$ arising from a path connecting $\Cr_1$ and $\Cr_2$ in $\Cs_k(Y)$. The is a natural map
\begin{equation}\label{parmod}
\pi: M_{z_0}([\Cr_1][\Cr_2])\rightarrow [C]
\end{equation}
sending a trajectory to the reducible critical point it lies over. We then have the following result, see Section $14.6$ in the book. 
\begin{lemma}\label{dimreducible}
We have the following:
\begin{equation*}
\gr(\Cr_1,\Cr_2)=
\begin{cases}
2i, & \text{if } \lambda_1\text{ and }\lambda_2 \textit{ have the same sign,}\\
2i-1, & \text{if } \lambda_1\text{ is positive and }\lambda_2 \text{ is negative,}\\
2i+1 & \text{if } \lambda_1\text{ is negative and }\lambda_2 \text{ is positive.}
\end{cases}
\end{equation*}
In the last case the relative grading is a negative integer. The moduli spaces are empty when the relative grading is negative. When the relative grading is a non negative number $N$, the map in equation (\ref{parmod}) is a fibration and each fiber is diffeomorphic to the complement in the projective space $\mathbb{C}P^{N+\mathrm{dim}[\Cr_1]}$ of two hyperspaces of codimension respectively $\mathrm{dim}[\Cr_1]$ and $\mathrm{dim}[\Cr_2]$. Furthermore both evaluation maps from $M_{z_0}([\Cr_1][\Cr_2])$ are fibrations.
\end{lemma}
\begin{proof}
This is proved as in Proposition $14.6.1$ in the book by characterizing the moduli spaces of trajectories lying over a given reducible solution $[B,0]$. In particular, these trajectories are identified as the quotient by the action of $\mathbb{C}^*$ of the space of solutions of the translation invariant Dirac equation
\begin{equation*}
\frac{d}{dt}\Phi(t)=-D_{\q,B}\Phi(t)
\end{equation*}
with asymptotics
\begin{align*}
\Phi(t)&\sim c_1 e^{-\lambda_1t}\qquad \text{as }t\rightarrow+\infty\\
\Phi(t)&\sim c_2 e^{-\lambda_2t}\qquad \text{as }t\rightarrow-\infty.
\end{align*}
In particular, we can write such a solution as
\begin{equation*}
\Phi=\sum_{\mu}c_{\mu}e^{-\mu t}\phi_{\mu}
\end{equation*}
where $\mu$ is an eigenvalue in the interval $[\lambda_2,\lambda_1]$, with the coefficients $c_{\lambda_1}$ and $c_{\lambda_2}$ both non zero. In particular the evaluation maps for these spaces are submersions, and the result follows because by definition also the blow down map from $[\Cr_1]$ and $[\Cr_2]$ is a fibration.
\end{proof}

\vspace{0.8cm}

We are now ready to discuss the main transversality result of the section. Recall we have introduced in Definition \ref{adapted} the notion of \textit{adapted} perturbation.

\begin{teor}\label{transversalitymain}
For a fixed Morse-Bott perturbation $\q_0$ and for $\delta>0$ small enough there is a residual subset of the space of adapted perturbations $\mathcal{P}_{\mathcal{O}}$ which consists of regular perturbations.
\end{teor}

\begin{proof}
The proof, which is analogous to the one of Theorem $15.1.1$ in the book, goes through the usual construction of a universal moduli space and Sard-Smale's theorem (see Lemma $12.5.1$ in the book). We first discuss the Smale regularity for the moduli spaces.
We can just work locally with some contractible open sets $[\mathfrak{U}_-]$ and $[\mathfrak{U}_+]$, with smooth lifts in Coulomb gauge $\mathfrak{U}_-$ and $\mathfrak{U}_+$.
Define the smooth map of Banach manifolds
\begin{IEEEeqnarray*}{c}
\mathfrak{N}:\Ct_{k,\delta}(\mathfrak{U}_-,\mathfrak{U}_+)\times\mathcal{P}_{\mathcal{O}}\rightarrow \V^{\tau}_{k-1,\delta}(Z)\\
(\gamma,\q)\mapsto \F^{\tau}_{\q}(\gamma).
\end{IEEEeqnarray*}
The derivative of $\mathfrak{N}$ at a point $(\gamma,\q)$ is a map
\begin{equation*}
\mathcal{D}_{(\gamma,\q)}\mathfrak{N}: \T^{\tau}_{k,\delta,\gamma}\times T_{\q}\mathcal{P}_{\mathcal{O}}\rightarrow \V^{\tau}_{k-1,\delta,\gamma}(Z).
\end{equation*}
We first focus in the case in which $\gamma$ is an irreducible configuration, and need to show that $\mathfrak{N}$ is surjective. After we use the same trick already employed in the proof of Proposition \ref{fredholmQ} in order to reduce to a problem on the unweighted spaces, the proof of Proposition $15.1.3$ in the book works verbatim to prove that the restriction of the map
\begin{equation*}
(\mathcal{D}_{(\gamma,\q)}\mathfrak{N})_0:  \T^{\tau}_{k,\delta,0,\gamma}\times T_{\q}\mathcal{P}_{\mathcal{O}}\rightarrow \V^{\tau}_{k-1,\delta,\gamma}(Z)
\end{equation*}
is surjective. Using then the familiar strategy, we can define the universal moduli space
\begin{align*}
\mathfrak{M}_z([\mathfrak{U}_-],[\mathfrak{U}_+])&=\mathfrak{N}^{-1}(0, s_0)/\G_{k+1,\delta}\\
&\subset\Bt_{k,\delta,z}([\mathfrak{U}_-],[\mathfrak{U}_+])\times\mathcal{P}_{\mathcal{O}},
\end{align*}
which is a smooth Banach manifold and has the property that the projection to $\mathcal{P}_{\mathcal{O}}$ is Fredholm. Hence the Sard-Smale theorem provides a residual subset of regular values in $\mathcal{P}_{\mathcal{O}}$, and we restrict our attention to the subset of adapted ones. 
\par
The case of $\gamma$ reducible not projecting to a reducible critical point in the blow down is analogous to the classical case, as in this case the derivative of $\mathfrak{N}$ has the summand
\begin{equation*}
(\mathcal{D}_{(\gamma,\q)}\mathfrak{N})^{\partial}: (\T^{\tau}_{k,\delta,\gamma})^{\partial}\times T_{\q}\mathcal{P}_{\mathcal{O}}\rightarrow (\V^{\tau}_{k-1,\delta,\gamma}(Z))^{\partial},
\end{equation*}
and the result follows in an identical way using the surjectivity of the restriction
\begin{equation*}
(\mathcal{D}_{(\gamma,\q)}\mathfrak{N})_0^{\partial}:  (\T^{\tau}_{k,\delta,0,\gamma})^{\partial}\times T_{\q}\mathcal{P}_{\mathcal{O}}\rightarrow (\V^{\tau}_{k-1,\delta,\gamma}(Z))^{\partial},
\end{equation*}
whose proof is also contained in the proof of Proposition $15.1.3$ in the book. When the trajectory $\gamma$ lies over a single reducible $(B,0)$, as in the proof in the usual case we need to show that the corresponding linearized operators is surjective. Considering as an example the spinor part, its adjoint in the weighted Sobolev norm in equation (\ref{weightedsp}) can be identified (as in the proof of Proposition \ref{fredholmQ}) with the operator on the unweighted Sobolev spaces
\begin{equation*}\label{adjointsurj}
\psi\mapsto-\frac{d}{dt}\psi+D_{\q,B}\psi+\sigma(t)\psi.
\end{equation*}
Notice that this is not the adjoint in the equivalent norm in equation (\ref{weightedspequiv}).
The operator
\begin{equation*}
\psi\mapsto-\frac{d}{dt}\psi+D_{\q,B}\psi
\end{equation*}
is an isomomorphism on the unweighted spaces as zero is not in the spectrum of $D_{\q,B}$ (see Lemma \ref{nonzerolambda}). Hence for $\delta>0$ small enough (and appropriate choice of the function $f$ defining the weighted Sobolev norm) the operator in equation (\ref{adjointsurj}) will also be an isomorphism, and indeed we can find a $\delta>0$ that works for all the reducible critical submanifolds by compactness.
\par
The regularity statement follows essentially in the same way. For example in the case of a sequence of length three $[\Cr_-],[\Cr_0],[\Cr_+]$ at a point in which both trajectories $([\gamma_-],[\gamma_+])$ are irreducible. The regularity property at this point is equivalent to the fact that the map
\begin{IEEEeqnarray*}{c}
\mathfrak{N}':\Ct_{k,\delta}(\mathfrak{U}_-,\mathfrak{U}_0)\times \Ct_{k,\delta}(\mathfrak{U}_0,\mathfrak{U}_+)\times\mathcal{P}_{\mathcal{O}}\rightarrow \V^{\tau}_{k-1,\delta}(Z)\oplus \V^{\tau}_{k-1,\delta}(Z)\times \mathfrak{U}_0\times \mathfrak{U}_0\\
(\gamma_-,\gamma_+,\q)\mapsto \left(\F^{\tau}_{\q}(\gamma_-),\F^{\tau}_{\q}(\gamma_+), \ev_+(\gamma_-), \ev_-(\gamma_+)\right).
\end{IEEEeqnarray*}
is transverse to $\left(\{0\}\oplus \{0\}\right)\times\Delta$, where
\begin{equation*}
\Delta\subset \mathfrak{U}_0\times \mathfrak{U}_0
\end{equation*}
is the diagonal. This will follow if we can prove that the restrictions of the linearizations
\begin{equation*}
\mathcal{D}_{(\gamma_-,\gamma_+,\q)}\mathfrak{N}'_0:  \T^{\tau}_{k,\delta,0,\gamma_-}\times \T^{\tau}_{k,\delta,0,\gamma_+}\times T_{\q}\mathcal{P}_{\mathcal{O}}\rightarrow \V^{\tau}_{k-1,\delta,\gamma_-}(Z)\times \V^{\tau}_{k-1,\delta,\gamma_+}(Z)\times\{0\}\times\{0\}
\end{equation*}
are surjective. Even though this does not follow from the surjectivity of the maps $(\mathcal{D}_{(\gamma,\q)}\mathfrak{N})_0$ above, the proof of Proposition $15.1.3$ is readily adapted to this case using unique continuation and the properties of cylinder functions discussed regarding embedding of compact sets of the based configuration space, see Proposition \ref{densepert} in Chapter $1$. Finally, in the case $[\Cr_1]$ and $[\Cr_2]$ are critical submanifolds lying over the same reducible $[C]$ the evaluation maps on the moduli spaces lying over constant trajectories are submersions because of Lemma \ref{dimreducible}.
\end{proof}

\vspace{1.5cm}
\section{Compactness and finiteness}
In this section, which closely follows Chapter $16$ in the book, we discuss the compactess properties for moduli spaces of trajectories, and construct the space of unparametrized broken trajectories.
\\
\par
We say that a trajectory in the moduli space $M_z([\Cr_-],[\Cr_+])$ is \textit{non trivial} if it is not invariant under the action of $\R$ by translations on the infinite cylinder. This is equivalent to say that either it has distinct endpoints or if they coincide then the relative homotopy class $z$ is non trivial.
\begin{defn}
An \textit{unparametrized trajectory} connecting $[\Cr_-]$ to $[\Cr_+]$ is an equivalence class of non trivial trajectories in $M_z([\Cr_-],[\Cr_+])$ under the action of translations. We write
\begin{equation*}
\breve{M}_z([\Cr_-],[\Cr_+])
\end{equation*}
for the space of unparametrized trajectories.
\end{defn}

\begin{defn}
An \textit{unparametrized broken trajectory} joining two critical submanifolds $[\Cr_-]$ to $[\Cr_+]$ consists of the following data:
\begin{itemize}
\item an integer $n\geq 0$, the \textit{number of components};
\item an $(n+1)$-tuple of critical submanifolds $[\Cr_0],\dots, [\Cr_n]$ with $[\Cr_0]=[\Cr_-]$ and $[\Cr_n]=[\Cr_+]$, the \textit{resting submanifolds};
\item for each $1\leq i\leq n$, an unparametrized trajectory
\begin{equation*}
[\breve{\gamma}_i]\in\breve{M}_{z_i}([\Cr_{i-1}],[\Cr_i]),
\end{equation*}
the \textit{i-th} \textit{component} of the broken trajectory, with the property that
\begin{equation*}
\ev_+[\breve{\gamma}_i]=\ev_-[\breve{\gamma}_{i+1}]
\end{equation*}
and we call this critical point the $i$\textit{th} \textit{restpoint}.
\end{itemize}
The \textit{homotopy class} of the broken trajectory is the relative homotopy class of the path obtained by concatenating representatives of the classes $z_i$. We write $\breve{M}^+_z([\Cr_-],[\Cr_+])$ for the space of unparametrized trajectories in the homotopy class $z$, and write the typical element as
\begin{equation*}
[\boldsymbol{\breve{\gamma}}]=([\breve{\gamma}_1],\dots,[\breve{\gamma}_n]).
\end{equation*}
Finally, there are naturally defined evaluation maps
\begin{equation*}
\ev_{\pm}:\breve{M}^+_z([\Cr_-],[\Cr_+])\rightarrow [\Cr_{\pm}].
\end{equation*}
\end{defn}
\begin{remark}We consider also broken trajectories with $n=0$ components for bookkeping purposes. If $z$ is the class of the constant path for the submanifold $[\Cr]$, then $\breve{M}^+_z([\Cr],[\Cr])$ consists of a single point, the broken trajectory with no components.
\end{remark}

\vspace{0.8cm}

The space of unparametrized broken trajectories is topologized as follows. Consider an element
\begin{equation*}
[\boldsymbol{\breve{\gamma}}]=([\breve{\gamma}_1],\dots,[\breve{\gamma}_n])\in\breve{M}^+_z([\Cr_-],[\Cr_+]),
\end{equation*}
with $[\breve{\gamma}_i]\in\breve{M}_{z_i}([\Cr_{i-1}],[\Cr_i])$ being represented by a parametrized trajectory
\begin{equation*}
[\gamma_i]\in M_{z_i}([\Cr_{i-1}],[\Cr_i]).
\end{equation*}
Let $U_i\subset \Bt_{k,\mathrm{loc}}(Z)$ an open neighborhood of $[\gamma_i]$, and let $T\in\R^+$. We define
\begin{equation*}
\Omega=\Omega(U_1,\dots U_n,T)
\end{equation*}
to be the subset of $\breve{M}^+_z([\Cr_-],[\Cr_+])$ consisting of broken unparametrized trajectories
\begin{equation*}
[\boldsymbol{\breve{\delta}}]=([\breve{\delta}_1],\dots,[\breve{\delta}_m])
\end{equation*}
satisfying the following condition. There exists a map
\begin{equation*}
(\jmath,s):\{1,\dots n\}\rightarrow \{1,\dots,m\}\times \R
\end{equation*}
such that
\begin{itemize}
\item $[\tau^*_{s(i)}\delta_{\jmath(i)}]\in U_i$;
\item if $1\leq i_1\leq i_2\leq n$, then either $\jmath(i_1)\leq \jmath(i_2)$, or $\jmath(i_1)= \jmath(i_2)$ and $s(i_1)+T\leq s(i_2)$.
\end{itemize}
We take the sets of the form $\Omega=\Omega(U_1,\dots U_n,T)$ to be a neighborhood base for $[\breve{\gamma}]$ in $\breve{M}^+_z([\Cr_-],[\Cr_+])$.

\vspace{0.8cm}
The first goal of the section is to prove the following result (see Proposition $16.1.4$ in the book). While in the proof we will not require transversality hypothesis on the moduli spaces, the latter will be useful later in the section when discussing finiteness conditions.
\begin{teor}\label{compactness}
For any $C>0$ and critical submanifolds $[\Cr_{\pm}]$, there are only finitely many $z$ with energy $\mathcal{E}_{\q}(z)\leq C$ for which $\breve{M}^+_z([\Cr_-],[\Cr_+])$ is non-empty. Furthermore each space $\breve{M}^+_z([\Cr_-],[\Cr_+])$ is compact.
\end{teor}
The proof of this proposition is essentially identical to the one in the Morse case (Proposition $16.1.4$ in the book), with just an additional detail to be fixed. We first show that a stronger compactness result holds \textit{downstairs}, i.e. for blown down trajectories. This will follow from the compactness properties for the Seiberg-Witten equations on a finite cylinder (see Chapter $5$ in the book). We will then deduce the case we are actually interested in. 
\\
\par
Suppose our critical submanifolds $[\Cr_{\pm}]$ blow down to critical submanifolds $[C_{\pm}]$ in $\Bo_k(Y)$, we introduce moduli spaces
\begin{IEEEeqnarray*}{c}
N_z([C_-],[C_+])\subset \Bo_{k,\mathrm{loc}}(Z) \\
N([C_-],[C_+])=\bigcup_z N_z([C_-],[C_+])
\end{IEEEeqnarray*}
of solutions to the perturbed equations asymptotic to $[C_-]$ and $[C_+]$ at $\pm\infty$. This comes with a blow down map
\begin{equation*}
\mathbf{\pi}:M_z([\Cr_-],[\Cr_+])\rightarrow N_z([C_-],[C_+]).
\end{equation*}
We define a trajectory in $N_z([C_-],[C_+])$ to be non-trivial if it is not invariant under translation. We can introduce for these moduli spaces the analogues of the unparametrized trajectories $\breve{N}_z([C_-],[C_+])$ and broken trajectories $\breve{N}^+_z([C_-],[C_+])$, and the disjoint union
\begin{equation*}
\breve{N}^+=\bigcup_{[C_-],[C_+]} \breve{N}^+([C_-],[C_+]).
\end{equation*}

For this space, we have the following compactness result, which differs from Proposition \ref{compactness} as it deals with the union over all critical submanifolds, see Proposition $16.2.1$ in the book.
\begin{prop}\label{compdownstairs}
For any $C>0$ there are only finitely many $[C_-],[C_+]$ and $z$ such that $\mathcal{E}_{\q}(z)\leq C$ and the space $\breve{N}^+_z([C_-],[C_+])$ is non empty. Furthermore, each $\breve{N}^+_z([C_-],[C_+])$ is compact. In other words, the space of broken trajectories in $\breve{N}^+$ with energy $\mathcal{E}_{\q}\leq C$ is compact.
\end{prop}

The key ingredient in the proof is the following basic lemma. As this is the only point in which the proof of Proposition \ref{compdownstairs} differs from the one of Proposition $16.2.1$ in the book, we discuss its proof in more detail.
\begin{lemma}\label{convergencefinite}
Let $[\gamma]\in \Bo_{k,\mathrm{loc}}$ be a solution of the equations with finite energy. Then $[\gamma]\in N([C_-],[C_+])$ for some critical submanifolds $[C_-],[C_+]$.
\end{lemma}
Before proving this lemma, we recall a basic lemma following from the compactness properties of the moduli space on a finite cylinder (see Lemma $16.2.2$ in the book) and a useful definition (see also Definition $16.2.3$ in the book).
\begin{lemma}\label{convfinite1}
Fix a collection $\mathcal{A}$ of critical points $[\alpha]$ in $\Bo_k(Y)$ and for each of them a gauge invariant open neighborhood $U_{\alpha}\subset \Co_k(I\times Y)$ of the translation invariant configuration $\gamma_{\alpha}$ such that their union contains all translation invariant solutions. Let $C_0$ be any constant, and $I'\subset I$ any other interval of non zero length. Then there exists $\epsilon>0$ such that if $\gamma$ is a trajectory satisfying
\begin{equation*}
\mathcal{E}_{\q}(\gamma)\leq C\qquad\text{and}\qquad \mathcal{E}^{I'}_{\q}(\gamma)\leq\epsilon,
\end{equation*}
then $\gamma|_{I\times Y}\in U_{\alpha}$ for some critical point $[\alpha]$ in $\mathcal{A}$.
\end{lemma}

Here by $\mathcal{E}_{\q}^{I'}(\gamma)$ we mean the (perturbed) energy of the trajectory when restricted to the interval $I'$. As the critical submanifolds are compact, we can suppose that the family $\mathcal{A}$ contains only a finite number of manifolds from each of them.

\begin{defn}Fix an interval $I$. We say that a collection $\mathcal{U}$ of gauge invariant neighborhoods $U_{\alpha}\subset \Co_k(I\times Y)$ with $\alpha\in\mathcal{A}$ as in the previous lemma has the \textit{separating property} if the following holds. There should exist neighborhoods $V_{[\alpha]}\subset \Bo_{k-1}(Y)$ of the critical submanifolds such that
\begin{itemize}
\item $V_{[\alpha]}$ and $V_{[\alpha']}$ are disjoint if $\alpha$ and $\alpha'$ do not belong to the same critical submanifold;
\item each $V_{[\alpha]}$ is path-connected and simply connected;
\item if $\gamma\in U_{\alpha}$, then $[\breve{\gamma}(t)]\in V_{[\alpha]}$ for every $t\in I$.
\end{itemize}
For a critical submanifold $[C]$ we define $V_{[C]}$ to be neighborhood of $[C]$ obtained as the union of the $V_{[\alpha]}$ with $[\alpha]$ in $[C]$.
\end{defn}

\begin{proof}[Proof of lemma \ref{convergencefinite}]
Fix an interval $I$ and a collection of neighborhoods $\mathcal{U}$ with the separating property. The finite energy condition implies that the translates $\tau^*_t(\gamma)$ are such that
\begin{equation*}
\mathcal{E}_{\q}^I(\tau^*_t\gamma) \rightarrow0\qquad\text{for }t\rightarrow+\infty
\end{equation*}
so from Lemma \ref{convfinite1} above the translate $(\tau^*_t\gamma)|_I$ belongs to $U_{\alpha_t}$ for some critical point $[\alpha_t]$ in $\mathcal{A}$. Because of the separating property we have that $[\check{\gamma}(t)]\in V_{[C]}$ for some critical submanifold $[C]$ for all $t\geq t_0$. By choosing big intervals $I$ and small neighborhoods $V_{[C]}$ this shows that the function $\CSd(\check{\gamma}(t))$ converges for $t\rightarrow +\infty$. We then need to prove that the solution has exactly one limit point on such a critical submanifold, i.e. it does not ``spiral around". This phenomenon might happen already in finite dimensions when considering the gradient flow of a smooth function, but does not happen for analytic functions (see \cite{MMR} and \cite{Don}).
\par
We define the $L^2$ metric on $\Bo_{k-1}(Y)$ given by
\begin{equation*}
d([\alpha],[\alpha'])=\mathrm{inf}\left\{ \|\alpha-u\cdot \alpha'\|_{L^2(Y)}\mid u\in \G_{k}(Y)\right\}.
\end{equation*}
To check that this is a metric we need to show that two configurations $\alpha,\alpha'\in \Co_{k-1}(Y)$ such that there is a sequence $\{u_n\}_{n\in\mathbb{N}}\subset \G_{k}(Y)$ with $u_n\cdot \alpha'$ converging to $\alpha$ in the $L^2$ norm are actually gauge equivalent. This is proved in a similar way as Proposition $9.3.1$ in the book. Indeed, we can suppose each $u_n$ is in the identity component of the gauge group, so it can be written as $e^{\xi^0_n+\xi^{\perp}_n}$. Calling $B,B'$ the connection component of the configurations, we have that
\begin{equation*}
d\xi_n^{\perp}-(B-B')\rightarrow0\text{ in }L^2.
\end{equation*}
In particular the $\xi_n^{\perp}$ form a Cauchy sequence in $L^2_1$, so they converge in the $L^2_1$ topology to a configuration $\xi^{\perp}$ which is in $L^2_{k}$ by elliptic regularity. This implies that the gauge transformations $u_n$ converge in the $L^2_1$ topology to a gauge transformation $u\in\G_{k}$. On the other hand, $u_n\cdot\alpha'$ converges in the $L^2$ norm to $u\cdot\alpha'$ (for the spinor part we use that in dimension three $L^2_1$ embeds in $L^6$ hence in $L^4$), so the latter coincides with $\alpha$.
\par
The result will follow by taking the initial interval $I$ as large as we wish and the separating neighborhood as small as we want if we show that the path $[\check{\gamma}(t)]$ has finite length in this metric. 
Let $I=[s-1,s+1]$, and suppose that $\gamma\mid_{I\times Y}\in U_{\alpha}$ for some $\alpha$ in $C$. We will suppose the configuration is in Coulomb-Neumann gauge with respect to $\gamma_{\alpha}$. After possibly restricting both family neighborhoods, we can carry over the estimates of Lemma \ref{expdecay} to the function ${\CSd}(\check{\gamma}(s))-{\CSd}(\check{\gamma}([C])$. Furthermore, because there are only finitely many neighborhoods $U_{\alpha}$ involved, we can choose a constant $C$ so that this in equality holds for all of them. As the Chern-Simons-Dirac operator $\CSd$ converges along the trajectory to $\CSd([C])$, this implies that its value along the trajectory converges exponentially fast. By applying the Cauchy-Schwarz inequality we have that
\begin{multline*}
\int_{s-1}^{s+1} \|\mathrm{grad}\CSd(\check{\gamma}(t))\|_{L^2(Y)}dt\leq C \left(\int_{s-1}^{s+1} \|\mathrm{grad}\CSd(\check{\gamma}(t))\|_{L^2(Y)}^2dt\right)^{1/2} =\\
=C\left({\CSd}(\check{\gamma}(s-1))-{\CSd}(\check{\gamma}(s+1))\right)^{1/2}.
\end{multline*}
The rightmost term is a well defined real number for $s$ big enough. The exponential convergence of $\CSd$ along the trajectory implies that the integral of $\|\mathrm{grad}\CSd(\check{\gamma}(t))\|_{L^2(Y)}$ is finite. The length of path $[\check{\gamma}(t)]$ in the $L^2$ metric is bounded above by such integral, as by the flow equations $\|\mathrm{grad}\CSd(\check{\gamma}(t))\|_{L^2(Y)}$ is exactly the $L^2$ norm of its derivative, and the result follows.
\end{proof}

\begin{proof}[Proof of Theorem \ref{compactness}]
This follows as the one in the book with modifications (to be made in the reducible case) analogous to those we discussed above for Lemma \ref{convergencefinite}. The additional complication comes from the function $\Lambda_{\q}$, which can be dealt with as in the end of Section $2$. It is useful to remark that this function has limits at both $\pm\infty$ on a given trajectory. This is because its value at a configuration $[\bcr]$ on a reducible critical submanifold depends only on its blowdown $\pi_*([\bcr])$.
\end{proof}

\vspace{1.5cm}

To get rid of the energy bound in the assumption of Proposition \ref{compactness}, we need some regularity assumptions on the moduli spaces, which assure some strong finiteness results on the set of non empty moduli spaces as in the following lemma (see Proposition $16.4.1$ and $16.4.3$ in the book).
\begin{prop}\label{finiteness}
Suppose that a regular Morse-Bott perturbation $\q_0$ has been fixed. Then for given critical submanifolds $[\Cr_-]$ and $[\Cr_+]$ there are only finitely many relative homotopy classes $z$ for which the moduli space $\breve{M}^+_z([\Cr_-],[\Cr_+])$ is non-empty. Furthermore:
\begin{itemize}
\item if $c_1(\spin)$ is torsion then for a given $[\Cr_-]$ and any $d_0$ there are only finitely many pairs $([\Cr_+],z)$ for which $\breve{M}^+_z([\Cr_-],[\Cr_+])$ is non-empty and of dimension at most $d_0$.
\item if $c_1(\spin)$ is not torsion, suppose that the perturbation has been chosen so that there are no reducible solutions (see Section $4.2$ in the book). Then there are only finitely many triples $([\Cr_-],[\Cr_+],z)$ such that the moduli space $\breve{M}^+_z([\Cr_-],[\Cr_+])$ is non-empty (without restrictions on the dimension).
\end{itemize}
\end{prop}
Before giving a proof of Proposition \ref{finiteness} we recall a useful definition from the book. Suppose that a reducible critical point $\acr$ lies over the configuration $(B,0)\in \Co_k(Y)$, and corresponds to the element $\lambda\in\mathrm{Spec}(D_{\q,B})$. In this case, we define
\begin{equation*}
\iota(\acr)=
\begin{cases}
|(\mathrm{Spec}(D_{\q,B})\cap[0,\lambda)|, & \text{if }\lambda> 0,\\
1/2-|(\mathrm{Spec}(D_{\q,B})\cap[0,\lambda)|, & \text{if } \lambda<0,
\end{cases}
\end{equation*}
where of course eigenvalues are counted with multiplicity. If $\acr$ is irreducible, we set $\iota(\acr)=0$. This definition is set up so that if $[\acr]$ and $[\acr']$ are two critical points whose blow down is the same critical point $[\alpha]\in \Bo_k(Y)$, then by Lemma \ref{dimreducible}
\begin{equation*}
\gr_{z_0}([\acr],[\acr'])=2(\iota(\acr)-\iota(\acr'))
\end{equation*}
for the trivial homotopy class $z_0$. Furthermore the value of $\iota$ is constant on a critical submanifold $[\Cr]$, hence we can univocally define the value value $\iota[\Cr]$.

\begin{proof}[Proof of Proposition \ref{finiteness}]
Suppose there is $[\boldsymbol{\breve{\gamma}}]\in \breve{M}_z([\Cr_-],[\Cr_+])$, and suppose that the resting manifolds are
\begin{equation*}
[\Cr_-]=[\Cr_0], [\Cr_1],\dots, [\Cr_{n-1}],[\Cr_n]=[\Cr_+].
\end{equation*}
The fact that the moduli space $M_z([\Cr_0],[\Cr_1])$ is Smale-regular implies that
\begin{equation*}
\mathrm{dim}[\Cr_0]+\gr_{z_0}([\Cr_0],[\Cr_1])\geq 0,
\end{equation*}
where the inequality is strict in the not boundary obstructed case. The regularity condition implies that the evaluation maps
\begin{IEEEeqnarray*}{c}
\ev_+:M_z([\Cr_0],[\Cr_1])\rightarrow[\Cr_1]\\
\ev_-:M_z([\Cr_1],[\Cr_2])\rightarrow[\Cr_1]
\end{IEEEeqnarray*} 
are transverse, and as they have non disjoint image we have that
\begin{equation*}
\left(\mathrm{dim}[\Cr_0]+\gr_{z_0}([\Cr_0],[\Cr_1])\right)+\left(\mathrm{dim}[\Cr_1]+\gr_{z_1}([\Cr_1],[\Cr_2])\right)\geq \mathrm{dim}[\Cr_1].
\end{equation*}
By induction on the number of components and using the additivity of the relative grading (Lemma \ref{additivegrading}) we then prove
\begin{equation*}
\gr_{z}([\Cr_-],[\Cr_+])\geq-\mathrm{dim}[\Cr_-].
\end{equation*}
The proof of Lemma $16.4.4$ in the book is then easily adapted to our case (using the compactness of the blow down of the critical set) to show that there exists $C_0$ such that for every $[\Cr_-],[\Cr_+]$ and $z$, and any broken trajectory $[\breve{\gamma}]\in \breve{M}_z^+([\Cr_-],[\Cr_+])$, we have the energy bound
\begin{equation*}
\mathcal{E}_{\q}(\breve{\gamma})\leq C + 8\pi^2(\iota[\Cr_-]-\iota[\Cr_+]).
\end{equation*}\end{proof}

\begin{remark}
Unlike the classical case, when the spin$^c$ structure is not torsion the moduli spaces of the form $M_z([\Cr],[\Cr])$ might be not empty.
\end{remark}

\vspace{0.8cm}

The rest of the present section (and chapter) is dedicated to understand in detail the structure of the moduli spaces of unparametrized broken trajectories. We first introduce a key definition from the book.

\begin{defn}
A topological space $N^d$ is a $d$\textit{-dimensional space stratified by manifolds} if there are closed subsets
\begin{equation*}
N^d\supset N^{d-1}\supset \dots \supset N^0\supset N^{-1}=\emptyset
\end{equation*}
such that $N^d\neq N^{d-1}$ and each space $N^e\setminus N^{e-1}$ (for $0\leq e\leq d$) is either empty or homeomorphic to a manifold of dimension $e$. We call $N^e\setminus N^{e-1}$ the $e$\textit{-dimensional stratum}. We will also call stratum any union of path components of $N^e\setminus N^{e-1}$.
\end{defn}

\begin{example}\label{stratexample}
Spaces stratified by manifolds (even compact ones) allow some pathologies.
Consider the space $N^1$ obtained as the union over all $n\in\mathbb{N}$ of all circles $C_n$ with center in $(-1/n,0)$ and radius $1/n$ and the segment joining $(0,0)$ and $(1,0)$. In this case $N^0=\{(0,0), (1,0)\}$, and $N^1\setminus N^0$ is a $1$-manifold with countably many path components.
\end{example}
Consider a sequence of critical manifolds $\mathcal{C}=\left([\Cr_i]\right)_{i=0,\dots, n}$ and corresponding relative homotopy classes $\mathbf{z}=\left(z_i\right)_{i=1,\dots, n}$. We can then define the subspace
\begin{equation*}
\breve{M}_{\mathbf{z}}(\mathcal{C})
\end{equation*}
consisting of unparametrized broken trajectories such that the $i$th restpoint lies on the critical manifold $[\Cr_i]$, and the $i$th component is in the relative homotopy class $z_i$. When the perturbation is regular, this subspace has a natural manifold structure as the quotient of open set the smooth fibered product $M_{\mathbf{z}}(\mathcal{C})$ introduced in Section $4$ consisting of $n$-uples such that each component is non trivial by the action of $\R^n$ given time translations on each component. The following result is then the analogue of Proposition $16.5.2$ and $16.5.5.$ in the book, and the proofs applies verbatim.
\begin{prop}\label{codim1}
Suppose we have fixed a regular perturbation $\q$, and let $M_z([\Lr_-],[\Lr_+])$ a $d$-dimensional moduli space containing irreducibles. Then the space of broken unparametrized trajectories $\breve{M}^+_z([\Lr_-],[\Lr_+])$ is a compact $(d-1)$-dimensional space stratified by manifolds, and the top stratum consists of the irreducible part of $\breve{M}_z([\Lr_-],[\Lr_+])$. Furthermore, the $(d-2)$-dimensional stratum of $\breve{M}^+_z([\Lr_-],[\Lr_+])$ is the union of pieces of three types:
\begin{itemize}
\item strata of the form
\begin{equation*}
 \breve{M}_{(z_1,z_2)}([\Lr_-],[\Cr],[\Lr_+])
\end{equation*}
where none of the pairs of consecutive critical manifolds is boundary obstructed;
\item strata of the form
\begin{equation*}
 \breve{M}_{(z_1,z_2,z_3)}([\Lr_-],[\Cr'],[\Cr''],[\Lr_+])
\end{equation*}
where only the middle moduli space is boundary obstructed;
\item the intersection of $\breve{M}_z([\Lr_-],[\Lr_+])$ with the reducibles, in the case $M_z([\Lr_-],[\Lr_+])$ contains both reducibles and irreducibles.
\end{itemize}
The last case happens only when $[\Cr_-]$ is unstable and $[\Cr_+]$ is stable.
\end{prop}
\vspace{0.8cm}
Finally, we discuss the case of reducible trajectories. We will write $M^{\mathrm{red}}_z([\Lr_-],[\Lr_+])$ for the subset of $M_z([\Lr_-],[\Lr_+])$ consisting of the reducible trajectories. Because of our transversality hypothesis, this is either empty, all of $M_z([\Lr_-],[\Lr_+])$, or the boundary of $M_z([\Lr_-],[\Lr_+])$ in the case $[\Lr_-]$ is boundary unstable and $[\Lr_+]$ is boundary stable. In any case we can introduce a modified relative grading given by
\begin{equation*}
\bar{\gr}_z([\Cr_-],[\Cr_+])=\gr_z([\Cr_-],[\Cr_+])-o[\Cr_-]+o[\Cr_+]
\end{equation*}
where we define
\begin{equation*}
o[\Cr]=
\begin{cases}
0, & \text{if }[\Cr]\text{ is boundary stable,}\\
1, & \text{if }[\Cr]\text{ is boundary unstable.}
\end{cases}
\end{equation*}
We can also introduce the spaces $\breve{M}^{\mathrm{red}}$ and $\breve{M}^{\mathrm{red}+}$, as the intersections of $\breve{M}$ and $\breve{M}^+$ with the reducibles. The situation for these moduli spaces is simpler, as we are essentially doing Morse theory on a closed manifold and there are no boundary obstructedness issues. For example $M^{\mathrm{red}}_z([\Cr_-],[\Cr_+])$ is always a smooth manifold without boundary, and its dimension is given by
\begin{equation*}
\bar{\gr}_z([\Cr_-],[\Cr_+])+\mathrm{dim}[\Cr_-]-1.
\end{equation*}
One also has the following result, the counterpart of Proposition $16.6.1$ in the book.

\begin{prop}
Suppose $M^{\mathrm{red}}_z([\Lr_-],[\Lr_+])$ is non empty and of dimension $d$. Then the space of unparametrized reducible trajectories $\breve{M}^{\mathrm{red}+}_z([\acr],[\bcr])$ is a compact $(d-1)$-dimensional space stratified by manifolds. The top stratum consists of $\breve{M}^{\mathrm{red}}_z([\Lr_-],[\Lr_+])$ alone, and the $(d-2)$-dimensional stratum consists of the space of unparametrized broken trajectories with exactly two components.
\end{prop}

\vspace{1.5cm}
\section{Gluing}

In this section we discuss a gluing result describing the structure of the space of unparametrized broken trajectories along a stratum. One would like these spaces to look like topological manifolds with boundary and corners. On the other hand this is in general false for the moduli space we are dealing with, and we will show that they have in general a slightly more complicated type of structure. Our characterization will be enough for the applications, and in particular to define Floer homology in the next chapter.
\par
This section consists of two parts. In the first part we state the gluing result, introducing the notion of thickened moduli space as in Chapter $19$ in the book, while in the second one we discuss the existence of stable ad unstable manifolds as in Chapter $18$. As the second part is the only part of the proof that requires some adaptations, we will describe it in quite detail.
\vspace{0.8cm}
We start by introducing a useful definition.
\begin{defn}\label{topsub}
Consider a pointed topological space $(Q,q_0)$, let $\pi:S\rightarrow Q$ be a continuous map and consider $S_0\subset \pi^{-1}(q_0)$. We say that $\pi$ is a \textit{topological submersion along} $S_0$ if for every $s_0\in S_0$ we can find a neighborhood $U\subset S$ and a neighborhood $Q'\subset Q$ of $q_0$ with a homeomorphism
\begin{equation*}
(U\cap S_0)\times Q'\rightarrow U
\end{equation*}
commuting with $\pi$.
\end{defn}

\begin{example}\label{corner}
Suppose $Q$ is $(0,\infty]^{n-1}$ and
\begin{equation*}
q_0=(\infty,\dots,\infty).
\end{equation*}
If we have a topological submersion
\begin{equation*}
\pi:S\rightarrow Q
\end{equation*}
along $\pi^{-1}(\mathbf{\infty})$ then total space $S$ is locally homeomorphic to the product $\pi^{-1}(\mathbf{\infty})\times(0,\infty]^{n-1}$, i.e. it locally looks like the product with a topological manifold with corners.
\end{example}
\vspace{0.8cm}

The following is the main gluing theorem, the counterpart of Theorem $19.4.1$ in the book. Similarly to that case the statement is quite elaborated.
\begin{teor}\label{gluing}
Consider a sequence of critical submanifolds $\mathcal{C}=\left([\Cr_i]\right)_{i=0,\dots, n}$ and corresponding relative homotopy classes $\mathbf{z}$. Suppose that the moduli space $\breve{M}_{\mathbf{z}}([\Cr_0],[\Cr_n])$ contains irreducibles, and define
\begin{equation*}
O\subset\{1,\dots, n\}
\end{equation*}
as the set of indices $i$ such thet the pair $([\Cr_{i-1}],[\Cr_i])$ is boundary obstructed.
Then there are an open neighborhood
\begin{equation*}
\breve{W}\supset \breve{M}_{z}(\mathcal{C})
\end{equation*}
inside $\breve{M}^+_{\mathbf{z}}([\Cr_0],[\Cr_n])$ and a continuous map
\begin{equation*}
\mathbf{S}: \breve{W}\rightarrow (0,\infty]^{n-1}
\end{equation*}
with the following properties.
\begin{enumerate}
\item There is a topological embedding $j$ of $\breve{W}$ into a space $E\breve{W}$ equipped with a map also denoted by $\mathbf{S}$ to $(0,\infty]^{n-1}$ such that $\mathbf{S}\circ j=\mathbf{S}$.
\item The map $\mathbf{S}:E\breve{W}\rightarrow (0,\infty]^{n-1}$ is a topological submersion along the fiber over $\boldsymbol{\infty}$.
\item The image of $j$ is the zero set of a continuous map
\begin{equation*}
\delta:E\breve{W}\rightarrow \R^O
\end{equation*}
vanishing on the fiber over $\boldsymbol{\infty}$. Hence the fiber over $\boldsymbol{\infty}$ in both $\breve{W}$ and $E\breve{W}$ is identified with the stratum $\breve{M}_{\mathbf{z}}(\mathcal{C})$.
\item If $\breve{W}^o\subset \breve{W}$ and $E\breve{W}^o\subset E\breve{W}$ are the subset where none of the components of $\mathbf{S}$ is infinite, then the restriction of $j$ to $\breve{W}^o$ is an embedding of smooth manifolds, and the restriction of $\delta$ to $E\breve{W}^o$ is transverse to zero.
\item Let $i_0\in O$ and $\delta_{i_0}$ be the corresponding component of $\delta$. Then for all $z\in E\breve{W}$ we have:\\
$\bullet$ if $i_0\geq2$ and $S_{i_0-1}(z)=\infty$ then $\delta_{i_0}(z)\geq0$\\
$\bullet$ if $i_0\leq n-1$ and $S_{i_0}(z)=\infty$ then $\delta_{i_0}(z)\leq0$.
\end{enumerate}
\end{teor}

\begin{defn}\label{thickening}
In the same setting as the theorem above, we call $E\breve{W}$ a \textit{thickening} of the moduli space $\breve{M}^+_{z}([\Cr_0],[\Cr_n])$ along the stratum $\breve{M}_{\mathbf{z}}(\mathcal{C})$.
\end{defn}

In the simplest case in which none of the intermediate pairs is boundary obstructed, the thickening coincides with the open neighborhood $\check{W}$, so the theorem tells us that along the stratum the moduli space looks like the product of the stratum and a topological manifold with corners (see Example \ref{corner}).
On the other hand, in the presence of boundary obstructedness the local structure becomes more complicated, and a good image to have in mind is that of Example \ref{stratexample}. The simplest (but central) example is the case in which the stratum consists of broken trajectories with three components, the middle one being boundary obstructed. In this case we have the following sharper statement (see Corollary $19.5.2$ in the book).
\begin{lemma}\label{codimension1}
In the setting of Theorem \ref{gluing}, suppose $n=3$ and $O=\{2\}$. Then $\delta$ is strictly positive if $S_1=\infty$ and $S_2$ is finite, and $\delta$ is strictly negative if $S_1$ is finite and $S_2$ is infinite.
\end{lemma}
This is an important point in the definition of \textit{codimension-$1$ $\delta$-structures}, see Definition $19.5.3$ in the book. For example, the space in Example \ref{stratexample} has such a structure along its zero dimensional stratum. As we will deal with more general objects, we defer a formal definition to Definition \ref{codimensionc}

\vspace{0.8cm}

The key notion in the construction of the thickening is that of \textit{extended moduli space} (see Section $19.2$ in the book). Forgetting about regularity issues, the moduli space $\tilde{M}([\Cr_-],[\Cr_+])$ can be seen as the fibered product of the two Hilbert manifolds 
\begin{IEEEeqnarray*}{c}
\tilde{M}(\R^{\leq}\times Y,[\Cr_-])\\
\tilde{M}(\R^{\geq}\times Y,[\Cr_+]),
\end{IEEEeqnarray*}
consisting of solutions on the negative (resp. positive) half cylinder which are asymptotic to a point in $[\Cr_-]$ (resp. $[\Cr_+]$), where the maps are the restrictions $R_{\pm}$ to the boundary $\tilde{\mathcal{B}}^{\sigma}(\{0\}\times Y)$. We have the projection to the boundary
\begin{IEEEeqnarray*}{c}
\pi^{\partial}: \tilde{\mathcal{B}}^{\sigma}(Y)\rightarrow \partial \Bs(Y)\\
{[}B,s,\phi]\mapsto [B,0,\phi]
\end{IEEEeqnarray*}
and we define the extended moduli space
\begin{equation*}
E\tilde{M}([\Cr_-],[\Cr_+])
\end{equation*}
as the fibered product of the composite maps $\pi^{\partial}\circ R_{\pm}$. One should think of an element $[\gamma]$ in $E\tilde{M}([\acr],[\bcr])$ as a trajectory defined on the whole line but having a discontinuity
\begin{equation*}
\delta=s([\gamma_+]|_{\{0\}\times Y})-s([\gamma_-]|_{\{0\}\times Y})
\end{equation*}
in the $s$ coordinate across $\{0\}\times Y$. There is no action by translation on such an element $[\gamma]$, it still has a well defined relative homotopy class $z$, and we can partition the moduli space accordingly.
The usual moduli space arises as the fiber over zero of the map
\begin{equation*}
\tilde{M}([\Lr_-],[\Lr_+])\hookrightarrow E\tilde{M}([\Lr_-],[\Lr_+])\stackrel{\delta}{\longrightarrow} \R.
\end{equation*}
If $([\Lr_-],[\Lr_+])$ is boundary obstructed then $\delta$ is identically $0$, hence the extended moduli space coincides with the usual one. In the boundary unobstructed case if $\tilde{M}([\Lr_-],[\Lr_+])$ is regular, then $E\tilde{M}([\Lr_-],[\Lr_+])$ is regular in a neighborhood of $\tilde{M}([\Lr_-],[\Lr_+])$ and contains it as a smooth codimension $1$ submanifold. In both cases, $E\tilde{M}_z([\Lr_-],[\Lr_+])$ has dimension $\mathrm{gr}_z([\Cr_-],[\Cr_+])+\mathrm{dim}[\Lr_-]+1$. 
\par
The thickening $E\breve{W}$ of a neighborhood $\breve{W}$ is then the space in which we allow the components to belong to the respective extended moduli space. From the discussion above, it follows that $\breve{W}$ sits inside it as the zero locus of all the maps $\delta$, and its codimension is exactly the number of consecutive pairs which are boundary obstructed.

\vspace{0.8cm}
An important feature is that the extended moduli space still is still equipped with continuous evaluations maps
\begin{equation*}
\ev_{\pm}: E\tilde{M}_z([\Lr_-],[\Lr_+])\rightarrow[\Cr_{\pm}],
\end{equation*}
which are smooth in a neighborhood of $\tilde{M}_z([\Lr_-],[\Lr_+])$. The following is the main additional result we will need.
\begin{prop}\label{evthickening}
Referring to the statement of Theorem \ref{gluing}, the evaluation maps $\ev_{\pm}$ on $\breve{W}$ extend to continuous evaluation maps
\begin{equation*}
\ev_{\pm}:E\breve{W}\rightarrow [\Cr_{\pm}]
\end{equation*}
which are smooth in a neighborhood of $\breve{W}^o\subset E\breve{W}^o$. Furthermore, suppose we are given three critical submanifolds $[\Cr_-],[\Cr_0]$ and $[\Cr_+]$, and sequences of critical submanifolds and relative homotopy classes between them $(\mathcal{C}_1,\mathbf{z}_1)$ and $(\mathcal{C}_2,\mathbf{z}_2)$. For $m=1,2$, let $\breve{W}_m$ the open neighborhoods of the strata $M_{z_m}(\mathcal{C}_m)$ and $E\breve{W}_m$ their thickening. Then the evaluation maps
\begin{IEEEeqnarray*}{c}
\ev_+: E\breve{W}_1\rightarrow [\Cr_0]\\
\ev_-: E\breve{W}_2\rightarrow[\Cr_0]
\end{IEEEeqnarray*}
are transverse in a neighborhood of $\breve{W}^o_1\times\breve{W}^o_2\subset E\breve{W}^o_1\times E\breve{W}^o_2$.
\end{prop}
\begin{proof}
The first part follows from the construction of the extended moduli spaces, see Section $19.4$ in the book. Then the evaluation maps are transverse on $\breve{W}^o_1\times\breve{W}^o_2$ by the regularity assumptions on the moduli spaces (see Definition \ref{regularpert}), hence they are transverse in a neighborhood.
\end{proof}

\vspace{0.8cm}

To conclude the first part of this section, we just state the case of the reducible moduli spaces (see Section $19.6$ in the book). This case in simpler as we do not have to deal with boundary obstructedness phenomena.
\begin{teor}
Suppose the moduli space $M^{\mathrm{red}}_z([\Lr_-],[\Lr_+])$ is $d$-dimensional and non-empty, so that the space of unparametrized broken reducible trajectories $\breve{M}^{\mathrm{red}+}_z([\Lr_-],[\Lr_+])$ is a $(d-1)$-dimensional space stratified by manifolds. If $M'\subset\breve{M}^{\mathrm{red}+}_z([\Lr_-],[\Lr_+])$ is a component of the codimension-$e$ stratum, then along $M'$ the space $\breve{M}^{\mathrm{red}+}_z([\Lr_-],[\Lr_+])$ is the product of the stratum and a $C^0$ manifold with an $e$-dimensional corner.
\end{teor}

\vspace{1.5cm}
We now turn our attention to the new analytical issues in the proof of Theorem \ref{gluing}, namely the parametrization of the space of solutions on a cylinder which are close to a given constant solution.
\\
\par
We first recall the notion of \textit{spectral decomposition} for a $k$-\textsc{asafoe} operator $L$ on a vector bundle $E_0\rightarrow Y$ (see Chapter $17$ in the book for the details). This naturally arises when studying Atiyah-Patodi-Singer boundary value problems. On the pulled-back vector bundle on the negative half cylinder $Z=\R^{\leq0}\times Y$, which we call $E$, we can consider the translation invariant operators 
\begin{equation*}
D^{\pm}=\frac{d}{dt}\mp L: L^2_1(Z;E)\rightarrow L^2(Z;E).
\end{equation*}
The restriction map to the boundary $\{0\}\times Y$ extends for every $k\geq 1$ to a continuous map
\begin{equation*}
r: L^2_k(X;E)\rightarrow L^2_{k-1/2}(Y;E_0)
\end{equation*}
which is also surjective with left inverse (see Theorem $17.1.1$ in the book). We can then define the \textit{spectral subspaces} of $L$
\begin{equation*}
H^{\pm}\subset L^2_{1/2}(Y;E_0)
\end{equation*}
obtained as the boundary values of the kernel of $D^{\pm}$. Then one has the following lemma (see Lemma $17.2.2$ in the book).
\begin{lemma}\label{spectraldec}
Suppose $L$ is a $k$-\textsc{asafoe} hyperbolic operator. Then one has the direct sum of closed subspaces
\begin{equation*}
 L^2_{1/2}(Y;E_0)=H^+\oplus H^-.
\end{equation*}
Furthermore, for each integer $1\leq j\leq k+1$, we have
\begin{equation*}
L^2_{j-1/2}(Y;E_0)=(H^+\cap L^2_{j-1/2})\oplus (H^-\cap L^2_{j-1/2}).
\end{equation*}
\end{lemma}
As the operators that arise in our setting are slightly more general, we introduce the following definition, see Lemma \ref{hessian}.
\begin{defn}
A $k$-\textsc{asafoe} operator $L$ is called \textit{almost hyperbolic} if the intersection between its spectrum and the imaginary axis is $\{0\}$, and the generalized $0$-eigenspace coincides with the kernel.
\end{defn}
If $L$ is almost hyperbolic it has finite dimensional kernel $H_0=\mathrm{Ker}L$ and for each $h\in H_0\subset L^2_k(Y;E_0)$ the translation invariant configuration $s_h$ satisfies the equation
\begin{equation*}
\left(d/dt+L\right)s_h=0.
\end{equation*}
Notice though that these sections are \textit{not} in $L^2_k(Z;E)$. One can then define a decomposition, called \textit{generalized spectral decomposition},
\begin{equation*}
L^2_{1/2}(Y;E_0)= H^+\oplus H_0\oplus H^-
\end{equation*}
where $H^+$ (respectively $H^-$) is the the positive (negative) spectral subspace of $L-\delta$ ($L+\delta$) for $\delta>0$ sufficiently small. This decomposition satisfies the obvious adaptation of Lemma \ref{spectraldec}, as it can be thought as the spectral decomposition for the hyperbolic operator $L\pm \delta$. An interesting space to study for our purposes is
\begin{equation}\label{infiniteweighted}
\begin{aligned}
L^2_{k,\delta}(Z; E, H_0)&= \left\{ s\mid s-s_h\in L^2_{k,\delta}(Z; E)\text{ for some }h\in H_0\right\}\\
&=H_0+ L^2_{k,\delta}(Z;E)\subset L^2_{k,\mathrm{loc}}(Z;E),
\end{aligned}
\end{equation}
where on the negative half cylinder we use the weight function $f(t)=e^{-\delta t}$.
This space comes with the linear map
\begin{equation*}
\Pi_0:L^2_{k,\delta}(Z; E, H_0)\rightarrow H_0
\end{equation*}
sending a configuration to its (unique) limit point. The norm of a configuration is defined by
\begin{equation*}
\|s\|^2=\|s-\Pi_0(s)\|^2_{L^2_{k,\delta}}+\|\Pi_0(s)\|^2.
\end{equation*}
We then have the following result.
\begin{lemma}\label{almostAPS}
Consider an operator of the form $D=d/dt+L$ where $L$ is $k$-\textsc{\textsc{asafoe}} and almost hyperbolic. Let
\begin{equation*}
\Pi: L^2(Y;E_0)\rightarrow H^-
\end{equation*}
the the projection with kernel $H_0\oplus H^+$.Then the operator
\begin{equation*}
D\oplus(\Pi\circ r)\oplus \Pi_0: L^2_{j,\delta}(Z;E,H_0)\rightarrow L^2_{j-1,\delta}(Z;E)\oplus (H^-\cap L^2_{j-1/2}(Y;E_0))\oplus H_0
\end{equation*}
is an isomorphism for $1\leq j\leq k$ for $\delta>0$ sufficiently small. Finally, the image of $\mathrm{ker}D$ under $r$ in $L^2_{j-1/2}(Y;E_0)$ is precisely $H^-\cap L^2_{j-1/2}(Y;E_0)\oplus H_0$.
\end{lemma}
\begin{proof}
Using the same trick as in Proposition \ref{fredholmQ}, the operator 
\begin{equation*}
D\oplus(\Pi\circ r)\oplus: L^2_{j,\delta}(Z;E)\rightarrow L^2_{j-1,\delta}(Z;E)\oplus (H^-\cap L^2_{j-1/2}(Y;E_0))
\end{equation*}
is equivalent to the operator
\begin{equation*}
(D+\delta) \oplus(\Pi\circ r): L^2_{j}(Z;E)\rightarrow L^2_{j-1}(Z;E)\oplus (H^-\cap L^2_{j-1/2}(Y;E_0))
\end{equation*}
for which the statement follows from Proposition $17.2.7$ in the book as $H^-$ is the negative spectral subspace for the hyperbolic operator $L+\delta$. The result then follows from the definition of the space $L^2_{j,\delta}(Z;E,H_0)$.
\end{proof}
\begin{example}\label{spectraldecmorsebott}
We are especially interested in the following situation. The Hilbert vector bundle
\begin{equation*}
\K^{\sigma}_j\rightarrow \Cs_k(Y)
\end{equation*}
carries a smooth family of operators with real spectrum, the Hessians $\{\Hess_\q^{\sigma}\}$. At a Morse-Bott critical point $\acr$, such an operator is almost hyperbolic, as it is a direct summand of the weighted extended Hessian (see Lemma \ref{weighhessian}), hence we have the (generalized) spectral decomposition
\begin{equation*}
\K^{\sigma}_{k-1/2,\acr}=\K^+\oplus T_{\acr}\Cr\oplus\K^-.
\end{equation*}
By quotienting by $T_{\acr}\Cr$, this also induces a spectral
\begin{equation}\label{spectralN}
\N^{\sigma}_{k-1/2,\acr}=\N^+_{\acr}\oplus \N^-_{\acr}.
\end{equation}
on the normal bundles.
\end{example}

\vspace{0.8cm}
Consider now a critical point $\acr\in C\subset\Cs_k(Y)$, and consider an $L^2_{k}$-compatible product chart $(\mathcal{U},\varphi)$ of a neighborhood $U$ of $\acr$. As pointed out in Remark \ref{jkcompatible}, we can consider such a chart $\varphi$ as the restriction of a $L^2_{k-1/2}$-compatible prduct chart, which we also call $\varphi$. Define
\begin{equation*}
Z^{\infty}=(\mathbb{R}^{\leq}\times Y)\amalg (\mathbb{R}^{\geq}\times Y).
\end{equation*}
Given an open neighborhood $[\mathfrak{U}]\subset [\Cr]$ of $[\acr]$ we can define the configuration space
\begin{equation*}
\Bt_{k,\delta}(Z^{\infty}, [\mathfrak{U}])= \{\gamma\in\Bt_{k,\mathrm{loc}}\mid \gamma\in\Bt_{k,\delta}(Z^{\infty}, [\bcr]) \text{ for some }[\bcr]\in [\mathfrak{U}]\}
\end{equation*}
consisting of configurations exponentially asymptotic to the same configuration $[\bcr]\in [\mathfrak{U}]$ on both ends (up to gauge equivalence). We can define the subspace
\begin{equation*}
{M}(Z^{\infty},[\mathfrak{U}])\subset\Bt_{k,\delta}(Z^{\infty}, [\mathfrak{U}])
\end{equation*}
consisting of solutions to the Seiberg-Witten equations. Given $[\bcr]\in[\mathfrak{U}]$ we will write
\begin{equation*}
{M}(Z^{\infty},[\bcr])\subset{M}(Z^{\infty},[\mathfrak{U}])
\end{equation*}
for the subset of solutions converging to $[\bcr]$. As usual, we can also introduce the versions of the moduli spaces with the tildes. Similarly, we define the finite cylinders
\begin{equation*}
Z^T=[-T,T]\times Y.
\end{equation*}
Theorem $17.3.1$ in the book and its proof are still valid in our situation, and give us the following transversality results for the moduli spaces of trajectories ${M}(Z^T)$ on finite cylinders.
\begin{prop}
For finite $T$, the space $\tilde{M}(Z^T)\subset \tilde{\mathcal{B}}^{\tau}_k(Z^T)$ is a closed Hilbert submanifold. The subset $M(Z^T)$ is a Hilbert submanifold with boundary, identified with the quotient of $\tilde{M}(Z^T)$ by the involution $\mathbf{i}$.
\end{prop}
The manifolds $Z^{\infty}$ and all $Z^T$ for $T>0$ have boundary $Y\amalg \bar{Y}$, so we have the continuous restriction maps
\begin{align*}
R: \tilde{M}(Z^{\infty},[\mathfrak{U}])&\rightarrow \tilde{\Bo}_{k-1/2}^{\sigma}(Y\amalg \bar{Y}) \\
R:\tilde{M}(Z^T,[\mathfrak{U}])&\rightarrow \tilde{\Bo}_{k-1/2}^{\sigma}(Y\amalg \bar{Y}).
\end{align*}

We will need to introduce suitable charts for the configuration spaces on the cylinders. We explain an idea that unfortunately does not immediately work. Suppose we are given a $L^2_k$-compatible product chart $(\varphi, \mathcal{U})$ which is the restriction of a $L^2_{k-1/2}$-compatible product chart. This defines a correspondence between elements in a neighborhood of zero $\tilde{\mathcal{U}}_T\subset\T^{\tau}_{k,\gamma_{\acr}}(Z^T)$ and elements in a neighborhood of $\gamma_{\acr}$ in  $\tilde{\mathcal{C}}^{\tau}_k(Z^T)$. Indeed, forgetting for a moment about regularity issues, we can think of an element of $v\in\T^{\tau}_{\gamma_{\acr}}(Z^T)$ as a path $\check{v}(t)$ with values in $\T^{\sigma}_{\acr}(Y)$ together with a path of imaginary valued $1$-forms $\check{c}(t)$ on $Y$. By suitably restricting $\tilde{\mathcal{U}}_T$, we can suppose that the configuration $\check{v}(t)$ always lies in the domain $\mathcal{U}$ of the chart (because of the continuity of the restriction maps). In a similar fashion, we can think of an element as a path $\check{\gamma}(t)$ with values in $\Cs(Y)$ together with a path of imaginary valued $1$-forms $c$ on $Y$. The identification then just sends the path $\check{v}(t)$ to the path $\varphi(\check{v}(t))$ in $\Cs(Y)$, and leaves $\check{c}(t)$ unchanged. The issue with this approach is that even though the chart we have constructed in Lemma \ref{compchart} (which has the additional desirable property of being somehow gauge invariant) is the restriction of a $L^2_{k-1/2}$-compatible product chart (see Remark \ref{jkcompatible}), the estimates on the $L^2_k$ norm of the derivative do not hold on the domain of this larger chart. We fix this issue in the next result.
\begin{lemma}\label{cylinderchart}
There exists an $L^2_k$-compatible product chart $(\varphi,\mathcal{U})$ around $\acr$ which is the restriction of an $L^2_{k-1/2}$-compatible product chart such that the construction above gives rise to a well defined chart
\begin{equation*}
\varphi: \tilde{\mathcal{U}}_T\subset\T^{\tau}_{k,\gamma_{\acr}}\rightarrow \tilde{\mathcal{C}}^{\tau}_k(Z^T)
\end{equation*}
onto a neighborhood of the constant solution $\gamma_{\acr}$. The similar statement holds on the infinite cylinder (with the weighted Sobolev spaces).
\end{lemma}
\begin{proof}
As in the proof of Lemma \ref{compchart} we can locally identify the slice through $\acr$ with its tangent space $\T_{\acr}\Cr\oplus \N^{\sigma,\acr}_k$, and the intersection of the critical submanifold $\Cr$ with the slice with the image of the graph of a smooth function
\begin{equation*}
f: \T_{\acr}\Cr\rightarrow \N^{\sigma}_k
\end{equation*}
defined in a neighborhood of the origin. This induces the map (defined in an $L^2_k$ neighborhood of the origin)
\begin{align*}
\tilde{f}: \T^{\sigma}_{k,\acr}(Y)&\rightarrow  \T^{\sigma}_{k,\acr}(Y)\\
(v^t, v^n, v^j)&\mapsto \left(v^t, v^n+f(v^t), v^j\right)
\end{align*}
where we have identified the tangent space with the direct sum $\T_{\acr}\Cr\oplus \N^{\sigma}_{k,\acr}\oplus \J^{\sigma}_{k,\acr}$. The nice feature of this map is that as in Remark \ref{jkcompatible} it is defined for each $1\leq j\leq k$ in an $L^2_j$ neighborhood of the origin, and furthermore norms of the derivatives and their inverses in the $L^2_k$ norm is bounded in this neighborhood. Here we use the fact that $\T_{\acr}\Cr$ does not depend on $j\leq k$. This implies that the analogue of the fiberwise construction discussed above at the level of the tangent spaces induces a diffeomorphism
\begin{equation*}
\tilde{f}:\tilde{\mathcal{U}}_T\subset\T^{\tau}_{k,\gamma_{\acr}}\rightarrow\T^{\tau}_{k,\gamma_{\acr}}
\end{equation*}
from a neighborhood of the origin onto a neighborhood of the origin. The desired chart is defined by  composing this map with the chart form Section $18.4$ in the book
\begin{align*}
i:\T^{\tau}_{k,\gamma_{\acr}}&\rightarrow \tilde{\mathcal{C}}^{\tau}_k(Z^T)\\
(a,r,\psi)&\mapsto (A_0+a\otimes 1, s_0+r, \phi)
\end{align*}
where $\gamma_{\acr}=(A_0,s_0,\phi_0)$ and
\begin{equation*}
\check{\phi}(t)=\frac{\check{\phi}_0+\psi(t)}{\sqrt{1+\|\psi(t)\|^2}}.
\end{equation*}
The result follows because the chart $i$ is also defined fiberwise.
\end{proof}

\vspace{0.8cm}
This section is devoted to the proof of the following theorem. Here we use the notation
\begin{equation*}
\N=\K^{\sigma}_{k-1/2,\acr}(Y)/T_{\acr}\Cr.
\end{equation*}
From this the proof of Theorem \ref{gluing} follows with no essential modifications as in Chapter $19$ in the book.
\begin{teor}\label{stableman}
Fix an $L^2_{k-1/2}$-compatible product chart $(\mathcal{U},\varphi)$ in a neighborhood of $\acr$ as in Lemma \ref{cylinderchart}. Then there exists $T_0$ such that for all $T\geq T_0$ we can find smooth parametrizations
\begin{align*}
u(T,-): B(T_{\acr}\Cr)\times B(\N)&\rightarrow\tilde{M}(Z^T) \\ u(\infty,-): B(T_{\acr}\Cr)\times B(\N)&\rightarrow \tilde{M}(Z^{\infty})
\end{align*}
which are diffeomorphisms from a product of balls $B(T_{\acr}\Cr)\times B(\N)$ onto a neighborhood of $[\gamma_{\acr}]$. The map $u(\infty,-)$ respects the decomposition of the moduli space in the sense that for every $\xi\in B(T_{\acr}\Cr)$ one has 
\begin{align*}
u\left(\infty, (\xi,-)\right)(B(\N))&\subset M(Z^{\infty},[\varphi(\xi,0,0)]) \\
u\left(\infty, (\xi,0)\right)&=[\gamma_{\varphi(\xi,0,0)}].
\end{align*}
Furthermore, these parametrizations can be chosen so that the map
\begin{IEEEeqnarray*}{c}
\mu_T:B(T_{\acr}\Cr)\times B(\N)\rightarrow \Bs_{k-1/2}(Y\amalg \overline{Y})\\
h\mapsto Ru\left(T,h\right)
\end{IEEEeqnarray*}
is a smooth embedding of $B(T_{\acr}\Cr)\times B(\N)$ for every $T\in[T_0,\infty]$ with the following additional properties. The function
\begin{equation*}
\mu_T:[T_0,\infty)\times B(T_{\acr}\Cr)\times B(\N)\rightarrow \Bs_{k-1/2}(Y\amalg \overline{Y})
\end{equation*}
is smooth for finite $T$ and furthermore
\begin{equation*}
\mu_T\stackrel{C^{\infty}_{\mathrm{loc}}}{\longrightarrow} \mu_{\infty} \text{ as } T\rightarrow\infty.
\end{equation*}
Finally there is an $\eta>0$ independent of $T$ such that the images of the maps $u(T,-)$ can be taken to contain all solutions $[\gamma]\in M(Z^T)$ with $\|\gamma-\gamma_{\bcr}\|_{L^2_k(Z^T)}\leq\eta$ for some $\bcr\in \mathfrak{U}$.
\par
In the case that $\acr$ is reducible, the parametrizations $u(T,-)$ for $T\in(T_0,\infty]$ are equivariant for the $\mathbb{Z}/2\mathbb{Z}$ action of $\mathbf{i}$.
\end{teor}

The parametrization $u(\infty,-)$ provided by the theorem also respects the fibered product structure on $M(Z^{\infty},[\mathfrak{U}])$. This can be identified with
\begin{align*}
\big\{([\gamma_+],[\gamma_-])&\mid \lim_{t\rightarrow\infty} [\breve{\gamma}_+(t)]=\lim_{t\rightarrow-\infty}[\breve{\gamma}_-(t)]\big\}\\
&\subset M(\R^{\geq0}\times Y, [\mathfrak{U}])\times M(\R^{\leq0}\times Y, [\mathfrak{U}])
\end{align*}
and and the map $u(\infty,-)$ provides local diffeomorphisms
\begin{IEEEeqnarray*}{c}
B(T_{\acr}\Cr)\times B(\N^+)\rightarrow M(\R^{\geq0}\times Y,[\mathfrak{U}]) \\
B(T_{\acr}\Cr)\times B(\N^-)\rightarrow M(\R^{\leq 0}\times Y,[\mathfrak{U}]).
\end{IEEEeqnarray*}
We can think of the boundary values of these moduli spaces as the (local) stable and unstable manifolds in a neighborhood of $[\acr]$.
\par
\vspace{0.8cm}

The proof of Theorem \ref{stableman} follows the analogous result in the Morse setting very closely. We first discuss a more abstract version for general operators of the form $d/dt+L$ on the cylinders $Z^T$, and then apply it to our special case. Here $L$ is an almost hyperbolic $k$-\textsc{asafoe} operator. Suppose $T>2$, and consider a smooth even function $g_{T,\delta}$ on $[-T,T]$ such that
\begin{equation*}
g_{T,\delta}(t)=
\begin{cases}
e^{\delta(t+T)}\text{ if }t\leq -1 \\
e^{-\delta(t-T)}\text{ if } t\geq 1.
\end{cases}
\end{equation*} 
The function $\sigma(t)_{T,\delta}=g'_{T,\delta}(t)/g_{T,\delta}(t)$ is $\delta$ for $t\leq-1$ and $-\delta$ for $t\geq 1$, and we can choose the family so that for each $\delta>0$ it is independent of $T$ on the interval $[-1,1]$. On the finite cylinder we will consider the $L^2_{k,\delta}$ norm defined to be
\begin{equation*}
\|u\|_{L^2_{k,\delta}}=\|g_{T,\delta}\cdot u\|_{L^2_k}.
\end{equation*}
We denote this normed space by $L^2_{k,\delta}(Z^T,E)$. The new norm, which we call the weighted Sobolev norm, is obviously equivalent to the original one on each finite cylinder, but it will be useful to study the behaviour of the operators
\begin{equation*}
D=d/dt+L: L^2_{k,\delta}(Z^T;E)\rightarrow L^2_{k-1,\delta}(Z^T;E)
\end{equation*}
for $T$ going to infinity. Inside this space there is the subspace $\mathcal{H}_0$ consisting of the constant sections $s_{h_0}$ for $h_0\in H_0=\mathrm{ker}L$. It is important to remark that the canonical identification
\begin{equation}\label{hequalsh}
H_0\rightarrow \mathcal{H}_0
\end{equation}
sending an element $h_0$ to the constant section $s_{h_0}$ has norm growing exponentially in the time parameter $T$. There is the map
\begin{equation*}
\Pi_0: L^2_{k,\delta}(Z^T,E)\rightarrow \mathcal{H}_0
\end{equation*}
given by $L^2_{k,\delta}$ projection, or, equivalently $L^2_{\delta}$ projection as the elements in $\mathcal{H}_0$ are constant. We still denote by $\Pi_0$ the composition of such a map with the identification (\ref{hequalsh}). We denote the kernel of such map by $L^2_{k,\delta}(Z^T;E)^{\perp}$. On the infinite cylinder, we introduced the space $L^2_{k,\delta}(Z^{\infty};E, H_0)$ in (\ref{infiniteweighted}). In this case, we will consider the operator
\begin{equation*}
D=d/dt+L: L^2_{k,\delta}(Z^{\infty};E,H_0)\rightarrow L^2_{k-1,\delta}(Z^{\infty};E)
\end{equation*}
Suppose we are also given a bounded linear operator
\begin{equation*}
\Pi:L^2_{k-1/2}(Y\amalg\overline{Y}; E_0)\rightarrow H
\end{equation*}
for some Hilbert space $H$. This induces by restriction to the boundary (which we omit from the notation) a map
\begin{align*}
\Pi&:L^2_k(Z^T;E)\rightarrow H\\
\Pi&: L^2_k(Z^{\infty};E;H_0)\rightarrow H.
\end{align*}
For simplicity, we introduce the notations
\begin{align*}
\mathcal{E}^T_{\delta}&=L^2_{k,\delta}(Z^T; E) \\
\mathcal{F}^T_{\delta}&=L^2_{k-1,\delta}(Z^T;E)
\end{align*}
and
\begin{align*}
\mathcal{E}^{\infty}_{\delta}&=L^2_{k.\delta}(Z^{\infty};E,H_0),\\
\mathcal{F}^{\infty}_{\delta}&=L^2_{k-1,\delta}(Z^{\infty};E).
\end{align*}
The key assumption in what follows is that the linear operator
\begin{equation}\label{invertible}
(D,\Pi_0,\Pi):\mathcal{E}^{\infty}_{\delta}\rightarrow\mathcal{F}^{\infty}_{\delta}\oplus H_0\oplus H
\end{equation}
is invertible. This also implies the invertibility on weighted spaces with weight $\delta'$ sufficiently close to $\delta$. The problem we are interested in is non-linear, and in our abstract setting we suppose this non-linearity arises as a map
\begin{equation*}
\alpha: C^{\infty}(Z^T;E)\rightarrow L^2_{\mathrm{loc}}(Z^T;E)
\end{equation*} 
obtained from a map $\alpha_0:C^{\infty}(Y; E_0)\rightarrow L^2(Y;E_0)$ by restriction to slices $\{t\}\times Y$. We will assume that $\alpha$ defines a smooth map
\begin{equation*}
\alpha: L^2_k([-1,1]\times Y;E)\rightarrow L^2_{k-1}([-1,1]\times Y; E)
\end{equation*}
with $\alpha(h_0)=0$ for every $h_0\in H_0$ and $\mathcal{D}_0\alpha=0$. This implies that $\alpha$ defines smooth maps
\begin{align*}
\alpha&: \mathcal{E}^T_{\delta}\rightarrow \mathcal{F}^T_{\delta}\\
\alpha&: \mathcal{E}^{\infty}_{\delta}\rightarrow \mathcal{F}^{\infty}_{\delta}.
\end{align*}
We are then interested in the study of the maps
\begin{align*}
F^T&=D+\alpha:\mathcal{E}_{\delta}^T\rightarrow \mathcal{F}_{\delta}^T \\
F^{\infty}&=D+\alpha:\mathcal{E}^{\infty}_{\delta}\rightarrow \mathcal{F}^{\infty}_{\delta}
\end{align*}
and especially in the spaces of solutions
\begin{align*}
M(T)&=(F^T)^{-1}(0)\subset\mathcal{E}^T\\
M(\infty)&=(F^{\infty})^{-1}(0)\subset\mathcal{E}^{\infty}_{\delta}.
\end{align*}
Our abstract version of the Theorem \ref{stableman} is the following.

\begin{prop}\label{parabstract}
Suppose the hypothesis above are satisfied, and in particular that the map (\ref{invertible}) is invertible. Then for $T\geq T_0$ the sets $M(T)$ and $M(\infty)$ are Hilbert submanifolds of $\mathcal{E}^T$ and $\mathcal{E}^{\infty}$ in a neighborhood of $0$. There exist $\eta>0$ and smooth maps
\begin{equation*}
u(T,-):B_{\eta}(H_0)\times B_{\eta}(H)\rightarrow M(T)\qquad T\in[T_0,\infty]
\end{equation*}
which are diffeomorphisms onto their image and preserve the product structure, i.e. for every $T\in [T_0,\infty]$ we have
\begin{align*}
(\Pi_0, \Pi) u(T,(h_0,h))&=(h_0,h)\\
u(T, (h_0,0))&=s_{h_0}.
\end{align*}
Furthermore, for $T\in[T_0,\infty]$ the map obtained by composing with the restriction to the boundary
\begin{IEEEeqnarray*}{c}
\mu_T: B_{\eta}(H_0)\times B_{\eta}(H)\rightarrow L^2_{k-1/2}(Y\amalg\overline{Y})\\
(h_0,h)\mapsto ru(T,h_0,h)
\end{IEEEeqnarray*}
is a smooth embedding. As a function on $[T_0,\infty)\times B_{\eta}(H_0)\times B_{\eta}({H})$ the map $(T,h_0,h)\mapsto \mu_T(h_0,h)$ is smooth for finite $T$, and
\begin{equation*}
\mu_T\stackrel{C^{\infty}_{\mathrm{loc}}}{\longrightarrow} \mu_{\infty}\qquad\text{ for }T\rightarrow \infty.
\end{equation*}
Finally, there is an $\eta'>0$ independent of $T$ such that the images of the maps $u(T,-)$ contain all solutions $u\in M(T)$ with $\|u\|_{L^2_{k,\delta}}\leq \eta'$.
\end{prop}

The strategy of the proof follows the classical case. We first prove the existence of the solution $u(\infty,h_0,h)$ for $\|(h_0,h)\|$ small.
\begin{lemma}\label{parinfty}
There exist $\eta_1, C_0>0$ such that for every $(h_0,h)$ with 
\begin{equation*}
\|h_0\|,\|h\|\leq \eta_2=\eta_1/2C_0
\end{equation*}
there exists a unique $u=u(\infty,h_0,h)$ in $B_{\eta_2}(\mathcal{E}^{\infty}_{\delta})$ satisfying
\begin{IEEEeqnarray*}{c}
F^{\infty}(u)=0 \\
\Pi u=h.
\end{IEEEeqnarray*}
Furthermore the map
\begin{equation*}
u(\infty,-): B_{\eta_2}(H_0)\times B_{\eta_2}(H)\rightarrow B_{\eta_1}(\mathcal{E}^{\infty})
\end{equation*}
is smooth, sends $(h_0,0)$ to the constant section $s_{h_0}$ and satisfies
\begin{equation*}
\|u(\infty,h_0, h)\|\leq 2C_0\|(h_0,h)\|.
\end{equation*}
\end{lemma}
\begin{proof}This follows from the application of the inverse function theorem (Proposition $18.3.6$ in the book). The the key point is that the non linear part of the function
\begin{equation*}
\alpha: \mathcal{E}^{\infty}_{\delta}\rightarrow \mathcal{F}^{\infty}_{\delta}
\end{equation*}
is $C^1$ and has vanishing derivative at the origin so it is uniformly Lipschitz with small Lipschitz constant on small balls around zero, i.e. for any $\varepsilon>0$ there is $\eta>0$ such that for each $u,u'\in \mathcal{E}_{\delta}^{\infty}$ we have
\begin{equation*}
\|u\|,\|u'\|\leq \eta\implies \| \alpha(u)-\alpha(u')\|\leq \varepsilon \|u-u'\|.
\end{equation*}
We can then just apply the inverse function theorem to the map
\begin{equation*}
(F^{\infty},\Pi,\Pi_0):\mathcal{E}^{\infty}_{\delta}\rightarrow \mathcal{F}^{\infty}\oplus H\oplus H_0.
\end{equation*}
The fact that $u(\infty, (h_0,0))=s_{h_0}$ follows from the fact that $\alpha(s_{h_0})=0$.
\end{proof}

\vspace{0.8cm}

We then focus on the operators $F^T$ acting on finite cylinders. First we study their linearizations.
\begin{lemma}\label{boundedinv}
There exists a $T_0$ such that for all $T\geq T_0$ the operator
\begin{equation*}
P^T=D\oplus\Pi:\mathcal{E}^{T,\perp}_{\delta}\rightarrow\mathcal{F}^T_{\delta}\oplus H
\end{equation*}
is invertible. Furthermore, for $T\rightarrow \infty$ the operator norm $\|(P^T)^{-1}\|$ is bounded by a constant independent of $T$.
\end{lemma}
This implies in particular that the whole linearization of the operator $F^T$ is invertible. On the other hand it is clear that the norm of the inverses grows exponentially because of the constant sections $\mathcal{H}_0$. The key point of the lemma is the fact that on the complement of this subspace the norm of the inverse is bounded.

\begin{proof}
This follows from modifying the usual patching argument (see the proof of Lemma $18.3.8$ in the book). Call $N^{\infty}$ the inverse of $P^{\infty}$. On $Z^T$, pick a smooth partition of unity $\beta_1,\beta_2$ with 
\begin{align*}
\beta_1(t)&=
\begin{cases}
1\text{ for }t\leq -1\\
0 \text{ for } t\geq 1
\end{cases}\\
\beta_2(t)&=\beta_1(-t)
\end{align*}
and smooth functions $\phi_1,\phi_2$ so that
\begin{align*}
\phi(t)&=
\begin{cases}
1\text{ for }t\leq T/2-1\\
0 \text{ for } t\geq T/2+1
\end{cases}\\
\phi_2(t)&=\phi_1(-t).
\end{align*}
We can also suppose that the non constant part of these functions does not depend on $T$. We construct an almost right inverse for $P^T$ with the following patching argument. Call $N^{\infty}$ the inverse of the linear operator
\begin{equation*}
P^{\infty}: \mathcal{E}^{\infty}_{\delta,0}\rightarrow \mathcal{F}^{\infty}_{\delta}\oplus H,
\end{equation*}
where $\mathcal{E}^{\infty}_{\delta,0}$ is the subspace of $\mathcal{E}^{\infty}_{\delta}$ consisting of configurations asymptotic to zero. We can then define
\begin{IEEEeqnarray*}{c}
\rho: \mathcal{F}_{\delta}^T\rightarrow\mathcal{F}^{\infty}_{\delta}\\
v\mapsto\rho(v)=
\begin{cases}\tau_{-T}^*\beta_1 v & \text{on }[0,\infty)\times Y \\
\tau_T^*\beta_2 v& \text{on } (-\infty,0]\times Y.
\end{cases}
\end{IEEEeqnarray*}
and
\begin{IEEEeqnarray*}{c}
\pi: \mathcal{E}^{\infty}_{\delta}\rightarrow \mathcal{E}_{\delta}^T \\
u\mapsto \phi_1\tau^*_T u_++\phi_2\tau^*_{-T}u_-
\end{IEEEeqnarray*}
where $u_+$ and $u_-$ are the parts of $u$ on the two components $\R^{\geq0}\times Y$ and $\R^{\leq0}\times Y$. We now define the map
\begin{IEEEeqnarray*}{c}
\tilde{N}^T: \mathcal{F}_{\delta}^T\oplus H\rightarrow \mathcal{E}_{\delta}^T\\
(v,h)\mapsto (\mathrm{Id}-\Pi^{\perp}_0)\pi\circ N^{\infty}(\rho(v),h)
\end{IEEEeqnarray*}
which has the property that $P^T\circ \tilde{N}^T=1+K^T$ where the operator norm of $K^T$ going to $0$.
In fact, $\tilde{N}^T(u)$ solves the equation outside the intervals $[-T/2-1,-T/2+1]$ and $[T/2-1,T/2+1]$. In the remaining to intervals, for example in $[T/2-1,T/2+1]$, the failure is described by the map
\begin{equation*}
u\mapsto  \frac{d}{dt}\phi_1\cdot\tau^*_T P^{\infty}_+\tau^*_{-T}\beta_1 u.
\end{equation*} 
This map has norm bounded above by the quantity $Ce^{-\delta T}$ for some time independent constant $C$. This is because the multiplication by $\beta_1$ (seen as a map from the finite cylinder to the infinite one) has norm bounded by a time independent constant, while the multiplication by $\frac{d}{dt}\phi_1$ (seen as a map from the infinite cylinder to the finite one) has norm decreasing as $e^{-\delta T}$. In fact weight function for the finite cylinder is approximatively $e^{\delta T}$ at zero and $e^{\delta T/2}$ at $T/2$, while the (translation of) the weight function on the infinite cylinder is $e^{\delta T}$ at zero and $e^{3\delta T/2}$ at $T/2$.  The operator $\tilde{N}^T$ has bounded norm for the same reason. The operator $P^T$ is injective as its kernel (on the whole space) consists of $\mathcal{H}_0$, so the existence of the right inverse for $T\geq T_0$ implies that it is invertible in the same range. Its inverse $N^T$ has bounded norm because $\|N^T-\tilde{N}^T\|$ is going to $0$ and the operator norm of $\tilde{N}^T$ remains bounded.
\end{proof}
We can then deduce the existence of solutions on a finite cylinder.
\begin{cor}
Let $\eta_1,\eta_2$ and $C_0$ as in Lemma \ref{parinfty}. Then there exists $T_0$ such that for every $T\in [T_0,\infty]$ and $(h_0,h)\in H_0\oplus H$ with $\|h\|\leq \eta_1$ there exists a unique $u=u(T,h_0,h)$ in $B_{\eta_2}(\mathcal{E}^T_{\delta})$ satisfying
\begin{align*}
F^T(u)&=0\\
 (\Pi_0,\Pi_0) u&=(h_0,h).
\end{align*}
The map
\begin{equation*}
u(T,-): B_{\eta_1}(H_0\oplus\tilde{H})\rightarrow B_{\eta_1}(\mathcal{E}^T)
\end{equation*}
is smooth and satisfies
\begin{equation*}
u(T,(h_0,0))=s_{h_0}.
\end{equation*}
Finally, we have the estimate
\begin{equation*}
\|u(T,h_0,h)\|\leq 2C_0(e^{\delta T}\|h_0\|+\|h\|).
\end{equation*}
\end{cor}
\begin{proof}
This is again an application of the inverse function theorem, but one needs some extra care to obtain maps defined on a time independent ball in $H_0$. Consider the non-linear map
\begin{equation*}
\left(F^T(s_{h_0}+-),\Pi\right): \mathcal{E}^{T,\perp}_{\delta}\rightarrow \mathcal{F}^T_{\delta}\oplus H,
\end{equation*}
which sends $0$ to $0$ as $\alpha(h_0)=0$ and has linearization at the origin given by the operator
\begin{equation*}
\left({d}/{dt}+L+\mathcal{D}_{h_0}\alpha,\Pi\right): \mathcal{E}^{T,\perp}_{\delta}\rightarrow \mathcal{F}^T_{\delta}\oplus H.
\end{equation*}
We can find a small ball $B_{\eta_1}(H_0)$ for which this operator has norm very close to that of $\frac{d}{dt}+L$, independently of time, as the difference is given by a small operator acting fiberwise. In particular, the previous lemma tells us that the linearization is invertible and its norm is bounded by a fixed constant uniformly in time and $h_0\in B_{\eta_1}(H_0)$. Furthermore we can choose such ball small enough so that all the non-linear parts are also uniformly Lipschitz with fixed small constant on a ball of radius $\eta_2$. Then for each fixed $h_0$ the inverse function theorem provides us with a (unique) solution $u_{h_0}(T,h)\in L^2_{k,\delta}(Z^T)$ with the additional property that
\begin{equation*}
\|u_{h_0}(T,h)- s_{h_0}\|\leq C' \|h\|
\end{equation*}
for some constant $C'$ independent of $h_0$ and $T$. We can interpret this as a map
\begin{equation*}
B_{\eta}(H_0)\times B_{\eta}(H)\rightarrow L^2_{k,\delta}(Z^T).
\end{equation*} 
The smoothness of this map follows from applying the inverse function theorem to the whole map $(F^T,\Pi_0,\Pi)$. Notice that in this case the ball on which the inverse function theorem applies has size decreasing exponentially fast.
\end{proof}
Notice that the map $u$ depends on $\delta$ via the choice of the projection $\Pi_0$. To underline this we respectively denote these by $u_{\delta}$ and $\Pi_0^{\delta}$, as it will be important in the next result. We need to compare the solution $u_{\delta}(T,h_0,h)$ on the finite cylinder $Z^T$ with the solutions $u(\infty,h,h_0)$ (which are independent of $\delta$). As above we denote the two components of the solution on the infinite cylinder as $u_+(h_0,h)$ and $u_-(h_0,h)$. We then define for $(h_0,h)$ the section
\begin{equation*}
U(T,h_0,h)=\tau^*_T u_+(T,h_0,h)+\tau^*_Tu_-(T,h_0,h)-h_0,
\end{equation*}
This section is close to $u_{\delta}(T,h,h_0)$, as the next lemma shows.
\begin{lemma}For each $\eta<\eta_2$, consider the function
\begin{align*}
\xi_{\delta}(T,-): B_{\eta}(H_0)\times B_{\eta}({H})\rightarrow \mathcal{E}^{T}_{\delta}\\
 (h_0,h)\mapsto u_{\delta}(T,h_0,h)-U(T,h_0,h).
\end{align*}
Then there is $\delta>0$ so that the previous result still holds and this function converges to zero in the $C^{\infty}_{\mathrm{loc}}$ topology for $T\rightarrow \infty$.
\end{lemma}
\begin{proof}
Let us write
\begin{equation}\label{configsmall}
(F^T, \Pi_0^{\delta},\Pi)U(T,h_0,h)=(\zeta(h_0,h), h_0+\nu^{\delta}(h_0,h),h+g(h_0,h)).
\end{equation}
Fix a $\delta_0>0$ so that the previous results hold. The functions $\zeta$ and $g$ (which are independent of $\delta$) and their derivatives all have decay of the form
\begin{equation*}
K(h_0,h)e^{-(\delta'+\delta_0) T}
\end{equation*}
for $0<\delta'\leq\delta_0$ where $K$ is a continuous function. The proof of Lemma $18.3.1$ in the book proves the claim for the function $g$ (where we can pick $\delta'$ to be $\delta_0$), and shows that the norm of the map to the unweighted spaces
\begin{equation*}
\zeta:B_{\eta}(H_0)\times B_{\eta}(H)\rightarrow L^2_{k}(Z^T)
\end{equation*}
and its derivatives are bounded by functions of the form $K(h_0,h)e^{-2{\delta_0}T}$. In particular, its norm seen as a map with values in $L^2_{k, \delta_0-\delta'}(Z^T)$ is bounded by $K'(h_0,h)e^{-(\delta_0+\delta')T}$. For any $\delta>0$ the function $\nu^{\delta}$ has norm decreasing as $K(h_0,h)e^{-2\delta T}$. This follows from the fact that the constant sections involved in the orthogonal projection have norm growing like $e^{\delta T}$.
\par
The result then follows by picking $\delta=\delta_0-\delta'$ with $\delta'$ small enough so that the previous results still hold. Indeed from the estimates provided by inverse function theorem the linearization of the local inverse
\begin{equation*}
(F^T, \Pi^{\delta}_0,\Pi)^{-1}: B(\mathcal{F}^{T}_{\delta}\oplus H_0\oplus H)\rightarrow \mathcal{E}^{T}_{\delta}
\end{equation*}
has norm growing like $C(h_0,h)e^{\delta T}$ at each point $(0,h_0,h)$, and is defined in a ball of radius $C(h_0,h)e^{-\delta T}$. Our configuration in equation (\ref{configsmall}) belongs to this ball for $T$ big enough, and for the same reason the result follows.
\end{proof}

\vspace{0.8cm}

With this in hand, we are finally ready to complete the proof of Proposition \ref{parabstract}, following the same arguments of the book.
\begin{proof}[Proof of proposition \ref{parabstract}]
There are two things left to check: the convergence of the restriction maps $\mu_T$ and the smoothness of such a map on the product $[T_0,\infty)\times B_{\eta}(H_0)\times B_{\eta}(\tilde{H})$. For the first one, the previous proposition tells us that it is sufficient to study the convergence of
\begin{equation*}
\tilde{\mu}_T(h,h_0)=rU(T,h,h_0).
\end{equation*}
On the other hand, if $h=(h_0,\tilde{h})$, the component in $L^2_{k-1/2}(\overline{Y}; E_0)$ is the sum
\begin{equation*}
\tilde{\mu}_T(h)=u_+(h)|_{\{0\}\times Y}+u_-(h)|_{\{-2T\}\times Y}-h_0=\mu_{\infty}(h)+(u_-(h)|_{\{-2T\}\times Y}-h_0)
\end{equation*}
and the second term converges to zero in the $C^{\infty}_{\mathrm{loc}}$ topology on the ball because $u_-\in L^2_{k,\delta}(Z^{\leq0};E,H_0)$. The same argument applies to the other boundary component. Finally,
the proof of smoothness is obtained by pulling back the whole family to a fixed cylinder $Z^T$, see the book.
\end{proof}

\vspace{1cm}
We now show how to deduce Theorem \ref{stableman} from Proposition \ref{parabstract}. First of all, we will work with in the slices
\begin{align*}
\mathcal{S}^{\tau}_{k,\acr}(Z^T)&\subset \Ct_k(Z^T)\\
\mathcal{S}^{\tau}_{k,\delta,\acr}(Z^{\infty},\mathfrak{U})&\subset \Ct_{k,\delta}(Z^{\infty},\mathfrak{U})
\end{align*}
both defined by the equations
\begin{align*}
\langle a,n\rangle&=0\text{ at }\partial Z\\
-d^*a+is_0^2 \mathrm{Re}\langle i\phi_0,\phi\rangle+i|\phi_0|^2\mathrm{Re}\mu_Y\langle i\phi_0,\phi\rangle &=0
\end{align*}
where as usual we write $A=A_0+a\otimes1$. As usual, $\mathfrak{U}$ is a neighborhood of $\acr$ in the intersection of the critical set $\Cr$ and the slice $\Sl^{\sigma}_{\acr}$. On the half infinite cylinder, we are using a slice which is different from the one we adopted in Section $3$, see the discussion before Proposition \ref{closedaction}. This equation defines a slice only for a choice of $\delta>0$ small enough depending of the configuration, which is the reason why we did not choose it in that context. Nevertheless, in the present discussion we are considering the neighborhood of a single constant trajectory, and this slice fits our approach better. Using the notation of equation (\ref{gaugefixedeq}) the operator $Q_{\gamma_0}$ obtained by studying the gauge fixed Seiberg-Witten equations can be written as
\begin{equation*}
(V,c)\mapsto \frac{D}{dt}(V,c)+ L_{\gamma_0(t)}(V,c).
\end{equation*}
where $L_{\gamma_0(t)}$ is not the usual extended Hessian, and not its weighted counterpart. In particular, this family of operators is constant at the constant trajectory $\gamma_{\acr}$. We will study a small neighborhood of $[\gamma_{\acr}]$ in $M(Z^T)$ by studying solutions of the gauge-fixing equation together with the perturbed Seiberg-Witten equations. To reduce ourselves to the linear setting of Proposition \ref{parabstract}, we introduce a chart $\varphi$ induced by a suitable $L^2_k$-compatible product chart (which is the restriction of an $L^2_{k-1/2}$-compatible product chart) as in Lemma \ref{cylinderchart}. We can suppose that the neighborhood $\tilde{\mathcal{U}}_T$ of the origin on which the chart is defined contains all the configurations that have distance at most $\eta_0>0$ for a constant configuration $\gamma_{\bcr}$ where $\bcr$ is a critical point in $\mathfrak{U}$ for some constant $\eta_0$ is independent of $T$ large enough.

\par
We have the restriction map
\begin{equation*}
r:\Ct_k(Z^T)\rightarrow \Cs_{k-1/2}(\overline{Y}\amalg Y)\times L^2_{k-1/2}(\overline{Y}\amalg Y;i\R)
\end{equation*}
where the second component records the normal component of the connection at the boundary. We will use the usual decomposition at a critical point
\begin{equation*}
\T^{\sigma}_{k-1/2,\acr}= \J^{\sigma}_{k-1/2,\acr}\oplus \K^{\sigma}_{k-1/2,\acr},
\end{equation*}
where the second summand is $\T_{\acr}\mathfrak{U}\oplus\N^{\sigma}_{k-1/2,\acr}$ and we have spectral decomposition
\begin{equation*}
\N=\N^{\sigma}_{k-1/2,\acr}=\N^+\oplus \N^-
\end{equation*}
as in equation (\ref{spectralN}).
We can define the subspaces of $\T^{\sigma}_{k-1/2,\acr}\oplus L^2_{k-1/2}(Y;i\R)$
\begin{IEEEeqnarray*}{c}
H^-_Y=\{0\}\oplus\{0\}\oplus \N^-\oplus L^2_{k-1/2}(Y;i\R) \\
H^-_{\overline{Y}}=\{0\}\oplus \{0\}\oplus\N^+\oplus L^2_{k-1/2}(Y;i\R)
\end{IEEEeqnarray*}
and define the projections
\begin{IEEEeqnarray*}{c}
\Pi^-_Y: \T^{\sigma}_{k-1/2,\acr}\oplus L^2_{k-1/2}(Y;i\R)\rightarrow H^-_Y \\
\Pi^-_{\overline{Y}}: \T^{\sigma}_{k-1/2,\acr}\oplus L^2_{k-1/2}(Y;i\R)\rightarrow H^-_{\overline{Y}}
\end{IEEEeqnarray*}
with kernels respectively
\begin{IEEEeqnarray*}{c}
J^{\sigma}_{k-1/2,\acr}\oplus (T_{\acr}\mathfrak{U}\oplus \N^+)\oplus\{0\} \\
J^{\sigma}_{k-1/2,\acr}\oplus (T_{\acr}\mathfrak{U}\oplus \N^-)\oplus\{0\},
\end{IEEEeqnarray*}
and we can set 
\begin{IEEEeqnarray*}{c}{H}= H^-_{\bar{Y}}\oplus H^-_Y \\
{\Pi}=\Pi^-_{\bar{Y}}\oplus \Pi^-_Y.
\end{IEEEeqnarray*}
Finally we let $H_0=\T_{\acr}\mathfrak{U}$ and we have the map for $T=\infty$
\begin{equation*}
\Pi_0: \T^{\tau}_{k,\delta,\gamma_{\acr}}\rightarrow H_0
\end{equation*}
simply sends a configuration to its limit point. Given a path $\gamma$ in $\Ct_k(Z^T)$ such that the restrictions to the boundary lie in the domain of the chart $U$,  we are interested in the system of equations given by
\begin{IEEEeqnarray*}{c}
\F^{\tau}_{\q}\gamma=0 \\
\mathrm{Coul}^{\tau}_{\acr}\gamma=0 \\
(\Pi_0,\Pi)\circ (\varphi^{-1}\circ r)\gamma=(h_0,h)
\end{IEEEeqnarray*}
where $(h_0,{h})\in H_0\oplus {H}$. 
Hence given any element $\gamma$ lying in a small neighborhood of $\gamma_{\acr}$ so that both restrictions lie in the domain $\mathcal{U}$ of the chart $\varphi$, we can alternatively write the equations as
\begin{align*}
(Q_{\gamma_{\acr}}+\alpha)\tilde{\gamma}&=0 \\
(\Pi_0,\Pi)\circ r\circ\tilde{\gamma}&=(h_0,h)
\end{align*}
where $Q_{\gamma_{\acr}}$ is the linear part, $\alpha$ is the remainder of the terms (and defined slicewise) and $\tilde{\gamma}$ is the configuration in $\T^{\tau}_{k,\gamma_{\acr}}(Z^T)$ corresponding to $\gamma$ under the local chart. The notation is justified because $Q_{\gamma_{\acr}}$ is the linearization of the Seiberg-Witten equations with Coulomb gauge condition arising in Section $4$, because of the condition $\mathcal{D}_{0}\varphi=\mathrm{Id}$ in the definition of an $L^2_{k-1/2}$-compatible product chart.
\\
\par
We then turn to study this problem with the abstract gluing result proved above. Here neither the domain $\T^{\tau}_{k,\gamma_{\acr}}$ nor the range $\V^{\tau}_{k-1}$ are in the form required by Proposition \ref{parabstract}, but they can be converted to spaces of sections of a finite dimensional vector bundle by the same device as Section $1$. We next verify the key hypothesis of Proposition \ref{parabstract}.
\begin{lemma}The linearized equations on the infinite cylinder $Z^{\infty}$ are invertible at $0$ for $\delta>0$ small enough.
\end{lemma}
\begin{proof}This is essentially a parametric version of the argument that settles the Morse case. One just has to look at the operator on each component, and we will focus on the operator on $Z^{\leq 0}$ given by
\begin{equation*}\label{operQgamma}
(Q_{\gamma_{\acr}}, \Pi_Y^-\circ r, \Pi_0):\T^{\tau}_{k,\delta,\gamma_{\acr}}\rightarrow \V^{\tau}_{k-1,\delta,\gamma_{\acr}}\oplus H^-_Y\oplus H_0.
\end{equation*}
If we write $Q_{\gamma_{\acr}}=d/dt+L$ and call $H_L^{\pm}$ the spectral subspaces of the almost hyperbolic operator $L$ (which is the weighted extended Hessian at the point for $t=\infty$, see Lemma \ref{weighhessian}), Lemma \ref{almostAPS} tells us that the operator
\begin{equation*}
(Q_{\gamma_{\acr}}, \Pi_L^-\circ r, \Pi_0):\T^{\tau}_{k,\delta,\gamma_{\acr}}\rightarrow \V^{\tau}_{k-1,\delta,\gamma_{\acr}}\oplus (H^-_L\cap L^2_{k-1/2})\oplus H_0
\end{equation*}
is an isomorphism, where $\Pi^-_L$ is the negative spectral projection. If we decompose now the domain a the direct sum $C\oplus K$, where $K$ is the kernel of $Q_{\gamma_{\acr}}$, we can write the operator
\begin{equation*}
\begin{bmatrix}
Q_{\gamma_{\acr}}|_C & 0\\
x & (\Pi^-_L\circ r)\oplus \Pi_0|_K.
\end{bmatrix}
\end{equation*}
On the other hand Lemma \ref{almostAPS} tells us that the image of $K$ under $r$ in $L^2_{k-1/2}(Y;E_0)$ is exactly $H_L^-\oplus H_0$. Furthermore the proof of Lemma $17.3.3$ in the book still applies in this case to show that $\Pi_Y^-$ is an isomorphism on $H_L^-$, so the matrix 
\begin{equation*}
\begin{bmatrix}
Q_{\gamma_{\acr}}|_C & 0\\
x & (\Pi^-_Y\circ r)\oplus \Pi_0|_K.
\end{bmatrix}
\end{equation*}
also defines an invertible operator. This is exactly the operator in equation (\ref{operQgamma}).
\end{proof}
Hence Proposition \ref{parabstract} provides us a solution $\gamma=u(T,h,h_0)$ to these equations for any $T\in[T_0,\infty]$ and $\|h_0\|,\|{h}\|\leq \eta$. The boundary conditions include also the normal component of the connections at the boundary,
\begin{IEEEeqnarray*}{c}
\langle a, n\rangle=c_1\text{ on }\bar{Y}\\
\langle a, n\rangle=c_2\text{ on }Y.
\end{IEEEeqnarray*}
By restricting to boundary conditions with $c_1=c_2=0$ we get trajectories in Neumann-gauge, and hence we obtain a parametrization of the solutions in the slice
\begin{IEEEeqnarray*}{c}
u(T,-): B_{\eta_1}(T_{\acr}\mathfrak{U})\times B_{\eta_1}(\tilde{\N})\rightarrow \Sl^{\tau}_{k,\gamma_{\acr}} \\
u(\infty,-): B_{\eta_1}(T_{\acr}\mathfrak{U})\times B_{\eta_1}(\tilde{\N})\rightarrow \Sl^{\tau}_{k,\delta,\gamma_{\acr}}
\end{IEEEeqnarray*}
such that there exists $\eta_2>0$ so that for every $T\geq T_0$ the image of $u(T,-)$ contains all solutions in $\mathcal{S}^{\tau}_{k,\gamma_{\acr}}$ with $\|\gamma-\gamma_{\bcr}\|_{L^2_{k,\delta}}\leq \eta_2$ for some $\bcr\in \mathfrak{U}$. The final step to prove Theorem \ref{stableman} is to extend the result to a uniform neighborhoods in the moduli space of solutions defined by the inequality $\|\gamma-\gamma_{\bcr}\|_{L^2_{k,\delta}}\leq \eta_2$ where $\gamma$ is not necessarily in the slice. This is proved in the same fashion as Proposition $18.4.1$ in the book, with the same adaptation discussed above (in particular Lemma \ref{boundedinv}). Finally, we notice that for $\delta>0$ the norms
\begin{equation*}
\| \gamma-\gamma_{\acr}\|_{L^2_k(Z^T)}\text{ and }\| \gamma-\gamma_{\acr}\|_{L^2_{k,\delta}(Z^T)}
\end{equation*}
are equivalent on a small neighborhood $\left\{\| \gamma-\gamma_{\acr}\|_{L^2_k(Z^T)}\leq \eta_3\right\}$ of $\gamma_{\acr}$ in $M(Z^T)$ by a constant independent of time, for some constant $\eta_3$ also independent of time. This follows from the exponential decay of a solution to the Seiberg-Witten equations always close to a constant one, as discussed in Section $2$.

\vspace{1.5cm}
\section{The moduli space on a cobordism}
In this section we briefly discuss the adaptation of the theory we have developed so far to the moduli spaces of solutions to the perturbed Seiberg-Witten equation on a general manifold with (possibly disconnected) boundary. This generalizes the content of Chapter $24$ in the book. The proofs of the results we are going to state can be easily obtained from those in the Morse case, using the same techniques we have used throughout the present chapter.
\\
\par
Let $X$ be a compact connected oriented Riemannian $4$-manifold with non empty (and possibly disconnected) boundary
\begin{equation*}
Y=\amalg Y^{\alpha},
\end{equation*}
and let us suppose that the metric is cylindrical in the neighborhood of the boundary, so it contains an isometric copy of $I\times Y$ where $I=(-C,0]$, with $\partial C=\{0\}\times Y$. A spin$^c$ structure $\spin_X$ on $X$ determines a spin$^c$ structure $\spin$ on $Y$. We define the configuration space and the space of tame perturbations
\begin{align*}
\Bs_k(Y,\spin)=&\prod \Bs_k(Y^{\alpha},\spin^{\alpha})\\
\mathcal{P}(Y,\spin)=&\prod \mathcal{P}(Y^{\alpha},\spin^{\alpha}).
\end{align*}
Suppose from now on that a given tame perturbation $\q_0=\{\q_0^{\alpha}\}$ which is Morse-Bott has been fixed on the boundary.
\\
\par
Our approach to the construction of maps induced by cobordisms (and in particular of the module structures) will be different from the one in the book, and it will heavily rely on the structure of moduli spaces on manifolds with more than one end. In that case,  we will partition the boundary $Y$ in two classes: \textit{outgoing} boundary components, which will have the boundary orientation, and \textit{incoming} boundary components, with the opposite orientation. In the present section however, we will suppose that all the boundary components have the boundary orientation in order to keep the discussion as uniform as possible. Notice that the orientation reversal changes the role of boundary stable and unstable critical points on the incoming ends, which may cause some confusion with the gradings.
\\
\par

The main difference with the setup we have worked so far is that in the case of a general cobordism we cannot rely on the $\tau$ model, which is only defined for cylinders. Recall from Section $2$ in Chapter $1$ that we have the configuration space
\begin{equation*}
\Cs(X,\spin_X)=\left\{(A,s,\phi)\mid s\geq 0,\|\phi\|_{L^2(X)}=1\right\}\subset \mathcal{A}\times \R\times \Gamma(X;S^+)
\end{equation*}
and its Hilbert completion $\Cs_k(X,\spin_X)$. The quotient by the gauge group action is denoted by $\Bs_k(X,\spin_X)$ and is a Hilbert manifold with boundary. The Seiberg-Witten equations define a smooth section $\F^{\sigma}$ of the vector bundle
\begin{equation*}
\V_{k-1}^{\sigma}\rightarrow \Cs_k(X,\spin_X), 
\end{equation*}
and we introduce perturbations as follows. Pick a cut-off function $\beta$ equal to $1$ near $t=0$ and $0$ near $t=-C$, and a bump function $\beta_0$ with compact support in $(-C,0]$. If we regard these as functions on the collar $I\times Y$, any tame perturbation $\mathfrak{p}_0$ in $\mathcal{P}(Y,\spin)$ defines the section
\begin{IEEEeqnarray*}{c}
\hat{\mathfrak{p}}:\Co_k(X,\spin_X)\rightarrow \V_k \\
\hat{\mathfrak{p}}=\beta\hat{\q}_0+\beta_0\hat{\mathfrak{p}}_0,
\end{IEEEeqnarray*}
where $\q_0$ is the fixed Morse-Bott perturbation.
This section can be extended as in Section $3$ of Chapter $1$ to the blown-up setting, and we can hence define
\begin{equation*}
\F_{\mathfrak{p}}^{\sigma}=\F^{\sigma}+\hat{\mathfrak{p}}^{\sigma}:\Cs_k(X,\spin_X)\rightarrow \V_{k-1}^{\sigma}
\end{equation*}
The perturbed Seiberg-Witten equations $\F_{\hat{\mathfrak{p}}}^{\sigma}=0$ are invariant under $\G_{k+1}(X)$, and we have the moduli spaces of solutions
\begin{align*}
M(X,\spin_X)&\subset \Bs_k(X,\spin_X)\\
M(X,\spin_X)&=\left\{(A,s,\phi)\mid\F_{\hat{\mathfrak{p}}}^{\sigma}=0\right\}\big/\G_{k+1}(X).
\end{align*}
Similarly, we have the larger moduli space
\begin{equation*}
\tilde{M}(X,\spin_X)\subset \tilde{\Bo}^{\sigma}_k(X,\spin_X).
\end{equation*}
obtained by dropping the condition $s\geq 0$, and $M(X,\spin_X)$ is identified with its quotient by the involution $\mathbf{i}$ switching the sign of $s$. The unique continuation property of the Seiberg-Witten equations (see Section $10.8$ in the book) and the equivalence between the $\sigma$ and $\tau$ models on a finite cylinder tells us that there are well defined restriction maps to the cylindrical end
\begin{equation*}
M(X,\spin_X)\rightarrow M(I\times Y,\spin_X)\subset\Bt_k(I\times Y,\spin_X) 
\end{equation*}
and to the boundary
\begin{equation*} 
R: M(X,\spin_X)\rightarrow \Bs_{k-1/2}(Y,\spin).
\end{equation*}
Also, the section $\F_{\hat{\mathfrak{p}}}^{\sigma}$ is transverse to zero, and $\tilde{M}(X,\spin_X)$ and $M(X,\spin_X)$ are respectively a smooth Hilbert manifold and a smooth Hilbert manifold with boundary (see Proposition $24.3.1$ in the book).

\vspace{0.8cm}
We then turn our interest on the non-compact Riemannian manifold $X^*$ obtained by attaching cylindrical ends to $X$, namely
\begin{align*}
X^*&=X\cup_Y Z \\
Z&=[0,\infty)\times Y.
\end{align*}
On such a space we have the $L^2_{k,\mathrm{\mathrm{loc}}}$ configuration space
\begin{equation*}
\Co_{k,\mathrm{loc}}(X^*,\spin_X)=\mathcal{A}_{k,\mathrm{loc}}\times L^2_{k,\mathrm{loc}}(X^*;S^+).
\end{equation*}
As we are not dealing with a Banach space anymore, we have to define the blow-up as follows (see Section $6.1$ in the book). Define the sphere $\mathbb{S}$ as the topological manifold obtained as the quotient of $L^2_{k,\mathrm{loc}}(X^*;S^+)\setminus 0$ by the action of $\R^+$. We then define
\begin{align*}
\Cs_{k,\mathrm{loc}}(X^*,\spin_X)&=\left\{(A,\R^+\phi,\Phi)\mid\Phi\in \R^{\geq}\phi\right\}\\
&\subset \mathcal{A}_{k,\mathrm{loc}}\times\mathbb{S}\times L^2_{k,\mathrm{loc}}(X^*; S^+),
\end{align*}
which comes with a canonical blow down map $\pi$. If we call $\mathcal{O}(-1)$ the complex tautological line bundle on $\mathbb{S}$, one can define the bundle
\begin{equation*}
\V^{\sigma}_{k-1}=\mathcal{O}(-1)^*\otimes \pi^*(\V_{k-1})\rightarrow\Cs_{k,\mathrm{loc}}(X^*,\spin_X),
\end{equation*}
and its section
\begin{align*}
\F^{\sigma}:\mathcal{O}(-1)&\rightarrow  \pi^*(\V_{k-1})\\
\F^{\sigma}(A,\R^+\phi, \Phi)(\psi)&=\left(\frac{1}{2}\rho(F_{A^t}^+)-(\Phi\Phi^*)_0,D_A^+\psi\right).
\end{align*}
We can consider again the perturbed equation, given by the continuous  gauge-invariant section $\F^{\sigma}_{\mathfrak{p}}=\F^{\sigma}+\hat{\mathfrak{p}}^{\sigma}$, where the perturbing term $\hat{\mathfrak{p}}^{\sigma}$ is defined as before on the collar $I\times Y$ and is extended to be $\hat{\q}^{\sigma}_0$ on the cylindrical end $Z$. The unique continuation property tells us again that there is a restriction map
\begin{equation*}
\left\{[\gamma]\in \Bs_{k,\mathrm{loc}}(X^*,\spin_X)\mid \F^{\sigma}_{\mathfrak{p}}(\gamma)=0\right\}\rightarrow \Bt_{k,\mathrm{loc}}(Z,\spin),
\end{equation*}
and clearly this restriction satisfies the equation on the cylinder.
\begin{defn}
For a critical submanifold
\begin{equation*}
[\Lr]=\prod [\Cr_{\alpha}]\subset \Bs_k(Y,\spin),
\end{equation*}
define the moduli space
\begin{equation*}
M(X^*,\spin_X;[\Lr])\subset \Bs_{k,\mathrm{loc}}(X^*,\spin_X)
\end{equation*}
as the set of all $[\gamma]$ with $\F^{\sigma}_{\mathfrak{p}}(\gamma)=0$ and such that its restriction is asymptotic to a configuration of $[\Cr]$ on the cylindrical end $Z$.
\end{defn}
We can also consider the union over all spin$^c$ structures $\Bs_{k,\mathrm{loc}}(X^*)$, and the moduli space
\begin{equation*}
M(X^*;[\Cr])= \coprod_{\spin_X} M(X^*,\spin_X;[\Cr])\subset \Bs_{k,\mathrm{loc}}(X^*),
\end{equation*}
where we implicitly restrict ourselves to the union to the spin$^c$ structures inducing the one to which $[\Cr]$ belongs. Similarly, we introduce the set of reducible elements
\begin{equation*}
M^{\mathrm{red}}(X^*, \spin_X;[\Cr])\subset M(X^*,\spin_X;[\Cr])
\end{equation*}
as the configurations with representatives $(A,s,\phi)$ with $s=0$. The configuration space can be decomposed along the elements
\begin{equation*}
z\in\pi_0(\Bs(X;[\Cr]))
\end{equation*}
and so can the moduli space
\begin{equation*}
M(X^*;[\Cr])=\bigcup_zM_z(X^*;[\Cr]).
\end{equation*}
Also, given an element $z_1\in \pi_1(\Bs(Y);[\Cr],[\Cr'])$, we obtain by concatenation a new element
\begin{equation*}
z_1\circ z\in\pi_0(\Bs(X;[\Cr']).
\end{equation*}
Finally,  there are natural continuous evaluation maps
\begin{equation*}
\ev^{\alpha}: M(X^*;[\Cr])\rightarrow [\Cr^{\alpha}]\subset \Bs_k(Y^{\alpha},\spin^{\alpha})
\end{equation*}
to each component of the critical submanifold.
\par
\vspace{0.8cm}

The clearest way to discuss regularity in this framework is to pass to a fiber product description. Both manifolds $X$ and $Z$ have boundary $Y$, and we have well defined restriction maps
\begin{align*}
R_+: M(X,\spin_X)&\rightarrow \Bs_{k-1/2}(Y,\spin)\\
R_-:  M(Z;[\Lr]) &\rightarrow \Bs_{k-1/2}(Y,\spin).
\end{align*}
Letting $\mathrm{Fib}(R_+,R_-)$ be the fibered product of these two maps, we have the restriction map
\begin{equation*}
\rho: M(X^*,\spin_X;[\Lr])\rightarrow \mathrm{Fib}(R_+,R_-).
\end{equation*}
In fact Lemma $24.2.2$ in the book tells us that this map is in fact a homeomorphism. The following definition is the analogue of the transversality hypothesis introduced in Definition \ref{smalereg}.
\begin{defn}\label{smaleregcob}
Let consider a solution $[\gamma]$ in $M(X^*,\spin_X;[\Cr])$. If $[\gamma]$ is irreducible, we say that the moduli space is \textit{Smale-regular} at $[\gamma]$ if the maps of Hilbert manifolds
\begin{equation*}
(R_-,R_+): M(X,\spin_X)\times M(Z;[\bcr]) \rightarrow \Bs_{k-1/2}(Y,\spin)\times \Bs_{k-1/2}(Y,\spin)
\end{equation*}
is transverse to the diagonal at $\rho[\gamma]$. In the reducible case, we consider instead the map
\begin{equation*}
(R_-,R_+): M^{\mathrm{red}}(X,\spin_X)\times M^{\mathrm{red}}(Z;[\bcr]) \rightarrow \partial\Bs_{k-1/2}(Y,\spin)\times\partial \Bs_{k-1/2}(Y,\spin).
\end{equation*}
We say that $M(X^*,\spin_X;[\Cr])$ is \textit{Smale-regular} if it is regular at every point. Finally, we say that a perturbation $\hat{\mathfrak{p}}$ is Smale-regular if all the moduli spaces are Smale-regular.
\end{defn}
One has the following result (see Proposition $24.4.3$ in the book).
\begin{lemma}
Let $[\Lr]=\prod[\Lr^{\alpha}]$ be a critical submanifold. If the moduli space $M(X^*,\spin_X;[\Lr])$ is non empty and Smale-regular, then it is
\begin{itemize}
\item a smooth manifold consisting only of irreducibles, if any $[\Cr^{\alpha}]$ is irreducible;
\item a smooth manifold consisting only of reducibles, if any $[\Cr^{\alpha}]$ is boundary stable;
\item a smooth manifold with (possibly empty) boundary if all $[\Cr^{\alpha}]$ are boundary stable.
\end{itemize}
In the last case, the boundary consists of the reducible elements of the moduli space.
\end{lemma}
\begin{defn}
Suppose that exactly $c+1$ of the $[\Cr^{\alpha}]$ are boundary unstable, with $c\geq 1$. Then we say that $[\Cr]$ is \textit{boundary obstructed with corank} $c$. 
\end{defn}
\vspace{0.8cm}
We can define the grading as follows. Let $[\gamma]$ be any element of $\Bs_{k,z}(X^*;[\Cr])$ and $\gamma$ a gauge representative. Suppose $[\gamma]$ is asymptotic to the critical point $[\bcr]$  and let $[\gamma_{\bcr}]$ be the corresponding constant trajectory in $\Bt_{k,\delta}(Z,\spin;[\Cr])$. We have the operator
\begin{equation*}
Q^{\sigma}_{\gamma}=\mathcal{D}_{\gamma}\F^{\sigma}_{\mathfrak{p}}\oplus \mathbf{d}^{\sigma,\dagger}_{\gamma}
\end{equation*}
on $X$ and the translation invariant operator $Q_{\gamma_{\bcr}}$ on $Z$. There are restriction maps
\begin{align*}
r_+&:\mathrm{ker}(Q^{\sigma}_{\gamma})\rightarrow L^2_{k-1/2}(Y;iT^*Y\oplus S\oplus i\R)\\
r_-&:\mathrm{ker}(Q_{\gamma_{\bcr}})\rightarrow L^2_{k-1/2}(Y;iT^*Y\oplus S\oplus i\R).
\end{align*} 
Then the operator
\begin{equation*}
r_+-r_-:\mathrm{ker}(Q^{\sigma}_{\gamma})\oplus\mathrm{ker}(Q_{\gamma_{\bcr}})\rightarrow L^2_{k-1/2}(Y;iT^*Y\oplus S\oplus i\R)
\end{equation*}
is Fredholm (see Proposition $24.3.2$ in the book), and we define $\mathrm{Ind}_z(X;[\Cr])$ to be its index. The point of the definition is that is makes sense even when the moduli space is empty.
We the define the relative grading to be
\begin{equation*}
\gr_z(X,[\Cr])=\mathrm{Ind}_z(X;[\Cr])-\mathrm{dim}[\Cr]
\end{equation*}
where the las term denotes the sum of the dimensions of the critical manifolds. This grading has the simple additivity property
\begin{equation*}
\gr_{z_1\circ z}(X;[\Cr])=\gr_z(X;[\Cr_0])+\gr_{z_1}([\Cr_0],[\Cr]),
\end{equation*}
where the last term is the sum over all components. Notice that in the case $X$ is a finite cylinder (hence $X^*$ is an infinite cylinder), the definition coincides with the usual one. This is because of the orientation convention the new index is obtained by considering the operator as acting on the (unproperly called) weighted Sobolev space on the doubly infinite cylinder defined by the function
\begin{equation*}
f(t)=e^{\delta t}.
\end{equation*}
In particular its index differs by the one we used in Section $4$ by $\mathrm{dim}[\Cr_-]$ by simple a spectral flow argument. We also have the following (see Proposition $24.4.6$ in the book).
\begin{prop}
If the moduli space $M_z(X;[\Lr])$ is non-empty and Smale-regular, its dimension is $\gr_z(X;[\Cr])+\mathrm{dim}[\Lr]$ in the boundary unobstructed case. If the moduli space is boundary obstructed with corank $c$, then its dimension is $\gr_z(X;[\Cr])+\mathrm{dim}[\Lr]+c$.
\end{prop}
\vspace{0.8cm}

In order to have a nice characterization of the compactifications of the moduli spaces we will need a stronger transversality assumption on the evaluation maps, as in Section $4$. Suppose we have  for each $\alpha$ a sequence of critical submanifolds
\begin{equation*}
\mathcal{C}^{\alpha}=\left([\Cr^{\alpha}_0],[\Cr^{\alpha}_1],\dots, [\Cr^{\alpha}_{n^{\alpha}}]\right),
\end{equation*}
and corresponding homotopy classes of relative homotopy classes $\mathbf{z}^{\alpha}$. We denote the product simply by $\mathcal{C}$ and $\mathbf{z}$.
We can then consider the space
\begin{equation*}
M_{\mathbf{z}}(X^*,\mathcal{C})
\end{equation*}
consisting of tuples $\left([\gamma_0],[\gamma^{\alpha}_i]\right)$  with
\begin{align*}
{[}\gamma_0]&\in M_{z_0}(X^*; [\Cr])\\
{[}\gamma^{\alpha}_i]&\in M_{z^{\alpha}_i}([\Cr_{i-1}^{\alpha}],[\Cr_{i}^{\alpha}])\text{ for } 1\leq i\leq n^{\alpha},
\end{align*}
such that the evaluations agree, i.e. we have
\begin{align*}
\ev^{\alpha}[\gamma_0]&=\ev_-[\gamma_1^{\alpha}],\\
\ev_+[\gamma_i^{\alpha}]&=\ev_-[\gamma_{i+1}^{\alpha}]\text{ for }i=1,\dots, n_{\alpha}-1.
\end{align*}
The space $M_{\mathbf{z}}(X^*,\mathcal{C})$ is naturally equipped with an evaluation map
\begin{equation*}
\ev:M_{\mathbf{z}}(X^*,\mathcal{C})\rightarrow \prod_{\alpha} [\Cr^{\alpha}_{n^{\alpha}}].
\end{equation*}
We then introduce the following definition, which has to be interpreted in an inductive fashion as Definition \ref{regularpert} in Section $4$.

\begin{defn}\label{regcob}
Suppose we are given a Smale-regular perturbation $\hat{\mathfrak{p}}$. We then say that ${\mathfrak{p}}$ is \textit{regular} if the following holds. For every sequence of critical manifolds $\mathcal{C}$ and corresponding relative homotopy classes $\mathbf{z}$ (see the notation above), and any other critical submanifold $[\Cr_+]=\prod [\Cr_+^{\alpha}]$ and relative homotopy class $z_+^{\alpha}\in \pi_1(\Bs_k(Y),[\Cr^{\alpha}_{n^{\alpha}}],[\Cr_+])$, the evaluations
\begin{IEEEeqnarray*}{c}
\ev:M_{\mathbf{z}}(X^*,\mathcal{C})\rightarrow \prod[\Cr_{n^{\alpha}}^{\alpha}]\\
\prod\ev_-^{\alpha}: \prod M_{z_+^{\alpha}}([\Cr_{n^{\alpha}}^{\alpha}],[\Cr_+^{\alpha}])\rightarrow\prod [\Cr_{n^{\alpha}}^{\alpha}]
\end{IEEEeqnarray*}
are transverse smooth maps.
\end{defn}

\vspace{0.8cm}
We now discuss the transversality result in the general case of a family of perturbations and metrics, which will be needed in the rest of the present work. Let $P$ be a smooth finite dimensional manifold (possibly with boundary) parametrizing a smooth family of Riemannian metrics $g^P$ on $X$, all of which contain an isometric copy of $I\times Y$. Similarly, let $\mathfrak{p}_0^P\in \mathcal{P}(Y,\spin)$ be a smooth family of perturbations, and write
\begin{equation*}
\mathfrak{p}^p=\beta(t)\q_0+\beta_0(t)\mathfrak{p}_0^p
\end{equation*}
as usual. Let $M(X^*,\spin_X;[\Cr])_p$ the moduli space corresponding to the metric $g_p$ and perturbation $\mathfrak{p}^p$, and consider the total space
\begin{align*}
M(X^*,\spin_X;[\Cr])_P&=\bigcup_p\{p\}\times M(X^*,\spin_X;[\Cr])_p\\
&\subset P\times \Bs_{k,\mathrm{loc}}(X^*,\spin_X),
\end{align*}
where the identification with a fixed spin bundle $S^{\pm}_{p_0}$ is implicitly made so that the configuration space can be thought to be metric independent. The notions of Smale-regularity and regularity readily extend to this context by considering the restriction maps and the evaluation maps from   the moduli spaces parametrized by the strata $P\setminus \partial P$ and $\partial P$ (and \textit{not} for a fixed $p$). The following is the key transversality result, which is proved as Proposition $24.4.10$ in the book with the same modifications as Proposition \ref{transversality}.
\begin{prop}\label{transversecob}
Suppose we have a fixed regular Morse-Bott pertubation $\q_0$ on $Y$. Let $g^P$ be a smooth family of metrics as above and $\mathfrak{p}^P$ a family of tame perturbations. Suppose that the family parametrized by the boundary $\partial P$ is regular. Then there is a new family of perturbations $\tilde{\mathfrak{p}}^P$ over $P$ such that
\begin{equation*}
\tilde{\mathfrak{p}}^p=\mathfrak{p}^p\text{ for all }p\in \partial P
\end{equation*}
and such that the corresponding parametrized moduli space is regular at every point.
\end{prop}
\begin{remark}
This result readily generalizes to families of metrics and perturbations parametrized by polyhedra. This will be useful in some other aspects of the theory.
\end{remark}

\vspace{0.8cm}
There is a natural way to introduce the compactification of the moduli space on a general manifold with cylindrical ends, analogous to the one for trajectories defined in Chapter $4$.
\begin{defn}
Consider a critical submanifold $[\Cr]$. A \textit{broken} $X$-\textit{trajectory} asymptotic to $[\Cr]$ consists of pairs $([\gamma_0],[\breve{\gamma}])$ where:
\begin{itemize}
\item $[\gamma_0]$ belongs to a moduli space $M_{z_0}(X^*,[\Cr_0])$;
\item $[\boldsymbol{\breve{\gamma}}]$ consists of an unparametrized broken trajectory $[\boldsymbol{\breve{\gamma}}^{\alpha}]$ in $\breve{M}^+_{z_1^{\alpha}}([\Cr_0^{\alpha}],[\Cr^{\alpha}])$ for every $\alpha$ such that its negative evaluation coincides with the evaluation of $[\gamma_0]$ at the end $Y^{\alpha}$.
\end{itemize}
The \textit{homotopy class} the broken trajectory is given by
\begin{equation*}
z=z_1\circ z_0\in \pi_0(\Bs(X,[\Cr]),
\end{equation*}
and we denote by $M^+_{z}(X^*,[\Cr])$ the space of broken $X$-trajectories in this homotopy class.
\end{defn} 
This space is topologized in a way analogous to the spaces $\breve{M}^+([\Cr_-],[\Cr_+])$, see Section $24.6$ in the book for details. Notice also that $M_z(X^*,[\Cr])$ sits inside $M^+_z(X^*,[\Cr])$ as the special case in which each $[{\gamma}^{\alpha}]$ has zero components. Furthermore, there is a natural continuous evaluation map
\begin{equation*}
\ev: M^+_{z}(X^*,[\Cr])\rightarrow [\Cr].
\end{equation*}
It is important to notice that even though when $X$ is a finite cylinder $X^*$ is an infinite cylinder, the compactification we have just constructed is quite different from the space of broken unparametrized trajectories

\begin{example}Here we deal with a simple (reducible) example. Consider a non degenerate perturbation $\q$, and two boundary stable critical points $\acr_1$ and $\acr_2$ over the same reducible critical point $\alpha$ and corresponding to the smallest and second smallest eigenvalues. Let $z_0$ the relative homotopy class corresponding to the path connecting $\acr_2$ and $\acr_1$ in $\Cs_k(Y)$. Then the moduli space
\begin{equation*}
M_{z_0}^{\mathrm{red}}([\acr_2],[\acr_1])
\end{equation*}
is identified with a copy of $\C P^1$ with two points removed, see Section $14.6$ in the book. This implies that $\breve{M}_{z_0}^{\mathrm{red}}([\acr_2],[\acr_1])$ is diffeomorphic to $S^1$, and hence coincides with $\breve{M}_{z_0}^{\mathrm{red}+}([\acr_2],[\acr_1])$. We can compare this space with its compactification given by the space of $I\times Y$-trajectories. This perturbation does not fall in the class studied above (it is translation invariant), but everything still makes perfectly sense. In particular, the compactification $M_{z_0}^{\mathrm{red}+}([\acr_2],[\acr_1])$ is identified with $S^1\times [0,1]$, where
\begin{equation*}
\{0,1\}\times S^1=M_{z_2}([\acr_2],[\acr_2])\times \breve{M}_{z_0}^{\mathrm{red}+}([\acr_2],[\acr_1])\coprod \breve{M}_{z_0}^{\mathrm{red}+}([\acr_2],[\acr_1])\times M_{z_1}([\acr_1],[\acr_1])
\end{equation*}
and $z_2$ and $z_2$ are the corresponding trivial homotopy classes.
\end{example}

Given a family of metrics $g^P$ and perturbations $\mathfrak{p}^P$ as above, we can form the parametrized space of broken $X$-trajectories
\begin{equation*}
M^+(X^*,[\Cr])_P=\bigcup_p\{p\}\times M^+(X^*,[\Cr])_p 
\end{equation*}
which can be given a natural topology using the fact that the metric and perturbation on the cylindrical part are independent of $p$. The main compactness and finiteness theorem is then the following (see Theorem $24.6.7$ in the book).
\begin{teor}\label{finitenesscob}
Suppose that the families of metric and perturbations $\{(g^p,\mathfrak{p}^p)\}_{p\in P}$ is regular. Then for each $[\Cr]$ the family of moduli spaces $M^+_z(X^*,[\Cr])_P$ is proper over $P$. For fixed $[\Cr]$, this family of moduli spaces is non empty for only finitely many components $z\in\pi_0(\Bs_k(X,[\Cr]))$.
\end{teor}
\vspace{0.5cm}
We discuss the structure of the compactifications as spaces stratified by manifolds (see Proposition $24.6.10$ in the book). The space $M^+_z(X^*,[\Cr])_P$ can be written as the disjoint union of the following subspaces. For any other critical submanifold $[\Cr_0]$ and strata $M^{\alpha}$ of the stratification of $\breve{M}^+([\Cr_0^{\alpha}],[\Cr^{\alpha}])$, we consider the subspace $M'$ of pairs $([\gamma_0],[\boldsymbol{\breve{\gamma}}])$ such that
\begin{itemize}
\item $[\gamma_0]\in M(X^*,[\Cr_0])_P$;
\item $[\boldsymbol{\breve{\gamma}}^{\alpha}]\in M^{\alpha}$;
\item for each $\alpha$, the evaluations $\ev^{\alpha}[\gamma_0]$ and $\ev_-[\boldsymbol{\breve{\gamma}}^{\alpha}]$ agree.
\end{itemize}
The space $M'$ has a natural structure of smooth manifold (with boundary, if $P$ has boundary) induced by its fibered product description, as the evaluation maps are transverse by the regularity assumption. These subspaces define a decomposition of the space in smooth manifolds, and this defines a structure of stratified space, as the next result states.
For a typical element
\begin{equation*}
([\gamma_0],[\boldsymbol{\breve{\gamma}}])\in M_{z_0}(X^*,[\Cr_0])_P\times \prod_{\alpha} \breve{M}^+_{z^{\alpha}}([\Cr_0^{\alpha}],[\Cr^{\alpha}])
\end{equation*}
we write $n^{\alpha}$ for the number of components of $[\boldsymbol{\breve{\gamma}}^{\alpha}]$, and denote its $i$th component by $[{\breve{\gamma}}^{\alpha}_i]$

\begin{prop}\label{cod1cob}
Suppose a regular family of perturbations and metrics parametrized by $P$ is fixed. If $M_z(X^*,[\Cr])_P$ contains irreducible solutions and has dimension $d$, then $M^+_z(X^*,[\Cr])_P$ is a $d$-dimensional space stratified by manifolds, with top stratum the irreducible part of $M_z(X^*,[\Cr])_P$. The $(d-1)$-dimensional stratum in $M^+_z(X^*,[\Cr])_P$ consists of elements of the following types.
\begin{itemize}
\item The elements with $n^{\alpha}=1$ for exactly one index $\alpha=\alpha_*$ and all the others $n^{\alpha}$ zero. In this case neither $[\gamma^{\alpha_*}_1]$ or $[\gamma_0]$ are boundary obstructed.
\item The elements with $n^{\alpha}=2$ for exactly one index $\alpha=\alpha_*$ and all the others $n^{\alpha}$ zero. In this case, $[\gamma^{\alpha_0}_1]$ is boundary obstructed but $[\gamma^{\alpha_0}_2]$ and $[\gamma_0]$ are not.
\item The elements with $[\gamma_0]$ boundary obstructed of corank $c$. In this case $n^{\alpha}=1$ for exactly $c+1$ indices $\alpha$, and $n^{\alpha}=0$ in the other cases. If $n^{\alpha}=1$ then the trajectory $[\breve{\gamma}_1^{\alpha}]$ is not boundary obstructed.
\item The unbroken reducible solutions, if the moduli space contains both reducibles and irreducibles.
\item The unbroken irreducible solutions lying over $\partial P$, if $P$ has boundary.
\end{itemize}
In the first three cases above, if any of the moduli spaces involved contains both reducibles and irreducibles then only the irreducibles contribute to the $(d-1)$-dimensional stratum.\end{prop}

\vspace{0.8cm}
As in the case of broken trajectories, the gluing properties along a stratum involve the construction of local thickenings of the moduli spaces, which have the additional properties that the evaluation maps extend to them. The local structure along the strata is slightly more general than that of Theorem \ref{gluing} and Proposition \ref{evthickening} because of the higher corank boundary obstructed trajectories. The following is Definition $24.7.1$ in \cite{KM}, and is useful to describe the structure along the codimension one strata of the moduli space.
\begin{defn}\label{codimensionc}
Let $N$ be a $d$-dimensional space stratified by manifolds and $M^{d-1}$ a union of $(d-1)$ dimensional components. We say that $N$ has a \textit{codimension-$c$ $\delta$-structure} along $M^{d-1}$ if there is an open set $W\supset M^{d-1}$ a topological embedding $j:W\rightarrow EW$ and a map
\begin{equation*}
\mathbb{S}=(S_1,\dots, S_{c+1}):EW\rightarrow (0,\infty]^{c+1}
\end{equation*}
with the following properties:
\begin{enumerate}
\item the fiber along $\mathbb{\infty}$ is identified with $j(M^{d-1})$ and $\mathbb{S}$ is a topological submersion along it (see Definition \ref{topsub});
\item the subset $j(W)\subset EW$ is the zero set of a continuous map $\delta:EW\rightarrow \Pi^{c}$ where $\Pi^c\subset \mathbb{R}^{c+1}$ is the hyperplace $\{\delta\in \mathbb{R}^{c+1}\mid \sum\delta_i=0\}$;
\item if $e\in EW$ has $S_{i_0}=\infty$ for some index $i_0$, then $\delta_{i_0}\leq 0$ with equality only if $\mathbb{S}(e)=\mathbb{\infty}$;
\item on the subset of $EW$ where all the $S_i$ are finite, $\delta$ is smooth and transverse to zero.
\end{enumerate}
\end{defn}

The main example to have in mind is given by $EW=\{x\mid x_i\geq 0\text{ for all }i\}\subset \mathbb{R}^{c+1}$, $S_i=1/x_i$ and
\begin{equation*}
\delta_i=cx_i-\sum_{j\neq i}x_j.
\end{equation*}
In this case the zero locus of $\delta$ is given by the half-line $W\subset EW$ where all the $x_i$ are equal. With this definition the gluing theorem is then the following.
\begin{prop}\label{gluingX}
Suppose a regular family of perturbations and metrics parametrized by $P$ is fixed. Then the space $M_P(X^*;[\Cr])$ has a codimension-$c$ $\delta$-structure stratum along each codimension one stratum.
\end{prop}
Here we have discussed the simplest possible case, but it is not hard to generalize this result to obtain a statement analogous to that of Theorem \ref{gluing}. Also as usual the case of the reducible moduli spaces is significantly easier.

\chapter{Floer homology for Morse-Bott singularities}

In this chapter we construct monopole Floer homology when the singularities of the perturbed Chern-Simons-Dirac functional are Morse-Bott. There are in literature many different approaches to the definition of the homology on a smooth finite dimensional manifold equipped with a Morse-Bott function. In the present work we will follow that of Fukaya (\cite{Fuk}), which was also introduced for gauge-theoretic purposes (see also Hutchings' lecture notes \cite{Hut}). First, he introduces a variant of the chain complex of a smooth manifold consisting of smooth maps from simplicial complexes which are transverse to a given countable collection of smooth maps. Then he defines the chain complex for the original manifold as the direct sum of the modified chain complexes of the critical submanifolds where the simplicial complexes are required to be transverse to the evaluation maps of the moduli spaces. The differential is given by the push-forward of these chains via the flow defined through a fiber product construction.
\\
\par
The choice of Fukaya's construction is mainly motivated by the fact that the transversality that we can generically achieve is not sufficient to pursue other constructions which require the evaluation maps to be submersions. Furthermore our compactified moduli spaces are not smooth compact manifolds with corners. Also, the boundary obstructedness phenomena makes other approaches unfeasible. On the other hand, Fukaya's approach can be readily adapted to our setting, and that is the content of the first two section of the present chapter. In Section $1$ we introduce the notion of \textit{abstract $\delta$-chain}, and construct a variant of the singular homology of a smooth manifold $M$ obtained by considering smooth maps from stratified spaces with nice properties into $M$ which are transverse to a given collection of smooth maps. In Section $2$ we use these in order to define the Morse-Bott chain complexes associated to a three manifold, and prove the basic properties of the invariants following Chapter $22$ of the book. In Section $3$ we construct the maps induced by a cobordism and use these to prove the invariance of the homology, which shows among the other things that the new approach gives rise to the same invariants defined in the book. It is worth noting that our definition of the module structure will follow the construction in \cite{Blo1} which fits our problem better. 
\\
\par
From now on we will work over $\ztwo$, the field with two elements.

\vspace{1.5cm}

\section{Homology of smooth manifolds via stratified spaces}

In this section we discuss an alternative construction of the homology of a smooth manifold following the treatment of \cite{Lip}. The next definition is rather awkward but should not surprise the reader, as it is made to fit the results of Theorem \ref{gluing} and Proposition \ref{gluingX} in the previous chapter.

\begin{defn}\label{abstractgeom}
A topological space $N^d$ is a \textit{$d$-dimensional abstract $\delta$-chain} if it is a stratified space
\begin{equation*}
N^d\supset N^{d-1}\supset\cdots \supset N^0\supset N^{-1}=\emptyset
\end{equation*}
of dimension $d$ with the following additional structures. We are given a finite partition of each stratum
\begin{equation*}
N^e\setminus N^{e-1}=\coprod_{i=1}^{m_e} M^e_{i}
\end{equation*}
for each $e=0,\dots d$, and we call the closure of each $M^e_i$ an $e$ dimensional \textit{face}. We denote the top stratum of a face $\Delta$ by $\mathring{\Delta}$. The set of faces satisfies the following combinatorial property: whenever a codimension $e$ face $\Delta''$ is contained in a codimension $e-2$ face $\Delta$, there are exactly two codimension $e-1$ faces contained in $\Delta$ that contain $\Delta''$.
\par
Furthermore, each pair of faces $\Delta'\subset \Delta$ has an associated finite set $N(\Delta,\Delta')$, together with a subset $O(\Delta,\Delta')\subset N(\Delta,\Delta')$, where the $O$ is for obstructed, and their local structure satisfy the following properties. There is an open neighborhood of $\breve{W}(\Delta, \Delta')$ of $\mathring{\Delta}'$ inside $\Delta$ together with a topological embedding $j$ inside a space $E\breve{W}(\Delta, \Delta')$ endowed with a topological submersion
\begin{equation*}
\mathbf{S}:E\breve{W}(\Delta, \Delta')\rightarrow (0,\infty]^{N(\Delta,\Delta')},
\end{equation*}
called the \textit{local thickening}, such that:
\begin{enumerate}
\item the map $\mathbf{S}$ is a topological submersion along the fiber over $\boldsymbol{\infty}$, and this is identified with $\mathring{\Delta}'$;
\item the image $j(\breve{W})\subset E\breve{W}$ is the zero set of a map
\begin{equation*}
\delta:E\breve{W}\rightarrow \R^{O(\Delta,\Delta')}
\end{equation*}
vanishing along the fiber of $\mathbf{S}$ over $\boldsymbol{\infty}$;
\item calling $\breve{W}^o\subset\breve{W}$ and $E\breve{W}^o\subset E\breve{W}$ the subsets where none of the components of $\mathbf{S}$ is infinite, the restriction of $j$ to $\breve{W}^o$ is a smooth embedding, and the restriction of $\delta$ to $E\breve{W}^o$ is transverse to zero.
\end{enumerate}

This collection of local thickenings is \textit{compatible} in the following sense. Whenever we have three faces $\Delta''\subset \Delta'\subset \Delta$, we have a canonical inclusion
\begin{equation*}
N(\Delta',\Delta'')\hookrightarrow N(\Delta,\Delta'')
\end{equation*}
with identifications
\begin{equation*}
O(\Delta',\Delta'')=O(\Delta,\Delta'')\cap N(\Delta',\Delta'')
\end{equation*}
and
\begin{equation*}
(0,\infty]^{N((\Delta',\Delta'')}=\{ (x_1,\dots, x_{N(\Delta,\Delta'')})\mid x_{\alpha}=\infty\text{ if } \alpha\not\in N(\Delta',\Delta'')\}\subset  (0,\infty]^{N((\Delta,\Delta'')}
\end{equation*}
so that we also have the identifications
\begin{align*}
E\breve{W}(\Delta',\Delta'')&\equiv E\breve{W}(\Delta,\Delta'')\cap\mathbf{S}^{-1}\left((0,\infty]^{N(\Delta',\Delta'')}\times \{\infty\}\right)\\
\mathbf{S}(\Delta',\Delta'')&\equiv\mathbf{S}(\Delta,\Delta'')\lvert_{E\breve{W}(\Delta',\Delta'')}: E\breve{W}(\Delta',\Delta'')\rightarrow (0,\infty]^{N(\Delta',\Delta'')}\\
\delta(\Delta',\Delta'')&\equiv\delta(\Delta,\Delta'')\lvert_{E\breve{W}(\Delta',\Delta'')}: E\breve{W}(\Delta',\Delta'')\rightarrow \R^{O(\Delta',\Delta'')}.
\end{align*}
In the last line we are implicitly stating that the restriction of $\delta(\Delta,\Delta'')$ has image contained in the subspace corresponding to $\R^{O(\Delta',\Delta'')}$ under the identifications above. Finally, whenever $\Delta'\subset\Delta$ has codimension one, for each face $\Delta$ the data above defines a codimension-$c$ $\delta$-structure along $\Delta'$ in the sense of Definition \ref{codimensionc} of Chapter $2$, in the sense that there is an identification of $\R^{O(\Delta,\Delta')}$ with the subspace $\Pi^c\subset \mathbb{R}^{c+1}$.
\end{defn}

Thinking about our moduli spaces the set $N(\Delta,\Delta')$ describes the parameters along which the trajectories in $\Delta$ can break in order to become trajectories in $\Delta'$, and the subset $O(\Delta,\Delta')$ keeps track of which of these are boundary obstructed.

\begin{example}
It is clear that any manifold with corners is an abstract $\delta$-chain by taking the thickening of the neighborhood $\breve{W}$ to be the neighborhood itself. In this case, each subset $O(\Delta,\Delta')$ is empty. Also, the disjoint union of abstract $\delta$-chains (with the obvious decomposition in faces) and each face of an abstract $\delta$-chain are again abstract $\delta$-chains. For a more interesting example any space of broken unparametrized trajectories $\breve{M}^+_z([\Cr_-],[\Cr_+])$ is an abstract $\delta$-chain, where the partition of the strata is given by fixing the resting submanifolds and the relative homotopy classes of the components. The only part which does not follow from Theorem \ref{gluing} and Corollary \ref{codimension1} is the combinatorial condition on the faces, which can be easily checked case by case. On the other hand, it is important to notice that there is not in general any similar combinatorial property for higher codimension faces.
\end{example}
\begin{remark}
We will consider two thickenings $E\breve{W}_1$ and $E\breve{W}_2$ of two neighborhoods $\breve{W}_1$ and $\breve{W}_2$ of a face to be \textit{equivalent} if they coincide (up to isomorphism respecting all the given structures) in a smaller neighborhood contained in both. In other words, we will always consider germs of thickenings. 
\end{remark}

\vspace{0.5cm}

The following result is the version of Stokes' theorem that we will need, and its proof follows with no modifications as in Section $21.3$ in the book. It is a simple generalization of the fact that the boundary of a compact $1$-dimensional manifold consists of an even number of points.

\begin{prop}\label{stokes}
Let $\Delta$ be a $1$-dimensional abstract $\delta$-chain. Then the union of its zero dimensional faces consists of an even number of points.
\end{prop}
We also have the following definition.
\begin{defn}\label{isomdelta}
A homeomorphism $\varphi:\Delta_1\rightarrow\Delta_2$ between two abstract $\delta$-chains is an \textit{isomorphism} if the following hold:
\begin{itemize}
\item it maps strata to strata, and each of these maps is a diffeomorphisms;
\item it induces a bijection between faces;
\item given faces $\Delta_1''\subset\Delta_1'$ and $\Delta_2''\subset\Delta_2'$ such that $\varphi(\Delta_1')=\Delta_2'$ and $\varphi(\Delta_1'')=\Delta_2''$, the homeomorphism
\begin{equation*}
\varphi\lvert_{\breve{W}(\Delta_1',\Delta_1'')}:\breve{W}(\Delta_1',\Delta_1'')\rightarrow\breve{W}(\Delta_2",\Delta_2'')
\end{equation*}
extends to a homeomorphism
\begin{equation*}
E\varphi\lvert_{E\breve{W}(\Delta_1',\Delta_1'')}:E\breve{W}(\Delta_1',\Delta_1'')\rightarrow E\breve{W}(\Delta_2',\Delta_2'')
\end{equation*}
commuting with the maps $\mathbf{S}$ and $\delta$ and is a diffeomorphism in a neighborhood of $j(\breve{W}^o)\subset E\breve{W}^o$.
\end{itemize}
As in the previous remark, we are only interested in the germ of the extensions of the homeomorphism $\varphi$ to the thickening.
\end{defn}

\vspace{0.8cm}
To define homology, we need to define maps from a given abstract $\delta$-chain inside a smooth manifold. This is done in order to fit the properties of evaluation maps, following Proposition \ref{evthickening} in Chapter $2$.

\begin{defn}\label{geomchain}
Let $X$ be a (possibly non compact) smooth manifold without boundary. A \textit{$\delta$-chain} in $X$ is a pair $\sigma=(\Delta,f)$ where
\begin{itemize}
\item $\Delta$ is an abstract $\delta$-chain, and $f:\Delta\rightarrow X$ is a continuous map;
\item the restriction of $f$ to each stratum of $\Delta$ is a smooth map;
\item for each pair of faces $\Delta'\subset \Delta''$, the map extends to a continuous map $Ef(\Delta',\Delta'')$ on the local thickening $E\breve{W}(\Delta',\Delta'')$ which is smooth in a neighborhood of $\breve{W}(\Delta,\Delta'')$;
\item the collection of extensions to the local thickenings is compatible, in the sense that for every triple $\Delta\supset\Delta'\supset\Delta''$ the map $Ef(\Delta',\Delta'')$ is the restriction of $Ef(\Delta,\Delta'')$ under the identification of Definition \ref{abstractgeom}.
\end{itemize}
As before, we consider the germ of the extensions of the map to the thickening.
We say that two $\delta$-chains $\sigma=(\Delta,f)$ and $\sigma'=(\Delta',f')$ are \textit{equivalent} if there exists an isomorphism $\varphi:\Delta\rightarrow \Delta'$ such that $f'\circ \varphi=f$. We denote the isomorphism class of $\sigma$ by $[\sigma]=[\Delta,f]$.
\end{defn}

We also introduce the notion of transversality in our context.

\begin{defn}
We say that two $\delta$-chains $\sigma_1=(\Delta_1,f_1)$ and $\sigma_2=(\Delta_2,f_2)$ are \textit{transverse} if the following hold:
\begin{itemize}
\item the restrictions to each pair of strata $f_1\lvert_{M_1}$ and $f_2\lvert_{M_2}$ are transverse smooth maps;
\item for each pair of faces $\Delta_1'\supset \Delta_1''$ and $\Delta_2'\supset \Delta_2''$, the extensions $Ef_1(\Delta_1',\Delta_1'')$ and $Ef_2(\Delta_2',\Delta_2'')$ are transverse in a neighborhood of $\breve{W}_1^o\times \breve{W}_2^o$.
\end{itemize}
We denote their \textit{fibered product} as $\sigma_1\times\sigma_2$.
\end{defn}

We have the following easy result.
\begin{lemma}\label{stratprod}
The fibered product $\sigma_1\times\sigma_2$ has a natural structure of a $\delta$-chain in $X$.
\end{lemma}

\begin{proof}
We show that the fibered product is an abstract $\delta$-chain. The strata of the fibered product are the fibered products of the strata on each factor, and they inherit a natural smooth structure because of the transversality hypothesis. Similarly, the local thickenings of the fibered product are the fibered products of the local thickenings, and the corresponding sets of components are obtained by concatenation. It is then easy to prove that the desired property hold. For example, the map
\begin{equation*}
(\delta_1,\delta_2): E\breve{W}_1\times_X E\breve{W}_2\rightarrow \R^{O_1}\times \R^{O_2}
\end{equation*} 
is transverse to zero in a neighborhood of  $\breve{W}_1^o\times_X\breve{W}_2^o$ because the maps $Ef_1$ and $Ef_2$ are extensions of maps which are already transverse.
\end{proof}

\vspace{0.8cm}

We are now ready to define a variant of the singular chain complex of a smooth (possibly non compact) manifold $X$. Suppose we are given a \textit{countable} collection of pairs
\begin{equation*}
\mathcal{F}=\left\{\sigma_{\alpha}=(\Delta_{\alpha},f_{\alpha})\right\}
\end{equation*}
of $\delta$-chains in $X$.
A \textit{geometric $\mathcal{F}$-transverse chain} of $X$ of dimension $d$ is a $d$-dimensional $\delta$-chain $\sigma$ which is transverse to all the chains in the family $\mathcal{F}$.
We define $\tilde{C}^{\mathcal{F}}_d(X)$ to be the $\ztwo$-vector space generated by all geometric transverse chains of dimension $d$ \textit{up to isomorphism} after we quotient out by the relations
\begin{equation*}
(\Delta, f)+(\Delta',f')\sim (\Delta \amalg \Delta', f\amalg f').
\end{equation*}
\begin{remark}
Notice that this definition has some set theoretic issues, as the collection of all abstract $\delta$-chains is not a set. Nevertheless we can restrict ourselves to consider a smaller collection which form a set. For example for this section we can consider all the manifolds with corners contained in a fixed Hilbert space (with the induced smooth structure). In the rest of the work, we can consider a slightly bigger family closed under fibered products with the compactified moduli spaces of trajectories.
\end{remark}
We denote its elements by $[\sigma]=[\Delta,f]$. We can then define the linear map
\begin{IEEEeqnarray*}{c}
\tilde{\partial}: \tilde{C}^{\mathcal{F}}_d(X)\rightarrow \tilde{C}^{\mathcal{F}}_{d-1}(X)\\
{[}\Delta,f]\mapsto \sum_{\Delta'} [\Delta', f\lvert_{\Delta'}],
\end{IEEEeqnarray*}
where the sum is taken over all codimension one faces $\Delta'\subset \Delta$. This map well defined because the restriction of $f$ to each face is still a smooth transverse map by definition. It is clear from the definition of geometric stratified space (and in particular the combinatorial condition on the set of faces) that $\partial^2$ is zero, hence the pair $(\tilde{C}^{\mathcal{F}}_*(X),\partial)$ is a chain complex. The key definition from \cite{Lip} is the following.

\begin{defn}\label{smallchain}
A $d$-dimensional chain $[\sigma]\in \tilde{C}^{\mathcal{F}}_d(X)$ is called $\mathcal{F}$-\textit{small} if it has a representative $\amalg (\Delta_i,f_i)$ such that the subset $\bigcup f_i(\Delta_i)$ is contained in the image $f(\Delta)$ of a $\delta$-chain of dimension $j<d$ which is transverse to $\mathcal{F}$. We say that a chain $[\sigma]$ is \textit{negligible} if both $[\sigma]$ and $\partial [\sigma]$ are small.

\end{defn}

For example, the chain represented by the only map to the point
\begin{equation*}
[0,1]\rightarrow \ast
\end{equation*}
is negligible. Define the subspace
\begin{equation*}
N_d^{\mathcal{F}}(X)\subset \tilde{C}^{\mathcal{F}}_d(X)
\end{equation*}
as the subspace generated by all negligible chains. As $\partial^2[\sigma]=0$ is clearly small, we have that the boundary of a negligible chain is again negligible. We define then $(C^{\mathcal{F}}_*(X),\partial)$ to be the quotient of $(\tilde{C}^{\mathcal{F}}_*(X),\partial)$ by the subcomplex generated by negligible chains, and denote its homology by $H_*^{\mathcal{F}}(X)$.
\begin{remark}
It is useful to notice that $C^{\mathcal{F}}_{k}(X)$ is trivial for $k\geq \mathrm{dim}(X)+2$.
\end{remark}

The following is the main result of the present section.

\begin{prop}\label{isomhom}
The homology $H_*^{\mathcal{F}}(X)$ is canonically isomorphic to the singular homology $H_*(X;\ztwo)$.
\end{prop}

Before proving this result we state the main transversality result.
\begin{lemma}\label{transext}
Suppose we are given a countable family of $\delta$-chains $\mathcal{F}$. Given any $\delta$-chain $(\Delta,f)$, a natural number $k$ and $\varepsilon>0$ there exists another $\delta$-chain $(\Delta,f')$ which transverse to all the $\delta$-chains in the family $\mathcal{F}$ and is $\varepsilon$-close to the original one in the $C^k$ topology on each stratum. The map $f'$ is homotopic to $f$, and there is a chain $(\Delta\times [0,1], F)$ such that the restrictions of $F$ to $\Delta\times\{0\}$ and $\Delta\times\{1\}$ are respectively $f$ and $f'$. Finally, if $f$ is $\mathcal{F}$-small (negligible), we can choose $f'$ to be also $\mathcal{F}$-small (negligible).
\end{lemma}
\begin{proof}
The only tricky point in the result is to preserve smallness (negligibility) in the small perturbation. Choose a smooth map
\begin{equation*}
F: X\times P\rightarrow X
\end{equation*}
where $(P,p_0)$ is a pointed smooth connected manifold such that
\begin{itemize}
\item each $x$ in $X$ the differential of the map $F(x,-)$ defined on $P$ is a submersion;
\item the map $F(-,p_0)$ is the identity.
\end{itemize}
Given a $\delta$-chain $(\Delta, f)$ we can consider the family of chains with values in $X$ parametrized by $p$ in $P$ given by $(\Delta, F(f(-),p))$. If the original chain was small (negligible) then each of these is small (negligible), and the result follows by a standard application Sard's theorem.
\end{proof}

\begin{remark}\label{transvhom}
If we suppose that our original chain $(\Delta,f)$ was already transverse, the proof shows that we can also arrange the chain $(\Delta\times[0,1],F)$ to be transverse by choosing a generic path in the manifold $P$.
\end{remark}

\vspace{0.8cm}
\begin{proof}[Proof of Proposition \ref{isomhom}]
We first consider the case in which $\mathcal{F}=\emptyset$. Following Chapter $4$ \cite{Sch}, we just need to show that the homology groups satisfy the classical Eilenberg-Steenrod axioms for homology for a restricted family of pairs of manifolds. We call a pair of smooth manifolds $(X,A)$ \textit{admissible} if $X$ is without boundary and $A$ is either a closed submanifold or a codimension zero submanifold whose boundary is a closed submanifold. It is immediate to extend the definition of the homology to the relative case, obtaining the group $H_*^{\emptyset}(X,A)$ as the homology of the quotient complex
$C^{\emptyset}_*(X)/C^{\emptyset}_*(A)$. It is also clear that a smooth map of pairs
\begin{equation*}
\varphi:(X,A)\rightarrow (Y,B)
\end{equation*}
induces a map in homology
\begin{equation*}
\varphi_*: H_*^{\emptyset}(X,A)\rightarrow H_*^{\emptyset}(Y,B)
\end{equation*}
by composition.
If fact, if $[\sigma]$ is small, so its image is contained in $f(\Delta)$ for a smaller dimensional $\delta$-chain $[\Delta,f]$, then the image of $\varphi_*[\sigma]$ is contained in $\varphi f(\Delta)$, hence is also small.
The desired result then follows once we prove that the functor $H^{\emptyset}_*$ satisfies the following axioms:
\begin{itemize}
\item the existence of a natural, long exact homology sequence for the pair;
\item the homotopy invariance;
\item the invariance under excision;
\item the dimension axiom.
\end{itemize}
The proof of the first two axioms is straightforward. To prove the dimension axiom, we need to show that every chain of dimension bigger or equal than one with values in a space consisting of a single point is negligible. This is obvious if the dimension is at least two, and in the case of a one dimensional chain it follows from Proposition \ref{stokes}. In order to prove the excision invariance we need to show that for any collection $\mathcal{U}=\{U_j\}$ of subspaces of $X$ such that their interiors form an open cover of $X$ the inclusion
\begin{equation*}
C^{\emptyset, \mathcal{U}}_*(X)\hookrightarrow C^{\emptyset}_*(X)
\end{equation*}
induces an isomorphism in homology, where the left hand side is the subspace generated by $\delta$-chain whose image is contained in some of the $U_j$. Given any $\delta$-cycle $\sigma=(\Delta,f)$, by applying Lemma \ref{transext} we can find by compactness a finite collection of closed balls $\{B_i\}$ which cover the image of $f(\Delta)$ such that the submanifolds $\{\partial B_i\}$ and $\sigma$ are transverse, and each $\{B_i\}$ is contained in the interior part of some set of $\mathcal{U}$. Setting
\begin{equation*}
\bigcup B_i\setminus\bigcup\partial  B_i=\coprod D_j
\end{equation*}
we have that
\begin{equation*}
\sum_j [f^{-1}(\bar{D}_j), f\lvert_{f^{-1}(\bar{D}_j)}]
\end{equation*}
is a cycle in $C^{\emptyset, \mathcal{U}}_*(X)$ whose image in $C^{\emptyset}_*(X)$ is homologous to $\sigma$ via the (subdivision of) $\sigma\times I$, hence the induced map in homology is surjective.
Suppose now $\sigma$ is a $\delta$-cycle in $C^{\emptyset, \mathcal{U}}_*(X)$ which is the boundary of a chain $\tau$ in $C^{\emptyset}_*(X)$. Using Lemma \ref{transext} as we did before we can perturb it to a new chain $\tilde{\tau}$ in $C^{\emptyset, \mathcal{U}}_*(X)$. Furthermore $\partial \tilde{\tau}$ is a cycle homologous to $\sigma$ in $C^{\emptyset, \mathcal{U}}_*(X)$ for perturbation small enough, so also injectivity follows.
\par
Finally the case of a non empty countable family $\mathcal{F}$ follows from the fact that the natural inclusion
\begin{equation*}
C^{\mathcal{F}}_*(X)\hookrightarrow C^{\emptyset}_*(X)
\end{equation*}
induces an isomorphism in homology. This is showed as in the proof of the excision property above by applying Lemma \ref{transext}, with the additional observation in Remark \ref{transvhom} in order to show injectivity. 
\end{proof}

\begin{remark}
In general, given any two homology theories satisfying the Eilenberg-Steenrod axioms, any isomorphism between the homologies of the point will give rise to a natural equivalence between the homology theories. In our case we have a canonical isomorphism, because the only automorphism of $\ztwo$ is the identity.
\end{remark}

\vspace{0.8cm}
Cohomology is defined in an analogous way, by considering non necessarily compact abstract $\delta$-chains $\Delta$ and \textit{proper} maps $f$ with values in our smooth manifold $X$. In particular, we define $\tilde{C}^q_{\mathcal{F}}(X)$ as the space generated by $\mathrm{dim}X-q$ dimensional such objects. The boundary of a cochain is the boundary of the respective abstract $\delta$-chain with the restriction of the maps. We can similarly define the notion of $\mathcal{F}$-small cochain, and define the cochain complex $({C}^*_{\mathcal{F}}(X), \delta)$ and its homology $H^*_{\mathcal{F}}(X)$. The analogue of Proposition \ref{isomhom} says that this is the usual singular cohomology of the space, and with this approach Poincar\'e duality for compact manifolds is tautological.
\par
The most important feature of this construction is that one can define an \textit{intersection pairing} between homology and cohomology classes
\begin{equation}\label{univcoeff}
H_q^{\mathcal{F}}(X)\otimes H^q_{\mathcal{F}'}(X)\rightarrow \ztwo.
\end{equation}
Given a homology and a cohomology class, one can represent them by a chain and a cochain transverse to each other (using Proposition \ref{transext}), and define their pairing as their intersection number. This is well defined because of Proposition \ref{stokes}.
\begin{remark}
In general one can define the intersection between classes of different dimensions, by requiring extra transversality conditions for small and negligible faces (which are vacuous in the context above), see \cite{Lip}.
\end{remark}

The following is the analogue of the universal coefficient theorem.
\begin{prop}\label{intpairing}
The intersection pairing \ref{univcoeff} is perfect.
\end{prop}
The proof of this fact can be achieved in a way completely analogous to the proof of Poincar\'e duality for the singular theory by defining a cap product (with the simplices of a fixed triangulation of the manifold). As we will need only the case in which $X$ is a point or the real projective plane (which can be easily worked out by hand), we will not spell out the details of the proof.

\vspace{1.5cm}

\section{Floer homology}

In this section we define monopole Floer homology using the transverse $\delta$-chains discussed in the previous section. The whole construction of the invariants carries over as in Chapter $22$ of the book with only minor changes, leading to \textit{a priori} proofs of the basic properties. In particular, we will not rely on the isomorphism with the Kronheimer and Mrowka's invariants (except for the duality results), which will be discussed together with the functoriality properties of the invariants later in the chapter. This will be useful in the next chapter, when we will construct the $\Pin$-theory. We start with a definition.
\begin{defn}
A Morse-Bott perturbation $\q_0$ is \textit{admissible} if it is regular in the sense of Definition \ref{regularpert} in Chapter $2$ and in the case $c_1(\spin)$ is not torsion there are no reducible critical points. 
\end{defn}
Notice that we have implicitly fixed a Riemannian metric $g$ on $Y$. Given an admissible perturbation $\q$, we will define three versions of the Floer homology groups
\begin{equation*}
\HMt_*(Y,\spin),\qquad \HMf_*(Y,\spin),\qquad \HMb_*(Y,\spin).
\end{equation*}
As the notation suggests, these groups will not depend on the choice of the Riemannian metric $g$ and the perturbation $\q$, but we will postpone this result to the next section. The construction closely follows Fukaya's approach (\cite{Fuk}). Let $\mathsf{C}\subset \Bs_k(Y,\spin)$ denote the set of critical submanifolds of the blow-up gradient of the Chern-Simons-Dirac functional. This can be written as a disjoint union
\begin{equation*}
\mathsf{C}^o\cup\mathsf{C}^s\cup \mathsf{C}^u,
\end{equation*}
consisting respectively of irreducible, boundary-stable and boundary-unstable critical submanifolds. We have the countable family $\mathcal{F}$ of $\delta$-chains
\begin{equation*}
\left(M^+_z([\Cr],[\Cr']),\mathrm{ev}_-\right)
\end{equation*}
for each pair of critical submanifolds $[\Cr],[\Cr']$ and relative homotopy class $z$. We define the vector spaces over $\ztwo$ given by the direct sum of the chain complexes of the critical submanifolds (defined in the previous section)
\begin{IEEEeqnarray*}{c}
C^o=\bigoplus_{[\Cr]\in \mathsf{C}^o} C^{\mathcal{F}}_*([\Cr])\\
C^s=\bigoplus_{[\Cr]\in \mathsf{C}^s} C^{\mathcal{F}}_*([\Cr]) \\
C^u=\bigoplus_{[\Cr]\in \mathsf{C}^u} C^{\mathcal{F}}_*([\Cr])
\end{IEEEeqnarray*}
where the transversality condition is with respect to the family $\mathcal{F}$ defined above, and we set
\begin{IEEEeqnarray*}{c}
\check{C}=C^o\oplus C^s\\
\hat{C}=C^o\oplus C^u \\
\bar{C}=C^s\oplus C^u.
\end{IEEEeqnarray*}
\vspace{0.5cm}
As in the Morse case, the vector spaces $\check{C},\hat{C}$ and $\bar{C}$ have a grading with values in a set $\mathbb{J}$ with an action by $\Z$, which is defined as follows (see Section $22.3$ in the book for details). Given an interval $I=[t_1,t_2]$, configurations $\acr_1,\acr_2\in\Cs_k(Y,\spin)$ and perturbations $\q_1,\q_2\in\mathcal{P}$, we can consider the space $\mathcal{C}$ of pairs $(\gamma,\mathfrak{p})$ where:
\begin{itemize}
\item $\gamma\in \Ct_k(I\times Y)$ is a configurations such that the restriction to $\{t_i\}\times Y$ is gauge-equivalent to $\acr_i$ for $i=1,2$;
\item $\mathfrak{p}$ is a continuous path in the Banach space $\mathcal{P}$ with $\mathfrak{p}(t_i)=\q_1$ for $i=1,2$.
\end{itemize}
Given such a pair, we can consider the operator
\begin{IEEEeqnarray*}{c}
P_{\gamma,\mathfrak{p}}=(Q_{\gamma,\mathfrak{p}},-\Pi^+_1,\Pi^-_2)\\
\T^{\tau}_{1,\gamma}(I\times Y)\rightarrow \left(\V^{\tau}_{0,\gamma}(I\times Y)\oplus L^2(I\times Y;i\R)\right)\oplus H^+_1\oplus H^-_2
\end{IEEEeqnarray*}
where $Q_{\gamma,\mathfrak{p}}$ is the linearization of the Seiberg-Witten equations together with gauge fixing of Section $4$ in Chapter $2$ (considered on a finite cylinder) and $\Pi^+_1,\Pi^-_2$ are spectral projections (See Section $20.3$ in the book and Section $6$ of Chapter $2$). Such an operator is Fredholm, see Chapter $20$ in the book. For a fixed metric $g$ and spin$^c$ structure $\spin$ on $Y$ we define the grading set as
\begin{equation*}
\mathbb{J}(Y,\spin)=\left(\Bs_k(Y,\spin)\times\mathcal{P}\times \Z)\right)/\sim
\end{equation*}
where we identify $([\acr], \q_1, m)$ and $([\bcr],\q_2, n)$ if there exists $(\gamma,\mathfrak{p})$ connecting $([\acr],\q_1)$ to $([\bcr],\q_2)$ such that the operator $P_{\gamma,\mathfrak{p}}$ has index $n-m$. The map
\begin{equation*}
([\acr],\q,m)\mapsto ([\acr],\q,m+1)
\end{equation*}
descends to the quotient defining a $\Z$ action on $\mathbb{J}(\spin)$. For a fixed Morse-Bott perturbation $\q$, it is clear that if $[\acr]$ and $[\bcr]$ are in the same critical submanifold $[\Cr]$ then we can identify
\begin{equation*}
([\acr], \q, n)\sim ([\bcr],\q, n).
\end{equation*}
For a $\delta$-chain $[\sigma]$ in $C^{\mathcal{F}}_d([\Cr])$ we can then define its grading as
\begin{IEEEeqnarray*}{c}
\mathrm{Gr}[\sigma]=([\acr],\q,d)/\sim \\
\in \mathbb{J}(\spin),
\end{IEEEeqnarray*}
for any choice of a point $[\acr]$ in $[\Cr]$.
By the additivity of the index and grading for any path $z$ joining $[\Cr]$ to another critical submanifold $[\Cr']$ that for every $[\sigma']$ in  $C^{\mathcal{F}}_{d'}([\Cr])$
\begin{equation*}
\mathrm{Gr}[\sigma]=\mathrm{Gr}[\sigma']+\mathrm{gr}_z([\Cr],[\Cr'])+(d-d')\in\mathbb{J}(\spin).
\end{equation*}
For a $\delta$-chain $\sigma$ in a reducible critical submanifold we can also introduce a modified grading by defining
\begin{equation*}
\overline{\mathrm{Gr}}[\sigma]=
\begin{cases}
\mathrm{Gr}[\sigma], & [\sigma]\in\mathsf{C}^s\\
\mathrm{Gr}[\sigma]-1, &[\sigma]\in\mathsf{C}^u.
\end{cases}
\end{equation*}
We can decompose each of $C^o, C^s$ and $C^u$ using the grading $\mathrm{Gr}$ as a direct sum over the components $C^o_j, C^s_j$ and $C^u_j$ generated by critical points of grading $j\in\mathbb{J}(\spin)$, and then define
\begin{align*}
\check{C}_j&=C^o_j\oplus C^s_j \\
\hat{C}_j&= C^o_j\oplus C^u_j\\
\bar{C}_j&=C^s_j\oplus C^u_{j+1}
\end{align*}
(notice that the last subspace is homogeneous for the modified grading).
\begin{remark}Even though there is no distinguished element in $\mathbb{J}(\spin)$, there is a canonical map
\begin{equation*}
\mathbb{J}(\spin)\rightarrow \Z/2\Z
\end{equation*}
which we can use to define a canonical mod two grading, see Section $22.4$ in the book. Also, we can identify the grading set $\mathbb{J}(\spin)$ with a more geometric object, namely the set of homotopy classes of plane distributions on $Y$, see Chapter 28 in the book. In the case $c_1(\spin)$ is torsion, one can also define absolute $\mathbb{Q}$-gradings (see Section 28.3 in the book). We will discuss this case in the next chapter, where this additional structure will turn out to be decisive.
\end{remark}

\vspace{0.5cm}
We now define the differential. Given an $\mathcal{F}$-transverse $\delta$-chain $[\sigma]=[\Delta,f]$ in a critical submanifold $[\Cr]$, for every moduli space $\breve{M}^+_z([\Cr],[\Cr'])$ the fibered product
\begin{equation*}
\sigma\times\left(\breve{M}^+_z([\Cr],[\Cr']),\ev_-\right)
\end{equation*}
is an abstract $\delta$-chain which naturally defines $\delta$-chain in $[\Cr']$ via the evaluation map
\begin{equation*}
\ev_+:\sigma\times\breve{M}^+_z([\Cr],[\Cr'])\rightarrow [\Cr'].
\end{equation*}
Such a map is transverse to all the evaluation maps $\ev_-$ with codomain $[\Cr']$ because $\sigma$ is by definition transverse to all the fibered products of the moduli spaces, so the $\delta$-chain is again $\mathcal{F}$-transverse. For simplicity, we will denote this $\delta$-chain simply by $\sigma\times\breve{M}^+_z([\Cr],[\Cr'])$.
We can then define the operators
\begin{align*}
\partial^o_o&: C^o_{*}\rightarrow C^o_{*}\\
\partial^o_s&: C^o_{*}\rightarrow C^s_{*}\\
\partial^u_s&: C^u_{*}\rightarrow C^s_{*}\\
\partial^u_o&: C^u_{*}\rightarrow C^o_{*}
\end{align*}
defined on a generator $[\sigma]$ of $C^{\mathcal{F}}_*([\Cr])$ as
\begin{align*}
\partial^o_o[\sigma]&=\partial[\sigma] +\sum_{[\Cr']\in \mathsf{C}^o}[\sigma\times\breve{M}^+([\Cr],[\Cr'])]\\
\partial^o_s[\sigma]&= \sum_{[\Cr']\in \mathsf{C}^s}[\sigma\times\breve{M}^+([\Cr],[\Cr'])]\\
\partial^u_s[\sigma]&=\sum_{[\Cr']\in \mathsf{C}^s}[\sigma\times\breve{M}^+([\Cr],[\Cr'])]\\
\partial^u_o[\sigma]&=\sum_{[\Cr']\in \mathsf{C}^o}[\sigma\times\breve{M}^+([\Cr],[\Cr'])]
\end{align*}
where in the first two cases $[\Cr]$ consists of irreducibles while in the last two it consists of boundary unstable critical points. Notice that in the first case we also have the summand corresponding to the differential in the $\delta$-chain complex $\partial$ of the critical manifold introduced in Section $1$.

\begin{lemma}\label{wellmap}
The maps above are well defined.
\end{lemma}
\begin{proof}
As remarked above, the fact that the fibered product is transverse to all the evaluation maps follows from the regularity of the moduli spaces. Also, if $[\sigma]$ is negligible then also the fibered product $[\sigma\times\breve{M}^+_z([\Cr],[\Cr'])]$ is negligible, because of the transversality condition in Definition \ref{smallchain}. Finally, we want to show that all but finitely many simplices $[\sigma\times\breve{M}^+_z([\Cr],[\Cr'])]$ are negligible, so that the sums are finite. Because there are only finitely many critical submanifolds in the blow down, we can focus in the case of reducible critical submanifolds. Suppose $[\sigma\times\breve{M}^+_z([\Cr],[\Cr'])]$ with $[\Cr']$ reducible (and blowing down to $[C]$) is non empty. Then if $[\Cr'']$ is a critical submanifold also blowing down to $[C]$ and corresponding to an eigenvalue negative enough, then Lemma \ref{dimreducible} in Chapter $2$ tells us that $[\sigma\times\breve{M}^+_z([\Cr],[\Cr''])]$ has dimension bigger than $\mathrm{dim}[\Cr'']+2$, hence it is negligible.
\end{proof}

\par
Similarly, in the case when both critical submanifolds are reducible, we can define the operators
\begin{align*}
\bar{\partial}^s_s&: C^s_*\rightarrow C^s_*\\
\bar{\partial}^s_u&: C^s_*\rightarrow C^u_*\\
\bar{\partial}^u_s&: C^u_*\rightarrow C^s_*\\
\bar{\partial}^u_u&: C^u_*\rightarrow C^u_*
\end{align*}
defined on an element $[\sigma]$ of $C^{\mathcal{F}}_*([\Cr])$ as
\begin{align*}
\bar{\partial}^s_s[\sigma]&=\partial[\sigma] +\sum_{[\Cr']\in \mathsf{C}^s}[\sigma\times\breve{M}^{\mathrm{red},+}([\Cr],[\Cr'])]\\
\bar{\partial}^s_u[\sigma]&= \sum_{[\Cr']\in \mathsf{C}^u}[\sigma\times\breve{M}^{\mathrm{red},+}([\Cr],[\Cr'])]\\
\bar{\partial}^u_s[\sigma]&=\sum_{[\Cr']\in \mathsf{C}^s}[\sigma\times\breve{M}^{\mathrm{red},+}([\Cr],[\Cr'])]\\
\bar{\partial}^u_u[\sigma]&=\partial[\sigma] +\sum_{[\Cr']\in \mathsf{C}^u}[\sigma\times\breve{M}^{\mathrm{red},+}([\Cr],[\Cr'])]
\end{align*}
where in the first two cases $[\Cr]$ consists of boundary stable critical points while in the last two it consists of boundary unstable critical points. Notice that the two maps
\begin{equation*}
\partial^u_s,\bar{\partial}^u_s: C^u_*\rightarrow C^s_*
\end{equation*}
involve different moduli spaces. Finally, we define the operators
\begin{IEEEeqnarray*}{c}
\check{\partial}:\check{C}_*\rightarrow \check{C}_*\\
\hat{\partial}:\hat{C}_*\rightarrow \hat{C}_*\\
\bar{\partial}:\bar{C}_*\rightarrow \bar{C}_*
\end{IEEEeqnarray*}
respectively as
\begin{align*}
\check{\partial}=&
\left[\begin{matrix}
\partial^o_o & \partial^u_o\bar{\partial}^s_u \\
\partial^o_s & \bar{\partial}^s_s+\partial^u_s\bar{\partial}^s_u
\end{matrix}\right] \\
\hat{\partial}=&
\left[\begin{matrix}
\partial^o_o & \partial^u_o \\
\bar{\partial}^s_u\partial^o_s & \bar{\partial}^u_u+\bar{\partial}^s_u\partial^u_s
\end{matrix}\right]\\
\bar{\partial}=&
\left[\begin{matrix}
\partial^s_s & \partial^u_s \\
\partial^s_u & \partial^u_u.
\end{matrix}\right]
\end{align*}
These operators have degree $-1$, and are chain maps as stated by the next result.

\begin{prop}\label{chaincomp}
The squares $\bar{\partial}^2, \check{\partial}^2$ and $\hat{\partial}^2$ are zero as operators on $\bar{C}, \check{C}$ and $\hat{C}$.
\end{prop}

\begin{proof}
The proof follows from the characterization of the codimension one strata in Proposition \ref{codim1} in the same way as Proposition $22.1.4$ in the book. We spell out the details for the case of the component of $\check{\partial}^2$ given by
\begin{equation*}
A=\partial^o_o\partial^o_o+\partial^u_o\bar{\partial}^s_u\partial^o_s: C^o\rightarrow C^o.
\end{equation*}
Given an irreducible critical manifold $[\Cr_-]$ and a transverse $\delta$-chain $\sigma$ in $C_*^{\mathcal{F}}([\Cr_-])$, we have that
\begin{align*}
A[\sigma]&=(\partial^{\mathcal{F}})^2[\sigma]\\
&+ \sum_{[\Cr_+]\subset\mathsf{C}^o} 
 \partial^{\mathcal{F}}[\sigma\times \breve{M}^+([\Cr_-],[\Cr_+])]\\
 &+\sum_{[\Cr_+]\subset\mathsf{C}^o} [\partial^{\mathcal{F}}[\sigma]\times \breve{M}^+([\Cr_-],[\Cr_+])]\\
&+\sum_{[\Cr_+]\subset\mathsf{C}^o}\sum_{[\Cr_1]\in\mathsf{C}^o} [\sigma\times \breve{M}^+([\Cr_-],[\Cr_1],[\Cr_+])]\\
&+\sum_{[\Cr_+]\subset\mathsf{C}^o}\sum_{[\Cr_1]\in\mathsf{C}^s}\sum_{[\Cr_2]\in\mathsf{C}^u} [\sigma\times \breve{M}^+([\Cr_-],[\Cr_1],[\Cr_2],[\Cr_+])],
 \end{align*}
 where we used that the strata of the compactified moduli spaces are defined as the fibered products.
It is clear that the term in the first row is zero. On the other hand, the codimension one faces of the  $\delta$-chain $[\sigma\times \breve{M}^+_z([\Cr_-],[\Cr_+])]$ are by definition the fibered product of a codimension one face in one factor and the whole space in the other. Hence by Proposition \ref{codim1} in Chapter $2$ its boundary in $C^{\mathcal{F}}_*([\Cr_+])$ is given by the sum
 \begin{align*}
 &[\partial^{\mathcal{F}}[\sigma]\times \breve{M}^+_z([\Cr_-],[\Cr_+])]\\
+&\sum_{[\Cr_1]\in\mathsf{C}^o} [\sigma\times M_z^+([\Cr_-],[\Cr_1],[\Cr_+])]\\
+&\sum_{[\Cr_1]\in\mathsf{C}^s}\sum_{[\Cr_2]\in\mathsf{C}^u} [\sigma\times M_z^+([\Cr_-],[\Cr_1],[\Cr_2],[\Cr_+])]
 \end{align*}
where the last term involves boundary obstructed moduli spaces. This completes the proof.
\end{proof}

\begin{defn}We define the \textit{monopole Floer homology groups} of $Y$ as the homologies of the three graded chain complexes $(\check{C}_*,\check{\partial}), (\hat{C}_*,\hat{\partial})$ and $(\bar{C}_*,\bar{\partial})$:
\begin{IEEEeqnarray*}{c}
\HMt_*(Y,\spin)=H(\check{C}_*,\check{\partial})\\
\HMf_*(Y,\spin)=H(\hat{C}_*,\hat{\partial})\\
\HMb_*(Y,\spin)=H(\bar{C}_*,\bar{\partial}).
\end{IEEEeqnarray*}
\end{defn}
The choice of metric and perturbation $(g,\mathfrak{p})$ is implicit in our notation, and we will make it explicit when needed.

\vspace{0.8cm}
The rest of the section is dedicated to the study of the basic properties of these objects, which will follow as those in Chapter $22$ in the book in exactly the same way. We expose the main constructions (without going too deep into details) as they will be useful for our purpose in the next chapter. Before doing this, we show that in the case of a non-degenerate perturbation our construction gives the same result.

\begin{lemma}\label{sameconstruction}
Suppose the admissible perturbation $\q_0$ is chosen so that all critical points are non-degenerate. Then the three flavors of monopole Floer homology coincide as graded groups with those defined in the book.
\end{lemma}
\begin{proof}
In the case the perturbation is non degenerate the chain complex is exactly the one defined in the book. Indeed by Proposition \ref{stokes} if the manifold is a point only zero dimensional chains can be not negligible. Hence the chain complex of each critical submanifold consists of a single $\ztwo$ in degree zero, and the moduli spaces of dimension one or higher do not contribute as any fibered product is necessarily negligible.
\end{proof}

\vspace{0.8cm}

\textit{Exact sequences.} The monopole Floer homology groups should be thought as half-infinite-dimensional homology groups of the moduli space of configurations
\begin{equation*}
\left(\Bs_k(Y,\spin),\partial\Bs_k(Y,\spin)\right),
\end{equation*}
and many of the basic properties of the usual homology groups can be performed. We first state the exact sequence for the pair in homology.

\begin{prop}\label{longexact}
For any $(Y,\spin)$, there is an exact sequence
\begin{equation*}
\dots\stackrel{i_*}{\longrightarrow} \HMt_k(Y,\spin)\stackrel{j_*}{\longrightarrow} \HMf_k(Y,\spin)\stackrel{p_*}{\longrightarrow} \HMb_{k-1}(Y,\spin)\stackrel{i_*}{\longrightarrow} \HMt_{k-1}(Y,\spin)\stackrel{j_*}{\longrightarrow}\dots
\end{equation*}
\end{prop}

The maps in the sequence are induced by the chain maps
\begin{equation*}
i:\bar{C}_*\rightarrow \check{C}_*,\qquad j:\check{C}_*\rightarrow \hat{C}_*, \qquad p:\hat{C}_*\rightarrow \bar{C}_*
\end{equation*}
given in components by
\begin{equation*}
i=\left[\begin{matrix}
0 & \partial^u_o \\
1 & \partial^u_s
\end{matrix}\right],\qquad
j=\left[\begin{matrix}
1 & 0 \\
0 & \bar{\partial}^s_u
\end{matrix}\right],\qquad
p=\left[\begin{matrix}
\partial^o_s & \partial^u_s \\
0 & 1
\end{matrix}\right],
\end{equation*}
see Section $22.2$ in the book for details.
\vspace{0.8cm}

\textit{Filtrations.}
The chain complex we have introduced for a Morse-Bott perturbation is much bigger than the usual one for a non degenerate perturbation, as it is for example not necessarily finitely generated in each dimension. When the perturbation satisfies a stronger transversality hypothesis, we can somehow handle the problem better by means of a spectral sequence (see for example \cite{Fuk} or \cite{AB}).

\begin{defn}
Suppose the spin$^c$ structure is torsion. We say that an admissible Morse-Bott perturbation $\q$ is \textit{weakly self-indexing} if for each pair of critical submanifolds $[\Cr_-],[\Cr_+]$ and relative homotopy class $z$ connecting them such that $M_z([\Cr_-],[\Cr_+])$ is not empty we have that 
\begin{equation*}
\gr_z([\Cr_-],[\Cr_+])\geq
\begin{cases}
1 \text{ if the pair is not boundary obstructed,}\\
0 \text{ otherwise.}
\end{cases}
\end{equation*}
\end{defn}

\begin{prop}
Suppose a spin$^c$ structure with $c_1(\spin)$ torsion and a weakly self-indexing Morse-Bott perturbation $\q$ are fixed. Then there is a spectral sequence $E^*_{ji}$ where $i$ is a positive integer and $j$ is an element of $\mathbb{J}(\spin)$ such that
\begin{equation*}
E^1_{ji}=\bigoplus_{\mathrm{Gr}[\Cr]=j, [\Cr]\subset \mathsf{C}^o\cup \mathsf{C}^s} H_i([\Cr];\ztwo);
\end{equation*}
and converges to $\HMt_{i+j}(Y,\spin)$. There are similar spectral sequences for the other flavors.
\end{prop}
\begin{proof}
For a fixed critical submanifold $[\Cr_0]$ there is a natural increasing filtration on the chain complex $\check{C}_*$ given by
\begin{equation*}
\mathcal{F}_k\check{C}_*=\bigoplus_{\gr_z([\Cr],[\Cr_0])\leq k} C_*^{\mathcal{F}}([\Cr]),
\end{equation*}
where the sum is taken over irreducible and boundary stable critical submanifolds. The differential respects this filtration because of the weakly self indexing hypothesis, and the spectral sequence in the proposition is exactly the one induced by this filtration. The spectral sequence converges for the same reason why the differential is well defined, see Lemma \ref{wellmap}.
\end{proof}

\begin{example}\label{S3} We study the case of $S^3$ with the round metric, see Section $22.7$ in the book. Suppose we have chosen any sufficiently small perturbation Morse-Bott perturbation $\q$ so that no irreducible critical point is introduced. The critical submanifolds consist of an infinite sequence of projective spaces all lying over the same critical reducible critical point. By the results of Lemma \ref{dimreducible} in Chapter $2$, it is clear that such a perturbation is weakly self-indexing, and for dimensional reasons the sequence collapses at the first page. Hence, we obtain again the calculation of Section $22.7$ in the book. Notice that by slightly more refined considerations one can show that the the chain complex is actually the direct sum of the single chain complexes, see the proof of Proposition \ref{positivity} in Section $4$.
\end{example}

Although they arise in many examples, not all Morse-Bott perturbations are weakly self-indexing. Indeed, it is not a generic property. To deal with the general case, we introduce another natural filtration (in the case of a torsion spin$^c$ structure) induced by the values of the perturbed Chern-Simons-Dirac functional $\CSd$. Let
these be
\begin{equation*}
\alpha_1<\cdots <\alpha_m,
\end{equation*}
and consider the \textit{energy filtration} given by
\begin{equation*}
\mathcal{G}_k\check{C}_*=\bigoplus_{\CSd([\Cr])\leq \alpha_k} C^{\mathcal{F}}_*([\Cr]),
\end{equation*}
which is well defined because the value of $\CSd$ is always decreasing along a trajectory. Notice that this filtration does not preserve gradings.
\par
This filtration turns in handy when we need to restrict our attention to $\delta$-chains which are transverse to a larger family of evaluation maps, as described in the next definition. 

\begin{defn}\label{compext}
We say that a countable family of $\delta$-chains $\mathcal{F}'$ containing the $\mathcal{F}$, the one given by the evaluation maps of the moduli spaces of broken unparametrized trajectories, is a \textit{compatible extension} if
\begin{itemize}
\item each $\sigma=(\Delta,f)$ in $\mathcal{F}'\setminus \mathcal{F}$ is transverse to all evaluation maps $\ev_+$ of the moduli spaces of trajectories;
\item for each such $\sigma$ and any moduli space $M=\breve{M}^+([\Cr],[\Cr'])$ the fibered product
\begin{equation*}
\left(\sigma\times (X,\ev_+), \ev_-\right)
\end{equation*}
is also an element of $\mathcal{F}'$.
\end{itemize}
\end{defn}
We are especially interested in the case when the family $\mathcal{F}'\setminus \mathcal{F}$ is consists of the moduli spaces on a cobordism equipped with the evaluation on the incoming end. This forms a compatible extension exactly when the perturbation is regular in the sense of Definition \ref{regcob} in Chapter $2$.

\begin{lemma}\label{moretransverse}
Consider compatible extension $\mathcal{F}'\supset\mathcal{F}$. Then the subspace defined as
\begin{equation*}
\check{C}'_*(Y,\spin)=\check{C}_*(Y,\spin)\cap\left(\bigoplus C^{\mathcal{F}'}_*([\Cr])\right).
\end{equation*}
is a subcomplex and the inclusion map is a quasi-isomorphism. The same statement holds for the other versions.
\end{lemma}
\begin{proof}
It is clear from the definition that the this is a subcomplex. To prove that the inclusion in a quasi isomorphism, consider first the case when the spin$^c$ structure is torsion. We can then consider the energy filtration for both chain complexes. The inclusion map $i$ is then a map of filtered chain complexes, and our goal is to show that the induced map on the $E^1$ page
\begin{equation}\label{e1}
i_*: H(\mathcal{G}_j\check{C}'_*/\mathcal{G}_{j-1}\check{C}'_*)\rightarrow H(\mathcal{G}_j\check{C}_*/\mathcal{G}_{j-1}\check{C}_*)
\end{equation}
is an isomorphism. In fact, this implies that the map induced on the $E^{\infty}$ page is an isomorphism, hence so is the map induced in homology. We have the splitting as chain complexes
\begin{equation*}
\mathcal{G}_j\check{C}_*/\mathcal{G}_{j-1}\check{C}_*=\bigoplus_{\CSd([C])=\alpha_j}\left(\bigoplus_{\pi([\Cr])=[C]} C^{\mathcal{F}}_*([\Cr])\right)
\end{equation*}
because we are only considering the trajectories with on which $\CSd$ is constant, hence we just need to study each of the summands. In the irreducible case, $i_*$ is an isomorphism because of Proposition \ref{isomhom}. In the reducible case the chain complexes
\begin{equation*}
\bigoplus_{\pi([\Cr])=[C]} C^{\mathcal{F}'}_*([\Cr])\hookrightarrow\bigoplus_{\pi([\Cr])=[C]} C^{\mathcal{F}}_*([\Cr])
\end{equation*}
have a filtration given by the ordering of the corresponding eigenvalues. The map induced on the first page of the associated spectral sequence is again an isomorphism thanks to Proposition \ref{isomhom}, so the result follows from the same reason as above.
\par
The case of a non-torsion spin$^c$ structure is slightly more complicated because there is not an energy filtration as the Chern-Simons-Dirac functional is circle valued. To tackle this issue, we follow the nice approach of \cite{FS} by constructing an auxiliary chain complex with an absolute $\mathbb{Z}$-grading. Consider any value $\vartheta\in \mathbb{R}/(4\pi^2\mathbb{Z})$ such that there are no critical submanifolds $[\Cr]$ with $\CSd([\Cr])=\vartheta$. We define the chain complex
\begin{equation*}
\left(\check{C}^{\vartheta}_*(Y,\spin),\check{\partial}^{\vartheta}\right)
\end{equation*}
whose underlying vector space is $\check{C}(Y,\spin)$ but the differential only counts moduli spaces consisting of unparametrized broken trajectories $\breve{\boldsymbol\gamma}$ such that $\CSd$ never achieves the value $\vartheta$ along them. There are no differences in the proof of the fact that this is actually a chain complex. Furthermore on such a complex there is a well defined energy filtration, hence we can prove as above that its homology does not depend on the choice of the family $\mathcal{F}'$. The result will follow if we can find a filtration on $(\check{C}_*(Y,\spin),\check{\partial})$ such that the associated spectral sequence is convergent and the $E^1$ page of the is the homology of this modified chain complex.
\par
Let $d\in 2\mathbb{N}^+$ the generator of the image of the evaluation of $c_1(\spin)$ on $H_2(Y,\mathbb{Z})$. For a fixed critical submanifold $[\Cr_0]$ we can define the $\mathbb{Z}$-valued grading
\begin{equation*}
\mathrm{gr}^{\vartheta}([\Cr])=\mathrm{gr}_{z_{\vartheta}}([\Cr],[\Cr_0])\in \mathbb{Z},
\end{equation*}
where $z_{\theta}$ is a relative homotopy class whose image under $\CSd$ avoids $\vartheta$. This is well defined because the quantity
\begin{equation}\label{sumsame}
\mathcal{E}_{\q}(z)+4\pi^2 \mathrm{gr}_z([\Cr],[\Cr'])
\end{equation}
is independent of the homology class, see Lemma $16.4.4$ in the book. With this convention, the relative $\mathbb{Z}/d\mathbb{Z}$ grading of $\check{C}_*(Y,\spin)$ can be lifted to a (not canonical) absolute $\mathbb{Z}$-grading, which we denote by the underlined index. For $l\in\mathbb{Z}/d\mathbb{Z}$ the (increasing) filtration for $s$ congruent to $l$ modulo $d$ we set
\begin{equation*}
\mathcal{G}_s \check{C}_{l}=\bigoplus_{j\geq 0} \check{C}_{\underline{s-jd}}.
\end{equation*}
It follows from the definition of the grading $\mathrm{gr}^{\vartheta}$ and the invariance of the quantiy (\ref{sumsame}) that this is a filtration on the $\mathbb{Z}/d\mathbb{Z}$-graded chain complex, i.e.
\begin{equation*}
\check{\partial}\left(\mathcal{G}_s\check{C}_l\right)\subset \mathcal{G}_{s-1}\check{C}_{l-1}.
\end{equation*}
The $E^0$ page coincides with $(\check{C}^{\vartheta}_*(Y,\spin),\check{\partial}^{\vartheta})$, and the filtration is bounded because of the finiteness properties of the moduli spaces (Proposition \ref{finiteness} in Chapter $2$), so the associated spectral sequence converges.
\end{proof}

\vspace{0.8cm}

\textit{Cohomology and duality.} The main drawback of our approach to the Morse-Bott case is that it is not clear how to define a duality map at the chain level because of the transversality conditions needed to define the intersection product. For this reason, we will define the cohomology tautologically as the homology of $-Y$, i.e. $Y$ with the orientation reversed, and define the pairing via intersection. The fact that this pairing is perfect will follow from the invariance of the homology, as in the non-degenerate case this is simply the universal coefficient theorem.

Recall that a spin$^c$ structure $(S,\rho)$ on $Y$ defines the spin$^c$ structure $(S,-\rho)$ on $-Y$. We can identify the configuration spaces $\Co_k(Y,\spin)$ and $\Co_k(-Y,\spin)$, and if we choose the perturbation $-\q$ on $Y$ we have the relation
\begin{equation*}
\CSd(Y)=-\CSd(-Y).
\end{equation*}
We can then identify the critical submanifolds, with the notion of boundary stable and boundary unstable switched. We also have the canonical identification
\begin{equation}\label{reversor}
M_z(Y;[\Cr_-],[\Cr_+])=M_{-z}(-Y;[\Cr_+],[\Cr_-]).
\end{equation}
There is a map
\begin{align*}
o:\mathbb{J}(-Y,\spin)&\rightarrow \mathbb{J}(Y,\spin)\\
([\acr],n)& \rightarrow \left([\acr],-n-N(\Hess^{\sigma}_{\q,\acr})\right),
\end{align*}
where $N(\Hess^{\sigma}_{\q,\acr})$ is the dimension of the generalized zero eigenspace, which changes the sign of the $\mathbb{Z}$ action, i.e.
\begin{equation}\label{reversegrading}
o(j+n)=o(j)-n.
\end{equation}
We then define the cochain complexes as
\begin{align*}
\check{C}^j(Y,\spin)&=\hat{C}_{o(j)}(-Y,\spin)\\
\hat{C}^j(Y,\spin)&=\check{C}_{o(j)}(-Y,\spin)\\
\bar{C}^j(Y,\spin)&=\bar{C}_{o(j)}(-Y,\spin),
\end{align*}
where the differentials $\check{\delta}, \hat{\delta}$ and $\bar{\delta}$ are the ones induced by this identification. In particular, the relation \ref{reversegrading} implies that these differentials have degree $1$. From this, we can define the \textit{Floer cohomology groups}
\begin{align*}
\HMt^*(Y,\spin)&=H(\check{C}^*,\check{\partial})\\
\HMf^*(Y,\spin)&=H(\hat{C}^*,\hat{\partial})\\
\HMb^*(Y,\spin)&=H(\bar{C}^*,\bar{\partial}).
\end{align*}
It is important to remark that a priori the chain complexes of $Y$ and $-Y$ involve different $\delta$-chains. In fact using the identification in equation (\ref{reversor}) in the former case they are required to be transverse to the maps $\ev_-$ while in the latter case to the maps $\ev_+$.
\par
Because of the definition of the map $o$ it is natural to interpret the groups $\check{C}^*, \hat{C}^*$ and $\bar{C}^*$ as the direct sum the cochain complexes of the critical submanifolds with the natural grading. Indeed, at a critical point $\acr$ the quantity $N(\Hess^{\sigma}_{\q,\acr})$ is the dimension of the critical submanifold to which the point belongs to. Hence via intersection theory we can define a natural pairing
\begin{equation*}
\HMt_j(Y,\spin)\otimes \HMt^j(Y,\spin)\rightarrow \ztwo
\end{equation*}
or, equivalently, a linear map
\begin{equation*}
\HMt_j(Y,\spin)\rightarrow \mathrm{Hom}(\HMt^j(Y,\spin),\ztwo).
\end{equation*}
To see that this is well defined, we need to show that given a cohomology class $[\sigma]$ and a homology class $[\tau]$ we can find a cocycle $\sigma$ in $\check{C}^k$ and a cycle $\tau$ in $\check{C}_k$ representing them which are transverse, so that they have a well defined intersection number. By definition, any cocycle $\sigma$ is transverse to all the evaluation maps $\ev_+$. In particular, we can define a compatible extension $\mathcal{F}'$ for the chain complex $\check{C}_*$ computing the homology (see Definition \ref{compext}) where
\begin{equation*}
\mathcal{F}'\setminus\mathcal{F}=\left\{(\breve{M}^+([\Cr],[\Cr'])\times \sigma, \ev_-)\right\}.
\end{equation*}
The existence of a representative $\tau$ for the homology class which is transverse to $\sigma$ follows then by applying Lemma \ref{moretransverse} to the family $\mathcal{F}'$. A similar argument implies that the intersection number only depends on the classes, and not by the representatives. Finally, this pairing is perfect because of the invariance properties and the universal coefficient theorem in the non-degenerate case.

\vspace{0.8cm}

\textit{Completions.}
Let $G_*$ be an abelian group graded by the set $\mathbb{J}$ equipped with a $\mathbb{Z}$ action. Let $O_{\alpha}$ ($\alpha\in A$) be the set of free $\Z$-orbits in $\mathbb{J}$, and fix an element $j_{\alpha}\in O_{\alpha}$ for each $\alpha$. Consider the subgroups
\begin{equation*}
G_*[n]=\bigoplus_{\alpha}\bigoplus_{m\geq n} G_{j_{\alpha}-m},
\end{equation*}
which form a decreasing filtration of $G_*$. We define the \textit{negative completion} of $G_*$ as the topological group $G_{\bullet}\supset G_*$ obtained by completing with respect to this filtration (which is clearly independent of the choice of the $j_{\alpha}$). We then define the negative completions
\begin{equation*}
\HMt_{\bullet}(Y,\spin),\qquad\HMf_{\bullet}(Y,\spin),\qquad\HMb_{\bullet}(Y,\spin)
\end{equation*}
of the Floer groups defined above. These also satisfy the properties discussed above, and they are natural objects to study when dealing with functoriality properties.

\vspace{1.5cm}
\section{Invariance and functoriality}

The aim of this final section is to construct maps on monopole Floer homology induced by cobordisms. These are the essential ingredient to study invariance and functoriality.
Our definition of the chain complexes is different from the one in the book, and we cannot follow the approach there in order to construct maps induced by a cobordism $X$ and a cohomology class $u^d$ on its configuration space. On the other hand there is an alternative approach (which can be found for example in \cite{Blo1}) leading to the same result which requires the addition of some incoming cylindrical ends of the form
\begin{equation*}
(-\infty,0]\times (S^3)
\end{equation*}
to the given cobordism $W$ from $Y_-$ to $Y_+$. In this case, we will obtain maps
\begin{equation*}
\HMt_{\bullet}(W):\HMt_{\bullet}(Y_-)\otimes \HMf_{\bullet}(S^3)\otimes \dots \HMf_{\bullet}(S^3)\rightarrow \HMt_{\bullet}(Y_+),
\end{equation*}
where each factor still depends on the choice of metric and perturbation. Using these, we will be able to prove invariance and functoriality.
\par
It is important to notice that in general given a manifold with more than one incoming end $Y_1,\dots ,Y_m$ and one outgoing end $Y_+$ there is \textit{not} a well-defined induced map
\begin{equation*}
\HMt_{\bullet}(Y_1)\otimes\dots\otimes \HMt_{\bullet}(Y_m)\rightarrow \HMt_{\bullet}(Y_+),
\end{equation*}
as there are substantial issues with the combinatorics of the codimension one strata. On the other hand, it is shown in \cite{Blo2} that the combinatorics of the degeneration and gluing of moduli spaces work out in order to define the map
\begin{equation*}
\HMt_{\bullet}(Y_1)\otimes \HMf(Y_2)\otimes\dots\otimes \HMf_{\bullet}(Y_m)\rightarrow \HMt_{\bullet}(Y_+)\\
\end{equation*}
and similarly for the map
\begin{equation*}
\HMf_{\bullet}(Y_1)\otimes \HMf(Y_2)\otimes\dots\otimes \HMf_{\bullet}(Y_m)\rightarrow \HMf_{\bullet}(Y_+).
\end{equation*}
In particular, Corollary $4.15$ in that paper provides an extremely general family (with possibly more incoming and outgoing ends at the same time) in which there is a natural induced map. In our case the situation is much more simple due to the absence of irreducible solutions on $S^3$ for the right choice of metric and perturbation (see Example \ref{S3}).

\vspace{0.8cm}
Suppose we are given a cobordism $W$ from $Y_-$ to $Y_+$ and consider a metric on $W$ which is cylindrical near the boundary components. Suppose we are also given a finite number of marked points $\mathbf{p}=\{p_1,\dots, p_m\}$ in the interior. Removing a small ball around the marked point we can think of each marked point $p_i$ to give rise and incoming end $S^3_i$. We also suppose that each of these has a metric with positive scalar curvature (e.g. the round metric), and that the metric is a product near the end. We will denote the manifold obtained by adding cylindrical ends by $W^*_{\mathbf{p}}$. We will refer to the ends corresponding to the marked points as punctures. Choose on each $S^3_i$ a small Morse-Bott perturbation $\q_i$ with only one reducible solution in the blow down and no irreducible solutions. Finally choose as in Section $6$ of Chapter $2$ an additional perturbation $\hat{\mathfrak{p}}$ on the cobordism so that all moduli spaces are regular (see Definition \ref{regcob} in that section).
\par
To define the maps we need to construct the fibered products with the moduli spaces on the cobordisms, hence we need to assume more transversality conditions on the set of $\delta$-chain we are considering. For each $(m+1)$-uple of critical submanifolds in the incoming end
\begin{equation*}
[\Cr_-],[\Cr_1],\dots,[\Cr_m]
\end{equation*}
respectively in $Y_-$ and the puncture $S^3_i$ we have a countable family of smooth maps
\begin{align*}
\mathcal{F}_W&=\{(f_i,\Delta_i)\}\\
f_i: \Delta_i&\rightarrow [\Cr_-]\times \prod_{i=1}^m[\Cr_i],
\end{align*}
where $\Delta_i$ is a moduli space on the cobordism
\begin{equation*}
M^+_z([\Cr_-],[\Cr_1],\dots, [\Cr_m], W^*_{\mathbf{p}}, [\Cr_+])
\end{equation*}
and the map $f_i$ is the product of the evaluation maps $\ev_-$ to each end.
We define the vector space $C_{\bullet}^o(Y_-,\mathbf{p})$ as the subspace of
\begin{equation*}
C_{\bullet}^o(Y_-)\otimes C_{\bullet}^{u}(S^3_1)\otimes\cdots \otimes C_{\bullet}^{u}(S^3_m)
\end{equation*}
generated by tensor products of $\mathcal{F}$-transverse $\delta$-chains $[\sigma_-]\otimes[\sigma_1]\otimes \dots \otimes[\sigma_m]$ where $[\sigma_-]\in C_*^{\mathcal{F}}([\Cr_-])$ and $[\sigma_i]\in C_*^{\mathcal{F}}([\Cr_i])$ such that
\begin{align*}
\sigma_-\times \sigma_1\times\dots\times \sigma_m: \Delta_-\times \Delta_1\times \dots \times \Delta_m\rightarrow [\Cr_-]\times \prod_{i=1}^m[\Cr_i]
\end{align*}
is transverse to all the maps in the family $\mathcal{F}_W$. To simplify the notation, we will set
\begin{align*}
\mathcal{C}&=\left( [\Cr_1],\dots, [\Cr_m]\right)\\
[\boldsymbol{\sigma}]&=[\sigma_1]\otimes\dots\otimes [\sigma_m].
\end{align*}
Similarly, we can define the vector spaces
\begin{align*}
C_{\bullet}^u(Y_-,\mathbf{p})&\subset C_{\bullet}^u(Y_-)\otimes C_{\bullet}^{u}(S^3_1)\otimes\cdots \otimes C_{\bullet}^{u}(S^3_m)\\
C_{\bullet}^s(Y_-,\mathbf{p})&\subset C_{\bullet}^s(Y_-)\otimes C_{\bullet}^{u}(S^3_1)\otimes\cdots \otimes C_{\bullet}^{u}(S^3_m),
\end{align*}
defined by the same transversality hypotheses. We can then define the vector spaces
\begin{align*}
\check{C}_{\bullet}(Y_-,\mathbf{p})&=C_{\bullet}^o(Y_-,\mathbf{p})\oplus C_{\bullet}^s(Y_-,\mathbf{p})\\
\check{C}_{\bullet}(Y_-,\mathbf{p})&=C_{\bullet}^o(Y_-,\mathbf{p})\oplus C_{\bullet}^u(Y_-,\mathbf{p})\\
\check{C}_{\bullet}(Y_-,\mathbf{p})&=C_{\bullet}^s(Y_-,\mathbf{p})\oplus C_{\bullet}^u(Y_-,\mathbf{p}).
\end{align*}
The vector spaces $C_{\bullet}^{u}(S^3_i)$ are the underlying vector spaces of the chain complex $\hat{C}_{\bullet}(S^3_i)$. These complexes associated to the three sphere have a natural $\mathbb{Z}$ valued grading so that the homology is canonically identified with the graded ring $\ztwo[[U]]$, where $U$ has degree $-2$ (see Example \ref{S3}). Notice that this grading differs from the canonical absolute grading, which is obtained by a total shift of $-1$. This implies that $\check{C}_{\bullet}(Y_-,\mathbf{p})$ and its companions have a natural $\mathbb{J}(Y_-)$ grading. The next result follows using the energy filtration as in Lemma \ref{moretransverse} (which deals with the case with no punctures).
\begin{lemma}\label{Yp}
The subspace $\check{C}_*(Y_-,\mathbf{p})$ of $\check{C}_*(Y_-)\otimes \hat{C}_*(S^3_1)\otimes\cdots \otimes \hat{C_*}(S^3_m)$ is a subcomplex, and the inclusion map is a quasi-isomorphism.
\end{lemma}
In particular in many occasions we will need to restrict to a smaller class of $\delta$-chains satysfying additional transversality hypothesis, and this result (and some slight modifications of it) will tell us that this will not effect the final result.
\par
We can define the maps on the completions

\begin{IEEEeqnarray*}{c}
m^o_o:  C^o_{\bullet}(Y_-,\mathbf{p})\rightarrow C^o_{\bullet}(Y_+)\\
m^o_s: C^o_{\bullet}(Y_-, \mathbf{p})\rightarrow C^s_{\bullet}(Y_+)\\
m^u_o:  C^u_{\bullet}(Y_-, \mathbf{p})\rightarrow C^o_{\bullet}(Y_+)\\
m^u_s: C^u_{\bullet}(Y_-, \mathbf{p})\rightarrow C^s_{\bullet}(Y_+)
\end{IEEEeqnarray*}
by the formulas
\begin{IEEEeqnarray*}{c}
m^o_o([\sigma_-]\otimes[\boldsymbol{\sigma}])=\sum_{[\Cr_+]\in \mathsf{C}^o(Y_+)}\sum_{z\in\pi_0}
([\sigma_-],[\boldsymbol{\sigma}])\times {M}^+_z([\Cr_-],\mathcal{C}, W^*_{\mathbf{p}}, [\Cr_+])\\
m^o_s([\sigma_-]\otimes[\boldsymbol{\sigma}])=\sum_{[\Cr_+]\in \mathsf{C}^s(Y_+)}\sum_{z\in\pi_0}
([\sigma_-],[\boldsymbol{\sigma}])\times{M}^+_z([\Cr_-],\mathcal{C}, W^*_{\mathbf{p}}, [\Cr_+])\\
m^u_o([\sigma_-]\otimes[\boldsymbol{\sigma}])=\sum_{[\Cr_+]\in \mathsf{C}^o(Y_+)}\sum_{z\in\pi_0}
([\sigma_-],[\boldsymbol{\sigma}])\times {M}^+_z([\Cr_-],\mathcal{C}, W^*_{\mathbf{p}}, [\Cr_+]) \\
m^u_s([\sigma_-]\otimes[\boldsymbol{\sigma}])=\sum_{[\Cr_+]\in \mathsf{C}^s(Y_+)}\sum_{z\in\pi_0}
([\sigma_-],[\boldsymbol{\sigma}])\times {M}^+_z([\Cr_-],\mathcal{C}, W^*_{\mathbf{p}}, [\Cr_+]).
\end{IEEEeqnarray*}

It is important to remark that these sums might be potentially infinite, but they are nevertheless well defined on the completions of the chain complexes because of the finiteness results in Theorem \ref{finitenesscob} in Chapter $2$.
Similarly, we can define maps 

\begin{IEEEeqnarray*}{c}
\bar{m}^s_s:  C^s_{\bullet}(Y_-, \mathbf{p})\rightarrow C^s_{\bullet}(Y_+)\\
\bar{m}^s_u: C^s_{\bullet}(Y_-, \mathbf{p})\rightarrow C^u_{\bullet}(Y_+)\\
\bar{m}^u_s:  C^u_{\bullet}(Y_-, \mathbf{p})\rightarrow C^s_{\bullet}(Y_+)\\
\bar{m}^u_u: C^u_{\bullet}(Y_-, \mathbf{p})\rightarrow C^u_{\bullet}(Y_+)
\end{IEEEeqnarray*}
using the reducible moduli spaces via the sums
\begin{IEEEeqnarray*}{c}
\bar{m}^s_s([\sigma_-]\otimes[\boldsymbol{\sigma}])=\sum_{[\Cr_+]\in \mathsf{C}^s(Y_+)}\sum_{z\in\pi_0}
([\sigma_-],[\boldsymbol{\sigma}])\times {M}^{\mathrm{red}+}_z([\Cr_-],\mathcal{C}, W^*_{\mathbf{p}}, [\Cr_+])\\
\bar{m}^s_u([\sigma_-]\otimes[\boldsymbol{\sigma}])=\sum_{[\Cr_+]\in \mathsf{C}^s(Y_+)}\sum_{z\in\pi_0}
([\sigma_-],[\boldsymbol{\sigma}])\times{M}^{\mathrm{red}+}_z([\Cr_-],\mathcal{C}, W^*_{\mathbf{p}}, [\Cr_+])\\
\bar{m}^u_s([\sigma_-]\otimes[\boldsymbol{\sigma}])=\sum_{[\Cr_+]\in \mathsf{C}^u(Y_+)}\sum_{z\in\pi_0}
([\sigma_-],[\boldsymbol{\sigma}])\times {M}^{\mathrm{red}+}_z([\Cr_-],\mathcal{C}, W^*_{\mathbf{p}}, [\Cr_+]) \\
\bar{m}^u_u([\sigma_-]\otimes[\boldsymbol{\sigma}])=\sum_{[\Cr_+]\in \mathsf{C}^u(Y_+)}\sum_{z\in\pi_0}
([\sigma_-],[\boldsymbol{\sigma}])\times {M}^{\mathrm{red}+}_z([\Cr_-],\mathcal{C}, W^*_{\mathbf{p}}, [\Cr_+]). 
\end{IEEEeqnarray*}
We then put the pieces together to define the operators between the chain complexes
\begin{IEEEeqnarray*}{c}
\check{m}:  \check{C}_{\bullet}(Y_-, \mathbf{p})\rightarrow \check{C}_{\bullet}(Y_+)\\
\hat{m}:  \hat{C}_{\bullet}(Y_-, \mathbf{p})\rightarrow \hat{C}_{\bullet}(Y_+)\\
\bar{m}:  \bar{C}_{\bullet}(Y_-, \mathbf{p})\rightarrow \bar{C}_{\bullet}(Y_+)
\end{IEEEeqnarray*}
by the formulas
\begin{equation}\label{cobmaps}
\begin{aligned}
\check{m}=&
\left[
\begin{matrix}
m^o_o & m^u_o \bar{\partial}^s_u(Y_-,\mathbf{p})+\partial^u_o(Y_+)\bar{m}^s_u \\
m^o_s & \bar{m}^s_s+m^u_s\bar{\partial}^s_u(Y_-,\mathbf{p})+\partial^u_s(Y_+)\bar{m}^s_u
\end{matrix}\right] \\
\\
\hat{m}=&
\left[
\begin{matrix}
m^o_o & m^u_o \\
\bar{m}^s_u\partial^o_s(Y_-,\mathbf{p})+\bar{\partial}^s_u(Y_+)m^o_s & \bar{m}^u_u+\bar{m}^s_u\partial^u_s(Y_-,\mathbf{p})+\bar{\partial}^s_u(Y_+)m^u_s
\end{matrix}\right]\\
\\
\bar{m}=&\left[
\begin{matrix}
\bar{m}^s_s & \bar{m}^u_s\\
\bar{m}^s_u & \bar{m}^u_u
\end{matrix}\right].
\end{aligned}
\end{equation}

\vspace{0.5cm}
In the next proposition, we use the natural identification between at the homology level provided by Lemma \ref{Yp}.

\begin{prop}\label{indepmap}
The operator $\check{m}$ is a chain map, i.e. it satisfies the relation
\begin{equation*}
\check{\partial}(Y_+)\circ\check{m}=\check{m}\circ\check{\partial}(Y_-, \mathbf{p})
\end{equation*}
and the induced map in homology
\begin{equation*}
\HMt_{\bullet}(W,\mathbf{p}):\HMt_{\bullet}(Y_-,\spin_-)\otimes \ztwo[[U_1]]\otimes \cdots \otimes \ztwo[[U_m]]\rightarrow \HMt(Y_+,\spin_+)
\end{equation*}
is independent of the choice of the punctures and of the metric on $W$ (isometric to the fixed cylindrical one in a collar of the boundary and standard around the punctures) and perturbation.
The same statement holds also for the operators $\hat{m}$ and $\bar{m}$.
\end{prop}

\begin{proof}
We first show that the map is a chain map. We will only show that the component of $\check{m}$ given by
\begin{equation*}
B= \partial^o_om^o_o+m^o_o\partial^o_o+\partial^u_o\bar{\partial}^s_um^o_s+m^u_o\bar{\partial}^s_u\partial^o_s+\partial^u_o\bar{m}^s_u\partial^o_s: C^o_{\bullet}(Y_-,\mathbf{p})\rightarrow C^o_{\bullet}(Y_+)
\end{equation*}
is zero. Consider $[\sigma_-]\otimes[\boldsymbol{\sigma}]\in C_{\bullet}^o(Y_-, \mathbf{p})$. For a given irreducible critical submanifold $[\Cr_+]\subset \Bs_k(Y_+)$, we have that the component of $B([\sigma_-]\otimes[\boldsymbol{\sigma}])$ in the summand $C_*^{\mathcal{F}}([\Cr_+])$ is given by
\begin{align*}
&\partial\left[(\sigma_-,\boldsymbol{\sigma})\times M^+([\Cr_-],\mathcal{C}, W^*_{\mathbf{p}}, [\Cr_+])\right]\\
&+\sum_{[\Cr'_+]\in \mathsf{C}^o(Y_+)}\left[(\sigma_-,\boldsymbol{\sigma}) \times M^+([\Cr_-],\mathcal{C}, W^*_{\mathbf{p}},[\Cr'_+])\right]\times \breve{M}^+([\Cr'_+],[\Cr_+])\\
&+\left[(\partial[\sigma_-],\boldsymbol{\sigma})\times M^+([\Cr_-],\mathcal{C}, W^*_{\mathbf{p}}, [\Cr_+])\right]+\left[\sum_{\mathcal{C}'}(\sigma_-, \hat{\partial}[\boldsymbol{\sigma}])\times M^+([\Cr_-],\mathcal{C}', W^*_{\mathbf{p}}, [\Cr_+])\right]\\
&+\sum_{[\Cr_-']\in \mathsf{C}^o(Y_-)}\left[\left((\sigma\times \breve{M}^+([\Cr_-],[\Cr_-']),\boldsymbol{\sigma}\right)\times M^+([\Cr_-'],\mathcal{C}, W^*_{\mathbf{p}}, [\Cr_+])\right]\\
&+\sum_{[\Cr_+']\in\mathsf{C}^s_+}\sum_{[\Cr_+'']\in\mathsf{C}^u_+}\left[(\sigma_-,\boldsymbol{\sigma})\times M^+([\Cr_-'],\mathcal{C}, W^*_{\mathbf{p}}, [\Cr_+'])\times \breve{M}^{\mathrm{red}+}([\Cr_+'],[\Cr_+''])\times \breve{M}^+([\Cr_+''],[\Cr_+])\right]\\
&+\sum_{[\Cr_-']\in\mathsf{C}^s_-}\sum_{[\Cr_-'']\in\mathsf{C}^u_-}
\left[\left(\sigma\times \breve{M}^+([\Cr_-],[\Cr_-'])\times \breve{M}^+([\Cr_-'],[\Cr_-'']), \boldsymbol{\sigma}\right)\times M^+([\Cr_-''],\mathcal{C}, W^*_{\mathbf{p}}, [\Cr_+])\right]\\
&+\sum_{[\Cr_-']\in\mathsf{C}^s_-}\sum_{[\Cr_+']\in\mathsf{C}^u_+}
\left[(\sigma\times \breve{M}^+([\Cr_-],[\Cr_-']),\boldsymbol{\sigma})\times M^+([\Cr_-'],\mathcal{C}, W^*_{\mathbf{p}}, [\Cr_+'])\times \breve{M}^+([\Cr_+'],[\Cr_+])\right].
\end{align*}
Here, in the second term of the third row we sum over the sequences $\mathcal{C}'$ that differ from $\mathcal{C}$ by at most one element (notice that for dimensional reasons very few terms actually contribute to the sum). The term $\partial^o_om^o_o$ corresponds to the first two rows and $m^o_o\partial^o_o$ corresponds to the third and fourth row. The remaining three terms correspond each to one of the last three lines. The claim follows because the first term is equal to the sum of all the others. This follows as in the proof of Proposition \ref{chaincomp} from the fact that the codimension $1$ faces of the fibered product of $\delta$-chains is given by the fibered products of a codimension $1$ face is one factor with the other factor, together with the classification of codimension one strata of the moduli space of solutions on a cobordism (see Theorem \ref{cod1cob}). Notice that the fact that the differentials in the factors corresponding to the punctures only involve reducible moduli spaces drastically simplifies the formula.
The proof in the other cases is essentially the same way, see also Lemma $25.3.6$ in the book.
\par
We then show that the map induced in homology does not depend on the choice of suitable metric and perturbations on the cobordism. Given two such choices $(g_0, \hat{\mathfrak{p}}_0)$ and $(g_1, \hat{\mathfrak{p}}_1)$, consider a path $P$ connecting them which such that the corresponding parametrized moduli spaces are regular in the sense of Definition \ref{regcob}. We will consider the chain complex
\begin{equation*}
\check{C}_{\bullet}(Y_-,\mathbf{p})_P
\end{equation*} 
which is the subchain complex of 
\begin{equation*}
\check{C}_{\bullet}(Y_-)\otimes {C}_{\bullet}^u(S^3_1)\otimes \cdots \otimes {C}_{\bullet}^u(S^3_m)
\end{equation*}
consisting of tuples of $\delta$-chains which are transverse to all evaluation maps arising from the moduli spaces with data $P$ (hence also to the ones arising from $(g_0, \hat{\mathfrak{p}}_0)$ and $(g_1, \hat{\mathfrak{p}}_1)$). The inclusion of this chain complex in each of the two subcomplexes of $\check{C}_{*}(Y_-,\mathbf{p})$ corresponding to a choice of metric and perturbation is a quasi-isomorphism (as in Lemma \ref{Yp}). Also, the two choices of metrics and perturbation give rise to two chain maps
\begin{equation*}
\check{m}_0,\check{m}_1: \check{C}_{\bullet}(Y_-,\mathbf{p})_P\rightarrow \check{C}_{\bullet}(Y_+),
\end{equation*}
and our goal is to show that these are chain homotopic. The chain homotopy is the map
\begin{equation*}
\check{m}(P): \check{C}_{\bullet}(Y_-,\mathbf{p})_P\rightarrow \check{C}_{\bullet}(Y_+)
\end{equation*}
obtained by the same formulas \ref{cobmaps} by substituting the moduli spaces on the cobordisms by the parametrized counterparts
\begin{equation*}
M^+([\Cr_-], \mathcal{C}, W^*_{\mathbf{p}},[\Cr_+])_P.
\end{equation*}
The discussion above can be carried over with the only difference in this case each moduli space has an additional codimension one face given by $M^+([\Cr_-],\mathcal{C}, W^*,[\Cr_+])_{\partial P}$. The same proof as above shows that
\begin{equation*}
\check{\partial}(Y_+) \circ\check{m}(P)+\check{m}(P)\circ\check{\partial}(Y_-,\mathbf{p})=\check{m}_0+\check{m}_1,
\end{equation*} 
which proves the result.
\end{proof}
\begin{remark}
It is important to remark that the groups we are dealing with depend on the choice of metric and perturbations on the three manifolds, even though our notation does not make that explicit. The key point of the result is that the induced map does not depend on the choice of the metric and perturbation on the cobordism.
\end{remark}
\vspace{0.8cm}

Given a cobordism $W$ from $Y_-$ to $Y_+$, we define the maps
\begin{equation*}
\HMt_{\bullet}(U^d\mid W): \HMt_{\bullet}(Y_-)\rightarrow \HMt_{\bullet}(Y_+)
\end{equation*}
as follows. Consider on $W$ marked points $\mathbf{p}=\{p_1,\dots, p_m\}$, and fix non-negative integers $\{d_i\}$ summing up to $d$. Then we set
\begin{equation*}
\HMt_{\bullet}(U^d\mid W)(x)=\HMt_{\bullet}(W,\mathbf{p})(x\otimes U_1^{d_1}\otimes \cdots \otimes U_m^{d_m})
\end{equation*}
where the map $\HMt_{\bullet}(W,\mathbf{p})$ is the above constructed above induced by the cobordism which has the marked points interpreted as incoming ends. Our next goal is to show that such a map is independent of the perturbation on the punctures, the number of punctures and the partition of $d$.

\begin{lemma}\label{indeppunctures}
The map $\HMt_{\bullet}(U^d\mid W)$ defined above and its analogues are independent of the choices made.
\end{lemma}
\begin{proof}We will use a metric stretching argument to show that the map is the one induced by a single puncture and the class $u^d$ on the corresponding end. We will focus on the \textit{to} case, as the other ones are essentially identical.
\par
Because of the metric independence on the maps, we can suppose we are in the following special situation. Let $p_0$ be any point in the interior of $W$, such that there is a ball $B_0$ with positive scalar curvature around $p_0$ and metric which is a product $S^3\times (-\varepsilon,0]$ near the boundary where $S^3$ has the standard round metric. We can assume that the neighborhoods around each marked point $p_i$ we used to define the chain map $\check{m}$ are all contained in such a ball. For every $T\geq0$, we can define the Riemannian manifold
\begin{equation*}
W(T)=  B_0\cup ([0,T]\times \partial B_0)\cup (W\setminus \mathrm{int} B_0)
\end{equation*}
obtained by inserting a cylinder of length $T$ along the boundary of the ball. We can interpret this family of manifolds also as a family of metrics parametrized by $[0,\infty)$ on the fixed manifold $W$, which is always standard near the punctures $p_i$. On the additional tube $[0,T]\times \partial B_0$, we consider a non degenerate perturbation $\q_0$ translation invariant perturbation not introducing any irreducible solution. Notice that the perturbations are not supported in the boundary anymore, but the details of the construction carry out without any problem in this slightly more general setting. We can then consider the moduli spaces on the cobordism
\begin{equation*}
M([\Cr_-],\mathcal{C}, W^*_{\mathbf{p}},[\Cr_+])_{[0,\infty)}
\end{equation*}
parametrized by the family of metrics in $[0,\infty)$. We can find time independent perturbations on $B_0$ and $W\setminus\mathrm{int}B_0$ so that these moduli spaces are regular, as it follows from their fibered product description (see Proposition $26.1.3$ in the book).
\par
Even if we compactify the space fiberwise, the result might not be a compact space, simply because the base space is not. As it is usual in neck-stretching arguments (see for example Chapter $26$ in the book), there is a natural compactification obtained by adding the fiber over $T=\infty$
\begin{equation*}
M^+([\Cr_-],\mathcal{C}, W^*_{\mathbf{p}}, [\Cr_+])_{\infty}.
\end{equation*}
This is a stratified space consisting of quintuples
\begin{equation*}
\left( [\check{\boldsymbol{\gamma}}], [\gamma_B],[\check{\gamma}],[\gamma_X],[\gamma_+]\right)
\end{equation*}
where
\begin{align*}
[\check{\boldsymbol{\gamma}}]&\in \breve{M}^+\left(([\Cr_-],\mathcal{C}),([\Cr_-'],\mathcal{C}')\right)\\
{[}\gamma_B]&\in M(\mathcal{C}', B_{0,\mathbf{p}}^*, [\Cr_0])\\
{[}\check{\gamma}]&\in \breve{M}^+([\Cr_0], [\Cr_0'])\\
{[}\gamma_X]&\in M([\Cr_-'],[\Cr_0'], X^*_{p_0}, [\Cr_+'])\\
{[}\breve{\gamma}_+]&\in \breve{M}^+([\Cr_+'],[\Cr_+])
\end{align*}
are such that the evaluations on the corresponding ends coincide. Here $[\Cr_0]$ and $[\Cr_0']$ are \textit{unstable} critical submanifolds for the flow on $S^3$ with the perturbation $\q_0$. These objects also have a well defined relative homotopy class. We define the total compactified space as
\begin{equation*}
\mathcal{M}^+_z([\Cr_-],\mathcal{C}, W^*_{\mathbf{p}}, [\Cr_+])=\bigcup_{S\in[0,\infty]}\{S\}\times M^+_z([\Cr_-]\mathcal{C}, W^*_{\mathbf{p}}, [\Cr_+])_S.
\end{equation*}
The topology on the total space is defined in the same fashion as the one on the space of broken trajectories, and it is not hard to prove that this space is an abstract $\delta$-chain whose codimension one faces are given by:
\begin{itemize}
\item the moduli space $M^+_z([\Cr_-],\mathcal{C}, W, [\Cr_+])_0$;
\item the union over $S\in[0,\infty]$ of the codimension one faces of $M^+_z([\Cr_-],\mathcal{C}, W, [\Cr_+])_S$ corresponding to the same sequences of critical submanifolds and homotopy class.
\item the top strata of the fiber over $\infty$.
\end{itemize}
Regarding the last item, the analogue of Proposition $26.1.6$ in the book holds in a simpler version as there are no boundary obstructed trajectories involved, and these top strata are fibered products of the form
\begin{equation*}
M(\mathcal{C}, B_{0,\mathbf{p}}^*, [\Cr_0])\times M([\Cr_-],[\Cr_0], X^*_{p_0}, [\Cr_+]).
\end{equation*}
Using these compactifications, we can define the map
\begin{align*}
H^o_o&: C^o_{\bullet}(Y_-, \mathbf{p})\rightarrow C^o_{\bullet}(Y_+)\\
[\sigma]\otimes[\boldsymbol{\sigma}]&\mapsto \sum_{[\Cr_+]\in \mathsf{C}^o} (\sigma,\boldsymbol{\sigma})\times \mathcal{M}^+_z([\Cr_-],\mathcal{C}, W^*_{\mathbf{p}}, [\Cr_+]),
\end{align*}
and similarly the maps $H^o_s, H^u_o, H^u_s$ and using the reducible counterparts of the moduli spaces the operator $\bar{H}^s_s, \bar{H}^u_u, \bar{H}^s_u, \bar{H}^u_s$. As it should be clear, here we restrict to the subcomplex generated by tuples of $\delta$-chains transverse to all the parametrized moduli spaces, which is quasi isomorphic to the total complex following Lemma \ref{Yp}. We then define
\begin{IEEEeqnarray*}{c}
\check{H}:\check{C}_{\bullet}(Y_-, \mathbf{p})\rightarrow \check{C}_{\bullet}(Y_+)\\
\check{H}=\left[
\begin{matrix}
H^o_o & H^u_o\bar{\partial}^s_u+\partial^u_o\bar{H}^s_u\\
H^o_s & \bar{H}^s_s+H^u_s\bar{\partial}^s_u+\partial^u_s\bar{H}^s_u
\end{matrix}\right].
\end{IEEEeqnarray*}
The following identity holds, 
\begin{equation*}
\check{H}\circ\check{\partial}+\check{\partial}\circ\check{H}=\check{m}_0+\check{m}_{\infty},
\end{equation*}
where $\check{m}_0$ is the chain map induced by the original cobordism and $\check{m}_{\infty}$ is analogous map defined by using the moduli spaces $M^+([\Cr_-],\mathcal{C}, W^*_{\mathbf{p}}, [\Cr_+])_{\infty}$. It is important to notice that although the latter are is not an abstract $\delta$-chain, it consists of a union of them along codimension one faces and the whole construction carries over without any complication.  
To verify the identity, we have for example that
\begin{equation*}\partial^o_oH^o_o+\partial^u_o\bar{\partial}^s_uH^o_s+H^o_o\partial^o_o+H^u_o\bar{\partial}^s_u\partial^o_s+\partial^u_o\bar{H}^s_u\partial^o_s= (\check{m}_0)^o_o+ (\check{m}_{\infty})_o^o,
\end{equation*}
which follows as usual by the characterization of the codimension one strata. The map $\check{m}_{\infty}$ is just the composition of two chain maps, namely the map
\begin{equation*}
\hat{m}(B_0, \mathbf{p}):\hat{C}_{\bullet}(S^3_1)\otimes\cdots \otimes\hat{C}_{\bullet}(S^3_m)\rightarrow \hat{C}_{\bullet}(S^3_0)
\end{equation*}
induced by the punctured cobordism $B_0$ and the map
\begin{equation*}
\check{m}(W,p_0): \check{C}_{\bullet}(Y_-, p_0)\rightarrow \check{C}_{\bullet}(Y_+)
\end{equation*}
that defines the map induced by the cobordism $W$ with the single puncture $p_0$. We again did not mention the transversality issue that can be handled in the usual way by restricting to a quasi-isomorphic chain complex. To conclude, we just need to show that the chain map $\hat{m}(B_0,\mathbf{p})$ induces at the homology level the multiplication map
\begin{align*}
\ztwo[[U_1]]\otimes \cdots\otimes \ztwo[[U_m]]&\rightarrow \ztwo[[U]]\\
U_1^{d_1}\otimes\cdots\otimes U_m^{d_m}&\mapsto U^{d_1+\dots+d_m}.
\end{align*}
The argument is very close to that of Lemma $27.4.2$ in the book. On the cobordism $(B_0)^*_{\mathbf{p}}$ there is only one anti-self dual connection $A$ up to gauge equivalence, and no irreducible solutions because of positive scalar curvature. Each generator $U_i^{d_i}$ can be realized as a projective subspace $M_i$ of the critical submanifold $[\Cr_i]$ representing a generator in homology, and similarly $U^d$ is represented by projective subspace in a critical submanifold $[\Cr_+]$. By linearity of the Dirac operator $D_A^+$, we have that the top stratum of the moduli space
\begin{equation*}
M^+([\Cr_1],\dots, [\Cr_m], (B_0)^*_{\mathbf{p}}, [\Cr_+])
\end{equation*}
is a complex projective space with some hyperplanes removed. Furthermore, the image of the evaluation maps in
\begin{equation*}
[\Cr_1]\times\dots \times[\Cr_m]
\end{equation*}
projects to a projective space in each factor and has the same dimension as the moduli space. The positivity of the scalar curvature implies that the $L^2$ index of the operator $D_A^+$ is zero. By simple dimensional considerations that the fibered product between $\prod M_i$ and the moduli space on $B_0$ represents the class $U^d$.
\end{proof}

\vspace{0.8cm}
The aim is now to check functoriality and invariance for the construction. The first verification is trivial.
\begin{prop}\label{identitymap}
If $X$ is a trivial cylindrical cobordism from $(Y;g,\q)$ to itself, then the maps
\begin{equation*}
\HMt_{\bullet}(1\mid X), \qquad\HMf_{\bullet}(1\mid X), \qquad \HMb_{\bullet}(1\mid X)
\end{equation*}
are the identity on the Floer homology groups.
\end{prop}
\begin{proof}
For this computation we can simply consider the product cobordism $I\times Y$ from $Y$ to itself with a translation invariant perturbation. We claim that the chain map is in this case the identity. In fact the evaluation maps of a moduli space consisting of non translationally invariant solutions factors through a smaller dimensional space (the space of unparametrized trajectories), hence the fibered products with it will always be negligible. This implies that we can restrict ourselves to consider the translationally invariant moduli spaces $M_z([\Cr],[\Cr])$ (with $z$ trivial), which induce the identity map at the chain level.
\end{proof}

The verification of the composition law is not as straightforward, and also uses a metric stretching argument similar to the one in Proposition \ref{indepmap}.
\begin{prop}\label{modulestr}
Let $Y_0,Y_1$ and $Y_2$ be $3$-manifolds, and let $W_{01}$ and $W_{12}$ be cobordisms from $Y_0$ to $Y_1$ and from $Y_1$ to $Y_2$ respectively which are cylindrical near their boundaries. If $d=d_1+d_2$, we have the identities 
\begin{equation*}
\HMt_{\bullet}(u^d\mid W)=\HMt_{\bullet}(u^{d_2}\mid W_{12})\circ\HMt_{\bullet}(u^{d_1}\mid W_{01})
\end{equation*}
where $W=W_{12}\circ W_{01}$ is the composite cobordism, and similarly for the from and bar versions.
\end{prop}

\begin{proof}
The proof is essentially the same as in Chapter $26$ in the book, to which we refer for the details.
By the independence of the map, we can pick a single puncture with weight $d_i$ on each of the cobordisms. For every $S\geq0$ define the manifold
\begin{equation*}
W(S)=W_{01}\cup ([0,S]\times Y_1)\cup W_{12},
\end{equation*}
the composite cobordism with a cylinder of length $S$ inserted in the middle. The manifolds $W(S)$ can also be interpreted as a family of metrics of the fixed manifold $X$. As in the proof of Proposition \ref{indepmap}, we can compactify the union of the moduli spaces for each metric, obtaining the space
\begin{equation*}
\mathcal{M}^+([\Cr_-],[\Cr_1],[\Cr_2],W_{p_1,p_2}^*,[\Cr_+])_{[0,\infty]},
\end{equation*}
where $[\Cr_i]$ is a critical manifold on the end corresponding to $p_i$, by adding a fiber over $S=\infty$ which we call
\begin{equation*}
{M}^+([\Cr_-],[\Cr_1],[\Cr_2], X(\infty)^*_{p_1,p_2}, [\Cr_+]).
\end{equation*}
This consists of tuples of the form
\begin{equation*}
\left( [\breve{\boldsymbol\chi}_0], [\chi_{01}], [\breve{\boldsymbol\gamma}_1],[\breve{\boldsymbol\chi}_1], [\chi_{12}], [\breve{\boldsymbol\gamma}_2],[\breve{\boldsymbol\chi}_2]\right)
\end{equation*}
where 
\begin{itemize}
\item each $[\breve{\boldsymbol\chi}_i]$ is a broken trajectories on $Y_i$;
\item $[\chi_{01}]$ is a solution on $(W_{01})^*_{p_1}$, and similarly $[\chi_{12}]$ is a solution on $(W_{12})^*_{p_2}$;
\item each $[\breve{\boldsymbol\gamma}_i]$ is a broken trajectory on the end corresponding to the puncture $p_i$,
\end{itemize}
such that the evaluation maps on the corresponding ends agree. The topology on the total space is defined in a similar fashion to that on a space of broken trajectories (see Section $5$ in Chapter $2$). As before, this space is obtained by gluing abstract $\delta$-chains along their codimension one faces, and the top strata of the fiber over $\infty$ consist of objects of the form:
\begin{itemize}
\item $M_{01}\times M_{12}$;
\item $\breve{M}_0\times M_{01}\times M_{12}$;
\item $M_{01}\times \breve{M}_1\times M_{12}$;
\item $M_{01}\times M_{12}\times \breve{M}_2$
\end{itemize}
where the $\breve{M}_i$ indicate the typical moduli space on $Y_i$, and $M_{01}$ and $W_{12}$ are the typical moduli spaces on the cobordisms with prescribed asymptotic at the puncture. As usual, in the last three cases, the middle moduli space is boundary obstructed, and we use the fact that there are no boundary obstructed trajectories on the three-sphere. The proof then follows the same lines of Lemma \ref{indeppunctures}. In particular, one can define the operators
\begin{equation*}
K^o_o, K^o_s, K^u_o, K^u_s\text{ and }\bar{K}^s_s,\bar{K}^s_u,\bar{K}^u_s,\bar{K}^u_u
\end{equation*}
via the fibered products with the moduli spaces $\mathcal{M}^+([\Cr_-],[\Cr_1],[\Cr_2],W_{p_1,p_2}^*,[\Cr_+])_{[0,\infty]}$, and define the operator
\begin{IEEEeqnarray*}{c}
\check{K}:\check{C}_{\bullet}(Y_0, \{p_1,p_2\})\rightarrow \check{C}_{\bullet}(Y_2)\\
\check{K}=\left[
\begin{matrix}
K^o_o & K^u_o\bar{\partial}^s_u+m^u_o\bar{m}^s_u+\partial^u_o\bar{K}^s_u\\
K^o_s & \bar{K}^s_s+K^u_s\bar{\partial}^s_u+m^u_s\bar{m}^s_u+\partial^u_s\bar{K}^s_u
\end{matrix}\right],
\end{IEEEeqnarray*}
where the terms $m^u_o\bar{m}^s_u$ and $m^u_s\bar{m}^s_u$ are the composition of the maps induced by the two cobordisms. It is then not hard to check that this map satisfies the identity
\begin{equation*}
\check{K}\circ \check{\partial}(Y_0, \{p_1,p_2\})+\check{\partial}(Y_2)\circ\check{K}=\check{m}(W,\{p_1,p_2\})+\check{m}(W_{12},p_2)\circ{m}(W_{01},p_1),
\end{equation*}
hence defines a chain homotopy between the two chain maps in consideration (see Chapter $26$ in the book for the details). 
\end{proof}

\vspace{0.8cm}
We recollect the invariance and functoriality results we have just proved in the following corollaries.

\begin{cor}
For different choices of metrics and perturbation $(g,\q)$ the Floer homology groups are canonically isomorphic, hence they are canonically isomorphic to the ones defined in the book. These identifications also preserve the maps induced by cobordisms. 
\end{cor}
\begin{proof}
The canonical isomorphism is the map induced by the trivial cobordism $I\times Y$ (with the appropriate metrics cylindrical near the boundary). Such a map is well defined because of Proposition \ref{indepmap}, and is hence an isomorphism by Lemma \ref{identitymap}. Furthermore, by Lemma \ref{sameconstruction} the construction gives the same result as the construction in the book, and in fact the maps induced by a cobordism are the same because in the case of a non degenerate perturbation only the zero dimensional moduli spaces come into definition of the differentials. \end{proof}

Recall from Definition $3.4.2$ in the book the category $\textsc{cob}$ whose objects are compact, connected, oriented $3$-manifolds and whose morphisms are isomorphism classes of cobordisms.

\begin{cor}
The Floer groups define covariant functors from the cobordism category $\textsc{cob}$ to the category of groups
\begin{align*}
\HMt_{\bullet}&:\textsc{cob}\rightarrow \textsc{group}\\
\HMf_{\bullet}&:\textsc{cob}\rightarrow \textsc{group}\\
\HMb_{\bullet}&:\textsc{cob}\rightarrow \textsc{group}.
\end{align*}
\end{cor}

\begin{remark}
There is a subtlety in the last result (see also the discussion after Corollary $23.1.6$ in the book). Even though the Floer groups for different choice of metric and regular triple are canonically isomorphic, they are different, and it would seem that the Floer groups define functors only to the category of groups up to canonical isomorphisms, i.e. the category where objects are families $\{G_a\}_{a\in A}$ with isomorphisms $\phi_{a_1a_2}$ satifying the compatibility condition
\begin{equation*}
\phi_{a_2a_3}\phi_{a_1a_2}=\phi_{a_1a_3}
\end{equation*}
and morphisms are collections of maps satifying the obvious compatibility relations. On the other hand there is a natural functor from this category to the category of groups: to a family $(\{G_a\}_{a\in A},\phi)$ we assign the subgroup of
\begin{equation*}
\prod_{a\in A} G_a
\end{equation*}
consisting of collections $\{g_a\}_{a\in A}$ with $g_b=\phi_{ab}g_a$ (which is isomorphic to each of the $G_a$).
\end{remark}
\vspace{0.8cm}

Finally the maps induced by cobordisms can also be used in order to define on the Floer homology groups of $Y$ a graded module structure over the ring $\ztwo[[U]]$. Indeed, given $\xi\in \HMt_{\bullet}(Y)$ one defines the cap product
\begin{IEEEeqnarray*}{c}
U^d\cap \xi = \HMt_{\bullet}(U^d\mid I\times Y)(\xi)
\end{IEEEeqnarray*}
and the analogues for the other versions. The fact that this is a module structure follows directly from Proposition \ref{modulestr}. Similarly one can define the cup product in cohomology.

\begin{remark}\label{exterioraction}
Notice that in order to make the discussion simpler we have not dealt with the whole action of the cohomology group with $\ztwo$ coefficients of the moduli space of configurations of $W$, but only with the elements of the form $U^d$. The action of the remaining terms, namely the ones of the form
\begin{equation*}
\Lambda^*(H_1(W;\Z)/\mathrm{Tor}\otimes \ztwo)
\end{equation*}
can be defined in an analogous way as follows. Given a degree one element $x$ in the exterior product, one chooses an embedded loop $\gamma$ representing it. A tubular neighborhood of the loop is diffeomorphic to $S^1\times D^3$, and we can suppose that the metric on it is the product of the standard ones. The boundary, which is a copy of $S^1\times S^2$, has positive scalar curvature, and we consider on it the spin$^c$ structure $\spin_0$ coming from the spin structure. We know from Chapter $36$ in the book that $\HMf_{\bullet}(S^1\times S^2,\spin_0)$ is canonically isomorphic to
\begin{equation*}
\ztwo[[U]]\oplus \ztwo[[U]]\{-1\},
\end{equation*}
and the group is trivial in other spin$^c$ structure (again, the actual absolute gradings are obtained by shifting by $-1$). In fact, we can fix a small perturbation so that there are only two reducible critical points downstairs and no irreducibles for the spin$^c$ structure $\spin_0$, and no solutions for the others. Call $W_{\gamma}$ the manifold obtained from $W$ by removing the neighborhood of $\gamma$. This has three boundary components, namely $Y_-$ and $S^1\times S^2$ as incoming ends and $Y_+$ as an outgoing end. The discussion of the present section then carries over without modifications to show that one can define a map
\begin{equation*}
\HMt_{\bullet}(Y_-)\otimes \HMf_{\bullet}(S^1\times S^2)\rightarrow \HMt_{\bullet}(Y_+)
\end{equation*}
by considering the moduli spaces on the manifold with cylindrical ends $W_{\gamma}^*$. The map $\HMt_{\bullet}(x\mid I\times Y)$ is then defined by evaluating this map at the element
\begin{equation*}
(0,1)\in \ztwo[[U]]\oplus \ztwo[[U]]\{-1\}=\HMf_{\bullet}(S^1\times S^2).
\end{equation*}
Similarly, we can define the action of an element
\begin{equation*}
x_1\wedge \dots \wedge x_n\otimes U^d
\end{equation*}
by deleting tubular neighborhoods of the loops $\gamma_i$ representing $x_i$ and a ball with weight $d$. The proofs that these are well defined are analogous to those discussed in the section.
\end{remark}

\chapter{$\Pin$-monopole Floer homology}

In this chapter we apply the theory we have developed in order to study a specific case arising from a natural action of conjugation on the set of spin$^c$-structures. The final product will be a set of invariants of three manifolds, called \textit{$\Pin$-monopole Floer homology}, which are the analogue of the $\Pin$-equivariant Seiberg-Witten-Floer homology recently defined by Manolescu in \cite{Man2}.
\par
The action by conjugation has been known and exploited since the early days of Seiberg-Witten theory. In order to remark this we named our first section with the same name as the corresponding section in the classical reference \cite{Mor}. On the other hand, the idea to exploit this additional symmetry of the equations at a more refined level is the key point behind Manolescu's breakthrough. In our setting, this will arise as an involution on the moduli spaces of configurations, which will in turn imply that the critical points are Morse-Bott and symmetrical in a very interesting way. Exploiting additional features at the algebraic level, and we will be able to define our new invariants building on the work of Chapter $3$.

\vspace{1.5cm}
\section{An involution in the theory}

We start with some remarks at the level of Clifford algebras (see for example \cite{Mor} for more details). The Lie group $\mathrm{Spin}(3)$ is naturally identified with the group $\mathrm{SU(2)}$ of unitary determinant one matrices. Furthermore its spin representation $S$ corresponds the space of quaternions $\mathbb{H}$ on which it acts via the natural action on the \textit{left} of $\mathrm{SU(2)}$ on the latter when thought as $\C^2$. We will always think of the complex structure on $\mathbb{H}$ as given by the multiplication on the \textit{right}, hence we have the identification of complex vector spaces
\begin{align*}
\mathbb{C}^2&\rightarrow \mathbb{H}\\
(z_1,z_2)&\mapsto z_1+jz_2.
\end{align*}
In particular, the multiplication by $j$ on the \textit{right} induces a complex antilinear isomorphism of $\mathbb{C}^2$, i.e. a complex isomorphism between $\mathbb{C}^2$ and its conjugate vector space.
\par
When we have a spin$^c$ structure $\spin=(S,\rho)$ on a three manifold $Y$, we can also consider its conjugate spin$^c$ structure $\bar{\spin}=(\bar{S},\rho)$, where $\bar{S}$ is the conjugate bundle of $S$ and the action via $\rho$ of the real one-forms is the same one. In this case, the action of $j$ induces a complex antilinear identification
\begin{equation*}
\jmath: \spin\rightarrow \bar{\spin}
\end{equation*}
which can also be extended to an action
\begin{IEEEeqnarray*}{c}
\jmath : \Co(Y,\spin)\rightarrow \Co(Y,\bar{\spin})\\
(B,\Psi)\mapsto (\bar{B}, \Psi\cdot j)
\end{IEEEeqnarray*}
where $\bar{B}$ is the conjugate connection of $B$. This also descends to an identification (still called $\jmath$) between the moduli spaces of configurations $\Bo(Y,\spin)$. Furthermore, $\jmath^2$ is the automorphism of $\Co(Y,\spin)$ given by
\begin{equation*}
(B,\Psi)\mapsto (B,-\Psi),
\end{equation*}
and therefore it acts as the identity on the moduli space of configurations $\Bo(Y,\spin)$.
\\
\par
The most interesting case of this isomorphism $\jmath$ is certainly the one in which the spin$^c$ structure $\spin$ is actually isomorphic to its conjugate, in which case we call it \textit{self-conjugate}. In this case the map $\jmath$ is an involution of the configuration space. The following result gives us a classification of such spin$^c$ structures.

\begin{prop}\label{selfconj}
Given an oriented Riemannian three manifold $Y$, there always exists a self-conjugate spin$^c$-structure $\spin_0$. Furthermore for a fixed one there is a one-to-one correspondence between:
\begin{enumerate}
\item isomorphism classes of self conjugate spin$^c$ structures $\spin$ on $Y$;
\item complex line bundles $L$ isomorphic to their conjugate line bundle $\bar{L}$;
\item the $2$-torsion of $H^2(Y;\Z)$.
\end{enumerate}
\end{prop}

\begin{proof}
This is achieved by the same construction as in Section $1$ of Chapter $1$, see also Proposition $1.1.1$ in the book. Because the tangent bundle of $Y$ is trivial, $Y$ admits a spin structure, and the spin$^c$ structure it defines via the spin representation is clearly self-conjugate. Furthermore, given such a structure $\spin_0=(S_0,\rho)$, for any line bundle $L$ which is isomorphic to its conjugate we have that the spin$^c$ structure
\begin{IEEEeqnarray*}{c}
S=S_0\otimes_{\C}L\\
\rho=\rho_0\otimes 1_L
\end{IEEEeqnarray*}
is again self-conjugate, and in fact every self-conjugate spin$^c$ structure arises in this way for an appropriate choice of a line bundle isomorphic to its conjugate. Finally, such line bundles are identified as those whose first Chern class is $2$-torsion, because $c_1(\bar{L})=-c_1(L)$.
\end{proof}

As we have seen every spin structure induces a self-conjugate spin$^c$ structure. By considering the Bockstein exact sequence
\begin{equation*}
\cdots\rightarrow H^1(X;\Z)\rightarrow H^1(X;\mathbb{Z}/2\mathbb{Z})\stackrel{\delta}{\rightarrow} H^2(X;\Z)\stackrel{\cdot2}{\rightarrow} H^2(X;\Z)\rightarrow\cdots
\end{equation*}
we see that the $2$-torsion classes in $H^2(X;\Z)$ are exactly the image of the Bockstein homomorphism $\delta$. At the level of spin$^c$ structures, this tells us that the self-conjugate ones are exactly those which come from a genuine spin structure. Furthermore, for a fixed spin structure $\spin_0$, the spin structures $\spin_0+x_1$ and  $\spin_0+x_2$ for $x_i\in H^1(X;\ztwo)$ determine the same spin$^c$ structure if and only if $\delta(x_1)$ coincides with $\delta(x_2)$. Hence each self-conjugate spin$^c$ structure comes from exactly $2^{b_1(Y)}$ spin structures. For example, on $S^2\times S^1$ there are two spin structures that induce the same spin$^c$ structure while $\mathbb{R}P^3$ the two spin structures induce distinct spin$^c$ structures.
\vspace{0.8cm}

Given a self-conjugate spin$^c$ structure, we can then consider it as coming from a fixed spin structure $\spin$, so that it is comes with a canonical choice of the automorphism $j$ of the spinor bundle $S$ given by right multiplication. The action of $j$ makes $S$ into a quaternionic vector bundle and we can think of this additional structure as endowing the equations with a $\Pin$-symmetry, where
\begin{equation*}
\Pin=S^1\cup j\cdot S^1\subset \mathbb{H}.
\end{equation*}
There is a unique $\jmath$-invariant (hence $\Pin$-invariant) base connection $B_0$, namely the one induced by the Levi-Civita connection on the tangent bundle and the conjugation invariant connection on the determinant line bundle. This has the additional property that the Dirac operator commutes with the action of $j$, i.e.
\begin{equation}\label{jequiv}
D_{B_0}(\Psi\cdot j)=(D_{B_0}\Psi)\cdot j.
\end{equation}
When working with a self-conjugate spin$^c$ structure, we will always suppose to have fixed this as a base connection, and we can write the action of $\jmath$ on the configuration space $\Co(Y,\spin)$ as
\begin{equation*}
\jmath\cdot(B_0+b,\Psi)=(B_0-b, \Psi\cdot j).
\end{equation*}
Furthermore, the Dirac operators
\begin{equation*}
D_B: \Gamma(S)\rightarrow \Gamma(S)
\end{equation*}
are compatible with the action of $\jmath$ in the sense that
\begin{equation}\label{conjdirac}
D_B(\Psi\cdot j)= D_{\bar{B}}\Psi\cdot j.
\end{equation}
This can be easily checked with a local computation. For a fixed orthonormal frame $\{e_i\}$, if $B=B_0+b$ one has
\begin{multline*}
D_B(\Psi\cdot j)=\sum_k \rho(e_k)(\nabla^{B_0}_{e_k}+b(e_k))(\Psi\cdot j)=\sum_k \rho(e_k)(\nabla^{B_0}_{e_k}(\Psi)\cdot j+(\Psi\cdot j)\cdot b(e_k))\\
=\sum_k \rho(e_k)(\nabla^{B_0}_{e_k}(\Psi)\cdot j-(\Psi\cdot  b(e_k))\cdot j)=D_{\bar{B}}\Psi\cdot j
\end{multline*}
where we used the fact that $b$ is purely imaginary.
\\
\par
As we have chosen the basepoint $B_0$, the Chern-Simons-Dirac functional $\CSD$ (which is well defined up to gauge because $c_1$ is torsion) is $\jmath$-invariant. Hence its gradient is $\jmath$-equivariant, and $\jmath$ also defines an involution of the set of critical points.
Because the action of $j$ on $\mathbb{C}^2$ does not have eigenvectors, a fixed point of the action of $\jmath$ on $B(Y,\spin)$ is necessarily of the form $[B,0]$, and there are $2^{b_1(Y)}$ such fixed points corresponding to the fixed points of the involution on the moduli space of flat connections
\begin{align*}
\jmath: \mathbb{T}\rightarrow \mathbb{T}\\
x\mapsto -x.
\end{align*}
These are exactly the the flat connections whose representatives are gauge equivalent to their conjugate or, equivalenty, the flat connections which have holonomy $\pm1$ around each loop.
The induced action on $\Bo(Y,\spin)$ does \textit{not} depend on the actual choice of the spin structure inducing the given one. This is a manifestation of Schur's Lemma, as such a $j$ is pointwise an isomophism of irreducible $SU(2)$-modules, hence any other such $j$ is necessarily gauge equivalent to it. This independence statement implies that the fixed points are the gauge equivalence classes of the canonical connections of the $2^{b_1(Y)}$ spin structures inducing the given one. We will refer to these connections as the \textit{spin} connections.
\\
\par
Notice that the relation (\ref{jequiv}) implies that the operator $D_{B_0}$ is quaternionic linear, hence in particular eigenspaces are always even dimensional. This is true in a more general context, as in the next lemma.
\begin{lemma}\label{quaternion}
Suppose the connection $B$ is gauge equivalent to a spin connection, hence it is gauge equivalent to its conjugate via a gauge transformation $u$, i.e. $\bar{B}=u\cdot B$. Then the right multiplication by 
\begin{equation*}
j'=ju^{-1}
\end{equation*}
determines a quaternionic structure on the eigenspaces of $D_B$, which are in particular always even dimensional. 
\end{lemma}
\begin{proof}
Suppose we have
\begin{equation*}
D_B\Psi=\Psi\cdot \lambda.
\end{equation*}
Then because of the gauge invariance of the equations, the relation (\ref{conjdirac}) and the fact that $\lambda$ is real we have the identities
\begin{equation*}
(\Psi\cdot j)\cdot \lambda=(D_B\Psi)\cdot j= D_{\bar{B}}(\Psi\cdot j)=D_{u\cdot B}(\Psi\cdot j)= D_B(\Psi\cdot j\cdot u^{-1})\cdot u,
\end{equation*}
hence
\begin{equation*}
D_B(\Psi\cdot j\cdot u^{-1})=(\Psi\cdot j\cdot u^{-1})\cdot \lambda.
\end{equation*}
The lemma then follows from the fact that $ju^{-1}$ anticommutes with the complex multiplication by $i$, and has square $-1$, so it defines a quaternionic structure on the eigenspace.
\end{proof}

The map $\jmath$ naturally induces an automorphism of the blow-up
\begin{IEEEeqnarray*}{c}
\jmath: \Cs(Y,\spin)\rightarrow \Cs(Y,\spin)\\
(B_0+b, s, \psi)\mapsto (B_0-b, s, \psi\cdot j)
\end{IEEEeqnarray*}
which descends to a fixed point free involution
\begin{equation*}
\jmath: \Bs(Y,\spin)\rightarrow \Bs(Y,\spin),
\end{equation*}
and it is immediate to check that the blown up gradient of the Chern-Simons-Dirac functional $(\mathrm{grad}\CSD)^{\sigma}$ is equivariant under this action. Furthermore, the action induces a smooth involution in the completion of the configuration spaces in each Sobolev norm.
\\
\par
The case of a four dimensional Riemannian manifold $X$ is completely analogous. In this setting, the spin representation $S=S^+\oplus S^-$ can by identified with $\mathbb{H}\oplus \mathbb{H}$, and as before there is an action from the right by quaternion multiplication by $j$, which determines an isomorphism between $S$ and its conjugate $\bar{S}$.
At a global level, everything regarding the involution $\jmath$ for a given self-conjugate spin$^c$-structure holds with the obvious translations. The only different point is that a self conjugate spin$^c$ structure does not exist unless the manifold is spin (or equivalently, the second Stiefel-Whitney class $w_2(X)$ is zero), as one can easily check using the construction of Proposition \ref{selfconj}. This involution naturally induces involutions of the various configurations spaces, and their properties are the same those described above for the configuration spaces of three manifolds.
\par
Similarly, a self conjugate spin$^c$ structure on a four manifold with boundary $X$ induces a self conjugate spin$^c$ structure on its boundary, and the quaternionic structures are compatible. The induced involutions on the configurations spaces are also compatible with the restriction maps.

\vspace{0.8cm}

In the same way the cohomology ring of $BS^1$ is important in the usual monopole Floer homology, the cohomology ring of $B\Pin$ will have a special role in our construction. In particular, it will arise as the Floer homology groups associated to $S^3$. It can be calculated by means of the Serre spectral sequence associated to the fiber bundle
\begin{equation*}
\R P^2 \hookrightarrow B\Pin\rightarrow \mathbb{H}P^{\infty}=B \mathrm{SU}(2)
\end{equation*}
coming from the identification $\mathrm{SU}(2)/\Pin=\R P^2$, which collapses at the $E^2$ page for degree reasons. The ring on which our modules will be defined is obtained from this by completion and degree reversal.
\begin{defn}\label{R}
We denote by $\Rin$ the ring
\begin{equation*}
\Rin=\ztwo[[V]][Q]/(Q^3)\\
\end{equation*}
where $V$ and $Q$ have degrees respectively $-4$ and $-1$.
\end{defn}

\vspace{1.5cm}
\section{Equivariant perturbations and Morse-Bott transversality}

As we saw in the previous section, when the spin$^c$ structure is self-conjugate, the monopole equations have an additional symmetry provided by $\jmath$, the action of the quaternionic multiplication by $j$. We would like to exploit this symmetry in order to construct a Floer chain complex with additional algebraic properties, and in order to do so we need to ensure that the perturbations are compatible with the extra action. The aim of the present section is to construct $\Pin$-equivariant tame perturbations, and use them to achieve Morse-Bott transversality in the sense of Chapter $3$.
\\
\par
From now on we suppose that the spin$^c$ structure $\spin$ is self-conjugate and fix a base spin connection $B_0$. Recall the construction of cylinder function from Section $4$ of Chapter $1$. First of all, notice that for a coclosed $1$-form $c\in\Omega(Y;i\R)$, the map
\begin{IEEEeqnarray*}{c}
r_c:\Co(Y,\spin)\rightarrow \R\\
(B_0+b\otimes 1,\Psi)\mapsto \int_Y b\wedge \ast\bar{c}
\end{IEEEeqnarray*}
has the symmetry
\begin{equation*}
r_c(\jmath\cdot(B,\Psi))=-r_c(B,\Psi).
\end{equation*}
Also, for a given smooth section $\Upsilon$ of the bundle $\mathfrak{S}\rightarrow \mathbb{T}\times Y$, one can define the $\G^o(Y)$-invariant map
\begin{align*}
\tilde{q}_{\Upsilon}:\Co(Y,\spin)&\rightarrow \mathbb{H} \\
(B,\Psi)&\mapsto q_{\Upsilon}(B,\Psi)-j\cdot\overline{q_{\Upsilon}(\jmath\cdot(B,\Psi))}\\
\end{align*}
where we recall
\begin{equation*}
q_{\Upsilon}(B,\Psi)=\int_Y \langle \Psi, \Upsilon^{\dagger}(B,\Psi)\rangle\in\mathbb{C}.
\end{equation*}
This map is equivariant under the action of $\jmath$, as it easily follows from the property
\begin{equation*}
q_{\Upsilon}(B,-\Psi)=-q_{\Upsilon}(B,\Psi).
\end{equation*}
We stress again on the fact that $j$ acts on $\mathbb{H}$ by left multiplication. Hence, given a finite collection of coclosed $1$-forms $c_1,\dots, c_{n+t}$, the first $n$ being coexact and the remaining $t$ being a basis of harmonic forms, and a collection $\Upsilon_1,\dots, \Upsilon_m$ of sections of $\mathbb{S}$, we can define as in Section $4$ of Chapter $1$ the map
\begin{IEEEeqnarray*}{c}
\tilde{p}:\Co(Y,\spin)\rightarrow \R^n\times\mathbb{T}\times\mathbb{H}^m\\
(B,\Psi)\mapsto \left(r_{c_1}(B,\Psi),\dots, r_{c_{n+t}}(B,\Psi),\dots, \tilde{q}_{\Upsilon_1}(B,\Psi),\dots, \tilde{q}_{\Upsilon_m}(B,\Psi)\right)
\end{IEEEeqnarray*}
which is invariant under $\G^o(Y)$ and equivariant under the remaining action of $\Pin$. On the right hand side, $j$ acts as minus the identity on the $\R$ and $\mathbb{T}$ factors and via right multiplication on the $\mathbb{H}^m$ factor.

\begin{defn}
We call a real valued function $f$ on $\Co(Y,\spin)$ a $\Pin$-\textit{invariant cylinder function} if it arises as $g\circ \tilde{p}$ where:
\begin{itemize}
\item the map $\tilde{p}:\Co(Y,\spin)\rightarrow \R^n\times\mathbb{T}\times\mathbb{H}^m$ is defined as above;
\item the function $g$ is a $\Pin$-invariant function on $\R^n\times\mathbb{T}\times\mathbb{H}^m$ with compact support.
\end{itemize}
\end{defn}

Because of the symmetry of the problem, the perturbed functional $\CSd=\CSD+f$ for a $\Pin$-invariant cylinder function $f$ will always have the fixed points of the involution $\jmath$ on $\mathcal{B}(Y,\spin)$ as critical points. The analogue of Proposition \ref{densepert} in Chapter $1$ also holds for $\Pin$-invariant cylinder functions, as reassumed in the following result.
\begin{prop}\label{largepert2}
If $f$ is a $\Pin$-invariant cylinder function, then its gradient
\begin{equation*}
\mathrm{grad}f: \Co(Y)\rightarrow \T_0
\end{equation*}
is a tame perturbation (in the sense of Definition \ref{tamepert} in Chapter $1$) which is also $\jmath$-equivariant.
\par
Furthermore, for every compact subset $K$ of a finite dimensional submanifold $M$ of the base configurations space $\Bo_k^o(Y)$, both $\Pin$-invariant, one can find a collection of $c_{\mu}$, $\Upsilon_{\nu}$ and a neighborhood $U\supset K$ such that
\begin{equation*}
p:\Bo_k^o(Y)\rightarrow\R^n\times\mathbb{T}\times \mathbb{H}^m
\end{equation*}
gives an embedding of $U$. Hence for any $[B,\Psi]\in\Bo_k^*(Y)$ and any non zero tangent vector $v$ there exists a $\Pin$-invariant cylinder function whose differential $\mathcal{D}_{[B,\Psi]}f(v)$ is non-zero.
\end{prop}
\begin{proof}
The formal gradient of a $\jmath$-invariant function is a $\jmath$-equivariant vector field. Furthermore, because our map $\tilde{p}$ is determined in a very simple explicit way by the map $p$ defined by means of the same family of coclosed forms and sections of $\mathbb{S}$, it is immediate to see that the estimates of Chapter $11$ in the book continue to hold with small adaptations, hence the gradient of $f$ is a tame perturbation in the usual sense. The second part is clear as if we compose the map $\tilde{p}$ with the natural projection map on each $\mathbb{H}$ factor
\begin{IEEEeqnarray*}{c}
\mathbb{H}\rightarrow \C\\
z+jw\mapsto z,
\end{IEEEeqnarray*}
we obtain the corresponding map $p$, which in turn can be made into an embedding by Proposition \ref{densepert} in Chapter $1$. The last sentence then follows because the fixed points of the action are reducible.
\end{proof}

Finally, as in the end of Chapter $1$, and in particular Proposition \ref{largepert}, we can construct a Banach space $\widetilde{\mathcal{P}}$ with a linear map
\begin{align*}
\mathfrak{D}:\widetilde{\mathcal{P}}&\rightarrow C^0(\Co(Y),\mathcal{T}_0)\\
\lambda& \mapsto \q_{\lambda}
\end{align*} 
such that every $\q_{\lambda}$ is a $\jmath$-equivariant tame perturbation and the image contains the formal gradient of each element of a countable family of $\Pin$-invariant cylinder function $\{f_i\}$ arising as follows. For every pair $(n,m)$, choose a countable collection of $(n+m)$-tuples $(c_1,\dots, c_n,\Upsilon_1,\dots, \Upsilon_m)$ which are dense in the $C^{\infty}$ topology in the space of all such $(n+m)$-tuples. Choose a countable collection of compact subsets $K\subset \R^n\times\mathbb{T}\times \mathbb{H}^m$ which is dense in the Hausdorff topology and, for each $K$, a countable collection of $\Pin$-invariant functions $g_{\alpha}(K)$ with support in $K$ which are dense (in the $C^{\infty}$ topology) in the space of $\Pin$-invariant functions with support in $K$ and such the subset
\begin{equation*}
\left\{g_{\alpha}\mid g_{\alpha}\lvert_{K\cap(\R^n\times\mathbb{T}\times\{0\})}=0\right\}
\end{equation*}
is dense (in the $C^{\infty}$ topology) in the space of $\Pin$-invariant functions with support on $K$ and vanishing on $K\cap (\R^n\times \mathbb{T}\times\{0\})$. Then the family $\{f_i\}$ is
obtained composing in all possible ways the functions arising from our choices. We call such a pair $(\widetilde{\mathcal{P}},\mathfrak{D})$ (oftenly denoted just by $\widetilde{\mathcal{P}}$) a \textit{large Banach space of $\jmath$-equivariant perturbations}.

\vspace{0.8cm}
We now focus on the transversality problem. Notice that the same proof of Lemma \ref{quaternion} applies to show that if $[B,0]$ is one of the $2^{b_1(Y)}$ fixed points of $\jmath$ on $\Bo(Y)$, then for any $\jmath$-equivariant tame perturbation $\q$ the eigenspaces of the operator $D_{B,\q}$ are naturally quaternionic vector spaces, hence their spectrum (as complex-linear operators) cannot be simple. Furthermore, the points $[B,0]$ are always critical points for the perturbed Chern-Simons-Dirac operator when the perturbing function is a $\Pin$-cylinder function, hence our critical points will never be non-degenerate in the sense of the book. In light of our discussion, it is natural to introduce the following definition (see also Proposition $12.2.5$ in the book).

\begin{defn}
Let $\acr=(B,r,\psi)\in\Cs_k(Y)$ be a critical point of the $\jmath$-invariant vector field $(\mathrm{grad}\CSd)^{\sigma}$. We say that $\acr$ is \textit{$\Pin$-non-degenerate} if the following conditions hold:
\begin{enumerate}
\item if $r\neq0$, then the corresponding point $(B,r\psi)\in \Co_k^*(Y)$ is non-degenerate in the usual sense, i.e. $\mathrm{grad}\CSd$ is transverse to the subbundle $\J_{k-1}\subset \T_{k-1}$;
\item if $r=0$, then $B\in\mathcal{A}_k$ is a non-degenerate zero of $(\mathrm{grad}\CSd)^{\mathrm{red}}$ (i.e. $(\mathrm{grad}\CSd)^{\mathrm{red}}$ is transverse to the subbundle $\J^{\mathrm{red}}_{k-1}\subset \T^{\mathrm{red}}_{k-1}$) and furthermore, if $\lambda$ is the eigenvalue of $D_{B,\q}$ corresponding to $\psi$ then $\lambda$ is not zero, and
\begin{itemize}
\item if $B$ is \textit{not} gauge equivalent to a spin connection, the $\lambda$ eigenspace has complex dimension one;
\item if $B$ is gauge equivalent to a spin connection, the $\lambda$ eigenspace has complex dimension two.
\end{itemize}
\end{enumerate}
\end{defn}

A $\Pin$-non degenerate reducible critical point $(B,0,\psi)$ with $B$ not conjugate to a spin connection is non-degenerate in the usual sense, and its gauge equivalence class is an isolated singularity in $\Bs_k(Y)$. On the other hand when the connection $B$ is spin, the eigenspace in which $\psi$ lies is two dimensional, and its image in $\Bs_k(Y)$ is a two dimensional sphere arising as the quotient by the action of $S^1$ of the unit sphere of the eigenspace. This quotient is  a Hopf fibration. Furthermore, the next result tells us that in order to study our problem we can rely on the machinery developed in Chapters $2$ and $3$.

\begin{prop}
Suppose a $\jmath$-equivariant perturbation $\q$ such that the critical points of $(\mathrm{grad}\CSd)^{\sigma}$ are $\Pin$-non-degenerate. Then the singularities are Morse-Bott.
\end{prop}
\begin{proof}
We will focus on the case of a reducible critical point $\acr=(B,0,\psi)$ where $B$ is gauge equivalent to a spin connection, which is the only one which needs some adaptations from the proof in Chapter $12$ of the book. Let $\lambda$ be the eigenvalue corresponding to $\acr$. In this case, as in the proof of Lemma $12.4.3$ in the book), the extended Hessian $\widehat{\Hess}_{\q,\alpha}$ can be written as the map
\begin{align*}
(b,r,\psi,c)&\mapsto \\
&(-dc+\ast db+2\mathcal{D}_{(B_0,0)}\q^0(b,0), \lambda r,\\
&\Pi^{\perp}\big[(D_{\q,B_0}-\lambda)\psi+\rho(b)\psi_0+c\psi_0+\mathcal{D}^2\q^1_{(B_0,0)}((b,0),(0,\psi_0))\\
&+(r/2)\mathcal{D}^2\q^1_{(B_0,0)}((0,\psi_0),(0,\psi_0))\big],\\
&-d^*b+i|\psi_0|^2\mathrm{Re}\mu_Y(\langle i\psi_0,\psi\rangle)).
\end{align*}
where $\Pi^{\perp}$ is the orthogonal projection to the real-orthogonal complement of $\psi_0$. We further analyze this operator by decomposing
\begin{IEEEeqnarray*}{c}
c=i\epsilon_1+\hat{c}\\
\psi=i\epsilon_2\psi_0+\hat{\psi}
\end{IEEEeqnarray*}
where the $\epsilon_j$ are constants, $\hat{c}$ and $\hat{\psi}$ are orthogonal respectively to the constants and to $i\psi_0$ (with the real $L^2$ inner product). Exploiting $S^1$-equivariance and the fact that the configuration is a reducible critical point (hence $\psi_0$ is an eigenvector of $D_{\q,B_0}$), one can write the operator in the block lower triangular form
\begin{equation*}
\begin{bmatrix}
r \\ \epsilon_1 \\ \epsilon_2 \\ \hat{c} \\ b \\ \hat{\psi}
\end{bmatrix}
\mapsto
\begin{bmatrix}
\lambda & 0 & 0 & 0 & 0 & 0\\
\ast & 0& -1 & 0 & 0 & 0 \\
\ast & -1 & 0 & 0 & 0 & 0 \\
\ast & \ast & \ast & 0 & -d^* & 0\\
\ast & \ast & \ast & -d & *d_{\q} & 0\\
\ast & \ast & \ast & \ast & \ast & D_{\q, B_0}-\lambda
\end{bmatrix}
\begin{bmatrix}
r \\ \epsilon_1 \\ \epsilon_2 \\ \hat{c} \\ b \\ \hat{\psi}
\end{bmatrix}
\end{equation*}
where $\ast d_{\q}$ is the operator
\begin{equation*}
b\mapsto \ast db+2\mathcal{D}_{(B_0,0)}\q^0(b,0)
\end{equation*}
and $D_{\q,B_0}-\lambda$ is acting on $(\mathbb{C}\psi_0)^{\perp}$. Also, notice that the tangent space $\T_{\acr} C$ (in the sense of Section $1$ of Chapter $2$) to the critical submanifold is contained in the latter space, as it is simply the complex orthogonal inside the eigenspace.
\par
The interesting part of the operator, i.e. the one corresponding to the Hessian $\Hess^{\sigma}_{\q,\acr}$, is the operator on
\begin{equation*}
\R\oplus K\oplus (\C\psi_0)^{\perp}
\end{equation*}
(where $K$ is the space of coclosed imaginary $1$-forms) given by
\begin{equation*}
\begin{bmatrix}
\lambda & 0 & 0 \\
\ast & \ast d_{\q}|_K & 0 \\
\ast & \ast & D_{\q, B_0}-\lambda
\end{bmatrix}.
\end{equation*}
The $\Pin$-non-degeneracy of $\acr$ implies that $\lambda\neq 0$ and that $\ast d_{\q}|_K$ is an isomorphism, as it is  the restriction of the operator $\mathcal{D}_{B
}(\mathrm{grad}\CSd^{\mathrm{red}})$. On the other hand the operator, as $D_{\q,B}$ has a two dimensional $\lambda$-eigenspace spanned (over the complex numbers) by $\psi_0$ and $\T_{\acr}C$, and as $\T_{\acr}C$ is contained in $(\C\psi_0)^{\perp}$, the operator $D_{\q,B_0}-\lambda$ descends to an isomorphism on
\begin{equation*}
D_{\q,B_0}-\lambda:(\C\psi_0)^{\perp}/ \T_{\acr}C\rightarrow (\C\psi_0)^{\perp}/ \T_{\acr}C.
\end{equation*}
This is equivalent to the fact that $\Hess_{\q,\acr}^{\sigma,\nu}$ is an isomorphism, hence that the singularity is Morse-Bott.
\end{proof}
\begin{remark}
If the connection is the base one, the Dirac operator is genuinely quaternionic, hence the tangent space $\T_{\acr}C$ at a critical point $\acr=(B_0,0,\psi_0)$ is simply the complex span of $\psi_0\cdot j$.
\end{remark}

\vspace{0.8cm}
We now want to show that for a generic choice of the $\jmath$-equivariant perturbation $\q$ one can achieve $\Pin$-non-degeneracy at all critical points and regularity in the Morse-Bott sense (see Definition \ref{regularpert} in Chapter $2$). We will prove the following transversality result (see also Theorem $12.1.2$ in the book).
\begin{teor}\label{pinreg}
Let $\widetilde{\mathcal{P}}$ be a large Banach space of tame $\jmath$-equivariant perturbations. Then there is a residual subset of perturbations such that for every $\q$ in such a subset all the critical points of $(\mathrm{grad}\CSd)^{\sigma}$ are $\Pin$-non-degenerate, hence Morse-Bott, and all moduli spaces are regular.
\end{teor}

\begin{proof}
The proof follows with few small complications in the same way as in Section $12.5$ and $12.6$ and Chapter $15$ in the book. For the $\Pin$-non-degeneracy of the critical points, it is clear that the only new issues arise when dealing with the critical points of the form $(B,0,\psi)$ where $B$ is gauge equivalent to a spin connection. We show that in fact these critical points are non-degenerate for an open dense subset of perturbations in $\widetilde{\mathcal{P}}$. From this we can then carry the proofs in Sections $12.5$ and $12.6$ in the book to achieve transversality at every critical point
\par
We first show that for a generic choice of the $\jmath$-equivariant perturbation, the operators $D_{B,\q}$ with $B$ a spin connection have two dimensional eigenspaces and zero is not an eigenvalue. This set is also open, because we are considering only finitely many gauge equivalence classes. Following then Section $12.6$, we show that if
\begin{equation*}
\widetilde{\mathcal{P}}^{\perp}\subset \widetilde{\mathcal{P}}
\end{equation*}
is the set of tame perturbations which vanish along the reducible locus of $\Co(Y)$, we can find a generic subset of perturbations in $\widetilde{\mathcal{P}}$ such that $D_{B,\q}$ has two dimensional eigenspaces. By gauge invariance we just need to show this for one of the $2^{b_1(Y)}$ spin connections, and for simplicity we only consider the case of the base one $B_0$ (as the others can be obtained by changing the quaternionic structure).
\par
In this case, we consider the space $\mathrm{Op}^{\mathbb{H}}$ consisting of self adjoint \textit{quaternionic linear} Fredholm maps
\begin{equation*}
L^2_k(Y;S)\rightarrow L^2_{k-1}(Y;S)
\end{equation*}
of the form $D_{B_0}+h$, where $h$ is a quaternionic self-adjoint operator which extends to a bounded operator on all $L^2_j$ for $j\leq k$. This space is stratified by the dimension of the kernel, and the set of operators whose (quaternionic) spectrum is not simple is a countable image of Fredholm maps with negative index
\begin{IEEEeqnarray*}{c}
F_n:\mathrm{Op}^{\mathbb{H}}_n\times \R\rightarrow \mathrm{Op}^{\mathbb{H}}\\
(L,\lambda)\mapsto L+\lambda,
\end{IEEEeqnarray*}
where $\mathrm{Op}^{\mathbb{H}}_n\subset\mathrm{Op}^{\mathbb{H}}$ is the space of operators with kernel of dimension exactly $n$. Our claim is that the map
\begin{IEEEeqnarray*}{c}
\tilde{M}: \mathcal{P}^{\perp}\rightarrow \mathrm{Op}^{\mathbb{H}}\\
\q^{\perp}\mapsto D_{B,\q^{\perp}}
\end{IEEEeqnarray*}
is transverse to both the stratification of the maps according to the dimension of the kernel and the Fredholm maps $\{F_n\}$, from which the proof follows as in Theorem $12.1.2$ in the book. Suppose $\q^{\perp}=\mathrm{grad}f$ is a be a perturbation in $\mathcal{P}^{\perp}$, where $f=g\circ \tilde{p}$ with $g$ vanishing along $\R^n\times\mathbb{T}\times\{0\}$.
If we denote by $V$ the kernel of $D_{B_0,\q}$, the tangent space to the stratification by the dimension of the kernel at the point
\begin{equation*}
D_{B_0,\q}\in \mathrm{Op}^{\mathbb{H}}\
\end{equation*}
is the kernel of the compression map from $\mathrm{Op}^{\mathbb{H}}$ to the self adjoint quaternionic operators on $V$.
Furthermore, $V$ can be regarded as a subspace of the normal bundle to $\mathcal{A}_k$ in $\Co_k(Y)$, and by Proposition \ref{largepert2} we can choose a $p$ (defined by a collection of coclosed forms and sections in the countable collection defining the large space of perturbations) whose differential embeds this into a linear subspace of
\begin{equation*}
\mathbb{H}^m\subset T_{(0,0,0)}(\R\times\mathbb{T}\times \mathbb{H}^m)
\end{equation*} 
in a $\Pin$-equivariant fashion. We have that $p(B,0)$ is necessarily
\begin{equation*}
(0,0,0)\in\R^n\times\mathbb{T}\times \mathbb{H}^m
\end{equation*}
by $\Pin$-equivariance. By choosing a $\Pin$-invariant function $\delta g$ (which can we suppose in our defining collection) on $\R\times\mathbb{T}\times \mathbb{H}^m$ vanishing along $\R\times\mathbb{T}\times \{0\}$ we can find a perturbation
\begin{equation*}
\delta\q^{\perp}=\mathrm{grad}(\delta g\circ p)\in\mathcal{P}^{\perp}
\end{equation*}
such that the Hessian of $(\delta g\circ p)\lvert_V$ is any chosen $\Pin$-equivariant (i.e. quaternionic) self-adjoint endomorphism of $V$, which proves our claim. The second claim follows with the same proof from the fact that the normal bundle to the image of $F_n$ at $L+\lambda$ is naturally isomorphic to the space of traceless, self-adjoint quaternionic endomorphisms of $\mathrm{Ker}(L)$.
\par
We then turn to the arrange that $(\mathrm{grad}\CSd)^{\mathrm{red}}$ is a transverse section downstairs at these $2^{b_1(Y)}$ points. Of course the set of perturbations in $\widetilde{\mathcal{P}}$ for which this condition holds is open. To show that it is dense, we need to study the surjectivity at a configuration $(B,0)$ with $B$ a spin connection of the map 
\begin{align*}
(\mathrm{grad}\CSd)^{\mathrm{red}}&:\mathcal{K}^{\mathrm{red}}_{k}\rightarrow\mathcal{K}^{\mathrm{red}}_{k-1}\\
b&\mapsto \ast db+2\mathcal{D}_{(B,0)}\q^0(b,0).
\end{align*}
where $\mathcal{K}^{\mathrm{red}}_{j}$ is the space of coclosed $1$-forms. As $p(B,0)$ is necessarily $(0,0,0)$ the second term involves the Hessian of the function $g$ at the origin. The requirement that $g$ is $\Pin$-invariant implies that it is an even function when restricted to $\R^n\times\mathbb{T}\times \{0\}$. The result then follows by choosing the map $p$ so that its differential is an embedding on the kernel of the linearization of $(\mathrm{grad}\CSd)^{\mathrm{red}}$ and the fact that the Hessian of an even function on $\mathbb{R}^n\times\mathbb{T}$ at the origin can be any self-adjoint linear operator.
\par
The proof of transversality for moduli spaces follows with the similar adaptations from the one in Chapter $15$ in the book and Theorem \ref{transversalitymain} in Chapter $2$. It is clear that if the path of configurations contains a point which is not reducible with the connection equivalent to a spin one then the proof of transversality in the classical case applies without any significant change. The only observation to be made in order to run the transversality argument is that if a trajectory $\gamma$ connecting two critical points is such that there exists a time $t_0$ and $T>0$ such that
\begin{equation*}
[\check{\gamma}(t_0)]=[\jmath\cdot \check{\gamma}(t_0+T)]
\end{equation*}
then the trajectory is constant (this is required in order to be able to find a perturbation that has non zero inner product with an element in the cokernel of the map defining the universal moduli space, see equation ($15.3$) in the book). On the other hand this relation implies that $[\check{\gamma}(t_0+2T)]$ is the same as $[\check{\gamma}(t_0)]$, hence the trajectory is constant.
\par
Finally, the only possible complication regards the trajectories that connect critical points $[\acr]$ and $[\bcr]$ with blow down the reducible critical point $[B,0]$ with $B$ equivalent to a spin connection, and such that the homotopy class is trivial. On the other hand, the proof of transversality in this situation in Theorem \ref{transversalitymain} in Chapter $2$ only uses only the fact that the operator $D_{\q,B}$ is hyperbolic (and not that its spectrum is simple), so it still holds in our case without any modification.
\end{proof}

\vspace{0.8cm}
The set of perturbations we have introduced is enough to define the Floer chain complexes, but is not sufficient to prove the invariance and functoriality properties. The issue is that for such perturbations, on a spin cobordism the spin connection is always a solution of the Seiberg-Witten equations (in the blow down) hence transversality might not hold simply because of index reasons. A finite dimensional example of this phenomenon is that on the circle $S^1$ there is no regular (in the sense of Morse homology) family of functions connecting $\cos \vartheta$ and $-\cos \vartheta$ which is conjugation invariant at each point. We introduce further perturbations in the blow up as follows.

\begin{defn}\label{ASDpert}
Consider a self-conjugate spin$^c$ structure $\spin$ on $X$.
For $k\geq 2$, a $k$-tame \textit{$\Pin$-equivariant \textsc{asd}-perturbation} is a map
\begin{equation*}
\hat{\omega}: \Cs(X,\spin)\rightarrow L^2(X,i\mathfrak{su}(S^+))
\end{equation*}
with the following properties:
\begin{enumerate}
\item the map $\hat{\omega}$ is gauge invariant and $\jmath$-equivariant, where $\jmath$ acts on the right hand side as the multiplication by $-1$;
\item the map extends to an element
\begin{equation*}
\hat{\omega}\in C^{\infty}\left(\Cs_k(X), L^2_{k+1}(X,i\mathfrak{su}(S^+))\right);
\end{equation*}
restricting to zero at the boundary of $X$;
\item for every integer $j\in[-k,k]$, the first derivative
\begin{equation*}
\mathcal{D}\hat{\omega}\in C^{\infty}\left( \Cs_k(X), \mathrm{Hom}(T\Cs_k(X), L^2_{k+1}(X,i\mathfrak{su}(S^+)))\right)
\end{equation*}
extends to a map
\begin{equation*}
\mathcal{D}\hat{\omega}\in C^{\infty}\left( \Cs_k(X), \mathrm{Hom}(T\Cs_j(X), L^2_j(X,i\mathfrak{su}(S^+)))\right);
\end{equation*}
\item the image of $\hat{\omega}$ is precompact in the $L^2_{k+1}$ topology.
\end{enumerate}
We say that $\hat{\omega}$ is \textit{tame} if it is tame for every $k\geq 2$.
\end{defn}

It is important to remark that these perturbations are purely $4$-dimensional and are defined in the blow up, hence also on $\Co^*(X,\spin)$ via the blow down map. The conditions we impose are very similar to those of a tame perturbations (Definition \ref{tamepert} in Chapter $1$), and in fact they are introduced for the same reasons. In particular, condition $(2)$ is require to set up the perturbed equations and condition $(4)$ has the same role as condition $(5)$ in Definition \ref{tamepert} in Chapter $1$. A slight difference is that our requirements involve the $L^2_{k+1}$ norm, which will be needed in the proofs of compactness.
\\
\par
Suppose now we are in the same setting as is Section $6$  of Chapter $2$, so our manifold $X$ has cylindrical end $Z=I\times Y$ and we are given $k$-tame perturbations $\hat{\q},\hat{\mathfrak{p}}_0$ and cut-off functions $\beta,\beta_0$. Given a $k$-tame $\Pin$-equivariant \textsc{asd}-perturbation $\hat{\omega}$ we can define the gauge invariant and $\jmath$-equivariant section
\begin{equation*}\F^{\sigma}_{\mathfrak{p},\hat{\omega}}=\F^{\sigma}+\hat{\mathfrak{p}}^{\sigma}+\hat{\omega}: \Cs_k(X,\spin_X)\rightarrow \V^{\sigma}_{k-1},
\end{equation*}
and the perturbed Seiberg-Witten equations $\F^{\sigma}_{\mathfrak{p},\hat{\omega}}=0$. The definition the moduli spaces in this context carries over without particular modifications.
\par
These moduli spaces satisfy the same properties discussed in Section $6$ of Chapter $2$. We discuss the most significant one, namely the proof of the compactness of the moduli spaces on $X$ (see Theorem $24.5.2$ in the book). Here, we write
\begin{equation*}
X_{\epsilon}=X\setminus\left( (-\epsilon,0]\times Y\right),
\end{equation*}
and we take $X'\subset X_{\epsilon}$.
\begin{prop}\label{compasd}
Let $\gamma_n$ be a sequence in $\Cs_k(X')$ of solutions to the perturbed equations $\F^{\sigma}_{\mathfrak{p},\hat{\omega}}=0$. Suppose that there is a uniform bound on the perturbed topological energy
\begin{equation*}
\mathcal{E}^{\mathrm{top}}_{\q}(\gamma_n)\leq C_1,
\end{equation*}
and that for each component $Y^{\alpha}$ of $Y$ there is a uniform upper bound
\begin{equation*}
-\Lambda_{\q}(\gamma_n\mid_{\{-\epsilon\}\times Y^{\alpha}})\leq C_2.
\end{equation*}
Then there is a sequence of gauge transformations $u_n\in \G_{k+1}(X)$ such that after passing to a subsequence the restrictions $u_n(\gamma_n)\lvert_{X'}$ converge in the topology of $\Cs_k(X')$ to a solution $\gamma\in \Cs_k(X')$.
\end{prop}
\begin{proof}
We focus on the case on the specail case of a sequence of solutions $(A_n, s_n, \phi_n)$ on a cylinder $Z=I\times Y$ (where the perturbation $\mathfrak{p}$ is time independent so it is not supported in a collar). Indeed, the general case follows along the same lines of Theorem $24.5.2$ in the book once has the convergence result on cylinders. We can also assume that the sequence consists of entirely of irreducibles, as the case of a sequence of reducible solutions is easier. As we mentioned above, the perturbation $\hat{\omega}$ induces a perturbation downstairs on the irreducible locus, and write $\mathbf{\pi}(\gamma_n)=(A_n,\Phi_n)$.
\par
We introduce the perturbed analytical energy (compare Chapter $4$ in the book) on general manifold $X$ as
\begin{equation*}
\mathcal{E}^{\mathrm{an}}_{\q,\hat{\omega}}(A,\Phi)=\mathcal{E}^{\mathrm{top}}_{\q}+\|\F_{\mathfrak{p},\hat{\omega}}(A,\Phi)\|^2.
\end{equation*}
The imaginary valued $2$-form $\rho_Z^{-1}(\hat{\omega}(A,\Phi))$ is anti-self dual, so it can be written as
\begin{equation*}
\rho_Z^{-1}(\hat{\omega}(A,\Phi))=\frac{1}{2}\left( dt\wedge \omega_t(A,\Phi)+ \ast \omega_t(A,\Phi)\right)
\end{equation*}
where $\omega_t(A,\Phi)$ is a family of imaginary valued $1$-forms on $Y$ and $\ast$ is the Hodge star on the three manifold. These elements are all in $L^2_{k-1}$ because of Condition ($1$) in Definition \ref{ASDpert}. In particular we have
\begin{equation*}
\hat{\omega}(A,\Phi)=-\rho(\omega_t(A,\Phi))
\end{equation*}
as elements of $\mathfrak{sl}(S^+)\cong \mathfrak{sl}(S)$. Hence if $(A,\Phi)$ is a configuration in temporal gauge that solves $\F^{\sigma}_{\mathfrak{p},\hat{\omega}}=0$ then the path $(\check{A}(t),\check{\Phi}(t))$ is a flow line for the (partially defined) vector field
\begin{equation*}
\mathrm{grad}\CSd+(\hat{\omega}_t,0)
\end{equation*}
where $\hat{\omega}_t$ depends on the configuration on the whole cylinder. We can then write, for a configuration in temporal gauge,
\begin{align*}
\mathcal{E}^{\mathrm{an}}_{\q,\hat{\omega}}(A,\Phi)&=2(\CSd(t_1)-\CSd(t_2))+\int_{t_1}^{t_2}\|\frac{d}{dt}\check{\gamma}(t)+\mathrm{grad}\CSd+(\hat{\omega}_t,0)\|^2dt=\\
&\int_{t_1}^{t_2}\|\frac{d}{dt}\check{\gamma}\|^2+\|\mathrm{grad}\CSd+(\hat{\omega}_t,0)\|^2dt+2\int_{t_1}^{t_2}\langle \frac{d}{dt}\check{\gamma},(\hat{\omega}_t,0)\rangle.
\end{align*}
This can be written in a gauge invariant fashion as
\begin{align*}\label{analenerg}
\mathcal{E}^{\mathrm{an}}_{\q,\hat{\omega}}(A,\Phi)=\int_{t_1}^{t_2}\left\|\frac{d}{dt}\check{A}-dc\right\|^2dt+\int_{t_1}^{t_2}\left\|\frac{d}{dt}\check{\Phi}+c\Phi\right\|^2dt\\+\int_{t_1}^{t_2}\|\mathrm{grad}\CSd+(\hat{\omega}_t,0)\|^2dt+
4\int_{t_1}^{t_2} \mathrm{tr}\left((\frac{d}{dt}\check{A}-dc)\rho^{-1}_Z(\hat{\omega})\right)dt.
\end{align*}
Because of condition $(4)$ in Definition \ref{ASDpert}, the $L^2$-norm of $\hat{\omega}(A,\Phi)$ is bounded independently of $(A,\Phi)$, so by applying the Peter-Paul inequality to the last two terms we have that that
\begin{equation*}
\mathcal{E}^{\mathrm{an}}_{\q,\hat{\omega}}\geq \frac{1}{2}\mathcal{E}^{\mathrm{an}}_{\q}- C
\end{equation*}
for some constant $C$ independent of $(A,\Phi)$ and as in equation $(10.12)$ in the book
\begin{equation*}
\frac{1}{2}\mathcal{E}^{\mathrm{an}}_{\q}=\int_{t_1}^{t_2}\left\|\frac{d}{dt}\check{A}-dc\right\|^2dt+\int_{t_1}^{t_2}\left\|\frac{d}{dt}\check{\Phi}+c\Phi\right\|^2dt\\+\int_{t_1}^{t_2}\|\mathrm{grad}\CSd\|^2dt.
\end{equation*}
Then, Lemma $10.6.1$ in the book implies that for some constants $C',C''$ we have an inequality of the form
\begin{equation*}
\mathcal{E}^{\mathrm{an}}_{\q,\hat{\omega}}(A,\Phi)\geq \frac{1}{4}\int_Z\left(\frac{1}{4}|F_{A^t}|^2+|\nabla_A\Phi|^2+\frac{1}{4}(|\Phi|^2-C')^2\right)-C''(t_2-t_1).
\end{equation*}
From this we see that a bound on the topological energy implies as in the usual case an $L^2$ bound on $|F_{A_n^t}|$ and $|\nabla_{A_n}\Phi_n|$ and an $L^4$ bound on $|\Phi|$. As in the proof of Theorem $10.7.1$ in the book, after suitable gauge transformations the usual tame perturbations $\hat{\mathfrak{p}}(A_n,\Phi_n)$ have a subsequence converging in the $L^2$ topology. Furthermore, condition $(5)$ in Definition \ref{ASDpert} implies that we can pass to a subsequence for which $\hat{\omega}(A_n,\Phi_n)$ converges in the $L^2_{k+1}$ topology. Hence because of the properness of the Seiberg-Witten map we can pass to a subsequence that converges up to gauge in the interior domains in the $L^2_1$ topology (see Theorem $5.2.1$ in the book). From this we can prove convergence in the $L^2_2$ topology, using the fact that $\hat{\omega}(A_n,\Phi_n)$ converges in the $L^2_{k+1}$ topology and the standard perturbation $\hat{\mathfrak{q}}(A_n,\Phi_n)$ converges in the $L^2_1$ topology. We can continue like this to prove convergence in any interior domain in the $L^2_{k+1}$ topology.
\par
We need now to prove convergence in the blow up, so suppose we are given a sequence of irreducible solutions $(A_n,\Phi_n)$ converging in the $L^2_{k+1}$ topology so that the $L^2$ norm of $\Phi_n$ going to zero. We need to show that the renormalized spinors $\Phi_n/\|\Phi_n\|_{L^2}$ converge in the $L^2_{k+1}$ topology. Because our new perturbations only affect the connection component, the proof of Proposition $10.8.1$ in the book carries over with no modification to show that the if the spinor $\Phi_n$ restricts to zero in a slice $\{t\}\times Y$ then it is identically zero. In particular we can pass to the $\tau$-model description of the moduli spaces. The fact that $\hat{\omega}(A_n,\Phi_n)$ is convergent in the $L^2_{k+1}$ topology implies that
\begin{equation*}
t\mapsto \mathrm{grad}\CSd(\check{\gamma}_n(t))+(\hat{\omega}_t(\gamma_n),0)
\end{equation*}
is an $L^2$ path in the Hilbert space $L^2_k(Y;i T^*Y\oplus S)$ converging in the topology of $L^2$ paths. The proof of Proposition $10.9.1$ in the book carries over without modifications to show that there is a continuous positive function $\zeta$ on $\Cs_k(Y)$ such that for any solution $\gamma$ of the equations $\F_{\mathfrak{p},\hat{\omega}}(\gamma)$ we have that
\begin{equation*}
\frac{d}{dt}\Lambda_{\q}(\check{\gamma}^{\tau}(t))\leq \zeta(\check{\gamma}(t))\|\mathrm{grad}\CSd(\check{\gamma}(t))+(\omega_t(\gamma),0)\|_{L^2_k(Y)}.
\end{equation*}
Notice that the right hand side is only defined for almost every $t$ and is locally square integrable.
This is again because the additional perturbation only involves the connection components, while the proof of the inequality follows from differentiating the equation for the spinor part. Together with the upper bound on $\Lambda_{\q}$, this implies that we have a uniform bound
\begin{equation*}
|\Lambda_{\q}(\check{\gamma}_n(t))|\leq M,
\end{equation*}
so we can conclude as in the usual proof (as it again only involves manipulations with the spinor part).
\end{proof}

\begin{remark}
It is important to remark that for these perturbed equations the proof of unique continuation (see Proposition $10.8.2$ in the book) does not hold, because the perturbation in the connection component at a single time depends on the whole configuration on the cylinder. On the other hand, we will only need to use the unique continuation property for the spinor component discussed in the proof above.
\end{remark}

\vspace{0.8cm}

We now show how to construct a Banach space of such tame $\Pin$-equivariant \textsc{asd}-perturbations which is large enough to achieve transversality. We introduce a simple class of perturbations called \textit{squared projections}, and show that they define tame $\textsc{asd}$-perturbations. Consider an embedded closed ball $B^4$ together with a base point $p_0$. Consider a \textit{finite dimensional} subspace of compactly supported spinors
\begin{equation*}
V\subset C^{\infty}_c(B^4; S^+)
\end{equation*}
which is quaternionic, i.e. invariant under the action of $j$, and let $\pi_V$ be the $L^2$ orthogonal projection to the subspace $V$. On $B^4$ there is a unique spin connection $A_0$. For any other connection $A=A_0+a$, there is a unique gauge transformation $u_A$ such that $u_A\cdot A$ is in Coulomb-Neumann gauge with respect to $A_0$ and $u_A(p_0)$ is $1$. We also have
\begin{equation*}
u_A=e^{Ha}
\end{equation*}
where $H$ is a smoothing operator of order $1$. Finally, let $f$ be a compactly supported smooth real valued function in $B^4$. We then define the map \textit{squared projection} associated to the quadruple $(B^4,p_0, V,f)$ as the map
\begin{align*}
\hat{\omega}_{(B^4,p_0,V,f)}:\Cs(Z,\spin_Z) & \rightarrow L^2(X, i\mathfrak{su}(2))\\
(A,s,\phi)&\mapsto f\cdot\left(\pi_V(\phi \cdot u_A)\pi_V(\phi\cdot u_A)^*\right)_0.
\end{align*}
The main analytic result is the following.

\begin{prop}
Given a triple $(B^4,p_0,V,f)$ as above, the squared projection
\begin{equation*}
\hat{\omega}_{(B^4,p_0,V,f)}:\Cs(Z,\spin_Z)\rightarrow L^2(X, i\mathfrak{su}(2))
\end{equation*}
is a tame $\Pin$-equivariant \textsc{asd}-perturbation in the sense of Definition \ref{ASDpert}.
\end{prop}
\begin{proof}
Gauge invariance is clear from the definition. To see the equivariance under the $j$ action, one notices that
\begin{equation*}
u_{\bar{A}}=\bar{u}_A
\end{equation*}
hence we have
\begin{equation*}
\pi_V(\phi\cdot j\cdot u_{\bar{A}})=\pi_V(\phi\cdot u_A \cdot j)=\pi(\phi\cdot u_A)\cdot j.
\end{equation*}
In the last step we used the fact that the $L^2$ projection to the quaternionic subspace $V$ is quaternionic linear, which easily follows from the fact that
\begin{equation*}
\langle \phi_1\cdot j, \phi_2\cdot j\rangle_{L^2(X)}=\overline{\langle \phi_1,\phi_2\rangle}_{L^2(X)}.
\end{equation*}
We then prove condition $(4)$. We have for a constant $C$ independent of the configuration
\begin{multline*}
\|\hat{\omega}_{(B^4,p_0,V,f)}(A,s,\phi)\|_{L^2(X)}=\|\pi_V(\phi \cdot u_A)\|_{L^4(X)}\leq\\
\leq C \|\pi_V(\phi \cdot u_A)\|_{L^2(X)}\leq  C\|\phi\cdot u_A\|_{L^2(X)}=C.
\end{multline*}
In the second step we used the fact that $V$ is finite dimensional, so all norms are equivalent, while in the last step we used the fact that $\|\phi\|_{L^2(X)}=1$. In particular, because the image of the map is contained in a finite dimensional subspace, the map $\hat{\omega}_{(B^4,p_0,V,f)}$ is precompact in the $L^2_{k+1}$-topology.
\par
The proof of conditions $(2)$ and $(3)$ follows in the same way as the estimates in Chapter $11$ in the book. In fact, the proofs are much easier because our perturbation are defined in a purely four dimensional context, rather that being induced by three dimensional ones. The key estimate to be proven involves the derivatives of the function
\begin{equation*}
(A,s,\phi)\mapsto \phi\cdot u_A,
\end{equation*}
and in fact we only need to consider the $L^2$ norm of the target as again we are projecting to a finite dimensional subspace. Hence we are interested in bounding at a configuration $(A_0+a_0,s_0,\phi_0)$ the $L^2$ norms of the quantities
\begin{equation*}
(H\delta a_1)\cdots(H\delta a_n)e^{Ha_0}\delta\phi\quad\text{and}\quad(H\delta a_1)\cdots(H\delta a_n)e^{Ha_0}\phi_0,
\end{equation*}
where the $\delta$s indicate the corresponding tangent vectors. Because $k\geq 2$, the required bounds readily follow from the Sobolev multiplication theorem.
\end{proof}
\vspace{0.5cm}
We then turn in the construction of a Banach space of perturbations $\mathcal{P}_{\textsc{ASD}}$. Choose a countable collection of pointed balls $(B_{\alpha}, p_{\alpha})$ such that their union is the interior of $X$. Then for each $n$ we choose a countable collection of $n$-dimensional quaternionic subspaces $V_{\alpha}$ of $C^{\infty}_c(B_{\alpha},S^+)$ which is dense in the space of such subspaces. Finally choose a countable set of compactly supported smooth functions $f_{\alpha}$ which is dense in the $C^{\infty}$ topology of such space. From these we can construct a countable family $\{\hat{\omega}_i\}_{i\in \mathbb{N}}$ of squared projections. Using Floer's construction as in Section $11.6$ of the book, we prove the following result, see also Proposition \ref{largepert} in Chapter $1$ and Theorem $11.6.1$ in the book.
\begin{prop}
Given a family of squared projections $\{\hat{\omega}_i\}_{i\in \mathbb{N}}$ as above, there exists a separable Banach space $\mathcal{P}_{\textsc{asd}}$ and a linear map
\begin{align*}
\mathfrak{D}_{\textsc{asd}}:\mathcal{P}_{\textsc{asd}}&\rightarrow C^0\left(\Cs(X), L^2(X,i\mathfrak{su}(S^+))\right)\\
\lambda\mapsto \hat{\omega}_{\lambda}
\end{align*}
satisfying the following properties:
\begin{itemize}
\item for each $\lambda\in \mathcal{P}_{\textsc{asd}}$, the element $\hat{\omega}_{\lambda}$ is a tame $\textsc{asd}$ perturbation in the sense of Definition \ref{ASDpert}; 
\item the image of $\mathfrak{D}_{\textsc{asd}}$ contains all the perturbations in the collection $\{\hat{\omega}_i\}_{i\in \mathbb{N}}$;
\item for each $k\geq 2$ the map
\begin{align*}
\mathcal{P}_{\textsc{asd}}\times \Cs_k(X)&\rightarrow L^2_k(X,i\mathfrak{su}(S^+))\\
(\lambda, \gamma)&\mapsto \hat{\omega}_{\lambda}(\gamma)
\end{align*}
is a smooth map of Banach manifolds.
\end{itemize}
\end{prop}

\vspace{0.8cm}

Finally we prove the following transversality result, see also Proposition \ref{transversecob} in Chapter $2$.
\begin{prop}
Suppose we have a fixed $\Pin$-non-degenerate $\jmath$-equivariant $\q$ on $Y$. Consider the perturbation
\begin{equation*}
\hat{\mathfrak{p}}+\hat{\omega}=\beta(t)\hat{\q}+\beta_0(t)(\hat{\mathfrak{p}}_0)+\hat{\omega}
\end{equation*}
as above. Then there is a residual subset of $\mathcal{P}(Y,\spin)\times \mathcal{P}_{\textsc{asd}}(X)$ such that for all pairs $(\mathfrak{p}_0, \hat{\omega})$ in this set the moduli spaces of solutions $M(X^*,[\Cr])$ are regular in the sense of Definition \ref{regcob} in Chapter $2$. The statement for families of Proposition \ref{transversecob} in Chapter $2$ also holds.
\end{prop}
\begin{proof}
The proof goes as the proof of Proposition $24.4.7$ and $15.1.3$ in the book, and we quickly point out the complications. Consider first an irreducible configuration $\gamma$. As in the usual proof we need to show that the differential of the map appearing in Lemma $24.4.8$ defining the universal moduli space with boundary conditions
\begin{equation}\label{univasd}
(\mathfrak{M}, \tilde{R}_+):\Cs_k(X,\spin_X)\times \mathcal{P}(Y)\times \mathcal{P}_{\textsc{asd}}(X)\rightarrow \mathcal{V}^{\sigma}_{k-1}\times \Cs_{k-1/2}(Y,\spin)
\end{equation}
has dense image in the $L^2\times L^2_{1/2}$ topology. We restrict the study to configurations in Coulomb-Neumann gauge. Consider a configuration $(V,v)$ which is $L^2$ orthogonal to the image of the differential at a point. By elliptic regularity $V$ is in $L^2_1$ and an integration by parts argument shows that $v$ is the restriction at $\partial X$ of $V$. Because in our setting the perturbed Dirac equation still has the unique continuation property (see the proof of Proposition \ref{compasd}), the restriction of $\gamma$ to a slice $\{t\}\times Y$ on the collar is still irreducible, and we can show that $V$ vanishes by just using $\Pin$-equivariant tame perturbations. When $\gamma$ is reducible, the same argument works to show that its spinor part is vanishing. On the other hand, if the other component of $V$ does not vanish one can find a $\Pin$-equivariant $\textsc{asd}$-perturbation $\hat{\omega}$ in the defining family $\{\hat{\omega}_i\}_{i\in\mathbb{N}}$ such that
\begin{equation*}
\langle\mathfrak{m}({0, \hat{\omega}_i}(\gamma)), V\rangle_{L^2(X)}>0,
\end{equation*}
where $\mathfrak{m}$ is the differential of the map (\ref{univasd}) in the directions corresponding to the perturbations. Hence the result follows, and the case of a reducible configuration is analogous.
\end{proof}

\vspace{1.5cm}
\section{Invariant chains and Floer homology}

In this section we construct the analogue of Manolescu's recent invariants (\cite{Man2}) from the Morse homology approach, the $\Pin$ version of the Floer homology groups. This is done by exploiting the extra symmetry of the Floer chain complex that comes when the spin$^c$ structure is self-conjugate and we use a $\jmath$-equivariant perturbation.
\\
\par
Suppose from now on that a $\Pin$-non-degenerate $\jmath$-equivariant tame perturbation $\q$ which is also regular in the Morse-Bott sense as in Theorem \ref{pinreg} has been fixed. Recall from Section $1$ that we have a natural identification of a critical submanifold $[\Cr]$ corresponding to an eigenspace of $D_{\q,B}$ for a connection $B$ conjugate to a spin one with $S^2$, and the action of $\jmath$ on the moduli space of configurations is identified with the antipodal map on it. This action induces an involution (still denoted by $\jmath$) on the chain complex $C_*^{\mathcal{F}}([\Cr])$ sending the $\delta$-chain $[\sigma]=[\Delta,f]$ to the $\delta$-chain
\begin{equation*}
\jmath[\sigma]=[\Delta, \jmath\circ f].
\end{equation*}
If the family countable family of $\delta$-chains $\mathcal{F}$ is also invariant under this action then the involution is clearly a chain map on $C_*^{\mathcal{F}}([\Cr])$, i.e. we have
\begin{equation*}
\jmath\circ \partial=\partial\circ\jmath.
\end{equation*}
In particular the subspace of invariant $\delta$-chains
\begin{equation*}
C_*^{\mathcal{F}}([\Cr])^{\mathrm{inv}}=\{[\sigma]\mid \jmath[\sigma]=[\sigma]\}\subset C_*^{\mathcal{F}}([\Cr])
\end{equation*}
is a subcomplex. There is also the subcomplex $(\mathrm{id}+\jmath)\left(C_*^{\mathcal{F}}([\Cr])\right)$, which is also $\jmath$-invariant. The proof of the following lemma is straightforward.

\begin{lemma}\label{projplane}
The inclusion $(\mathrm{id}+\jmath)\left(C_*^{\mathcal{F}}([\Cr])\right)\hookrightarrow C_*^{\mathcal{F}}([\Cr])^{\mathrm{inv}}$ is a quasi-isomorphism, and the homology of both complexes is naturally $H_*(\mathbb{R}P^2)$.
\end{lemma}
Similarly, if $[\bcr]$ is an isolated critical point then also the configuration $[\jmath\bcr]$ is, hence there is a natural involution $\jmath$ on the chain complex $C^{\mathcal{F}}_*([\bcr]\cup[\jmath\bcr])$ exchanging the two points. We can define again the subcomplexes
\begin{equation*}
C^{\mathcal{F}}_*([\bcr]\cup\jmath[\bcr])^{\mathrm{inv}}\text{ and }(\mathrm{id}+\jmath)\left(C^{\mathcal{F}}_*([\bcr]\cup\jmath[\bcr])\right)
\end{equation*}
(which coincide in this case) and the analogous of Lemma \ref{projplane} is obvious in this case.  Furthermore the involution $\jmath$ also gives rise to a natural isomorphism (in the sense of Definition \ref{isomdelta} in Chapter $3$) between the moduli spaces
\begin{align*}
\jmath: \breve{M}^+_z([\Cr_-],[\Cr_+])&\rightarrow \breve{M}^+_{\jmath (z)}([\jmath\Cr_-],[\jmath\Cr_+])\\
\breve{\boldsymbol\gamma}=([\breve{\gamma}_1],\dots, [\breve{\gamma}_m])&\rightarrow ([\jmath\breve{\gamma}_1],\dots, [\jmath\breve{\gamma}_m]),
\end{align*}
with the additional property that $\ev_{\pm}\circ\jmath=\jmath\circ\ev_{\pm}$. This implies in particular that
given any $\delta$-chain $[\sigma]$ in a critical submanifold $[\Cr_-]$ we have the identity
\begin{equation*}
\jmath\left([\sigma]\times \breve{M}^+_z([\Cr_-],[\Cr_+]) \right)=\jmath[\sigma]\times \breve{M}^+_{\jmath (z)}(\jmath[\Cr_-],\jmath[\Cr_+]),
\end{equation*}
hence the operators $\partial^o_o, \partial^o_s, \partial^u_o, \partial^u_s$ and $\bar{\partial}^s_s,\bar{\partial}^s_u,\bar{\partial}^u_s,\bar{\partial}^u_u$ all commute with the action of $\jmath$ as they involve these fiber products.
\\
\par
The previous discussion implies that the chain complexes $(\check{C}_*,\check{\partial})$, $(\hat{C}_*,\hat{\partial})$ and $(\bar{C}_*,\bar{\partial})$ are all equipped with an involutory chain map $\jmath$. We can then define the subcomplexes consisting of the invariant chains
\begin{equation*}
(\check{C}_*^{\mathrm{inv}},\check{\partial}),\qquad(\hat{C}^{\mathrm{inv}}_*,\hat{\partial}),\qquad (\bar{C}^{\mathrm{inv}}_*,\bar{\partial}),
\end{equation*} 
and similarly the subcomplexes 
\begin{equation*}
\left((\mathrm{id}+\jmath\right)\check{C}_*,\check{\partial}),\qquad (\left(\mathrm{id}+\jmath\right)\hat{C}_*,\hat{\partial}),\qquad(\left(\mathrm{id}+\jmath\right)\bar{C},\bar{\partial}).
\end{equation*}
The following result readily follows along the same lines of Lemma \ref{moretransverse} in Chapter $3$ using Lemma \ref{projplane}.
\begin{lemma}
The inclusion $\left(\mathrm{id}+\jmath\right)\check{C}_*\hookrightarrow \check{C}_*^{\mathrm{inv}}$ is a quasi-isomorphism, and similarly for the other variants.
\end{lemma}

\vspace{0.8cm}

We are finally ready to introduce the main protagonist of the present chapter.
\begin{defn}
We define the \textit{$\Pin$-monopole Floer homology groups} of $Y$ equipped with the self-conjugate spin$^c$ structure $\spin$ denoted by
\begin{equation*}
\HSt_*(Y,\spin),\qquad \HSf_*(Y,\spin), \qquad \HSb_*(Y,\spin)
\end{equation*}
as the homology groups
\begin{align*}
\HSt_*(Y,\spin)&=H(\check{C}^{\mathrm{inv}}_*,\check{\partial})\\
\HSf_*(Y,\spin)&=H(\hat{C}^{\mathrm{inv}}_*,\hat{\partial})\\
\HSb_*(Y,\spin)&=H(\bar{C}^{\mathrm{inv}}_*,\bar{\partial}).
\end{align*}
Here again the choice of metric and perturbations is implicit in our notation.
\end{defn}
\vspace{0.5cm}
The objects we have just defined can be graded by the set $\mathbb{J}(Y,\spin)$ as in Section $2$ of Chapter $3$, and we can define their negative completions
\begin{equation*}
\HSt_{\bullet}(Y,\spin),\qquad \HSf_{\bullet}(Y,\spin), \qquad \HSb_{\bullet}(Y,\spin).
\end{equation*}
As the spin$^c$ structure is torsion, they also admit an absolute rational grading, which we will define in more detail in the next Section. Before describing and proving the properties of these invariants as in Chapter $3$, we perform an explicit calculation of these groups in the simplest possible case, which will also be central in the construction of the maps induced by cobordisms.
\begin{example}\label{HSS3}
Consider $S^3$ with the round metric and its unique spin$^c$ structure (which is obviously self-conjugate). When we consider the unperturbed monopole equations, because the metric has scalar positive curvature and there is no homology, there only one solution of the form $[B,0]$. In particular, $B$ is gauge equivalent to the spin connection $B_0$, as $[B_0,0]$ is always a critical point of the Chern-Simons-Dirac functional. In this case, the operator $D_{B_0}$ does not have simple (quaternionic) spectrum, but we can choose a small $\jmath$-equivariant perturbation such that
\begin{itemize}
\item $\q$ is $\Pin$-non-degenerate;
\item $[B,0]$ is still the only critical point.
\end{itemize}
Hence the critical submanifolds in $\Bs_k(S^3,\spin)$ consist of a doubly infinite sequence of spheres $S^2$, each corresponding to an eigenspace of $D_{\q,B}$. Again Lemma \ref{dimreducible} in Chapter $2$ tells us that the perturbation is weakly self-indexing, so we can run the associated spectral sequence which collapses at the first page for dimensional reasons. This implies that the $\Pin$-monopole Floer homology is just a direct sum of the homology groups of the invariant chains of the critical submanifolds. Following Lemma \ref{projplane} we have that:
\begin{itemize}
\item $\HSt_k(S^3,\spin)=\ztwo$ for $k$ non negative and congruent to $0,1$ or $2$ modulo $4$, and zero otherwise;
\item $\HSf_k(S^3,\spin)=\ztwo$ for $k$ negative and congruent to $1,2$ or $3$ modulo $4$, and zero otherwise;
\item $\HSb_k(S^3,\spin)=\ztwo$ for $k$ congruent to $0,1$ or $2$ modulo $4$, and zero otherwise.
\end{itemize}
Here we assigned grading zero to the zero dimensional $\delta$-chains in the first stable critical submanifold (which coincides with the absolute rational grading). In particular, up to grading shift of $-1$ the group $\HSf_{\bullet}(S^3,\spin)$ is naturally isomorphic as a graded group to the ring $\mathcal{R}$ in Definition \ref{R}.
\end{example}

\vspace{0.8cm}
The definition of the maps induced by cobordisms is essentially the same as Section $3$ in Chapter $3$, with some slight modifications to be made in order to perform a $\jmath$-invariant construction. We first focus on the case of a cobordism equipped with a self-conjugate spin$^c$ structure. In particular, given two three manifolds with self-conjugate spin$^c$ structures $(Y_{\pm},\spin_{\pm})$ and regular $\jmath$-equivariant perturbations $\q_{\pm}$, a cobordism $X$ endowed with a self-conjugate spin$^c$ structure $\spin_X$ between them and a invariant cohomology class of the form $Q^iV^n$ on the configuration space $\Bs_k(X,\spin)$, we want to define the maps
\begin{align*}
\HSf_{\bullet}(Q^iV^n\mid X,\spin_X)&: \HSf_{\bullet}(Y_-,\spin_-)\rightarrow \HSf_{\bullet}(Y_+,\spin_+)\\
\HSt_{\bullet}(Q^iV^n\mid X,\spin_X)&: \HSt_{\bullet}(Y_-,\spin_-)\rightarrow \HSt_{\bullet}(Y_+,\spin_+)\\
\HSb_{\bullet}(Q^iV^n\mid X,\spin_X)&: \HSb_{\bullet}(Y_-,\spin_-)\rightarrow \HSb_{\bullet}(Y_+,\spin_+).
\end{align*}
As in Section $3$ of the previous Chapter, this is done by considering on the cobordism a finite set of marked points $\mathbf{p}=\{p_1,\dots, p_m\}$. In this case, we also assume that the perturbation $\q_i$ on each end corresponding to each puncture is also $\jmath$-equivariant, regular and does not introduce irreducible critical points. We need to add extra $\Pin$-equivariant \textsc{asd}-perturbations on the blow up as defined in Section $2$ in order to achieve regularity of the moduli spaces. In this case completed chain complex
\begin{equation*}
\check{C}_{\bullet}(Y_-)\otimes {C}^u_{\bullet}(S^3_1)\otimes\cdots\otimes {C}^u_{\bullet}(S^3_m)
\end{equation*}
contains the subcomplex complex
\begin{equation}\label{tensorchain}
\check{C}_{\bullet}^{\mathrm{inv}}(Y_-)\otimes {C}_{\bullet}^{u,\mathrm{inv}}(S^3_1)\otimes\cdots\otimes {C}_{\bullet}^{u,\mathrm{inv}}(S^3_m)
\end{equation}
generated by elements $[\sigma]\otimes [\sigma_1]\otimes\cdots\otimes [\sigma_m]$ such that each factor is $\jmath$ invariant. Notice that is is not the subspace of the invariants of the natural action of $\jmath$ (because of the non indecomposable elements). The homology of the second chain complex is naturally identified with 
\begin{equation*}
\HSt_{\bullet}(Y_-)\otimes\HSf_{\bullet}(S^3_1)\otimes \cdots\otimes \HSf_{\bullet}(S^3_m).
\end{equation*}
Similarly there are natural isomorphisms
\begin{equation*}
\jmath: M^+_z([\Cr_-],\mathcal{C}, X^*,[\Cr_+])\rightarrow M^+_{\jmath{z}}([\jmath\Cr_-],\jmath\mathcal{C}, X^*,[\jmath\Cr_+])
\end{equation*}
commuting with each evaluation map. This last observation implies that the subcomplex $\check{C}_*(Y_-,\mathbf{p})$ defined by transversality conditions is also $\jmath$-invariant, and we can define as above its subcomplex
\begin{equation*}
\check{C}_{\bullet}(Y_-,\mathbf{p})^{\mathrm{inv}}=\check{C}_{\bullet}(Y_-,\mathbf{p})\cap\left(\check{C}_{\bullet}^{\mathrm{inv}}(Y_-)\otimes {C}_{\bullet}^{u,\mathrm{inv}}(S^3_1)\otimes\cdots\otimes {C}_{\bullet}^{u,\mathrm{inv}}(S^3_m)
\right).
\end{equation*}
It follows as in Lemma \ref{Yp} in Chapter $3$ (via Lemma \ref{projplane}) that its inclusion in the tensor product of the invariant subcomplexes is a quasi-isomorphism. The chain map $\check{m}$ satisfies the property $\jmath\circ\check{m}=\check{m}\circ\jmath$, hence it restricts to a chain map
\begin{equation*}
\check{m}^{\mathrm{inv}}: \check{C}_{\bullet}(Y_-,\mathbf{p})^{\mathrm{inv}}\rightarrow \check{C}_{\bullet}(Y_+)^{\mathrm{inv}},
\end{equation*}
which induces the map in homology
\begin{equation*}
\HSt_{\bullet}(X,\mathbf{p}):\HSt_{\bullet}(Y_-)\otimes{\Rin}\otimes \cdots\otimes {\Rin}\rightarrow \HSt_{\bullet}(Y_+).
\end{equation*}
Here we used the identification from Example \ref{HSS3} for each small regular $\jmath$-equivariant perturbation
\begin{equation*}
\HMf_{\bullet}(S^3,\q_0)\cong {\Rin}
\end{equation*}
where ${\Rin}$ is the ring in definition \ref{R}. Here again there is a grading shift.
Finally we define for any element $a\in \Rin$ the element
\begin{equation*}
\HSt_{\bullet}(a\mid X,\spin_X)(x)=\HSt_{\bullet}(X,\mathbf{p})(x\otimes a_1\otimes \cdots\otimes a_m).
\end{equation*}
where the product of the $a_i$ is $a$.
\\
\par
Having constructed these maps, we can discuss the general invariance and functoriality result, which follows along the same lines of Section$3$ in Chapter $3$. We first introduce an useful definition.
\begin{defn}
The category $\textsc{cob}_{\textsc{spin}}$ is the category whose objects are connected compact oriented $3$-manifolds with a fixed spin structure and whose objects are isomorphism classes of connected oriented cobordisms which admit a spin structure restricting to the given ones on the boundary.
\end{defn}
In the definition we consider spin structures and not self conjugate spin$^c$ structures because the composition of spin cobordisms is not necessarily spin. On the other hand it is clear from the definition that in our case morphisms compose well.
\par
As in the classical case, the key result in the proof of invariance and functoriality is the following.
\begin{teor}\label{mainHS}
The $\Pin$-monopole Floer homology groups and the maps induced by cobordisms do not depend on the choice of metric and $\jmath$-equivariant perturbation.
They define covariant functors
\begin{IEEEeqnarray*}{c}
\HSt_{\bullet}: \textsc{cob}_{\textsc{spin}}\rightarrow \textsc{mod}_{{\Rin}}\\
\HSf_{\bullet}: \textsc{cob}_{\textsc{spin}}\rightarrow \textsc{mod}_{{\Rin}}\\
\HSb_{\bullet}: \textsc{cob}_{\textsc{spin}}\rightarrow \textsc{mod}_{{\Rin}} 
\end{IEEEeqnarray*}
where $\textsc{mod}_{{\Rin}}$ is the category of graded topological ${\Rin}$-modules. More in general, if $Y_0,Y_1$ and $Y_2$ are $3$-manifolds, $X_{01}$ are $X_{12}$ are cobordisms from $Y_0$ to $Y_1$ and from $Y_1$ to $Y_2$ such that the composite $X_{12}\circ X_{01}$ admits a self conjugate spin$^c$ structure $\spin$ then
\begin{equation*}
\HSt_{\bullet}(a\mid X_{12}\circ X_{01}, \spin)=\sum_{\spin_2=\spin\lvert_{X_{12}}}\sum_{\spin_1=\spin\lvert_{X_{01}}}\HSt_{\bullet}(a_2\mid X_{12},\spin_2)\circ\HSt_{\bullet}(a_1\mid X_{01},\spin_1).
\end{equation*}
where $a_1a_2$ is $a$ in $\Rin$.
\end{teor}

Here the $\Rin$-module structure is defined as in the classical case by considering the maps induced by the product cobordism $I\times Y$, see Section $3$ in Chapter $3$.

\begin{proof}
The proof follows in the same way as the ones in Section $3$ of Chapter $3$, by taking the maps induced on the tensor product of the invariant subcomplexes. The only non obvious verification is the fact in the proof of Proposition \ref{modulestr} that the map $\hat{m}(B_0,\mathbf{p})$ defined by a disk with $m$ punctures induces the multiplication
\begin{align*}
{\Rin}\otimes\cdots\otimes {\Rin}&\rightarrow {\Rin}\\
a_1\otimes\cdots\otimes a_m&\mapsto a_1\cdots a_m.
\end{align*}
To do this, one can first reduce to the case with only two punctures
\begin{equation*}
{\Rin}\otimes{\Rin}\rightarrow {\Rin}
\end{equation*}
by using the associativity property (whose metric-stretching proof carries over without complications also in the case of more punctures). We then focus on the two properties
\begin{align*}
Q\otimes Q&\mapsto Q^2\\
Q^2\otimes Q^2&\mapsto 0,
\end{align*}
as the proof of all other relations follow in a very similar way. Let $[\Cr_{-1}]$ be the critical submanifold corresponding to the first negative eigenvalue. Then the moduli space
\begin{equation*}
M^+([\Cr_{-1}],[\Cr_{-1}], (B_0)^*_{\{p_1,p_2\}},[\Cr_{-1}])
\end{equation*}
is two dimensional and the evaluation map to each critical submanifold $[\Cr_{-1}]$ is surjective and has degree one. Because of the intersection structure of one dimensional chains in $\mathbb{R}P^2$, this implies that given a representative $[\sigma_1]\otimes[\sigma_2]$ of $Q\otimes Q$ its fibered product consists on an odd number of pairs consisting of a point and its image under $\jmath$, proving the first of the properties. For the second property we know that the moduli space
\begin{equation*}
M^+([\Cr_{-1}],[\Cr_{-1}], (B_0)^*_{\{p_1,p_2\}},[\Cr_{-2}])
\end{equation*}
is six dimensional. Any generic pair of points gives rise under the fibered product to a generator of $[\Cr_{-2}]$, as it follows from the product structure in the classical case. In particular, as an invariant generator in each $[\Cr_{-1}]$ consists of a even number of points, the fibered product is zero in homology.
\end{proof}

\begin{example}\label{S3modstr}
The proof shows that
\begin{equation*}
\HSf_{\bullet}(S^3,\q_0)\cong \ztwo[[V]][Q]/(Q^3)\{-1\}
\end{equation*}
where $\deg Q=-1$ and $\deg Q=-4$ as a ${\Rin}$-module.
\end{example}

\vspace{0.8cm}

Even thought in the rest of the present work we will only make use of the equivariant constructions, it will be interesting in some developments of the theory to consider also the interaction between these new invariants and the usual ones. Suppose we are given a pair $\spin\neq\bar{\spin}$ of conjugate but non self-conjugate spin$^c$ structures. Then the involution $\jmath$ can be thought as a diffeomorphism
\begin{equation*}
\jmath: \Co(Y,\spin)\rightarrow\Co(Y,\bar{\spin})
\end{equation*}
hence one can define a chain complex with an involution $\jmath$ by taking the direct sum
\begin{equation*}
\check{C}_{\bullet}(Y,\spin)\oplus \check{C}_{\bullet}(Y,\bar{\spin})
\end{equation*}
where we pick a non-degenerate tame perturbation $\q$ on $\Co(Y,\spin)$ and its image $\jmath_*\q$ on $\Co(Y,\bar{\spin})$. Here the action of $\jmath$ sends a critical point to its image under the diffeomorphism (which is still a critical point). The homology of the invariant subcomplex, which we denote by 
\begin{equation*}
\HSt_{\bullet}(Y,[\spin]),
\end{equation*}
is naturally isomorphic to the canonically isomorphic groups
\begin{equation*}
\HMt_{\bullet}(Y,\spin)\cong \HMt_{\bullet}(Y,\bar{\spin})
\end{equation*}
Here $[\spin]$ denotes the equivalence class of $\spin$ under the involution $\jmath$ on the set of spin$^c$ structures $\mathrm{Spin}^c(Y)$ given by conjugation. This group can be thought as an ${\Rin}$-module under the extension of coefficients
\begin{IEEEeqnarray*}{c}
\ztwo[[V]][Q]/(Q^3)\rightarrow \ztwo[[U]]\\
V\mapsto V^2\\
Q\mapsto 0.
\end{IEEEeqnarray*}
This last map is the one induced in cohomology by the double cover
\begin{equation*}
BS^1\rightarrow B\Pin.
\end{equation*}
One can also define maps induced by cobordisms equipped with a pair of non self conjugate spin$^c$ structures. As an example, if $(X,\spin_X)$ is a cobordism between $(Y_{\pm},\spin_{\pm})$ with $\spin_-$ self-conjugate and $\spin_+\neq \bar{\spin}_+$, then $(X,\bar{\spin}_X)$ is a cobordism between $(Y_-,\spin_-)$ and $(Y_+,\bar{\spin}_+)$, and we obtain a map of ${\Rin}$-modules
\begin{equation*}
\HSt_{\bullet}(X;[\spin_X]): \HSt_{\bullet}(Y_-,\spin_-)\rightarrow\HSt_{\bullet}(Y_+,[\spin_+])
\end{equation*}
from the $\jmath$-equivariant chain map
\begin{IEEEeqnarray*}{c}
\check{m}(\spin)\oplus \check{m}(\bar{\spin}):\check{C}_{\bullet}(Y_-,\spin_-)\rightarrow \check{C}_{\bullet}(Y_+,\spin_+)\oplus\check{C}_{\bullet}(Y_+,\bar{\spin}_+)
\end{IEEEeqnarray*}
where the maps $\check{m}(\spin)$ and $\check{m}(\bar{\spin})$ are the ones defined in Chapter $3$. Here, we chose the perturbations on the two collars to be $\jmath$-equivariant. We can define analogous maps in all the other cases. As discussed in the introduction, we can define the total group
\begin{equation*}
\HSt_{\bullet}(Y)=\bigoplus_{[\spin]\in \mathrm{Spin}^c(Y)/\jmath} \HSt_{\bullet}(Y,[\spin]),
\end{equation*}
for which the following result holds.
\begin{teor}
The total $\Pin$-monopole Floer homology groups define a functors from the category $\textsc{cob}$ of compact connected oriented three manifolds and isomorphism classes of cobordism between them to the category of topological $\ztwo[[V]]$-modules.
\end{teor}

\vspace{0.8cm}

We now discuss a few additional properties that our invariants satisfy. We focus on the case of a self-conjugate spin$^c$ structure $\spin$. The following result follows as Proposition \ref{longexact} in Chapter $3$ by taking the invariant subcomplexes.

\begin{prop}
For any $(Y,\spin)$, there is an exact sequence of graded $\Rin$-modules
\begin{equation*}
\dots\stackrel{i_*}{\longrightarrow} \HSt_{*}(Y,\spin)\stackrel{j_*}{\longrightarrow} \HSf_{*}(Y,\spin)\stackrel{p_*}{\longrightarrow} \HSb_*(Y,\spin)\stackrel{i_*}{\longrightarrow} \HSt_*(Y,\spin)\stackrel{j_*}{\longrightarrow}\dots
\end{equation*}
where the maps $i_*,j_*$ and $p_*$ have degree $0,0,-1$ respectively.
\end{prop}

The following result should be thought as a version of the Gysin exact sequence (in the simple case of $S^0$ bundles, i.e. double coverings).

\begin{prop}\label{gysin}
For every three manifold $Y$ and self conjugate spin$^c$ structure $\spin$, there is an exact sequence
\begin{equation*}
\cdots\rightarrow\HSt_{k+1}(Y,\spin)\stackrel{e_*}{\longrightarrow} \HSt_k(Y,\spin)\stackrel{\iota_*}{\longrightarrow} \HMt_k(Y,\spin)\stackrel{\pi*}{\longrightarrow} \HSt_k(Y,\spin)\rightarrow \HSt_{k-1}(Y,\spin)\rightarrow\cdots
\end{equation*}
The maps $e_*,\iota_*$ and $\pi_*$ are maps of $\Rin$-modules. The similar result holds for the other versions of $\Pin$-monopole Floer homology.
\end{prop} 
\begin{proof}
There is a short exact sequence of chain complexes
\begin{equation*}
0\rightarrow \check{C}_*^{\mathrm{inv}}\hookrightarrow \check{C}_*\stackrel{\mathrm{id}+\jmath}{\longrightarrow} (\mathrm{id}+\jmath)\check{C}_*\rightarrow 0.
\end{equation*}
The sequence in the statement is the associated long exact sequence, where we use Lemma \ref{projplane} to identify the homology groups.
\end{proof}

The cohomology groups $\HSt^*(Y,\spin), \HSf^*(Y,\spin)$ and $\HSb^*(Y,\spin)$ are defined as in Section $2$ of Chapter $3$ as the groups of $-Y$ with the appropriate change in the gradings. These come with canonical intersection pairings with the corresponding homology groups, for example
\begin{equation*}
\HSt^k(Y,\spin)\otimes \HSt_k(Y,\spin)\rightarrow \ztwo,
\end{equation*}
where we consider the intersection of $\jmath$-invariant as the intersection of the corresponding (after subdivision) cycles in the quotient by the action of $\jmath$. We cannot prove that such pairing is perfect for a general three manifold, but we expect that the proof of the next result could be adapted once has a better knowledge of the reducible solutions analogous to that of Chapter $35$ in the book.

\begin{prop}\label{univcoeffspin}
Suppose $Y$ is a rational homology sphere. Then the intersection pairing is perfect.
\end{prop}
\begin{proof}
We focus on the \textit{to} version of the invariants. Because the space is a rational homology sphere, we can choose a perturbation with a single reducible critical point downstairs. For the grading high enough, the result follows from the fact that the intersection pairing on $\R P^2$ is perfect (see Proposition \ref{intpairing} in Chapter $3$), as there are only finitely many irreducible solutions. We can then proceed with the following inductive proof. We have the commutative diagram
\begin{center}
\begin{tikzpicture}
  \matrix (m) [matrix of math nodes,row sep=3em,column sep=1em,minimum width=2em]
  {
     \HMt_{k+1}(Y) & \HSt_{k+1}(Y) & \HSt_{k}(Y) & \HMt_{k}(Y) & \HSt_{k}(Y) \\
     \mathrm{Hom}\HMt^{k+1}(Y) & \mathrm{Hom}\HSt^{k+1}(Y) & \mathrm{Hom}\HSt^{k}(Y) & \mathrm{Hom}\HMt^{k}(Y) & \mathrm{Hom}\HSt^{k}(Y) \\};
  \path[-stealth]

    (m-1-1) edge  (m-1-2)
    	edge node [left] {$a$} (m-2-1)
   
    (m-1-2) edge  (m-1-3)
          edge node [left] {$b$} (m-2-2)
        (m-1-3) edge  (m-1-4)   
        edge node [left] {$F$} (m-2-3)
        (m-1-4) edge  (m-1-5)    
        edge node [left] {$c$} (m-2-4)
        (m-1-5) edge node [left] {$F$} (m-2-5)

    (m-2-1) edge  (m-2-2)
    (m-2-2) edge (m-2-3)
        (m-2-3) edge  (m-2-4)   
        (m-2-4) edge  (m-2-5)

    ;

\end{tikzpicture}
\end{center}
where the upper row is the Gysin sequence form Proposition \ref{gysin} and the lower row is the dual of the corresponding sequence for cohomology (or equivalently for $-Y$), and the vertical maps are the ones induced by the intersection pairing. In particular we know that $a$ and $c$ are isomorphism. Suppose now $b$ is also an isomorphism. Then applying the four lemma to the left part of the diagram, we obtain that $F$ is injective. Applying then the four lemma to the right part of the diagram we obtain that $F$ is also surjective, hence it is an isomorphism.
\end{proof}

\vspace{1.5cm}
\section{Some computations}

In this section we provide some basic computations of the invariants which can be performed by explicitly solving the equations as in Chapters $36$ and $37$ of the book. Before doing this, we discuss a non-vanishing result and quickly review the definition of absolute gradings.

\begin{prop}
For every three manifold $Y$ and self conjugate spin$^c$ structure $\spin$, the groups
\begin{equation*}
\HSt_*(Y,\spin),\qquad \HSf_*(Y,\spin), \qquad \HSb_*(Y,\spin)
\end{equation*}
are non-zero in infinitely many gradings.
\end{prop}

\begin{proof}
This readily follows from the Gysin exact sequence in Proposition \ref{gysin} and the deep non-vanishing result for the classical monopole Floer groups when the spin$^c$ structure has torsion first Chern class, see Corollaries $35.1.3$ and $35.1.4$ in the book.
\end{proof}

\vspace{0.8cm}

We quickly review the discussion of the absolute rational gradings for torsion spin$^c$ structures in Section $28.3$ of the book. Notice that even though we could only deal with self-conjugate spin$^c$ structures our definition of absolute grading (as the three dimensional spin cobordism group is trivial) we will consider the more general case of a non-self conjugate spin$^c$ structure, as it sometimes makes computations easier. Furthermore, this will allow to consider also the interaction with the non-equivariant counterparts of the theory. Given an integral $2$-dimensional cohomology class on a cobordism $W$ that restricts to a torsion class on the boundary, we define
\begin{equation*}
\langle c,c\rangle=(\tilde{c}\cup\tilde{c})[W,\partial W]\in\mathbb{Q}
\end{equation*}
where $\tilde{c}$ is any lift the image of $c$ in the rational cohomology $H^2(W;\mathbb{Q})$ to $H^2(W,\partial W;\mathbb{Q})$. Given now a self-conjugate spin$^c$ structure $\spin$ on $Y$, choose any four manifold $\tilde{X}$ whose boundary is $Y$ and over which the $\spin$ extends. We will think of this manifold as a cobordism $X$ from $S^3$ to $Y$ obtained by removing a ball. In the following definition, we have a fixed equivariant perturbation $\q$ on $Y$ (see also Definition $28.3.1$ in the book), and we can choose any Morse-Bott perturbation on $S^3$. Recall the definition of relative grading on a cobordism from Section $6$ in Chapter $2$.

\begin{defn}
Given a self-conjugate spin$^c$ structure $\spin$ on $Y$, let $X$ be any cobordism from $S^3$ to $Y$ as above. For a $\delta$-chain $[\sigma]$ with value in a critical submanifold $[\Cr]$, we define its \textit{absolute grading} $\mathrm{gr}^{\mathbb{Q}}([\sigma])$ as the rational number
\begin{equation*}
\mathrm{gr}^{\mathbb{Q}}([\sigma])=-\mathrm{gr}_z([\Cr_0],X,[\Cr])+\mathrm{dim}[\sigma]+\frac{1}{4}\langle c_1(S^+), c_1(S^+)\rangle- \iota(X)-\frac{1}{4}\sigma(X)\in\mathbb{Q}
\end{equation*}
where
\begin{itemize}
\item $[\Cr_0]$ is the the stable reducible critical manifold of $S^3$ corresponding to the smallest positive eigenvalue of the Dirac operator, and $z$ is any relative homotopy class;
\item $S^+$ is the spinor bundle for the corresponding spin$^c$ structure on $X$;
\item $\iota(X)$ is the characteristic number
\begin{equation*}
\frac{1}{2}(\chi(X)+\sigma(X)-b_1(Y))
\end{equation*}
from Definition $25.4.1$ in the book (in general there is an additional term that vanishes in this case).
\end{itemize}
Similarly, for a reducible critical point, we can introduce the modified absolute grading
\begin{equation*}
\bar{\mathrm{gr}}^{\mathbb{Q}}([\acr])=
\begin{cases}
\mathrm{gr}^{\mathbb{Q}}([\acr]), & [\acr]\text{ boundary stable}\\
\mathrm{gr}^{\mathbb{Q}}([\acr])-1, & [\acr]\text{ boundary unstable}.
\end{cases}
\end{equation*}
\end{defn}
Recall that the characteristic number $\iota(X)$ is an integer as it is the index of an operator, see Lemma $25.4.2$ in the book.
We summarize main features of these gradings (all from Section $28.3$ in the book) in the following proposition.
\begin{prop}\label{absgrad}
The absolute grading $\mathrm{gr}^{\mathbb{Q}}([\sigma])$ is well defined, i.e. independent of the choice of $X$ and homotopy class $z$. Its fractional part is the same as the fractional part of $(\langle c_1,c_1\rangle-\sigma(X))/4$. If $Y$ is a homology three sphere and $[\sigma]$ is a $\delta$-chain in a boundary-stable reducible critical submanifold, its grading $\mathrm{gr}([\sigma])$ is an integer with the same parity as its dimension. Finally, in any case the duality isomorphism
\begin{equation*}
\check{\omega}: \HSt_{\bullet}(-Y,\spin)\rightarrow \HSf^{\bullet}(Y,\spin)
\end{equation*}
maps elements of grading $k$ to elements of grading $-1-b_1(Y)-k$.
\end{prop}

\begin{remark}
This absolute grading convention differs from the one adopted in Heegaard Floer homology (see \cite{OSd}), which can be recovered by subtracting $-b_1(Y)/2$.
\end{remark}

\vspace{0.8cm}

The next computation will be important for the construction of the integral invariants of homology spheres in the next section. We first introduce a useful notation.
\begin{defn}\label{standardL}
The \textit{standard $\Rin$-module} $\mathcal{M}$ is the graded $\Rin$-module
\begin{equation*}
\mathcal{M}=\ztwo[V^{-1},V]][Q]/(Q^3)\{-2\},
\end{equation*}
where $\ztwo[V^{-1},V]]$ denotes the ring of Laurent series in $V$.
In particular the element generating the homogeneos component of degree zero is $V^{-1}Q^2$.
\end{defn}
The grading shift is performed so that $\mathcal{M}$ agrees as a graded module with $\HSb_{\bullet}(S^3)$, see Corollary \ref{poshomology}.

\begin{prop}\label{HShomology}
Let $Y$ be a rational homology sphere, and $\spin$ a self-conjugate spin$^c$ structure on $Y$, then after grading shifts we have the isomorphism of graded $\Rin$-modules
\begin{equation*}
\HSb_{\bullet}(Y,\spin)\cong \mathcal{M}
\end{equation*}
If $Y$ is a homology sphere, the generator $V$ has even degree.
\end{prop}
The result tells us that in this case the only interesting content of the invariant are the absolute rational gradings.

\begin{proof}
The bar version of the Floer invariants only involves reducible critical points, and in the case of a rational homology sphere $Y$ there is only $[B_0,0]$ before blowing up. Hence, we are in the same situation as Example \ref{HSS3}. The determination of module structure follows with the same proof as in Theorem \ref{mainHS}, see also Example \ref{S3modstr}. The fact that in the case of a homology sphere $v$ has even degree follows the fact that it is represented by a $\delta$-chain of even dimension and Proposition \ref{absgrad}.
\end{proof}

\vspace{0.8cm}
The result we have just stated only deals with the reducible solutions, and we do not have in general any knowledge about the irreducible ones. In particular, the only computation we can perform by hand is the case in which the manifold has positive scalar curvature or is flat. We first have the following result.
\begin{prop}\label{positivity}
If the manifold $Y$ has positive scalar curvature, then the map $j_*$ is zero, hence the long exact sequence splits into a canonical direct sum decomposition
\begin{equation*}
\HSb_{\bullet}(Y,\spin)=\HSt_{\bullet}(Y,\spin)\oplus\HSf_{\bullet}(Y,\spin)\{-1\}.
\end{equation*}
where the braces indicate the grading shift.
\end{prop}
\begin{proof}
This follows in the same way as in Section $36.1$ of the book. For a $\jmath$-invariant Morse function on the torus
\begin{equation*}
\mathbb{T}=H^1(Y;\R)/ H^1(Y;\Z)
\end{equation*}
parametrizing the reducible solutions to the monopole equations, one can construct a $\Pin$-equivariant tame perturbation such that the boundary map $\bar{\partial}^u_s$ is zero (see Lemma $36.1.2$ in the book). The map $\bar{\partial}^s_u$ is also zero at the chain level. Indeed the image is always negligible for dimensional reasons unless we are considering two critical manifolds $[\Cr_0]$ and $[\Cr_{-1}]$ corresponding to the first and positive and negative eigenvalues respectively, and zero chains in the former. On the other hand, the subset of trajectories in the unparametrized moduli space $\breve{M}([\Cr_0],[\Cr_{-1}])$ converging to a given point is already compact simply because there are no possible intermediate resting points. So the $\delta$-chain defined in $[\Cr_{-1}]$ by a zero chain in $[\Cr_0]$ is three dimensional and has no boundary, so is negligible. The vanishing of this two maps implies that the chain complex $\bar{C}_*$ splits as the direct sum of $\check{C}_*$ and $\hat{C}_*$.
\end{proof}
\begin{remark}
By the characterization of the moduli spaces in Lemma \ref{dimreducible} in Chapter $2$ we can see that the chain defined above is indeed diffeomorphic to a three sphere.
\end{remark}

\begin{cor}\label{poshomology}
Suppose $Y$ is a rational homology sphere with positive scalar curvature, and $\spin$ a self-conjugate spin$^c$ structure. Then we have up to grading shifts the isomorphism of ${\Rin}$-modules
\begin{align*}
\HSb_{\bullet}(Y,\spin)&\cong \mathcal{M}\\
\HSt_{\bullet}(Y,\spin)&\cong \mathcal{M}/\Rin\cdot 1\\
\HSf_{\bullet}(Y,\spin)&\cong \Rin\{-1\}.
\end{align*}
Furthermore, if $Y$ is $S^3$, the absolute grading of the minimum non zero element in $\HSt_{\bullet}$ is $0$, and if $Y$ is the Poincar\'e homology sphere (oriented as the boundary of the plumbing along the graph $-E_8$) the absolute grading is $-2$.
\end{cor}
\begin{proof}
The computation of the gradings follows from the usual ones (see for example \cite{KMOS}), as the homology of the chain complex computes the usual invariants. 
\end{proof}
\vspace{0.8cm}

We then turn to the case of $S^1\times S^2$. This case is also tractable because of the positive scalar curvature, and in particular because of Proposition \ref{positivity} we only need to compute the bar version of the theory.
\begin{prop}\label{S2S1}
Letting $\spin$ is the only self-conjugate spin$^c$ structure on $S^1\times S^2$, up to grading shifts we have the isomorphisms of ${\Rin}$-modules
\begin{align*}
\HSb_{\bullet}(S^1\times S^2,\spin)&\cong L\{-1\}\otimes H_*(S^1;\ztwo)\\
\HSt_{\bullet}(S^1\times S^2,\spin)&\cong L\{-1\}/\Rin\cdot 1\otimes H_*(S^1;\ztwo)\\
\HSf_{\bullet}(S^1\times S^2,\spin)&\cong\Rin\{-2\}\otimes H_*(S^1;\ztwo).
\end{align*}
In particular, the minimum absolute grading of a non zero element in $\HSt_{\bullet}(S^1\times S^2,\spin)$ is $-1$.
\end{prop}
\begin{proof}
This computation can be performed directly by choosing an appropriate perturbation. Because of Proposition \ref{positivity}, we can just focus on the bar version of the chain complex, and can choose as in Section $36.1$ of the book a perturbation of the form $f\circ p$ where
\begin{equation*}
p: \Bo(Y,\spin)\rightarrow \mathbb{T}
\end{equation*}
is a smooth gauge-invariant retraction equivariant for the action of $\jmath$ and $f$ is a $\jmath$-invariant Morse function with exactly two critical points, which are the reducible critical points corresponding to the two spin connections on $S^1\times S^2$, namely $[B_1,0]$ with index $1$ and $[B_0,0]$ with index $0$. It is clear then that a differential between two different critical submanifolds lying over different spin connections starts from $[B_1,0]$ and ends at $[B_0,0]$. In fact there are two flow lines on $\mathbb{T}$ connecting such two connections, and they are conjugate under the action of $\jmath$. Consider the critical submanifolds $[\Cr^0_i]$ and $[\Cr^1_i]$ the critical submanifolds corresponding to the $i$th eigenvalue of $D_{\q,B_0}$ and $D_{\q,B_1}$. Because of dimensional reasons (see Proposition \ref{positivity}), the only interesting moduli spaces when computing differentials are those of the form $\breve{M}^+([\Cr^0_i],[\Cr^1_i])$. These are to two copies of $\mathbb{C}P^1$ identified via the action of $\jmath$, and each of the evaluation maps are either constant or a diffeomorphism onto the image. This implies that their total contribution on the invariant chains is trivial, hence the only non-zero differentials are the ones in the chain complexes of the critical submanifolds, which implies the result.
\par
The ${\Rin}$-module structure is determined in a similar way as in the proof of Theorem \ref{mainHS}. In particular, the only additional thing to check in this case is that the action of $Q$ and $V$ does not send a $\delta$-chain over one reducible over the other reducible. This follows as above from the fact that there are the corresponding moduli space has two components related by the action of $\jmath$ hence their contribution on the invariant chains is zero. Finally the fact that the lowest grading of a non zero element is $-1$ follows from the fact that the total chain complex computes the usual monopole Floer homology of $S^2\times S^1$.
\end{proof}
\begin{remark}
Given this computation, we can define the map induced by a cobordism and an element in the exterior algebra $\Lambda^*(H_1(W;\Z)/\mathrm{Tor}\otimes \ztwo)$ as in Remark \ref{exterioraction} in Chapter $3$.
\end{remark}
\vspace{0.8cm}

We now study the case of the three torus $T^3$, which is more involved.
\begin{prop}
Let $\spin$ be the only self conjugate spin$^c$ structure on the three torus. We have the isomorphisms of graded $\Rin$-modules
\begin{align*}
\HSb_{\bullet}(T^3,\spin)&\cong  \mathcal{M}\{-3\}\otimes\left(H_1(T^3;\ztwo)\oplus H_2(T^3;\ztwo)\right)\\
\HSt_{\bullet}(T^3,\spin)&\cong \mathcal{M}\{-3\}\otimes\left(H_1(T^3;\ztwo)\oplus H_2(T^3;\ztwo)\right)\\
\HSf_{\bullet}(T^3,\spin)&\cong\Rin\{-3\}\otimes \left(H_1(T^3;\ztwo)\oplus H_2(T^3;\ztwo)\right).
\end{align*}
In particular the absolute grading of the minimum non zero element in $\HSt_{\bullet}(T^3,\spin)$ is $-2$.
\end{prop}
\begin{proof}

The proof follows the same line of Proposition \ref{S2S1}, with some complications due to the fact that the chain complexes we are dealing with are much bigger than that case. As in Section $34.7$ in the book, we can choose $\epsilon$ and $\delta$ very small (see Proposition $37.1.1$) so that the perturbed functional
\begin{equation}\label{perttorus}
\CSd=\CSD-(\delta/2)\|\Phi\|^2+\epsilon f
\end{equation}
has no irreducible critical points. Here $f$ is a $\jmath$-invariant function obtained via a retraction as in Proposition \ref{S2S1} by Morse function on the space of reducible solutions of the unperturbed equations $\mathbb{T}$ (which is a three torus itself) which has a standard form with eight critical points (the gauge equivalence classes of spin connections): $x$ of index $0$ (corresponding to the only class of flat connections for which the Dirac operator has kernel), three points $y^{\alpha}$ of index $1$, three points $w^{\alpha}$ of index $2$ and $z$ of index $3$. Then for each $i\in \mathbb{Z}$ we have eight critical submanifolds
\begin{align*}
&[\Cr^w_i]\\
[\Cr^{z_1}_i]\qquad &[\Cr^{z_2}_i]\qquad [\Cr^{z_3}_i]\\
[\Cr^{y_1}_i]\qquad &[\Cr^{y_2}_i]\qquad [\Cr^{y_3}_i]\\
&[\Cr^x_i]
\end{align*}
corresponding to the $i$th eigenvalue of the corresponding Dirac operators. The interesting feature is that the grading on $[\Cr^x_i]$ is shifted up by two because of the spectral flow. First of all, we make some simple considerations on the moduli spaces that arise.
\begin{enumerate}
\item The moduli spaces $\check{M}^+([\Cr_i^w],[\Cr_i^{z_{\alpha}}])$ and $\check{M}^+([\Cr_i^{z_{\alpha}}],[\Cr_i^{y_{\beta}}])$ all consist of pairs of $\mathbb{C}P^1$ related by the action of $\jmath$ as in Proposition \ref{S2S1}, each evaluation map being constant or a diffeomorphism.
\item The moduli spaces $\check{M}^+([\Cr_i^{y_{\beta}}],[\Cr_i^{x}])$ consist of pairs of points related by the action of $\jmath$.
\item the moduli spaces $\check{M}^+([\Cr_i^w],[\Cr_i^{y_{\beta}}])$ is naturally decomposed as the union of two $3$-dimensional $\delta$-cycles in $[\Cr_i^{y_{\beta}}]$ related by the action of $\jmath$. To see this, notice that the space of unparametrized flow lines for $f$ between $w$ and $y_{\beta}$ can be identified with the union of four intervals $I,J, \jmath I$ and $\jmath J$, and the action of $\jmath$ in exchanges them in pairs. The observation in point ($1$) implies that the stratum of the moduli space over each boundary point of such interval (which corresponds to a broken trajectory) is either empty or consists of a copy of $\mathbb{C}P^1$. The claim follows by taking the cycles parametrized by the unions $I\cup J$ and $\jmath I\cup\jmath J$. Notice that the possible complication that the boundary consists of pairs of conjugate $\mathbb{C}P^1$ mapping to conjugate points is not an issue because these $\delta$-chains are negligible, hence zero in our context.
\item For the same reason, the moduli spaces $\check{M}^+([\Cr_i^{z_{\alpha}}],[\Cr_i^{x}])$ can be written as the disjoint union of two one dimensional $\delta$-cycles in $[\Cr_i^{x}]$ related by the action of $\jmath$.
\item Finally, analogous considerations imply that the moduli space $\check{M}^+([\Cr_i^w],[\Cr_i^{x}])$ is a two dimensional $\delta$-cycle $[\sigma_i]$ in $[\Cr_i^{x}]$.
\end{enumerate}
Our claim is that the $\delta$-cycle $[\sigma_i]$ represents a generator of the top homology of $[\Cr_i^{x}]$. Once this is settled, the computation can be performed by running the spectral sequence associated to the energy filtration for the perturbed functional $\CSd$ (see Section $2$ in Chapter $3$). Indeed first page of this spectral is the following
\begin{center}
\begin{tikzpicture}
\matrix (m) [matrix of math nodes,row sep=0.5em,column sep=0.5em,minimum width=2em]
  {
    H_2^{\mathrm{inv}}([\Cr_i^w]) & & &\\
    H_1^{\mathrm{inv}}([\Cr_i^w]) & \bigoplus H_2^{\mathrm{inv}}([\Cr_i^{z_{\alpha}}]) & &H_2^{\mathrm{inv}}([\Cr_i^x]) \\
    H_0^{\mathrm{inv}}([\Cr_i^w]) & \bigoplus H_1^{\mathrm{inv}}([\Cr_i^{z_{\alpha}}])& \bigoplus H_2^{\mathrm{inv}}([\Cr_i^{y_{\beta}}])& H_1^{\mathrm{inv}}([\Cr_i^x])\\
    & \bigoplus H_0^{\mathrm{inv}}([\Cr_i^{z_{\alpha}}]& \bigoplus H_2^{\mathrm{inv}}([\Cr_i^{y_{\beta}}])&H_0^{\mathrm{inv}}([\Cr_i^x])\\
    & & \bigoplus H_2^{\mathrm{inv}}([\Cr_i^{y_{\beta}}])&\\};    
    ;
\end{tikzpicture}
\end{center}
Of course the $\mathrm{inv}$ indicates that we are considering the corresponding $\jmath$-invariant subcomplex of the critical submanifolds, and each of the groups is a copy of $\ztwo$. Here we focus only on the part involving the eight critical manifolds with index $i$ because the map on the first page
\begin{equation*}
d_1: H_0^{\mathrm{inv}}(([\Cr_{i+1}^{y_{1}}])\oplus H_0^{\mathrm{inv}}([\Cr_{i+1}^{y_{1}}])\oplus H_0^{\mathrm{inv}}([\Cr_{i+1}^{y_{1}}])\rightarrow H_2^{\mathrm{inv}}([\Cr_{i}^x])
\end{equation*}
is zero by symmetry considerations. Exploiting the discussion on the moduli spaces above we can conclude that there are only two non trivial differentials:
\begin{itemize}
\item on the first page, there is a differential
\begin{equation*}
d_1: H_2^{\mathrm{inv}}(([\Cr_i^{y_{1}}])\oplus H_2^{\mathrm{inv}}([\Cr_i^{y_{1}}])\oplus H_2^{\mathrm{inv}}([\Cr_i^{y_{1}}])\rightarrow H_0^{\mathrm{inv}}([\Cr_i^x])
\end{equation*}
which is non zero on each summand coming from the moduli spaces in case $(2)$;
\item on the third page, there is a differential
\begin{equation*}
d_3: H_2^{\mathrm{inv}}([\Cr_i^{w}])\rightarrow H_2^{\mathrm{inv}}([\Cr_i^x])
\end{equation*}
coming from chain $[\sigma_i]$ in case $(5)$, which we know is a generator of the latter group.
\end{itemize} 
The module structure follows in a way similar to that of Proposition \ref{S2S1}. The only additional detail is that we can find for $\beta=1,2,3$ a generator for the group in the same grading of $H_2^{\mathrm{inv}}([\Cr_i^{y_\beta}])$ involving a top generator of exactly one of the $[\Cr_i^{y_{\beta}}]$, which again follows from the claim on $[\sigma_i]$.
\par
Finally, to prove our claim about the cycle $[\sigma_i]$, we consider the energy filtration on the total complex, not just the invariant subcomplex. In particular, symmetry considerations imply that
\begin{equation*}
H_0([\Cr_i^w]),\quad H_0([\Cr_i^{x_{\alpha}}]) \text{ for }\alpha=1,2,3,\quad H_0([\Cr_i^x])
\end{equation*} 
all survive until the $E^3$ page. On the other hand, the usual monopole Floer group in the grading corresponding to the last four groups has rank three, so that the differential
\begin{equation*}
d_3: H_0^{\mathrm{inv}}([\Cr_i^{w}])\rightarrow H_0^{\mathrm{inv}}([\Cr_i^x])
\end{equation*}
is non trivial. Hence $[\sigma_i]$ is a generator of the top homology of $[\Cr_i^x]$. \end{proof}
\begin{remark}
The perturbation in the last computation is \textit{not} weakly self-indexing.
\end{remark}
\vspace{0.8cm}
We finally provide some computations for the other flat three manifolds. The computation for the Hantzsche-Wendt manifold (which is a rational homology sphere) is analogous to that of Corollary \ref{poshomology} because one can achieve transversality using a perturbation as in the case of the three torus without introducing irreducible critical points.  The remaining case are all flat three manifolds which are torus bundles over $S^1$ with $b_1(Y)=1$.
\begin{prop}
Let $Y$ be flat torus bundle over the circle, and let $\spin_0$ be the self-conjugate spin$^c$ structure corresponding to the $2$-plane field $\xi_0$ tangent to the fiber. Then we have the isomorphisms of ${\Rin}$-modules
\begin{align*}
\HSb_{\bullet}(Y,\spin_0)&\cong Q\cdot \mathcal{M}\otimes(\ztwo\oplus \ztwo\{2\})\\
\HSt_{\bullet}(Y,\spin_0)& \cong Q\cdot \mathcal{M}/\Rin\cdot Q\otimes(\ztwo\oplus \ztwo\{2\}) \\
\HSf_{\bullet}(Y,\spin_0)&\cong \Rin\cdot Q \otimes (\ztwo\{-1\}\oplus \ztwo\{1\}).
\end{align*}
\end{prop}

\begin{remark}
Notice that some of these manifolds have other self-conjugate spin$^c$ structures other than the one we are considering.
\end{remark}
\begin{proof}
As in the case of the three torus, we can choose a perturbation as in equation (\ref{perttorus}) where the Morse function on the circle of flat connections $\mathbb{T}(\spin)$ is such that 
\begin{itemize}
\item $f$ is $\jmath$-invariant;
\item $f$ has two maxima corresponding to the spin connections $[B_0]$ and $[B_1]$ and two minima $[A_0]$ and $[A_1]$ corresponding to the flat connections such that the corresponding Dirac operator has kernel.
\end{itemize}
The two minima are conjugate via the involution $\jmath$. Here we set $[B_1]$ to be the spin connection such that on a path connecting it to a minimum the family of Dirac operators $D_A-\epsilon$ for $\varepsilon>0$ small has complex spectral flow $-1$. Call $[\Cr^0_i],[\Cr^1_i]$ and $[\bcr^0_i], [\bcr^1_i]$ the critical submanifolds corresponding to the $i$th eigenvalue of the respective Dirac operators. In particular the first two are copies of $\mathbb{C}P^1$ while the latter consist of single points. We have that for dimensional reasons (as in Proposition \ref{positivity}) the complex is a direct sum of pieces as in the following diagram:
\begin{center}
\begin{tikzpicture}
\matrix (m) [matrix of math nodes,row sep=1.5em,column sep=0.5em,minimum width=2em]
  {
    C_2([\Cr^0_i]) & & \\
    C_1([\Cr^0_i]) & C_0([\bcr^0_{2i+1}])\oplus C_0([\bcr^1_{2i+1}]) & \\
    C_0([\Cr^0_i]) & & C_2([\Cr^1_i])\\
    & C_0([\bcr^0_{2i}])\oplus C_0([\bcr^1_{2i}])& C_1([\Cr^1_i])\\
    & & C_0([\Cr^1_i])\\};
  \path[-stealth]

    (m-1-1) edge node [left] {$\partial^{\mathcal{F}}$}  (m-2-1)
    	edge node [above] {$\partial^s_s$} (m-2-2)
    
   (m-2-1) edge node [left] {$\partial^{\mathcal{F}}$}  (m-3-1)

    (m-3-3) edge node [left] {$\partial^{\mathcal{F}}$}  (m-4-3)
    	edge node [above] {$\partial^s_s$} (m-4-2)
    
   (m-4-3) edge node [left] {$\partial^{\mathcal{F}}$}  (m-5-3)
    
    ;
\end{tikzpicture}
\end{center}
Here we assume $i\geq 0$, and the chain complex of the critical submanifold $[\Cr^1_i]$ is shifted down by two because of the spectral flow. Each of the moduli spaces $M^+([\Cr^0_i],[\bcr^{0}_{2i+1}])$ and $M^+([\Cr^0_i],[\bcr^{1}_{2i+1}])$ consists of a single point, and these are conjugated under the action of $\jmath$. The same holds for $M^+([\Cr^1_i],[\bcr^{0}_{2i}])$ and $M^+([\Cr^1_i],[\bcr^{1}_{2i}])$, and the result readily follows.
\end{proof}

\vspace{1.5cm}
\section{Manolescu's $\beta$ invariant and the Triangulation conjecture}

In this final section we define the counterpart of Manolescu's $\beta$ invariant introduced in \cite{Man2} in our Morse-theoretic approach, and prove the main properties which are used in the disproof of the Triangulation conjecture as discussed in the Introduction. The essential ingredient is the absolute grading discussed in the previous section. The invariant $\beta$ and its companions $\alpha$ and $\gamma$ are the $\Pin$ analogue of the Froysh\o v invariant $h$ in monopole Floer homology (see \cite{Fro} and Chapter $39$ in the book) and the Heegaard Floer correction terms (\cite{OSd}). Let $Y$ be a rational homology sphere, and consider a self-conjugate spin$^c$ structure $\spin$ on it. Clearly, if $Y$ is actually a homology sphere, or more in general if it the first homology has no two-torsion, there is only one such spin$^c$-structure. As stated in Proposition \ref{HShomology}, up to grading shifts $\HSb_{\bullet}(Y,\spin)$ is isomorphic as a ${\Rin}$-module to the standard $\Pin$-module $\mathcal{M}$ from Definition \ref{standardL}. A rational homology sphere $Y$ with a self-conjugate spin$^c$ structure $\spin$ gives rise to a preferred proper homogeneous quotient of $\mathcal{M}$, namely the image of the map
\begin{equation}\label{imap}
i_*: \HSb_{\bullet}(Y,\spin)\rightarrow \HSt_{\bullet}(Y,\spin).
\end{equation}
We use these submodules to define three numerical invariants of the rational homology sphere.

\begin{defn}
For $\alpha\geq \beta\geq \gamma$ rational numbers all congruent to $\epsilon$ modulo $1$, the \textit{standard $\Pin$-module} $\mathcal{M}_{\alpha,\beta,\gamma}$ is defined to be the $\Rin$-module obtained as the quotient of (a suitably shifted) standard module $\mathcal{M}$ by a graded submodule such that
\begin{itemize}
\item the element of lowest degree of the form $Q^2V^a$ has degree $2\alpha$;
\item the element of lowest degree of the form $QV^b$ has degree $2\beta+1$;
\item the element of lowest degree of the form $V^c$ has degree $2\gamma+2$.
\end{itemize}
\end{defn}

It follows from the structure as a $\Rin$-module that the image of the map $i_*$ in equation (\ref{imap}) is a module of the form $\mathcal{M}_{\alpha,\beta,\gamma}$.

\begin{defn}For a rational homology sphere $Y$ equipped with a self-conjugate spin$^c$ structure $\spin$, the \textit{Manolescu's invariants} $\alpha(Y,\spin)$, $\beta(Y,\spin)$ and $\gamma(Y,\spin)$ are defined as the rational numbers $\alpha\geq\beta\geq\gamma$ such that
\begin{equation*}
i_*\left(\HSb_{\bullet}(Y,\spin)\right)\cong \mathcal{M}_{\alpha,\beta,\gamma}
\end{equation*}
as graded $\Rin$-modules.
\end{defn}

\vspace{0.5cm}
\begin{example}
Following Corollary \ref{poshomology}, in the case of $S^3$ the grading preserving identification
\begin{equation*}
\HSb_{\bullet}(S^3)\cong \mathcal{M}.
\end{equation*}
Moreover $i_*$ is surjective and $\HSt_{\bullet}(S^3)$ consists of the part in non-negative. Hence we have
\begin{equation*}
\alpha(S^3)=\beta(S^3)=\gamma(S^3)=0.
\end{equation*}
Similarly, the case of the Poncar\'e homology sphere $Y$ is identical up to a shift in degree by $-2$, hence
\begin{equation*}
\alpha(Y)=\beta(Y)=\gamma(Y)=-1.
\end{equation*}
More in general, whenever $Y$ is a homology sphere these three numbers are integers thanks to Proposition \ref{absgrad}.
\end{example}

\begin{remark}\label{alternbeta}
The following interpretation is closer to Manolescu's original definition. We have that $i_*$ defines an injective map on $\HSb_{\bullet}(Y,\spin)/\mathrm{im}(p_*)$ so we can also define
\begin{align*}
\alpha(Y,\spin)&=\frac{1}{2}\left(\min\{ \mathrm{gr}^{\mathbb{Q}}(x) \mid x\in i_* \ztwo[V^{-1},V]], x\neq 0\}\right)\\
\beta(Y,\spin)&=\frac{1}{2}\left(\min\{ \mathrm{gr}^{\mathbb{Q}}(x) \mid x\in i_* Q\ztwo[V^{-1},V]],x\neq 0\}-1\right)\\
\gamma(Y,\spin)&=\frac{1}{2}\left(\min\{ \mathrm{gr}^{\mathbb{Q}}(x) \mid x\in i_* Q^2\ztwo[V^{-1},V]],x\neq 0\}-2\right)
\end{align*}
under the usual identification of $\HSb_{\bullet}(Y,\spin)$ with $\mathcal{M}$.
\end{remark}

\vspace{0.8cm}
The next result is the central theorem regarding Manolescu's invariants. Recall that the Rohklin invariant of a rational homology three sphere $Y$ equipped with a spin structure $\spin$ is defined as
\begin{equation*}
\mu(Y,\spin)=\frac{1}{8}\sigma(X)\quad\mod2\Z
\end{equation*}
for any smooth oriented four manifold $X$ with boundary $Y$ which admits a spin structure restricting to the given one on the boundary. This is well defined because of the celebrated Rokhlin's theorem, and is an integer modulo two in the case $Y$ is a homology sphere (see for example \cite{Sav} for an introduction to the subject).

\begin{teor}\label{mainbeta}
Let $Y$ be a rational homology sphere, and $\spin$ a self-conjugate spin$^c$ structure. Then the Manolescu's invariants $\alpha(Y,\spin),\beta(Y,\spin)$ and $\gamma(Y,\spin)$ satisfy the inequalities
\begin{equation*}
\alpha(Y,\spin)\geq \beta(Y,\spin)\geq \gamma(Y,\spin).
\end{equation*}
Furthermore, they satisfy the following properties:
\begin{enumerate}
\item if $-Y$ denotes $Y$ with the opposite orientation, then
\begin{align*}
\alpha(-Y,\spin)&=-\gamma(Y,\spin)\\
\beta(-Y,\spin)&=-\beta(Y,\spin)\\
\gamma(-Y,\spin)&=-\alpha(Y,\spin);
\end{align*}
\item the reduction modulo two of $\alpha(Y,\spin), \beta(Y,\spin)$ and $\gamma(Y,\spin)$ are the opposite of the Rokhlin invariant $-\mu(Y,\spin)$;
\item given a smooth spin cobordism $W$ with negative definite intersection form from $(Y_0,\spin_0)$ to $(Y_1,\spin_1)$, then
\begin{align*}
\alpha(Y_1,\spin_1)&\geq \alpha(Y_0,\spin_0)+\frac{1}{8} b_2(W)\\
\beta(Y_1,\spin_1)&\geq \beta(Y_0,\spin_0)+\frac{1}{8} b_2(W)\\
\gamma(Y_1,\spin_1)&\geq \gamma(Y_0,\spin_0)+\frac{1}{8} b_2(W).
\end{align*}
\end{enumerate}
In particular the invariants $\alpha(Y,\spin),\beta(Y,\spin)$ and $\gamma(Y,\spin)$ are invariant under homology cobordism.
\end{teor}

\begin{remark}
We expect our $\Pin$-monopole Floer homology invariants (and in particular these numerical invariants) to coincide with those defined in \cite{Man2}. 
\end{remark}

Restricting to the case in which $Y$ is a homology sphere, we obtain the following result. This implies Theorem \ref{triangulation} in the Introduction, see the discussion after the result.
\begin{cor}[See also \cite{Man2}] \label{beta}
There is an integer valued invariant $\beta$ associated to each homology sphere $Y$ which satisfies the following properties:
\begin{enumerate}
\item if $-Y$ is $Y$ with the opposite orientation, then $\beta(-Y)=-\beta(Y)$;
\item the reduction modulo two of $\beta(Y)$ is the Rokhlin invariant $\mu(Y)$;
\item if $W$ is a spin cobordism between two homology spheres $Y_0$ and $Y_1$ with negative definite intersection form, then
\begin{equation*}
\beta(Y_1)\geq \beta(Y_0)+\frac{1}{8}b_2(W).
\end{equation*}
\end{enumerate}
In particular, $\beta$ is invariant under homology cobordism.
\end{cor}
\vspace{0.5cm}
\begin{proof}[Proof of Theorem \ref{mainbeta}]
The fact that
\begin{equation*}
\alpha(Y,\spin)\geq \beta(Y,\spin)\geq \gamma(Y,\spin)
\end{equation*}
is implicit in the definition of the invariants and follows from the $\Rin$-module structure. Regarding property $(2)$, if $X$ is a spin four manifold bounding $Y$ we have (as in Lemma $28.3.2$ in the book and Proposition \ref{absgrad}) that if $[\sigma]$ is a zero dimensional $\delta$-chain is some critical submanifold, we have
\begin{equation*}
\bar{\mathrm{gr}}^{\mathbb{Q}}([\sigma])\equiv-\frac{1}{4}\sigma(X)\text{   } \mod 4\Z.
\end{equation*}
This follows from the fact that we can compute $\mathrm{gr}_z([\sigma_0],X,[\sigma])$ by means of the reducible configuration corresponding to a spin connection on the cobordism. If we choose a $\Pin$-equivariant perturbation on $S^3$ this relative grading turns out to be the sum of the (real) index of a perturbed Dirac operator (which is divisible by four as it is quaternionic) and $\iota(X)$. This readily implies that $\alpha(Y,\spin)$ is congruent to the opposite of the Rokhlin invariant modulo two, and the same proof applies to the invariants $\beta(Y,\spin)$ and $\gamma(Y,\spin)$.
\par
Property $(1)$ follows from the definitions of the invariant using the fact that if
\begin{equation*}
\HSb_k(Y,\spin)\stackrel{i_*}{\longrightarrow} \HSt_k(Y,\spin)
\end{equation*}
is injective, then by Poincar\'e duality also
\begin{equation*}
\HSb^{-2-k}(-Y,\spin)\stackrel{p^*}{\longrightarrow} \HSf^{-1-k}(-Y,\spin)
\end{equation*}
is injective (see Proposition \ref{absgrad} for the gradings), hence by the universal coefficient theorem (Proposition \ref{univcoeffspin}) the dual map
\begin{equation*}
\HSf_{-1-k}(-Y,\spin)\stackrel{p_*}{\longrightarrow}\HSb_{-2-k}(-Y,\spin)
\end{equation*}
is surjective. Here we use the alternative point of view from Remark \ref{alternbeta} and the fact that the kernel of $i_*$ is the image of $p_*$.
\par
To prove property $(3)$, we first modify the cobordism by surgery paying attention to the framings so that we obtain a spin four manifold with $b_1(W)$ zero and the same intersection form. We still call this cobordism $W$. For a fixed self-conjugate spin$^c$ structure $\spin$ on $W$ restriction to the given ones on the boundary, we can then consider the induced map
\begin{equation*}
\HSb_{\bullet}(W,\spin): \HSb_{\bullet}(Y_0,\spin_0)\rightarrow \HSb_{\bullet}(Y_1,\spin_1).
\end{equation*}
The claim is that this is an isomorphism of absolute degree $\frac{1}{4}b_2(X)$. The topological assumptions imply that on the cylindrical-end manifold $W^*$ there is a unique anti-self-dual spin$^c$ connection (up to gauge equivalence), the spin connection $A_0$. If we chose the perturbations on the two ends to be small enough, the hypothesis that the cobordism is negative definite implies that we can achieve transversality for the moduli spaces without appealing to $\Pin$-equivariant \textsc{asd}-perturbations, so that we are in the same setting of the proof of Theorem $39.1.4$ in the book. Given on each end a critical submanifold $[\Cr_0]$ and $[\Cr_1]$ lying over the unique reducible solution, it follows from the grading formulas and the fact that $\sigma(X)=-b_2(X)$ that the moduli space $M^+([\Cr^0], W^*,[\Cr^1])$ is two dimensional if and only if they have relative grading
\begin{equation*}
\mathrm{gr}_z([\Cr^0],W,[\Cr^1])=\frac{1}{4}b_2(X)-2.
\end{equation*}
It is important to remark here that given any critical manifold $[\Cr_0]$ there is \textit{exactly one} critical manifold $[\Cr^1]$ such that the above relation holds. This follows from the fact that the grading shifts of the reducible critical submanifolds is determined by the Rokhlin invariants, as discussed above. In the case this holds, the moduli space $M^+([\Cr^0], W^*,[\Cr^1])$ is a copy of $\mathbb{C}P^1$ on which $\jmath$ acts as the antipodal map. Furthermore both evaluation maps $\ev_{\pm}$ are diffeomorphisms because they are $\jmath$ equivariant hence they cannot be constant (cfr. the proof of Proposition \ref{positivity}). This proves our claim as all other maps are zero because of the usual dimensional argument. Finally, this implies that the commutative diagram of $\Rin$-modules
\begin{center}
\begin{tikzpicture}
  \matrix (m) [matrix of math nodes,row sep=2em,column sep=5em,minimum width=2em]
  {
  \HSb_{\bullet}(Y_0,\spin_0) & \HSb_{\bullet}(Y_1,\spin_1)\\
    \HSt_{\bullet}(Y_0,\spin_0) & \HSt_{\bullet}(Y_1,\spin_1)\\
   };
  \path[-stealth]
  
  (m-1-1) edge node [above] {$\HSb_{\bullet}(W,\spin)$}  (m-1-2)
    edge node [left] {$i_*$} (m-2-1)
 (m-2-1) edge node [above] {$\HSt_{\bullet}(W,\spin)$} (m-2-2)
 (m-1-2) edge node [left] {$i_*$} (m-2-2)
    ;
   
\end{tikzpicture}
\end{center}
is identified with the commutative diagram of ${\Rin}$-modules
\begin{center}
\begin{tikzpicture}
  \matrix (m) [matrix of math nodes,row sep=3em,column sep=4em,minimum width=2em]
  {
  \mathcal{M}\{\epsilon_0\} & \mathcal{M}\{\epsilon_1\}\\
    \mathcal{M}_{\alpha_0,\beta_0,\gamma_0}& \mathcal{M}_{\alpha_1,\beta_1,\gamma_1}\\
   };
  \path[-stealth]
  
  (m-1-1) edge node [above] {$x$}(m-1-2)
 (m-1-1)   [->>]edge  (m-2-1);
 \path(m-2-1) [->]edge node [above] {$x$} (m-2-2);
 \path(m-1-2) [->>]edge (m-2-2)
    ;
   
\end{tikzpicture}
\end{center}
where the $\epsilon_i$ indicate suitable grading shifts, the $x$ on the top row is an isomorphism of degree $\frac{1}{4}b_2(X)$, and the groups in the bottom rows are the images of the respective maps $i_*$. In particular the vertical arrows are surjective, and Property ($3$) follows. Finally, the invariance under homology cobordism invariance follows because if $W$ is a homology cobordism from $Y_0$ to $Y_1$ then $\bar{W}$ is a homology cobordism from $Y_1$ to $Y_0$.
\end{proof}

\bibliographystyle{alpha}
\bibliography{biblio}

\begin{thebibliography}{KMOS07}

\bibitem[AB95]{AB}
D.~M. Austin and P.~J. Braam.
\newblock Morse-{B}ott theory and equivariant cohomology.
\newblock In {\em The {F}loer memorial volume}, volume 133 of {\em Progr.
  Math.}, pages 123--183. Birkh\"auser, Basel, 1995.

\bibitem[AM90]{AM}
Selman Akbulut and John~D. McCarthy.
\newblock {\em Casson's invariant for oriented homology {$3$}-spheres},
  volume~36 of {\em Mathematical Notes}.
\newblock Princeton University Press, Princeton, NJ, 1990.
\newblock An exposition.

\bibitem[Blo10]{Blo1}
Jonathan Bloom.
\newblock A link surgery spectral sequence in monopole {F}loer homology.
\newblock {\em Advances in Mathematics}, doi:10.1016/j.aim.2010.10.014, 2010.

\bibitem[Blo13]{Blo2}
Jonathan Bloom.
\newblock The combinatorics of {M}orse theory with boundary.
\newblock {\em Proceedings of 19th Gokova Geometry-Topology Conference}, pages
  44--88, 2013.

\bibitem[Don02]{Don}
S.~K. Donaldson.
\newblock {\em Floer homology groups in {Y}ang-{M}ills theory}, volume 147 of
  {\em Cambridge Tracts in Mathematics}.
\newblock Cambridge University Press, Cambridge, 2002.
\newblock With the assistance of M. Furuta and D. Kotschick.

\bibitem[Fr{\o}10]{Fro}
Kim~A. Fr{\o}yshov.
\newblock Monopole {F}loer homology for rational homology 3-spheres.
\newblock {\em Duke Math. J.}, 155(3):519--576, 2010.

\bibitem[FS90]{FS1}
Ronald Fintushel and Ronald~J. Stern.
\newblock Instanton homology of {S}eifert fibred homology three spheres.
\newblock {\em Proc. London Math. Soc. (3)}, 61(1):109--137, 1990.

\bibitem[FS92]{FS}
Ronald Fintushel and Ronald~J. Stern.
\newblock Integer graded instanton homology groups for homology three-spheres.
\newblock {\em Topology}, 31(3):589--604, 1992.

\bibitem[Fuk96]{Fuk}
Kenji Fukaya.
\newblock Floer homology of connected sum of homology {$3$}-spheres.
\newblock {\em Topology}, 35(1):89--136, 1996.

\bibitem[Fur90]{Fur1}
Mikio Furuta.
\newblock Homology cobordism group of homology {$3$}-spheres.
\newblock {\em Invent. Math.}, 100(2):339--355, 1990.

\bibitem[GS80]{Gal}
David~E. Galewski and Ronald~J. Stern.
\newblock Classification of simplicial triangulations of topological manifolds.
\newblock {\em Ann. of Math. (2)}, 111(1):1--34, 1980.

\bibitem[Hut]{Hut}
Michael Hutchings.
\newblock Lectures on {M}orse homology (with an eye towards {F}loer theory and
  pseudoholomorphic curves.
\newblock {\em Preprint}.

\bibitem[Ker69]{Ker}
Michel~A. Kervaire.
\newblock Smooth homology spheres and their fundamental groups.
\newblock {\em Trans. Amer. Math. Soc.}, 144:67--72, 1969.

\bibitem[KM07]{KM}
Peter Kronheimer and Tomasz Mrowka.
\newblock {\em Monopoles and three-manifolds}, volume~10 of {\em New
  Mathematical Monographs}.
\newblock Cambridge University Press, Cambridge, 2007.

\bibitem[KMOS07]{KMOS}
P.~Kronheimer, T.~Mrowka, P.~Ozsv{\'a}th, and Z.~Szab{\'o}.
\newblock Monopoles and lens space surgeries.
\newblock {\em Ann. of Math. (2)}, 165(2):457--546, 2007.

\bibitem[Kne26]{Kne}
Hellmuth Kneser.
\newblock {D}ie {T}opologie der {M}anningfaltigkeiten.
\newblock {\em Jahresbericht der Deut. Math. Verein.}, 34:1--13, 1926.

\bibitem[KS77]{KS}
Robion~C. Kirby and Laurence~C. Siebenmann.
\newblock {\em Foundational essays on topological manifolds, smoothings, and
  triangulations}.
\newblock Princeton University Press, Princeton, N.J.; University of Tokyo
  Press, Tokyo, 1977.
\newblock With notes by John Milnor and Michael Atiyah, Annals of Mathematics
  Studies, No. 88.

\bibitem[Lin15]{Lin}
Francesco Lin.
\newblock The surgery exact triangle in {P}in(2)-monopole {F}loer homology.
\newblock {\em To appear}, 2015.

\bibitem[Lip14]{Lip}
Max Lipyanskiy.
\newblock Geometric homology.
\newblock {\em Preprint}, arXiv:math/1303.2354, 2014.

\bibitem[Man03]{Man1}
Ciprian Manolescu.
\newblock Seiberg-{W}itten-{F}loer stable homotopy type of three-manifolds with
  {$b_1=0$}.
\newblock {\em Geom. Topol.}, 7:889--932 (electronic), 2003.

\bibitem[Man13a]{Man2}
Ciprian Manolescu.
\newblock {P}in(2)-equivariant {S}eiberg-{W}itten {F}loer homology and the
  {T}riangulation {C}onjecture.
\newblock {\em preprint}, arXiv:math/1303.2354, 2013.

\bibitem[Man13b]{Man3}
Ciprian Manolescu.
\newblock {T}he {C}onley index, gauge theory, and triangulations.
\newblock {\em preprint}, arXiv:math/1308.6366, 2013.

\bibitem[Mat78]{Mat}
Takao Matumoto.
\newblock Triangulation of manifolds.
\newblock In {\em Algebraic and geometric topology ({P}roc. {S}ympos. {P}ure
  {M}ath., {S}tanford {U}niv., {S}tanford, {C}alif., 1976), {P}art 2}, Proc.
  Sympos. Pure Math., XXXII, pages 3--6. Amer. Math. Soc., Providence, R.I.,
  1978.

\bibitem[MMR94]{MMR}
John~W. Morgan, Tomasz Mrowka, and Daniel Ruberman.
\newblock {\em The {$L^2$}-moduli space and a vanishing theorem for {D}onaldson
  polynomial invariants}.
\newblock Monographs in Geometry and Topology, II. International Press,
  Cambridge, MA, 1994.

\bibitem[Moi52]{Moi}
Edwin~E. Moise.
\newblock Affine structures in {$3$}-manifolds. {V}. {T}he triangulation
  theorem and {H}auptvermutung.
\newblock {\em Ann. of Math. (2)}, 56:96--114, 1952.

\bibitem[Mor96]{Mor}
John~W. Morgan.
\newblock {\em The {S}eiberg-{W}itten equations and applications to the
  topology of smooth four-manifolds}, volume~44 of {\em Mathematical Notes}.
\newblock Princeton University Press, Princeton, NJ, 1996.

\bibitem[MOY97]{MOY}
Tomasz Mrowka, Peter Ozsv{\'a}th, and Baozhen Yu.
\newblock Seiberg-{W}itten monopoles on {S}eifert fibered spaces.
\newblock {\em Comm. Anal. Geom.}, 5(4):685--791, 1997.

\bibitem[MW01]{MW}
Matilde Marcolli and Bai-Ling Wang.
\newblock Equivariant {S}eiberg-{W}itten {F}loer homology.
\newblock {\em Comm. Anal. Geom.}, 9(3):451--639, 2001.

\bibitem[OS03]{OSd}
Peter Ozsv{\'a}th and Zolt{\'a}n Szab{\'o}.
\newblock Absolutely graded {F}loer homologies and intersection forms for
  four-manifolds with boundary.
\newblock {\em Adv. Math.}, 173(2):179--261, 2003.

\bibitem[Rad25]{Rad}
Tibor Rad\'o.
\newblock \"{U}ber den {B}egriff der {R}iemannschen {F}l\"{a}che.
\newblock {\em Acta Sci. Math. (Szeged)}, 2:101--121, 1925.

\bibitem[Sav02]{Sav}
Nikolai Saveliev.
\newblock {\em Invariants for homology {$3$}-spheres}, volume 140 of {\em
  Encyclopaedia of Mathematical Sciences}.
\newblock Springer-Verlag, Berlin, 2002.
\newblock Low-Dimensional Topology, I.

\bibitem[Sch93]{Sch}
Matthias Schwarz.
\newblock {\em Morse homology}, volume 111 of {\em Progress in Mathematics}.
\newblock Birkh\"auser Verlag, Basel, 1993.

\end{thebibliography}

\end{document}